\documentclass[a4paper,11pt]{article}
\usepackage[margin=2.5cm]{geometry}
\usepackage{graphicx}
\usepackage[utf8]{inputenc}

\usepackage{amsmath}
\usepackage{amsfonts}
\usepackage{amssymb}
\usepackage{amsthm}
\newtheorem{theorem}{Theorem}
\newtheorem{corollary}{Corollary}
\newtheorem{lemma}{Lemma}
\newtheorem{remark}{Remark}
\numberwithin{equation}{section}
\numberwithin{theorem}{section}
\numberwithin{lemma}{section}
\numberwithin{remark}{section}
\numberwithin{corollary}{section}

\usepackage{todonotes}
\usepackage{hyperref}
\usepackage{bm}
\usepackage{bbm}
\usepackage{enumitem}
\usepackage{overpic}
\usepackage{algpseudocode}
\usepackage[ruled, lined, longend, linesnumbered]{algorithm2e}
\usepackage{hhline}
\usepackage{tabto}

\usepackage{subfig}
\usepackage{overpic}
\usepackage{pgfplots,pgfplotstable}
\pgfplotsset{compat=newest}
\usepackage{tikz}
\usetikzlibrary{shapes,arrows,snakes,mindmap,patterns}
\tikzstyle{decision} = [diamond, draw, fill=blue!20, text badly centered, node distance=3cm, inner sep=0pt]
\tikzstyle{sdecision} = [diamond, draw, fill=blue!20, text centered, node distance=3cm, inner sep=0pt, text width=3cm]

\tikzstyle{block} = [rectangle, draw, fill=blue!20, text centered, rounded corners]
\tikzstyle{sblock} = [rectangle, draw, fill=blue!20, text centered, rounded corners, text width = 6cm]

\tikzstyle{line} = [draw, -latex']
\tikzstyle{cloud} = [draw, ellipse,fill=red!20, node distance=3cm,
    minimum height=2em]
\tikzstyle{scloud} = [draw, ellipse,fill=red!20, node distance=3cm,
    minimum height=2em,text width = 5 cm]

\definecolor{blau0} {RGB}{ 131 198 216} 
\definecolor{grau}  {RGB}{  0  84 159}
\definecolor{rot}   {RGB}{204   7  30}
\definecolor{blau2} {RGB}{  0  61 128}
\definecolor{grun2} {RGB}{  0 85   0}
\definecolor{rot2}  {RGB}{120   7  30}
\definecolor{gelb}  {RGB}{70 70 70}

\definecolor{blau}{RGB}{0 144 188}
\definecolor{newblue1}{RGB}{0 144 188}
\definecolor{newblue2}{RGB}{197 216 227}
\definecolor{newgreen1}{RGB}{0 144 118}
\definecolor{grun}{RGB}{255 137 0}
\definecolor{newgreen2}{RGB}{197 222 215}
\definecolor{neworange1}{RGB}{255 137 0}
\definecolor{neworange2}{RGB}{255 205 105}

\renewcommand{\vec}{\bm}
\newcommand{\tensor}{\bm}

\newcommand{\x}{x}
\newcommand{\y}{y}

\newcommand{\stress}{\tensor{\sigma}}
\newcommand{\generalizedstress}{\tensor{\Sigma}}
\newcommand{\testgeneralizedstress}{\tensor{T}}

\newcommand{\g}{\vec{g}}
\newcommand{\f}{\vec{f}}
\renewcommand{\u}{\vec{u}}
\newcommand{\U}{\vec{U}}

\renewcommand{\v}{\vec{v}}

\newcommand{\w}{\vec{w}}
\newcommand{\flux}{\vec{q}}

\newcommand{\entropyflux}{\vec{j}}
\newcommand{\entropyfluxint}{\vec{j}_{\int}}
\newcommand{\dotentropyfluxint}{\dot{\vec{j}}_{\int}}

\newcommand{\z}{\vec{z}}
\newcommand{\n}{\vec{n}}
\newcommand{\q}{q}

\newcommand{\viscostrain}{\tensor{\varepsilon}_\mathrm{v}}
\newcommand{\viscostrainhydrostatic}{\varepsilon_\mathrm{v}^\mathrm{h}}
\newcommand{\viscostraindeviatoric}{\tensor{\varepsilon}_\mathrm{v}^\mathrm{d}}
\newcommand{\zeroviscostrain}{\tensor{\varepsilon}_\mathrm{v,0}}
\newcommand{\dotviscostrain}{\dot{\tensor{\varepsilon}}_\mathrm{v}}
\newcommand{\dotviscostrainhydrostatic}{\dot{\varepsilon}_\mathrm{v}^\mathrm{h}}
\newcommand{\dotviscostraindeviatoric}{\dot{\tensor{\varepsilon}}_\mathrm{v}^\mathrm{d}}
\newcommand{\testviscostrain}{\tensor{t}}
\newcommand{\epshydrostatic}[1]{\varepsilon^\mathrm{h}\!\left(#1\right)}
\newcommand{\epsdeviatoric}[1]{\tensor{\varepsilon}^\mathrm{d}\!\left(#1\right)}
\newcommand{\masssource}{q_\fluidmass}
\newcommand{\entropysource}{q_S}
\newcommand{\intmasssource}{Q_\fluidmass}
\newcommand{\intentropysource}{Q_S}

\newcommand{\erroru}{\bm{e}_{\u}}
\newcommand{\errorviscostrain}{\bm{e}_{\viscostrain}}
\newcommand{\errorflux}{\bm{e}_{\flux}}
\newcommand{\errorentropyflux}{\bm{e}_{\entropyflux}}
\newcommand{\errorstress}{\bm{e}_{\stress}}
\newcommand{\errorpressure}{e_p}
\newcommand{\errortemperature}{e_T}

\newcommand{\NABLA}{\vec{\nabla}}
\newcommand{\DIV}{\NABLA\cdot}
\newcommand{\GRAD}{\NABLA}
\newcommand{\eps}[1]{\tensor{\varepsilon}\!\left(#1\right)\hspace{-0.025cm}}

\newcommand{\conductivity}{\tensor{\kappa}_\mathrm{F}}

\newcommand{\permeability}{\tensor{\kappa}}

\newcommand{\fluidmass}{\theta}

\newcommand{\fext}{\f_{\mathrm{ext}}}
\newcommand{\dotfext}{\dot{\f}_\mathrm{ext}}
\newcommand{\gext}{\g_{\mathrm{ext}}}

\newcommand{\alphamatrix}{\mbox{$\alpha$}}

\newcommand{\llangle}{\left\langle}
\newcommand{\rrangle}{\right\rangle}

\def\mathcolor#1#{\@mathcolor{#1}}
\def\@mathcolor#1#2#3{%
  \protect\leavevmode
  \begingroup
    \color#1{#2}#3%
  \endgroup
}

\begin{document}

\title{The gradient flow structures of thermo-poro-visco-elastic processes in porous media}

\author{Jakub W.\ Both\thanks{Department of Mathematics, University of Bergen, Bergen, Norway; $\{$\texttt{erlend.storvik@uib.no, jakub.both@uib.no, jan.nordbotten@uib.no, florin.radu@uib.no}$\}$} \and
       Kundan Kumar\thanks{Department of Mathematics and Computer Science, Karlstad University, Karlstad, Sweden; $\{$\texttt{kundan.kumar@kau.se}$\}$}  \and 
       Jan M.\ Nordbotten\footnotemark[1] \and
       Florin A.\ Radu\footnotemark[1]
}

\date{}


\maketitle

\begin{abstract}
In this paper, the inherent gradient flow structures of thermo-poro-visco-elastic processes in porous media are examined for the first time. In the first part, a modelling framework is introduced aiming for describing such processes as generalized gradient flows requiring choices of physical states, corresponding energies, dissipation potentials and external work rates. It is demonstrated that various existing models can be in fact written within this framework. Ultimately, the particular structure allows for a unified well-posedness analysis performed for different classes of linear and non-linear models. In the second part, the gradient flow structures are utilized for constructing efficient discrete approximation schemes for thermo-poro-visco-elasticity -- in particular robust, physical splitting schemes. Applying alternating minimization to naturally arising minimization formulations of (semi-)discrete models is proposed. For such, the energy decrease per iteration is quantified by applying abstract convergence theory only utilizing convexity and Lipschitz continuity properties of the problem -- a fairly simple but flexible machinery. By this approach, e.g., the widely used undrained and fixed-stress splits for the linear Biot equations are derived and analyzed. By application of the framework to more advanced models, novel splitting schemes with guaranteed theoretical convergence rates are naturally derived. Moreover, based on the minimization character of the (semi-)discrete equations, relaxation of splitting schemes by line search is proposed; numerical results show a potentially great impact on the acceleration of splitting schemes for both linear and nonlinear problems.\end{abstract}


\section{Introduction}


Gradient flows describe the evolution of purely dissipative systems. Given an initial state $x_0$, a state $x$ evolves along the negative gradient of an energy $\mathcal{E}$ under the influence of an external force $f_\mathrm{ext}$, i.e., 
\begin{align}\label{introduction:classical-gradient-flow}
 \dot{x} + \GRAD \mathcal{E}(x) = f_\mathrm{ext},\quad \text{a.e.\ in }(0,T),\quad x(0)=x_0,
\end{align}
where $\dot{x}$ denotes the temporal derivative of $x$ and $\GRAD \mathcal{E}$ denotes the G$\hat{\text{a}}$teux-derivative of $\mathcal{E}$ wrt. $x$.

The formal gradient flow structure~\eqref{introduction:classical-gradient-flow} is ubiquitous in a broad set of applications and has been therefore of great research interest since the fundamental works by Komura~\cite{Komura1967}, Crandall and Pazy~\cite{Crandall1969} and Brezis~\cite{Brezis1971,Brezis1973}. Meanwhile, gradient flows have been studied in Hilbert spaces and metric spaces~\cite{Ambrosio2008}; in particular, since the seminal work by Otto~\cite{Otto2001}, much attraction has been paid to gradient flows in probability spaces endowed with the Wasserstein metric. It is not our intention to review the vast literature on the topic; we mention a small fragment of the long list of applications with an inherent gradient structure: heat conduction, the Stefan problem, Hele-Shaw cell, flow in porous media, parabolic variational inequalities, degenerate and quasi-linear parabolic PDEs, and transport.

Classical gradient flows are limited to dissipation mechanisms induced by a quadratic potential, which is quite restrictive for many practical situations. Far more systems can be modelled using the notion of \textit{generalized gradient flows} as, e.g., described by Peletier~\cite{Peletier2014}. Those allow in particular for non-quadratic dissipation potentials, including those which are vanishing, not finite, positively homogeneous of degree 1 or state-dependent. Additionally, generalized gradient flows allow for relating the tangent space of the state space with a process space. In this perspective, generalized gradient flows are formally defined by five components:
\begin{enumerate} 
 \item A state space $\mathcal{X}$.
 \item A process space $\mathcal{P}_{\dot{\mathcal{X}}}$ together with an instruction how states change $\dot{x} = \mathcal{T}(x)p$, where $x\in\mathcal{X}$, $p\in\mathcal{P}_{\dot{\mathcal{X}}}$, and $\mathcal{T}(x)$ a transformation operator.
 \item An (internal) free energy $\mathcal{E}(x)$ for states $x\in\mathcal{X}$.
 \item An (external) work rate $\mathcal{P}_\mathrm{ext}(x;p)$ in terms of the process vectors.
 \item A dissipation potential $\mathcal{D}(x;p)$ in terms of process vectors inducing the cost of the change of state.
\end{enumerate}
Then for given state $x\in\mathcal{X}$, the current change of state $\dot{x}$, in terms of the corresponding process vector $p\in\mathcal{P}_{\dot{\mathcal{X}}}$, is defined by
\begin{align}\label{generalized-gradient-flow-structure}
\begin{split}
 \dot{x} &= \mathcal{T}(x)p \\
 p &= \underset{q\in\mathcal{P}_{\dot{\mathcal{X}}}}{\text{arg\,min}}\, \Big\{ \mathcal{D}(x;q) + \llangle \GRAD\mathcal{E}(x), \mathcal{T}(x)q \rrangle - \mathcal{P}_\mathrm{ext}(x;q)  \Big\},
\end{split}
\end{align}
i.e., the loss of energy is maximized along the steepest descent of the energy under minimum cost. Again, many applications can be modelled by generalized gradient flows. We mention incompressible, immiscible two-phase flow in porous media~\cite{Cances2015}, doubly non-linear Allen-Cahn equations~\cite{Mielke2013}, rate-independent finite elasticity~\cite{Mielke2013}, rate-dependent visco-plasticity at finite strain~\cite{Mielke2018}.

Apart from the structure itself, a (generalized) gradient flow interpretation may be beneficial in many ways. A wide range of abstract theory for gradient systems has been established dealing, e.g., with the well-posedness analysis~\cite{Brezis1971,Brezis1973,Colli1992}, \textit{a posteriori} error analysis for time discretizations~\cite{Nochetto2000}, or \textit{a priori} error analysis for numerical discretizations in time and space~\cite{Bartels2014}. Furthermore, energy preserving time discretizations can be constructed~\cite{Jungel2018}, and optimization algorithms can be utilized for the construction of robust numerical solvers.


In this work, for the first time, we explore the gradient structure in the  consolidation of fluid-saturated porous media, also called \textit{theory of poro-elasticity}, and provide a generalized gradient flow formulation~\eqref{generalized-gradient-flow-structure} for various poro-elasticity models. Coupled thermo-hydro-mechanical-chemical processes in porous media have been of great research interest recently, due to the presence of many practical applications of societal and industrial relevance. We mention not only classical, geotechnical applications within soil and reservoir mechanics, but also geothermal reservoirs~\cite{Segall1998}, CO$_2$ storage~\cite{Bjornara2016}, deformation of hydrogels~\cite{Hong2008} or biomechanical applications~\cite{Cowin1999} among others.  
 
The theory of poro-elasticity goes back to the early seminal contributions by Terzaghi~\cite{Terzaghi1936} and Biot~\cite{Biot1941}; 
since then many mathematical models for thermo-hydro-mechanical-chemical processes in porous media have been established utilizing, e.g.,  averaging processes~\cite{Lewis1998}, thermodynamic arguments~\cite{Coussy2004}, or homogenization~\cite{Burridge1981,Mikelic2012,Brun2018,Vanduijn2017}. Traditionally, corresponding models are formulated as partial differential equations (PDE). Based on those formulations, there exists a mature literature on both analytical and numerical, rigorous mathematical theory for specific poro-elasticity models. It is beyond the scope of this work to give a comprehensive review; we only point out results connected to this paper: For the linear Biot model well-posedness has been showed using semigroup theory~\cite{Showalter2000}. Recent advances on extensions of the linear Biot equations include well-posedness for the dynamic poro-elasticity~\cite{Mikelic2012}, thermo-poro-elasticity with non-linear, thermal convection~\cite{Brun2018b}, poro-visco-elasticity with a purely visco-elastic strain~\cite{Showalter2000,Bociu2016}, and linear poro-elasticity with a deformation-dependent, non-linear permeability~\cite{Bociu2016}. We are not aware of any explicit result on the well-posedness of linear poro-visco-elasticity models, which consider strains composed of an elastic and a visco-elastic contribution as modelled by~\cite{Coussy2004}, or non-linear poro-elasticity under an infinitesimal strain assumption as studied in a fully discretized setting by~\cite{Borregales2018}. In terms of numerical discretization and solution of the linear Biot model, stable, spatial discretizations for various choices of primary variables have been introduced~\cite{Haga2012,Keilegavlen2017,Ambartsumyan2018b,Korsawe2005}. Furthermore, physically motivated, robust operator splittings have been of great, recent interest, allowing for either using independent, tailored simulators for different physics or developing good block preconditioners for monolithic Krylov subspace methods. Such have been developed and studied for in particular the linear Biot model~\cite{Kim2011,Haga2012b,Mikelic2013,Almani2016,Castelletto2016,White2016,Both2017,Bause2017,Adler2017,Castelletto2019}, for non-linear poro-elasticity under an infinitesimal strain assumption~\cite{Borregales2018}, thermo-poro-elasticity~\cite{Kim2018a} and large strain poro-elasticity~\cite{Kim2018b}. 

To our knowledge, the connection between gradient flows and poro-elasticity from a mathematical point of view has not yet been studied in the literature. However, we have to honor the work by Miehe~\cite{Miehe2015}, which has also been an inspiration for this paper. In the aforementioned work with a focus on general modelling, isothermal flow in fully-saturated poro-elastic media under large strains is formulated using minimization principles, which eventually can be identified as a generalized gradient flow~\eqref{generalized-gradient-flow-structure}. The authors have furthermore noted that the minimization structure allows arbitrary pairs of finite elements as spatial discretization of the coupled problem.


The aim of this paper is not only to reveal a natural gradient structure of thermo-poro-visco-elasticity, but also to discuss how to exploit this structure to study well-posedness and naturally develop numerical methods. Serving as proof of concept, we explore thoroughly the linear Biot equations: We highlight the gradient structure of the linear Biot equations; well-posedness results are deduced employing abstract theory for doubly non-linear evolution equations and convex analysis; additionally, we identify widely used splitting schemes~\cite{Kim2011,Kim2011b} as alternating minimization, which are \textit{a priori} guaranteed to converge. Utilizing abstract convergence theory, we are able to prove the same convergence rates as previously reported in the literature~\cite{Mikelic2013}, in which problem-specific proofs are performed. In addition, \textit{a posteriori} error bounds are provided, similar to Ostrowski-type bounds~\cite{Kumar2018}. We further apply the same workflow to more advanced poro-elasticity models with increased complexity and derive novel robust splitting schemes with guaranteed theoretical convergence rates. The findings are presented in two parts.

Part~I (Sec.~\ref{section:frameworks}--\ref{section:thermo-poro-elasticity-gradient-flow}) is concerned with two aspects: (i) The modelling of coupled processes in  poro-elastic materials as generalized gradient flows, and (ii) a subsequent well-posedness analysis. By combining the abstract generalized gradient flow formulation~\eqref{generalized-gradient-flow-structure} with conceptual considerations regarding poro-elasticity, an abstract modelling framework for thermo-poro-visco-elasticity is established in Sec.~\ref{section:frameworks}. In its most general form, it allows for non-isothermal, (non-)Darcy flow in a saturated, non-linearly poro-visco-elastic material governed by dissipation only. Specific models are then obtained by involving common thermodynamic knowledge on free energies and dissipation potentials: A gradient flow formulation is derived for linear poro-elasticity (Sec.~\ref{section:linear-biot-gradient-flow}), linear poro-visco-elasticity (Sec.~\ref{section:poro-visco-elasticity-wellposedness}), non-linear poro-elasticity in the infinitesimal strain regime (Sec.~\ref{section:non-linear-biot-linear-coupling}), non-Newtonian Darcy and non-Darcy flows in poro-elastic media (Sec.~\ref{section:non-newtonian-fluids}), and thermo-poro-elasticity without thermal convection (Sec.~\ref{section:thermo-poro-elasticity-gradient-flow}), all consistent with previously employed PDE-based models~\cite{Coussy2004}. Regarding the well-posedness analysis for poro-elasticity models, the main difficulty is the characteristic fact that the dissipation potential does not depend on all process vectors, as e.g., the change in mechanical displacement; this is solved by combining an abstract decoupling approach~\cite{Rossi2006} with classical convex analysis~\cite{Ekeland1999} and theory on doubly non-linear evolution equations~\cite{Colli1990,Mielke2013}, tailored to our needs, cf.\ Appendix~\ref{appendix:classic-theorems}. It is summarized in a unified well-posedness result, Thm.~\ref{thm:abstract-well-posedness-of-dne}, the main theoretical result of Part~I. Furthermore, it is applied to practically all models listed above; in particular it gives a new concise proof for the well-posedness of the linear Biot equations.
  
Part~II (Sec.~\ref{section:general-time-discretization}--\ref{section:numerical-results}) deals with the robust, numerical solution of aforementioned thermo-poro-visco-elasticity models, by exploiting the generalized gradient flow structure discussed in Part~I. More precisely, after a semi-implicit time discretization along the lines of the minimizing movement scheme for gradient flows~\cite{Ambrosio1995}, the generalized gradient flow formulation translates into a minimization problem. For models discussed in Sec.~\ref{section:linear-biot-gradient-flow}--\ref{section:thermo-poro-elasticity-gradient-flow}, the minimization problem is convex, enabling the vast literature on convex optimization for efficient numerical solution, i.e., a problem non-specific machinery. Motivated by the recent advances on splitting schemes in the community, we discuss in particular the application of the plain alternating minimization or cyclic block coordinate descent methods~\cite{Bertsekas1999,Luo1992}. They allow for natural decoupling of the entire problem into its physical subproblems.
Additionally, guaranteed convergence follows directly from abstract optimization theory, adjusted to our needs, cf.\ Appendix~\ref{appendix:section:alternating-minimization}.
By this, we provide a new perspective on widely used, physically motivated splitting schemes, as the undrained and fixed-stress splits~\cite{Kim2011,Kim2011b} for linear poro-elasticity (Sec.~\ref{section:splitting-linear-biot}). In particular, we provide a simple, mathematical intuition why those schemes are natural choices among predictor-corrector methods for which physical variables are simply fixed in the predictor step -- in contrast for example to the drained and fixed-strain splits which are only conditionally stable~\cite{Kim2011,Kim2011b}. In addition, by applying the unified approach, we derive novel, robust splitting schemes for linear poro-visco-elasticity (Sec.~\ref{section:splitting-poro-visco-elasticity}) and nonlinear poro-elasticity under infinitesimal strains (Sec.~\ref{section:splitting-non-linear-poro}), and provide a theoretical basis for the undrained-adiabatic and extended fixed-stress splits~\cite{Kim2018a} for thermo-poro-elasticity (Sec.~\ref{section:splitting-thermo-poro-elasticity}). This annexes the mathematically intuitive interpretation of directional minimization to the physical motivation of the splitting schemes. Finally, the minimization formulation allows for acceleration of the previously discussed splitting schemes, using a line search relaxation strategy (Sec.~\ref{section:minimization:line-search}). In the context of poro-elasticity, this has not yet been observed in the literature. In particular, for linear problems, exact line search can be performed cheaply using quadratic interpolation due to the quadratic nature of the time-discrete minimization problem; the same technique is proposed as inexact line search for semi-linear models. We close the second part with a succinct numerical study (Sec.~\ref{section:numerical-results}) aiming for answering four questions: (i) what is the impact of the relaxation of splitting schemes by line search; (ii) how does it relate to the optimization of tuning parameters employed within splitting schemes; (iii) how do relaxed splitting schemes perform for poro-visco-elasticity and (iv) and non-linear poro-elasticity? We observe that applying line search is effectively identical with optimizing splitting schemes, but no \textit{a priori} knowledge or user-interaction is required. Furthermore, splitting schemes for poro-visco-elasticity and non-linear poro-elasticity show similar performance as for linear poro-elasticity.

\subsection{Notation}

Throughout this work, let $\Omega\subset\mathbb{R}^d$, $d\in\{2,3\}$, be an open, connected domain, with Lipschitz boundary $\partial\Omega$ and outward normal $\n$; let $[0,T]$ denote a finite time interval with finite time $T>0$. 

We use the following notation for standard function spaces and their norms~\cite{Adams2003}: Let $L^p(\Omega)$ be the space of functions for which the p-th power of the absolute value is Lebesgue integrable. For $L^2(\Omega)$, let $\llangle \cdot,\cdot\rrangle=\llangle \cdot,\cdot\rrangle_{L^2(\Omega)}$ denote the standard $L^2(\Omega)$ scalar product, $\|\cdot\|=\|\cdot\|_{L^2(\Omega)}$ the associated norm. Let $\llangle \cdot,\cdot\rrangle_\Gamma:=\llangle \cdot,\cdot\rrangle_{L^2(\Gamma)}$ for measurable boundary segments $\Gamma\subset\Omega$. Let $W^{1,p}(\Omega)$, $p\geq 1$,  denote the usual Sobolev space, consisting of functions in $L^p(\Omega)$ with a weak derivative in $L^p(\Omega)$, $H^1(\Omega)=W^{1,2}(\Omega)$ and $H_0^1(\Omega)$ its subspace with zero trace on $\partial\Omega$. Furthermore, $H^p_\mathrm{div}(\Omega)$, $p\geq 1$, denotes vectorial functions with $d$ components in $L^p(\Omega)$ with a weak divergence in $L^2(\Omega)$; and $H(\mathrm{div};\Omega)=H^2_\mathrm{div}(\Omega)$.

We use bold symbols for vectors and tensors. Similarly, we use bold symbols for vector valued function spaces, e.g., $\bm{H}^1(\Omega)$. For elements of $\bm{H}^1(\Omega)$, let $\eps{\u}=\tfrac{1}{2} \left(\GRAD \u + \GRAD \u^\top \right)$ denote the symmetric gradient, also called linearized strain; $\GRAD$ denotes both the spatial gradient and the (partial) functional derivative given by the G$\hat{\text{a}}$teaux-derivative, depending on the context. For $\mathcal{V}$ a Banach space, let $L^p(0,T;\mathcal{V})$ and $H^1(0,T;\mathcal{V})$ denote standard Bochner spaces endowed with standard norms. Newton's notation is used for denoting temporal derivatives of variables, e.g., $\dot{x}$ for the temporal derivative of $x$, whereas partial temporal derivatives of functionals are denoted by $\partial_t$. For $\mathcal{V}$ a Banach space, we denote $\mathcal{V}^\star$ its dual space and $\llangle \cdot,\cdot\rrangle _{\mathcal{V}^\star\times\mathcal{V}}$ a duality pairing. If obvious, we omit the subscript.

Finally, let $|\cdot|$ denote the absolute value, the Euclidean distance and the Frobenius norm for scalars, vectors and second-order tensors, respectively. And let $\mathrm{tr}\,\mathbf{A}=\sum_i A_{ii}$ denote the trace of a quadratic second-order tensor $\mathbf{A}$. The inequality $a\lesssim b$ means there exists a generic constant $C>0$ independent of $a$ and $b$ such that $a \leq C b$.

Let $\otimes$ denote the Kronecker product, and for the special case of two vectors. Moreover, let $:$ denote the single, double or triple (depending on the context) inner product for tensors. For the double inner product of a fourth order and a second order tensor we often omit $:$ as often done in mathematical literature for linear elasticity. Finally, $\llangle \cdot,\cdot\rrangle$ with tensorial arguments of same order is equivalent to a Lebesgue integral over the double inner product of the arguments.

A nomenclature regarding notation for generalized gradient flows, physical fields, function spaces etc. is provided in Appendix~\ref{appendix:notation}.

\section*{Part I -- Modelling and analyzing thermo-poro-visco-elasticity as generalized gradient flow}
\addcontentsline{toc}{section}{Part I -- Modelling and analyzing thermo-poro-visco-elasticity as generalized gradient flow} 

The main objective of part~I is to highlight the inherent gradient  structure of various poro-elasticity models. Secondary, we prove well-posedness for such models. Sec.~\ref{section:frameworks}  lays a foundation for this, providing an abstract gradient flow modelling framework for poro-elasticity, and subsequently an abstract well-posedness result for degenerate, doubly non-linear evolution equations, which will allow for a unified well-posedness analysis of poro-elasticity models. Based on those tools, we discuss linear poro-elasticity (Sec.~\ref{section:linear-biot-gradient-flow}), linear poro-visco-elasticity (Sec.~\ref{section:poro-visco-elasticity-wellposedness}), non-linear poro-elasticity in the infinitesimal strain regime (Sec.~\ref{section:non-linear-biot-linear-coupling}), non-Darcy flows in poro-elastic media (Sec.~\ref{section:non-newtonian-fluids}), and linear thermo-poro-elasticity without thermal convection (Sec.~\ref{section:thermo-poro-elasticity-gradient-flow}).

\section{Foundation for modelling and analyzing poro-elasticity as generalized gradient flow}\label{section:frameworks}

In the following, tools are introduced which will be applied throughout Part~I of the paper. First, in Sec.~\ref{section:general-model-biot}, a general framework for modelling poro-elasticity based on the formal definition of generalized gradient flows~\eqref{generalized-gradient-flow-structure} is proposed. Additionally, in Sec.~\ref{section:general-well-posedness}, an abstract well-posedness result is derived,
which allows for a unified analysis of poro-elasticity in the subsequent sections.

\subsection{Formal modelling framework for non-isothermal flow in poro-visco-elastic media}\label{section:general-model-biot}

From a continuum mechanical perspective, it is fair to assume that fluid-saturated, deformable porous media are purely governed by dissipation. That remains true, when allowing for additional structural visco-elasticity or non-isothermal flow with negligible, thermal convection. Consequently, it is natural to expect that a wide class of poro-elasticity models have an inherent gradient flow structure.
Indeed, by incorporating thermodynamic interpretation into the notion of generalized gradient flows~\eqref{generalized-gradient-flow-structure}, we introduce a general modelling framework for non-isothermal flow in poro-visco-elastic media. 


To set modelling limits, we restrict the discussion to fully-saturated media which deform under an infinitesimal strain assumption. Visco-elastic and thermal effects are allowed. But it is implicitly assumed that the considered system can be formulated as a gradient flow. This cannot always be true, e.g., when thermal convection or non-quasi-static mechanical behavior are non-negligible.

In the following the single components of a generalized gradient flow are defined based on thermodynamic knowledge:

\begin{enumerate}

 \item As state space, we choose
 \begin{align}\label{general-biot-state-space}
  \mathcal{X}=\left\{ \left(\u,\fluidmass,\viscostrain,S\right) \right\}, 
 \end{align}
 where $\u$ is the displacement of the matrix with respect to a reference state $\Omega$; $\fluidmass$ is the change of the fluid mass on $\Omega$ with respect to some reference configuration scaled by the inverse of a reference fluid density; $\viscostrain$ is the visco-elastic strain such that $\eps{\u} - \viscostrain$ denotes the elastic strain; and $S$ is the total entropy. Depending on which processes are considered, we choose only a suitable subset of $\mathcal{X}$ as state space. \\
 
 \item Structural displacements $\u$ and visco-elastic strains $\viscostrain$ change with rates $\dot{\u}$ and $\dotviscostrain$, respectively. Instead of using the rates $\dot{\fluidmass}$ and $\dot{S}$ directly, we associate those with a volumetric flux $\flux$ and an entropy flux $\entropyflux$, respectively. Their relations are imposed by the conservation of mass and balance of entropy
 \begin{alignat}{2}
 \label{general-biot-process-vectors:start}
  \dot{\fluidmass} + \DIV \flux &= \masssource &&\text{ on }\Omega, \\
  \dot{S} + \DIV \entropyflux &= \entropysource &&\text{ on }\Omega,
 \label{general-biot-process-vectors:end}
 \end{alignat}
 where $\masssource$ and $\entropysource$ denote given, time-dependent production terms.
 
 Gradient flows effectively define changes of states, and boundary conditions can be imposed for those on boundary segments $\Gamma_{\u},\Gamma_{\flux},\Gamma_{\entropyflux}\subset\partial\Omega$. We define the function spaces for $t\in[0,T]$ (without specifying regularity for now)
 \begin{align}
 \label{general-function-spaces:start}
  \dot{\mathcal{V}}(t) &= \left\{ \v :\Omega \rightarrow \mathbb{R}^d \,|\, \v = \dot{\u}_{\Gamma}(t) \text{ on }\Gamma_{\u}\right\}, \\
  \mathcal{Z}(t) &= \left\{ \z :\Omega \rightarrow \mathbb{R}^d \,|\, \z\cdot\n = q_\mathrm{\Gamma,n}(t) \text{ on }\Gamma_{\flux} \right\}, \\
  \dot{\mathcal{T}}(t) &= \left\{ \testviscostrain :\Omega \rightarrow \mathbb{R}^{d\times d} \right\}, \\
  \mathcal{W}(t) &= \left\{ \w :\Omega \rightarrow \mathbb{R}^d \,|\, \w\cdot\n = j_\mathrm{\Gamma,n}(t) \text{ on }\Gamma_{\entropyflux} \right\}.
 \label{general-function-spaces:end}
 \end{align} 
 associated with the change of structural displacement, volumetric flux, the change of the visco-elastic strain and entropy flux. Function spaces associated with the states are implicitly defined. Due to its internal character, no boundary conditions are imposed for the change of the visco-elastic strain. We suppress the explicit time-dependence of function spaces and boundary data in the rest of the article; e.g. we write $\dot{\mathcal{V}}$ instead of $\dot{\mathcal{V}}(t)$.\\

 \item For given state $(\u,\fluidmass,\viscostrain,S)$, let the energy $\mathcal{E}$ be given by the Helmholtz free energy of the system. According to thermodynamic derivations~\cite{Coussy2004}, we can derive the total stress $\stress$, the fluid pressure $p$ and the temperature $T$ by
 \begin{align}\label{thermodynamic-interpretation}
  \stress := \partial_{\GRAD \u} \mathcal{E}, \qquad
  p := \partial_{\fluidmass} \mathcal{E}, \qquad
  T := \partial_{S} \mathcal{E}.
 \end{align}
 Those also act as dual variables to $(\u,\fluidmass,S)$, for which complementary boundary conditions to~\eqref{general-function-spaces:start}--\eqref{general-function-spaces:end} have to be prescribed
 \begin{alignat}{3}
 \label{boundary-conditions:dual:start}
  \stress\n &= \stress_\mathrm{\Gamma,n} &\quad& \text{on }\Gamma_{\stress}&&:=\partial\Omega\setminus\Gamma_{\u}, \\
  p &= p_\Gamma && \text{on }\Gamma_p&&:=\partial\Omega\setminus\Gamma_{\flux}, \\
  T &= T_\Gamma && \text{on }\Gamma_T&&:=\partial\Omega\setminus\Gamma_{\entropyflux}.
 \label{boundary-conditions:dual:end}
 \end{alignat}
 
 As common in poro-elasticity, we employ an effective stress approach. We assume therefore, the total energy $\mathcal{E}$ can be decomposed into three contributions
 \begin{align}\label{general-biot-energy}
  \mathcal{E}(\u,\fluidmass,\viscostrain,S) = \mathcal{E}_\mathrm{eff}(\GRAD \u,\viscostrain) + \mathcal{E}_\mathrm{v}(\viscostrain) + \mathcal{E}_\mathrm{fluid}(\GRAD \u,\fluidmass,\viscostrain,S),
 \end{align}
 where the first contribution is assigned to the solid and defines the effective stress and will finally depend only on the elastic strain; the second contribution is the energy stored (and potentially lost) due to inelastic effects; and the third contribution corresponds to the fluid, allowing for defining the fluid quantities. We obtain the effective stress $\stress_\mathrm{eff}$, $p$ and $T$ also from
 \begin{align}\label{thermodynamic-interpretation-2}
  \stress_\mathrm{eff} := \partial_{\GRAD\u} \mathcal{E}_\mathrm{eff}, \qquad
  p = \partial_\fluidmass \mathcal{E}_\mathrm{fluid}, \qquad
  T = \partial_{S} \mathcal{E}_\mathrm{fluid}.
 \end{align}
 
 \item The external work rate $\mathcal{P}_\mathrm{ext}$ acts as a negative potential for changes of state or associated process vectors. Throughout this work, we assume $\mathcal{P}_\mathrm{ext}$ is linear and state-independent, and we allow $\mathcal{P}_\mathrm{ext}$ to vary in time. Furthermore, it is natural to assume the total external work rate decomposes into separate, independent contributions
 \begin{align*}
  \mathcal{P}_\mathrm{ext}(t,\dot{\u},\flux,\dotviscostrain,\entropyflux) &= \mathcal{P}_\mathrm{ext,mech}(t,\dot{\u}) + \mathcal{P}_\mathrm{ext,fluid}(t,\flux) 
  + \mathcal{P}_\mathrm{ext,temp}(t,\entropyflux).
 \end{align*}
 Since the visco-elastic strain is interpreted as an internal variable, no external work rate is associated to $\dotviscostrain$. In the context of poro-elasticity, external work rates integrate external body and surface forces acting on the fluid and the matrix. In particular, surface forces can be identified as the boundary conditions imposed on the dual variables~\eqref{boundary-conditions:dual:start}--\eqref{boundary-conditions:dual:end}. All in all, we employ
 \begin{align*}
  \mathcal{P}_\mathrm{ext,mech}(t,\dot{\u}) \hspace{-1cm} & \hspace{1cm} = \llangle \fext(t), \dot{\u} \rrangle + \llangle \stress_\mathrm{\Gamma,n}(t) , \dot{\u} \rrangle_{\Gamma_{\stress}}, \\
  \mathcal{P}_\mathrm{ext,fluid}(t,\flux) \hspace{-1cm} & \hspace{1cm}= \llangle \gext(t) , \flux \rrangle + \llangle p_\Gamma(t), \flux \cdot \n \rrangle_{\Gamma_p}, \\
  \mathcal{P}_\mathrm{ext,temp}(t,\entropyflux) \hspace{-1cm} & \hspace{1cm}= \llangle T_\Gamma(t), \entropyflux \cdot \n \rrangle_{\Gamma_T}.
 \end{align*}
 Here, $\fext$ and $\gext$ denote external body forces applied to the matrix and the fluid, respectively. We stress that under the hypothesis of small perturbations of the Lagrangian porosity~\cite{Coussy2004}, often coming along with the assumptions of linear elasticity, it is indeed fair to assume that $\fext$ and $\gext$ are state-independent.\\
 
 \item Accounting for viscous dissipation, changes of states come at cost governed by a dissipation potential. For the poro-elasticity models considered in this work, it is adequate to assume that the underlying dissipation mechanisms are state-independent; e.g., for large strain poro-elasticity, this is not the case. Furthermore, we presume independent cost for each process such that the total dissipation potential decomposes into 
 \begin{align}\label{general-biot-dissipation}
  \mathcal{D}(\dot{\u},\flux,\dotviscostrain, \entropyflux) = \mathcal{D}_\mathrm{mech}(\dot{\u}) + \mathcal{D}_\mathrm{fluid}(\flux) + \mathcal{D}_\mathrm{v}(\dotviscostrain) + \mathcal{D}_\mathrm{th}(\entropyflux).
 \end{align}
 A common feature for many poro-elasticity models is to assume that structural displacements react instantaneously, which corresponds to the choice $\mathcal{D}_\mathrm{mech}=0$. The potentials $\mathcal{D}_\mathrm{fluid}$, $\mathcal{D}_\mathrm{v}$ and $\mathcal{D}_\mathrm{th}$ correspond to a (non-)Darcy-type law, some viscosity law for strain rates and a Fourier-type law, respectively. We will essentially consider quadratic dissipation potentials; besides in the context of non-Darcy flows  in poro-elastic materials, cf.\ Sec.~\ref{section:non-newtonian-fluids}.
  
\end{enumerate}

\noindent
Finally,~\eqref{generalized-gradient-flow-structure} yields an abstract model for describing the evolution of a fluid-saturated, deformable porous medium. The states $(\u,\fluidmass,\viscostrain,S)$ change in time $t\in(0,T)$ by
\begin{align}
\label{biot-generalized-gradient-flow-structure:start}
 \dot{\fluidmass} &= \masssource -\DIV\flux,  \\ 
\label{biot-generalized-gradient-flow-structure:mid}
 \dot{S} &= \entropysource -\DIV\entropyflux, \\
\label{biot-generalized-gradient-flow-structure:end}
 (\dot{\u},\flux,\dotviscostrain,\entropyflux) &=  \underset{(\v,\z,\testviscostrain,\w)\in\dot{\mathcal{V}}\times\mathcal{Z}\times\dot{\mathcal{T}}\times\mathcal{W}}{\text{arg\,min}}\, \Big\{ 
 \llangle \partial_{\GRAD \u}\mathcal{E}(\GRAD \u, \fluidmass,\viscostrain,S), \GRAD \v \rrangle - \mathcal{P}_\mathrm{ext,mech}(t,\v)  \\
\nonumber
 &\hspace{1cm}  + \mathcal{D}_\mathrm{fluid}(\z) - \llangle \partial_\fluidmass \mathcal{E}_\mathrm{fluid}(\GRAD \u, \fluidmass, \viscostrain, T), \DIV \z \rrangle - \mathcal{P}_\mathrm{ext,fluid}(t,\z) \\[0.5em]
\nonumber
 &\hspace{1cm}  + \mathcal{D}_\mathrm{v}(\testviscostrain) + \llangle \partial_{\viscostrain} \mathcal{E}(\GRAD \u, \fluidmass,\viscostrain,S), \testviscostrain \rrangle \\
\nonumber
 &\hspace{1cm} + \mathcal{D}_\mathrm{th}(\w) - \llangle \partial_S \mathcal{E}_\mathrm{fluid}(\GRAD \u, \fluidmass, \viscostrain,T), \DIV \w \rrangle - \mathcal{P}_\mathrm{ext,temp}(t,\w)
 \Big\}
\end{align}
and are subject to initial conditions at time $t=0$
\begin{align}
\label{general-biot-initial-conditions}
 \u = \u_0, \quad \ 
 \fluidmass  = \fluidmass_0,  \quad \ 
 \viscostrain = \zeroviscostrain, \quad \ 
 S = S_0 \qquad\quad\text{on }\Omega.
\end{align}

\begin{remark}[Primal and dual formulation]
 We distinguish between primal and dual variables. The gradient flow formulation~\eqref{biot-generalized-gradient-flow-structure:start}--\eqref{biot-generalized-gradient-flow-structure:end} governs primal variables. Hence, we will call this the primal formulation. In certain situations, a dual formulation governing dual variables can be derived from the primal formulation. This is, e.g., discussed for linear poro-elasticity, cf.\ Sec.~\ref{section:dual-formulation:linear-poro-elasticity}. 
\end{remark}

\subsection{Poro-elasticity formulated as doubly non-linear evolution equation}\label{section:abstract-dne}
 
The framework as described in the previous section is suitable for modelling poro-elasticity. Yet, in the next section, we provide tools for a unified well-posedness analysis of models of type~\eqref{biot-generalized-gradient-flow-structure:start}--\eqref{general-biot-initial-conditions}, which utilize the closely related reformulation of a generalized gradient flow as a doubly non-linear evolution equation. A natural reformulation of the general poro-elasticity model~\eqref{biot-generalized-gradient-flow-structure:start}--\eqref{general-biot-initial-conditions} is achieved by introducing accumulated fluxes
\begin{align}
\label{accumulated-darcy-flux}
 \flux_{\int}(t) &:= \int_0^t \flux(\tau) \, d\tau, \\
\label{accumulated-fourier-flux}
 \entropyfluxint(t) &:= \int_0^t \entropyflux(\tau) \, d\tau
\end{align}
as alternative states to $\fluidmass$ and $S$, respectively. Corresponding function spaces $\mathcal{Z}_{\int}$ and $\mathcal{W}_{\int}$ are implicitly defined by $\dot{\mathcal{Z}}_{\int} = \mathcal{Z}$ and $\dot{\mathcal{W}}_{\int} = \mathcal{W}$, i.e., $\flux_{\int} \in \mathcal{Z}_{\int}$ if and only if $\dot{\flux}_{\int}\in \mathcal{Z}$. Analogously, 
we set $\intmasssource(t) := \int_0^t \masssource(s)\, ds$ and $\intentropysource(t) := \int_0^t \entropysource(s)\, ds$. By~\eqref{biot-generalized-gradient-flow-structure:start}--\eqref{biot-generalized-gradient-flow-structure:mid}, the accumulated fluxes are associated with $\fluidmass$ and $S$ by
\begin{align}
\label{accumulated-mass-conservation}
 \fluidmass &= \fluidmass_0 + \intmasssource - \DIV \flux_{\int},\\
\label{accumulated-entropy-conservation}
 S &= S_0 + \intentropysource - \DIV \entropyfluxint.
\end{align}
By eliminating $\fluidmass$ and $S$, the generalized gradient flow formulation~\eqref{biot-generalized-gradient-flow-structure:start}--\eqref{general-biot-initial-conditions} becomes a degenerate doubly non-linear evolution equation for $(\u,\flux_{\int},\viscostrain,\entropyfluxint)\in\mathcal{V}\times \mathcal{Z}_{\int} \times \mathcal{T} \times \mathcal{W}_{\int}$
\begin{align}
\label{dne:biot-ggf}
 \GRAD \mathcal{D}(\dot{\u},\dot{\flux}_{\int},\dotviscostrain,\dotentropyfluxint) + \GRAD\tilde{\mathcal{E}}(t,\u, \flux_{\int}, \viscostrain, \entropyfluxint) = \GRAD\mathcal{P}_\mathrm{ext}(\dot{\u}, \dot{\flux}_{\int},\dotviscostrain,\dotentropyfluxint),
\end{align}
with reinterpreted (potentially, explicitly time-dependent) energy
\begin{align*}
 \tilde{\mathcal{E}}(t, \u, \flux_{\int}, \viscostrain, \entropyfluxint):= \mathcal{E}(\GRAD \u, \fluidmass_0 + \intmasssource(t) - \DIV \flux_{\int}, \viscostrain, S_0 + \intentropysource(t) - \DIV \entropyfluxint).
\end{align*} 
and initial conditions
\begin{align}
\label{dne:general-biot-initial-conditions}
 \u = \u_0, \quad \ 
 \flux_{\int}  = \bm{0},  \quad \ 
 \viscostrain = \zeroviscostrain, \quad \ 
 \entropyfluxint = \bm{0} \qquad\quad\text{on }\Omega.
\end{align}

\begin{remark}[Reformulation over linear spaces]
 The function spaces $\mathcal{V}\times \mathcal{Z}_{\int} \times \mathcal{T} \times \mathcal{W}_{\int}$ are given by linear spaces translated by some essential boundary conditions, cf.~\eqref{general-function-spaces:start}--\eqref{general-function-spaces:end}. By explicitly incorporating the translation into the definitions of $\mathcal{D}$, $\tilde{\mathcal{E}}$ and $\mathcal{P}_\mathrm{ext}$, the problem~\eqref{dne:biot-ggf} can be reformulated over time-independent, linear spaces; however, each functional becomes explicitly time-dependent.
\end{remark}


\begin{remark}[Time-independent dissipation potential and energy functional and linear function spaces]\label{remark:time-independent-energy-dissipation}
 From a modelling perspective, imposing time-dependent, essential boundary conditions is straightforward. However, the analysis of gradient systems under essential boundary conditions is known to be a delicate topic, cf., e.g.,~\cite{Peletier2014}. A model-specific discussion is most often required; e.g., for quadratic potentials and energies, boundary conditions or external sources can be equivalently reformulated as linear contributions of the external work rates, allowing for reducing the discussion to linear, time independent spaces and time-independent energy functionals. Non-homogeneous, time-independent boundary conditions are less of a problem, as the driving functional remains decreasing along solutions.
%
\end{remark}

\subsection{Abstract well-posedness result for degenerate doubly non-linear evolution equations}\label{section:general-well-posedness}

In the following, we establish an abstract well-posedness result which allows for a unified discussion of poro-elasticity models arising from the gradient flow modelling framework introduced above, cf.\ Sec.~\ref{section:linear-biot-gradient-flow}--\ref{section:thermo-poro-elasticity-gradient-flow}. For this, we first note that the problem~\eqref{dne:biot-ggf} falls into the category of degenerate, doubly non-linear evolution equations on Banach spaces. More specifically, the structural assumptions made in Sec.~\ref{section:general-model-biot}, and assuming solely external work rates are time-dependent, cf.\ Rem.~\ref{remark:time-independent-energy-dissipation}, motivates to consider the abstract evolutionary system
\begin{align}
\label{general-wellposedness:gradient-flow-system}
 \left(\dot{\x}_1,\dot{\x}_2\right) \,=\, \underset{(\y_1,\y_2)\in\mathcal{V}_1 \times\mathcal{V}_2}{\mathrm{arg\,min}}\, \bigg\{ &\Psi(\y_2) + \llangle \GRAD \mathcal{E}(\x_1,\x_2), (\y_1,\y_2) \rrangle \\[-1em]
 \nonumber 
 &\quad - \llangle \mathcal{P}_1(t),\y_1 \rrangle - \llangle \mathcal{P}_2(t),\y_2 \rrangle \bigg\}.
\end{align}
In particular, we assume:
\begin{itemize}

 \item[(P1)] The set of primary variables can be partitioned into two sets with either vanishing or non-vanishing dissipative character. Those can be respectively grouped in two (multi-valued) variables $\x_1$, $\x_2$. Let $\x_1$ denote the variables that change without cost. 

 \item[(P2)] The function spaces $\mathcal{V}_1$ and $\mathcal{V}_2$ corresponding to $x_1$ and $x_2$, respectively, are assumed to be time-independent and to have a linear structure. Thereby, they can be identified as both state and tangent spaces. Furthermore, let $\mathcal{V}_i$ be Banach spaces with norms $\|\cdot\|_{\mathcal{V}_i}$, $i=1,2$. In particular, assume there exists a semi-norm $|\cdot|_{\mathcal{V}_2}$ on $\mathcal{V}_2$ such that
 \begin{align*}
  \|y_2\|_{\mathcal{V}_2}^p = \| y_2 \|_{\mathcal{B}_2}^p + | y_2 |_{\mathcal{V}_2}^p,
 \end{align*}
 where $\mathcal{B}_2 \supset \mathcal{V}_2$ is a larger Banach space with norm $\|\cdot\|_{\mathcal{B}_2}$, and $p:=\mathrm{min}\{p_\psi,p_\mathcal{E}\}\in(1,\infty)$ with $p_\psi$ and $p_\mathcal{E}$ introduced in (P3) and (P4).

 \item[(P3)] The dissipation potential $\Psi:\mathcal{B}_2\rightarrow [0,\infty)$ is convex, continuously differentiable and coercive wrt.\ to $\mathcal{B}_2$. In particular, there exists a constant $C>0$ and $p_\psi\in(1,\infty)$ satisfying 
 \begin{align*}
  \Psi(y_2)\geq C \| y_2 \|_{\mathcal{B}_2}^{p_\psi},\quad y_2\in\mathcal{B}_2.
 \end{align*}

 \item[(P4)] The free energy of the system is convex, lower semi-continuous and continuously differentiable. Furthermore, it can be decomposed into a strictly convex part in the variable with vanishing dissipation, and a convex contribution in an affine combination of the primary variables
 \begin{align}
  \label{general-wellposedness:gradient-flow-system:energy}
  \mathcal{E}(\x_1,\x_2) = \mathcal{E}_1(\x_1) + \mathcal{E}_2(\Lambda(\x_1,\x_2)) , \ (\x_1,\x_2)\in\mathcal{V}_1\times\mathcal{V}_2,
 \end{align}
 i.e., $\mathcal{E}_1:\mathcal{V}_1 \rightarrow [0,\infty)$ is strictly convex; $\mathcal{E}_2:\tilde{\mathcal{V}} \rightarrow [0,\infty)$ is convex with $\tilde{\mathcal{V}}$ a (separable) Banach space; and $\Lambda:\mathcal{V}_1 \times \mathcal{V}_2 \rightarrow \tilde{\mathcal{V}}$ is an affine operator, satisfying
\begin{align*}
 \Lambda(\x_1,\x_2) - \Lambda(\y_1,\y_2)= \Lambda_1 (\x_1 - \y_1) + \Lambda_2 (\x_2-\y_2),\quad \forall \x_i,\y_i\in\mathcal{V}_i,\ i=1,2,
\end{align*}
for $\Lambda_i: \mathcal{V}_i \rightarrow \tilde{\mathcal{V}}$ two linear operators with adjoint operators $\Lambda_i^\star$, $i=1,2$. Furthermore, there exist constants $C_1,C_2,C_3$ and $p_1,p_\mathcal{E}\in(1,\infty)$ satisfying
\begin{align*}
 \mathcal{E}_1(x_1)     &\geq C_1\| x_1 \|_{\mathcal{V}_1}^{p_1},\\
 \mathcal{E}(\x_1,\x_2) &\geq C_2 |\x_2|_{\mathcal{V}_2}^{p_\mathcal{E}} - C_3.
\end{align*}

\item[(P5)] The external loads satisfy $\mathcal{P}_1\in C(0,T;\mathcal{V}_1^\star)\cap W^{1,p_1^\star}(0,T;\mathcal{V}_1^\star)$ and $\mathcal{P}_2\in C(0,T;\mathcal{V}_2^\star)\cap W^{1,p^\star}(0,T;\mathcal{V}_2^\star)$, where $\tfrac{1}{p_1}+\tfrac{1}{p_1^\star}=\tfrac{1}{p}+\tfrac{1}{p^\star}=1$.

\item[(P6)] The initial conditions $(x_1(0),x_2(0))\in \mathcal{V}_1 \times \mathcal{V}_2$ have finite energy $\mathcal{E}(x_1(0),x_2(0))<\infty$ and satisfy the compatibility condition
\begin{align*}
 x_1(0) = \underset{x_1\in\mathcal{V}_1}{\mathrm{arg\,min}}\, \Big\{\mathcal{E}(x_1,x_2(0)) - \llangle \mathcal{P}_1(0),x_1 \rrangle\Big\}.
\end{align*}
\end{itemize}

Quasi-static systems of type~\eqref{general-wellposedness:gradient-flow-system} have been studied in the literature before, cf., e.g.,~\cite{Rossi2006,Mielke2013}; however in the aforementioned works, energies with decompositions different than the poro-elasticity-specific choice~\eqref{general-wellposedness:gradient-flow-system:energy} are treated. And in the theory on doubly non-linear evolution equations in general, the external loading $\mathcal{P}_2$ would usually be assumed to be in the dual of a larger space ($\mathcal{B}_2^\star$ in our context). Here, the weaker regularity assumption on $\mathcal{P}_2$ originates from the nature of external loadings applied in the context of flow in porous media. In order to handle the weak, spatial regularity within the theory on doubly non-linear evolution equations, stronger temporal regularity is required along with above growth conditions on the energy functional. All in all, using similar ideas as~\cite{Rossi2006,Mielke2013} but tailored to the above problem structure, we prove well-posedness of~\eqref{general-wellposedness:gradient-flow-system} under (P1)--(P6).

\begin{theorem}[Well-posedness for generalized gradient flow system~\eqref{general-wellposedness:gradient-flow-system}]\label{thm:abstract-well-posedness-of-dne}
 Assuming~$\mathrm{(P1)}$--$\mathrm{(P6)}$, there exists a solution $(\x_1,\x_2)$ to~\eqref{general-wellposedness:gradient-flow-system} satisfying
 \begin{align*}
  x_1 &\in L^\infty(0,T;\mathcal{V}_1), \\
  x_2 &\in W^{1,p}(0,T;\mathcal{B}_2)\cap L^\infty(0,T;\mathcal{V}_2).
 \end{align*}
 If $\GRAD\Psi$ or $\GRAD\mathcal{E}$ is linear and self-adjoint, it is unique. 
\end{theorem}

\begin{proof}
 We follow ideas by~\cite{Rossi2006,Mielke2013} and decouple the system into a minimization problem and a gradient flow problem. The first is discussed using classical convex analysis (Thm.~\ref{appendix:well-posedness:convex-minimization}, cf.\ Appendix~\ref{appendix:classic-theorems}); the discussion of the gradient flow problem utilizes theory on doubly non-linear evolution equations (Thm.~\ref{preliminaries:thm:well-posedness:gradient-flows}, cf.\ Appendix~\ref{appendix:classic-theorems}).

 \paragraph{Decoupling.}
 
 For fixed time $t\in[0,T]$, the optimality conditions for $(x_1(t),x_2(t))$, derived as first variation, corresponding to~\eqref{general-wellposedness:gradient-flow-system} read
 \begin{alignat}{2}
\label{general-well-posedness:optimality-conditions:start}
  \GRAD \mathcal{E}_1(\x_1(t)) + \Lambda_1^\star \GRAD \mathcal{E}_2\big(\Lambda(\x_1(t),\x_2(t))\big) &= \mathcal{P}_1(t)  &\ \ &\text{in }\mathcal{V}_1^\star, \\
  \GRAD \Psi (\dot{\x}_2(t)) + \Lambda_2^\star \GRAD \mathcal{E}_2\big(\Lambda(\x_1(t),\x_2(t))\big) &= \mathcal{P}_2(t) &&\text{in }\mathcal{V}_2^\star.
\label{general-well-posedness:optimality-conditions:end}
 \end{alignat}
 We conclude that $\x_1$ is defined as solution to a minimization problem for given $\x_2$. Hence, given $t\in[0,T]$ and $\x_2\in\mathcal{V}_2$, we denote $\x_1^\star:=\x_1^\star(t,\x_2)\in\mathcal{V}_1$ to be the solution of the problem
 \begin{align}
 \label{general-well-posedness:aux-0}
  \underset{\y_1\in\mathcal{V}_1}{\mathrm{inf}}\, \mathcal{E}(\y_1,\x_2) - \llangle \mathcal{P}_1(t), \y_1 \rrangle.
 \end{align}
 Since $\mathcal{E}_1$ is strictly convex, $\x_1^\star$ is well-defined by Thm.~\ref{appendix:well-posedness:convex-minimization}. We introduce the reduced energy
 \begin{align*}
  \mathcal{E}_\mathrm{red}(t,\x_2):=\mathcal{E}\big(\x_1^\star(t,\x_2),\x_2\big) - \llangle \mathcal{P}_1(t), \x_1^\star(t,x_2) \rrangle, \ t\in[0,T],\ \x_2\in\mathcal{V}_2.
 \end{align*}
 We observe, that the optimality conditions~\eqref{general-well-posedness:optimality-conditions:start}--\eqref{general-well-posedness:optimality-conditions:end} can be decoupled into
 \begin{alignat}{2}
 \label{general-well-posedness:aux-1:start}
  \GRAD \mathcal{E}_1(\x_1^\star(t,\x_2))
  +
  \Lambda_1^\star \GRAD \mathcal{E}_2\big(\Lambda(\x_1^\star(t,\x_2),\x_2)\big)
  &=
  \mathcal{P}_1(t) && \text{in }\mathcal{V}_1^\star,\\
  \GRAD \Psi (\dot{\x}_2(t)) + \GRAD \mathcal{E}_\mathrm{red}(t,\x_2(t)) &= \mathcal{P}_2(t) &\ \ &\text{in }\mathcal{V}_{2}^\star.
 \label{general-well-posedness:aux-1:end}
 \end{alignat}

 \paragraph{Existence and uniqueness for the gradient flow problem.}
 
 Eq.~\eqref{general-well-posedness:aux-1:end} has the structure of a doubly non-linear evolution equation. The existence (and uniqueness) of a solution to~\eqref{general-well-posedness:aux-1:end} follows by employing Thm.~\ref{preliminaries:thm:well-posedness:gradient-flows}; under above assumptions, it is sufficient to check that $\mathcal{E}_\mathrm{red}$ complies with the assumptions of Thm.~\ref{preliminaries:thm:well-posedness:gradient-flows}.
 
 First, by~(P4) and (P5), it is simple to show that there exist constants $C_1>0$, $C_2\geq0$, independent of $t$, satisfying
 \begin{align*}
  &\mathcal{E}_\mathrm{red}(t,\x_2) 
  \geq
  C_1 | \x_2 |_{\mathcal{V}_2}^{p_\mathcal{E}} - C_2.
  \end{align*} 
  
  Second, $\mathcal{E}_\mathrm{red}(t,\cdot)$ is convex on $\mathcal{V}_2$: This follows from the fact that $\GRAD\mathcal{E}_\mathrm{red}$ is monotone~\cite{Rockafellar1970}. In order to see this, we derive an explicit expression for $\GRAD\mathcal{E}_\mathrm{red}$. Let $\x_2,\y_2\in\mathcal{V}_2$ be arbitrary, and let $D\x_1^\star(t,\x_2)[\y_2]:=\left.\tfrac{\text{d}}{\text{d}\delta}\right|_{\delta=0} \x_1^\star(t,\x_2 + \delta \y_2)$. Using the chain rule, the optimality condition corresponding to~\eqref{general-well-posedness:aux-0}, and the definitions of $\x_1^\star$ and $\mathcal{E}$, we obtain
 \begin{align}
 \nonumber
  &\llangle \GRAD \mathcal{E}_\mathrm{red}(t,\x_2), \y_2 \rrangle \\
  \nonumber
  &\quad= 
  \llangle \GRAD_{1} \mathcal{E}\big(\x_1^\star(t,\x_2),\x_2\big) - \mathcal{P}_1(t), D\x_1^\star(t,\x_2)[\y_2] \rrangle
  +
  \llangle \GRAD_{2} \mathcal{E}\big(\x_1^\star(t,\x_2),\x_2\big), \y_2 \rrangle \\
  \nonumber
  &\quad=
  \llangle \GRAD \mathcal{E}_2\big(\Lambda(\x_1^\star(t,\x_2),\x_2)\big), \Lambda_2 \y_2 \rrangle.
 \end{align}
 Hence, from the definition of $\Lambda$, we obtain
 \begin{align*}
  &\llangle \GRAD \mathcal{E}_\mathrm{red}(t,\x_2) - \GRAD \mathcal{E}_\mathrm{red}(t,\y_2), \x_2 - \y_2 \rrangle\\[0.5em]
  &\quad=
  \llangle \GRAD \mathcal{E}_2\big(\Lambda(\x_1^\star(t,\x_2),\x_2)\big) - \GRAD \mathcal{E}_2\big(\Lambda(\x_1^\star(t,\y_2),\y_2)\big), \Lambda_2 (\x_2 - \y_2) \rrangle  \\[0.5em]
  &\quad=
  \llangle \GRAD \mathcal{E}_2\big(\Lambda(\x_1^\star(t,\x_2),\x_2)\big) - \GRAD \mathcal{E}_2\big(\Lambda(\x_1^\star(t,\y_2),\y_2)\big), \right. \\
  &\quad\quad\quad \left.\big(\Lambda(\x_1^\star(t,\x_2),\x_2) - \Lambda(\x_1^\star(t,\y_2),\y_2)\big) - \Lambda_1\big(\x_1^\star(t,\x_2) - \x_1^\star(t,\y_2)\big) \rrangle.
 \end{align*}
 Subtracting the optimality condition for arbitrary $\x_2,\y_2\in\mathcal{V}_2$, yields
 \begin{align*}
  &\GRAD \mathcal{E}_1(\x_1^\star(t,\x_2)) - \GRAD \mathcal{E}_1(\x_1^\star(t,\y_2)) \\
  &\quad=
  - \Lambda_1^\star\left( \GRAD \mathcal{E}_2\big(\Lambda(\x_1^\star(t,\x_2),\x_2)\big) - \GRAD \mathcal{E}_2(\Lambda(\x_1^\star(t,\y_2),\y_2)) \right)
  \ \  \text{in }\mathcal{V}_1^\star.
 \end{align*}
 Hence, together, we obtain
 \begin{align*} 
  &\llangle \GRAD \mathcal{E}_\mathrm{red}(t,\x_2) - \GRAD \mathcal{E}_\mathrm{red}(t,\y_2), \x_2 - \y_2 \rrangle\\
    &\quad=
  \llangle \GRAD \mathcal{E}_2\big(\Lambda(\x_1^\star(t,\x_2),\x_2)\big) - \GRAD \mathcal{E}_2\big(\Lambda(\x_1^\star(t,\y_2),\y_2)\big), \Lambda(\x_1^\star(t,\x_2),\x_2) - \Lambda(\x_1^\star(t,\y_2),\y_2) \rrangle \\
  &\quad\quad +
  \llangle \GRAD \mathcal{E}_1\left(\x_1^\star(t,\x_2)\right) - \GRAD \mathcal{E}_1(\x_1^\star(t,\x_2)), \x_1^\star(t,\x_2) - \x_1^\star(t,\y_2) \rrangle.
 \end{align*}
 From the convexity of $\mathcal{E}_1$ and $\mathcal{E}_2$, we obtain the convexity of $\mathcal{E}_\mathrm{red}$.

 By exploiting the definition of $\mathcal{E}_\mathrm{red}$, the optimality condition~\eqref{general-well-posedness:aux-1:start},~(P4), and~(P5), we obtain for almost every $t\in(0,T)$ and $p_1^\star$ as in~(P5)
 \begin{align*}
  \left| \partial_t \mathcal{E}_\mathrm{red}(t,x_2) \right| 
  &= \big| \llangle \GRAD \mathcal{E}_1(\x_1^\star(t,x_2)), \dot{\x}_1^\star(t,x_2) \rrangle \\
  &\qquad + \llangle \GRAD \mathcal{E}_1\left(\Lambda\left(\x_1^\star(t,x_2),x_2\right)\right),\Lambda_1 \dot{\x}_1^\star(t,x_2) \rrangle - \llangle \mathcal{P}_1, \dot{\x}_1^\star(t,x_2) \rrangle\\
  &\qquad - \llangle \partial_t \mathcal{P}_1, \x_1^\star(t,x_2) \rrangle \big|\\
  &= \left| \llangle \partial_t \mathcal{P}_1, \x_1^\star(t,x_2) \rrangle \right|\\
  &\lesssim \left\| \partial_t \mathcal{P}_1 \right\|_{\mathcal{V}_1^\star}^{p_1^\star} + \left\|  \x_1^\star(t,x_2)  \right\|_{\mathcal{V}_1}^{p_1}.
 \end{align*}
 In addition, by employing~(P4), and using that $\mathcal{E}_2$ is positive, it follows
 \begin{align*}
 \left\|  \x_1^\star(t,x_2)  \right\|_{\mathcal{V}_1}^{p_1}
 \lesssim
  \mathcal{E}_1(x_1^\star(t,x_2))
  \leq
  \mathcal{E}_\mathrm{red}(t,x_2)
  +
  \llangle \mathcal{P}_1, \x_1^\star(t,x_2)  \rrangle.
 \end{align*}
 Employing (P5) and Young's inequality, yields
 \begin{align*}
  \left\|  \x_1^\star(t,x_2)  \right\|_{\mathcal{V}_1}^{p_1}
  \lesssim
  \mathcal{E}_\mathrm{red}(t,x_2)
  +
  \llangle \mathcal{P}_1, \x_1^\star(t,x_2)  \rrangle.
 \end{align*}
 Altogether, we obtain
 \begin{align}
 \label{general-well-posedness:aux-2}
  \left| \partial_t \mathcal{E}_\mathrm{red}(t,x_2) \right| 
  \lesssim  
  \left\| \mathcal{P}_1 \right\|_{\mathcal{V}_1^\star}^{p_1^\star} + \left\| \partial_t \mathcal{P}_1 \right\|_{\mathcal{V}_1^\star}^{p_1^\star} + \mathcal{E}_\mathrm{red}(t,x_2).
 \end{align}
 \indent
 Consequently, $\mathcal{E}_\mathrm{red}$ complies with Thm.~\ref{preliminaries:thm:well-posedness:gradient-flows}, and together with (P1)--(P6), there exists a solution $\x_2\in W^{1,p}(0,T;\mathcal{B}_2)\cap L^\infty(0,T;\mathcal{V}_2)$ to~\eqref{general-well-posedness:aux-1:start}. It is unique in case $\GRAD\Psi$ or $\GRAD\mathcal{E}_\mathrm{red}$ are linear and self-adjoint, cf.\ Thm.~\ref{preliminaries:thm:well-posedness:gradient-flows}. The latter follows for linear and self-adjoint $\GRAD\mathcal{E}_i$, $i=1,2$. 
 
 \paragraph{Finite energy.} By Thm.~\ref{preliminaries:thm:well-posedness:gradient-flows}, $x_2$ satisfies the characteristic energy identity
  \begin{align*}
  &\int_0^T \Psi(\dot{x}_2(t))\,dt + \mathcal{E}_\mathrm{red}(x_2(T)) - \llangle \mathcal{P}_2(T),x_2(T) \rrangle \\[-0.25em]
  \nonumber
  &\quad = \mathcal{E}_\mathrm{red}\left(x_2(0)\right) - \llangle \mathcal{P}_2(0), x_2(0) \rrangle + \int_0^T \partial_t \mathcal{E}_\mathrm{red}(t,x_2(t))\, dt - \int_0^T \llangle \dot{\mathcal{P}}_2(t), x_2(t) \rrangle \, dt.
 \end{align*}
 Using (P4) and~\eqref{general-well-posedness:aux-2}, we obtain
 \begin{align*}
  &\int_0^T \Psi(\dot{x}_2(t))\,dt + \mathcal{E}_\mathrm{red}(x_2(T)) \lesssim \llangle \mathcal{P}_2(T),x_2(T) \rrangle \\[-0.25em]
  &\quad 
  \leq 
  \mathcal{E}_\mathrm{red}\left(x_2(0)\right) - \llangle \mathcal{P}_2(0), x_2(0) \rrangle
  +
   \left\| \mathcal{P}_1 \right\|_{W^{1,p_1^\star}(0,T;\mathcal{V}_1^\star)}^{p_1^\star}\\
   &\quad 
  +
  \|\mathcal{P}_2(T)\|_{\mathcal{V}_2^\star}^{p_\mathcal{E}^\star} 
  +
  \left\| \mathcal{P}_2 \right\|_{W^{1,p_\mathcal{E}^\star}(0,T;\mathcal{V}_2^\star)}^{p_\mathcal{E}^\star}
  +C_3
  + 
  \int_0^T \mathcal{E}_\mathrm{red}(t,x_2) \, dt,
 \end{align*}
 with $\tfrac{1}{p_\mathcal{E}}+\tfrac{1}{p_\mathcal{E}^\star}=1$ and $C_3$ from (P4). The assumptions on the external data are chosen such that the right hand side is uniformly bounded in $T$ up to the last term. By the Gr\"onwall inequality it follows that $\mathcal{E}_\mathrm{red}(t,x_2(t))$ is uniformly bounded in time.
 
 \paragraph{Existence and uniqueness for the minimization problem.}
 Since $x_2 \in L^\infty(0,T;\mathcal{V}_2)$, \eqref{general-well-posedness:aux-0} is well-defined for $x_2=x_2(t)$ for a.e.\ $t\in(0,T)$; thereby also $x_1=x_1^\star(t,\x_2)$. Finally, by the definition of $\mathcal{E}_\mathrm{red}$ and~(P4), it follows  for a.e.\ $t\in(0,T)$
 \begin{align*}
  \|\x_1(t)\|_{\mathcal{V}_1}^{p_1} \lesssim \mathcal{E}_\mathrm{red}(t,\x_2) + \| \mathcal{P}_1(t) \|_{\mathcal{V}_1}^{p_1^\star}.
 \end{align*}
 Hence, by the above paragraph and (P5), $\x_1^\star\in L^\infty(0,T;\mathcal{V}_1)$. Altogether, we obtain existence (and uniqueness) of the coupled system~\eqref{general-wellposedness:gradient-flow-system}, which concludes the proof.

\end{proof}

\begin{remark}
 Detailed stability bounds can be derived using energy identities for gradient flows, cf., e.g.,~\eqref{appendix:well-posedness-dne:energy-identity}.
\end{remark}

\section{Linear Biot equations as generalized gradient flow}\label{section:linear-biot-gradient-flow}

The theory of linear poro-elasticity describes the continuum mechanics of coupled flow and geomechanics in porous media under several simplifying hypotheses: in particular, the fundamental linearizing assumptions of linear elasticity; the hypothesis of small perturbations of the Lagrangian porosity; and an at most slightly compressible, Newtonian fluid. Together with first principles and Darcy's law, the \textit{Biot's consolidation model}, also called \textit{linear Biot equations}, can be deduced, coupling elliptic and parabolic equations. For a detailed introduction, we refer to the seminal work by Biot~\cite{Biot1941} and the comprehensive  books~\cite{Coussy2004,Lewis1998}.

In this section, we provide a derivation of the linear Biot equations employing the modelling framework described in Sec.~\ref{section:general-model-biot}. Thereby we demonstrate the inherent gradient flow structure of the linear Biot equations. Acknowledging the fact that the linear Biot equations have been already studied quite thoroughly in the literature, the following discussion serves mostly as proof of concept and guide for subsequent discussions of more involved poro-elasticity models.

\subsection{Generalized gradient flow formulation of linear poro-elasticity}\label{section:linear-biot:primal-formulation}

Using the modelling approach described in Sec.~\ref{section:general-model-biot}, we derive Biot's consolidation model as a generalized gradient flow. It suffices to specify states, an associated Helmholtz free energy $\mathcal{E}$ and a dissipation potential $\mathcal{D}$. 


As states, we choose the mechanical displacement $\u$ and the volume content $\fluidmass$ with associated processes $\dot{\u}$ and the volumetric flux $\flux$, respectively. Suitable function spaces for the latter, incorporating essential boundary conditions are given by
\begin{align}
 \label{linear-poro:zero-spaces:start}
 \mathcal{V} &= \left\{ \v\in H^1(\Omega) \,|\, \v=\u_\Gamma\text{ on }\Gamma_{\u}\right\}, \\
 \dot{\mathcal{V}} &= \left\{ \v\in H^1(\Omega) \,|\, \v=\dot{\u}_\Gamma\text{ on }\Gamma_{\u}\right\}, \\
 \mathcal{Z} &= \left\{ \z\in H(\mathrm{div};\Omega) \,|\, \z\cdot\n = q_\mathrm{\Gamma,n}\text{ on }\Gamma_{\flux}\right\}.
\end{align}
For their variations, we define correspondingly
\begin{align}
 \mathcal{V}_0 &= \left\{ \v\in H^1(\Omega) \,|\, \v=\vec{0}\text{ on }\Gamma_{\u}\right\}, \\
 \mathcal{Z}_0 &= \left\{ \z\in H(\mathrm{div};\Omega) \,|\, \z\cdot\n = 0\text{ on }\Gamma_{\flux}\right\}.
\label{linear-poro:zero-spaces:end}
\end{align}

\noindent
The energy is chosen to be the Helmholtz free energy for linearly deformable porous media, cf.\ Ch.~4.2.2,~\cite{Coussy2004},
\begin{align*}
 \mathcal{E}(\u,\fluidmass) &= 
 \mathcal{E}_\mathrm{eff}(\u) + \mathcal{E}_\mathrm{fluid}(\u,\fluidmass),\\ 
 \mathcal{E}_\mathrm{eff}(\u) &=\tfrac{1}{2} \llangle \mathbb{C}\eps{\u},\eps{\u} \rrangle, \\
 \mathcal{E}_\mathrm{fluid}(\u,\fluidmass) &= \tfrac{M}{2} \left\| \fluidmass- \alpha \DIV \u \right\|^2,
\end{align*}
where $\mathbb{C}$ is a symmetric, uniformly positive definite, fourth-order stiffness tensor, $M$ can be identified as the inverse of the compressibility of the bulk and $\alpha$ is the Biot coefficient. In this work, we assume isotropic materials, modelled as St.\ Venant Kirchhoff material, i.e., there exist constants $\mu>0$ and $\lambda\geq 0$ satisfying
\begin{align*}
 \mathbb{C}\eps{\u} = 2\mu \eps{\u} + \lambda \DIV \u \,\mathbf{I}.
\end{align*}
From~\eqref{thermodynamic-interpretation}, we recover the classical relations
\begin{align}
\label{linear-biot-interpretation-m}
 \fluidmass&= \tfrac{1}{M} p + \alpha \DIV \u, \\
\label{linear-biot-interpretation-stress}
 \stress &= \mathbb{C}\eps{\u} - \alpha p \,\mathbf{I}, \\
\label{linear-biot-interpretation-effective-stress}
 \stress_\mathrm{eff} &= \mathbb{C}\eps{\u} = \stress + \alpha p \,\mathbf{I}.
\end{align}

A standard assumption in linear poro-elasticity is the quasi-static character of the mechanical problem. As consequence, mechanical deformation occur instantaneously and hence changes without any cost. Hence, allowing for viscous dissipation for the fluid, changes of displacements and volumetric fluxes come at costs based on the dissipation potentials
\begin{align*}
 \mathcal{D}_\mathrm{mech}(\dot{\u}) &= 0, \\
 \mathcal{D}_\mathrm{fluid}(\flux)   &= \tfrac{1}{2} \llangle \permeability^{-1} \flux, \flux \rrangle,
\end{align*}
where the conductivity $\permeability$ is a symmetric, uniformly positive definite and uniformly bounded second-order tensor. It can be identified as the permeability, scaled by the inverse of the fluid viscosity. 

Given the current state $(\u,\fluidmass)$, its change is then described by~\eqref{biot-generalized-gradient-flow-structure:start}--\eqref{biot-generalized-gradient-flow-structure:end}:
\begin{align}
\label{linear-biot-generalized-gradient-flow-structure:start}
 \dot{\fluidmass} &= \masssource -\DIV\flux \qquad \text{and} \\
\nonumber
 (\dot{\u},\flux) &= \underset{(\v,\z)\in\dot{\mathcal{V}}\times\mathcal{Z}}{\text{arg\,min}}\, \Big\{ 
 \llangle \mathbb{C}\eps{\u}, \eps{\v} \rrangle  - \alpha \llangle M(\fluidmass-\alpha\DIV\u), \DIV\v \rrangle 
 - \mathcal{P}_\mathrm{ext,mech}(\v) \\[-1em]
\label{linear-biot-generalized-gradient-flow-structure:end}
 &\qquad\qquad\qquad\quad + \tfrac{1}{2} \llangle \permeability^{-1}\z, \z\rrangle - \llangle M(\fluidmass- \alpha\DIV\u), \DIV \z \rrangle - \mathcal{P}_\mathrm{ext,fluid}(\z)   
 \Big\}.
\end{align}
The system~\eqref{linear-biot-generalized-gradient-flow-structure:start}--\eqref{linear-biot-generalized-gradient-flow-structure:end} can be reduced to a compact two-field formulation using ideas from Sec.~\ref{section:abstract-dne}. Recalling the definition of the accumulated flux $\flux_{\int}$, cf. Eq.~\eqref{accumulated-darcy-flux}, living in 
\begin{align}
\label{linear-poro:zero-spaces:second}
 \mathcal{Z}_{\int}(t) = \left\{ \z \in H(\mathrm{div};\Omega) \,\left|\, \z \cdot \n = \int_0^t q_\mathrm{\Gamma,n}\, dt\ \text{on }\Gamma_{\flux} \right. \right\},\quad t\in[0,T],
\end{align}
we introduce the generalized displacement $\bm{U}=(\u,\flux_{\int})$ and its change $\dot{\bm{U}}=(\dot{\u},\flux)$. Energies, external work rates and dissipation potentials can be naturally interpreted as functions of $\bm{U}$ and $\dot{\U}$, respectively. After all, the evolution of the generalized displacement $\U$ is governed by the generalized gradient flow
\begin{align}
\label{biot-generalized-gradient-flow-structure-U}
 \dot{\bm{U}}(t) 
 &= \underset{\bm{V}\in\dot{\mathcal{V}}(t)\times\dot{\mathcal{Z}}_{\int}(t)}{\text{arg\,min}}\, 
 \Big\{ 
 \mathcal{D}(\bm{V})
 + \llangle \GRAD \mathcal{E}(t,\bm{U}(t)), \bm{V} \rrangle 
 - \mathcal{P}_\mathrm{ext}(t,\bm{V}) \Big\}.
\end{align}
Formulations based on the generalized displacement are in the following referred to as the \textit{primal formulation} of linear poro-elasticity.

In order to verify that~\eqref{biot-generalized-gradient-flow-structure-U} is indeed formally equivalent to the Biot equations, we derive the corresponding optimality conditions. Written in variational form, they read: Find $(\u,\flux)\in\mathcal{V}\times\mathcal{Z}$ and $\fluidmass$ with suitable regularity such that
\begin{align}
\label{biot:mechanics}
\llangle \mathbb{C}\eps{\u}, \eps{\v} \rrangle - \alpha \llangle M(\fluidmass-\alpha \DIV\u), \DIV \v \rrangle &= \mathcal{P}_\mathrm{ext,mech}(\v) &&\forall\v\in\mathcal{V}_0,\\
\label{biot:darcy}
\llangle \permeability^{-1}\flux, \z \rrangle - \llangle M(\fluidmass-\alpha \DIV \u), \DIV \z \rrangle &= \mathcal{P}_\mathrm{ext,fluid}(\z), &&\forall\z\in\mathcal{Z}_0,\\
\label{biot:mass-conservation}
\dot{\fluidmass} + \DIV \flux &= \masssource, &&\text{in }L^2(\Omega).
\end{align}
Identifying the fluid pressure from~\eqref{linear-biot-interpretation-m}, we recover the three-field formulation of the classical quasi-static linear Biot equations.

\subsection{Dual formulation of linear poro-elasticity} \label{section:dual-formulation:linear-poro-elasticity}

For the special case of quasi-static linear poro-elasticity, a natural \textit{dual formulation} can be derived by applying the Legendre-Fenchel duality theory~\cite{Ekeland1999} to~\eqref{biot-generalized-gradient-flow-structure-U}. The procedure is analogous to the discussion of primal and dual formulations of linear elastostatics in the context of convex analysis~\cite{Ciarlet2011,Temam2000}. We skip the derivation here and present directly the dual formulation. It naturally employs the dual generalized stress $\generalizedstress=(\stress,p)$ as primary variable, pairing up the total mechanical stress and the fluid stress, i.e., fluid pressure. Suitable function spaces incorporating essential boundary conditions are given by
\begin{align}
 \dot{\mathcal{S}} &:= \left\{ \stress \in H(\text{div};\Omega)^d \,\left|\, 
 \begin{array}{l}  
 \stress\bm{n} = \dot{\stress}_{\Gamma,n} \text{ on }\Gamma_{\stress},\\[2pt]
 \DIV \stress + \dotfext = \bm{0} \text{ in } L^2(\Omega),\\[2pt]
 \llangle \stress,\bm{\gamma}\rrangle = 0 \ \forall \bm{\gamma}\in\bm{Q}_\mathrm{AS}
 \end{array} \right.\right\}, \\[2pt]
 \label{definition:skew-symmetric-tensors}
 \bm{Q}_\mathrm{AS} &:=\left\{ \bm{\gamma}\in L^2(\Omega)^{d\times d}\,|\, \bm{\gamma}\text{ skew-symmetric on }\Omega \right\} \\[2pt]
 \dot{\mathcal{Q}} &:= \left\{ q \in H^1(\Omega) \,|\, q=\dot{p}_\Gamma \text{ on }\Gamma_{p} \right\}, \\[2pt]
 \dot{\mathcal{H}}^\star &:= \dot{\mathcal{S}} \times \dot{\mathcal{Q}}.
\end{align}
We note, that the balance of linear momentum is incorporated intrinsically in $\dot{\mathcal{S}}$, which is characteristic for the dual formulation. Imposing only weak symmetry of stress tensors however is our choice, which is motivated by current advances in the robust discretization of the mixed formulation of elasticity and poro-elasticity, cf., e.g.,~\cite{Ahmed2019a,Ambartsumyan2018b,Baerland2017,Boffi2013,Keilegavlen2017}; imposing strong symmetry is also possible.

In between the primal and the dual formulation, the mathematical interpretation of dissipation and energy essentially swaps, similarly for essential and natural boundary conditions. Hence, utilizing~\eqref{linear-biot-interpretation-m}--\eqref{linear-biot-interpretation-stress} and Darcy's law, we define the dual energy, dissipation and external work rate by
\begin{align*}
 \mathcal{E}^\star(\generalizedstress) &= \mathcal{D}(\generalizedstress) = \frac{1}{2} \llangle \permeability (\GRAD p - \gext), \GRAD p - \gext \rrangle, \\
 \mathcal{D}^\star(\dot{\generalizedstress}) &= \mathcal{E}(\dot{\generalizedstress}) = \tfrac{1}{2} \llangle \mathbb{A} (\dot{\stress} + \alpha \dot{p} \,\mathbf{I}), \dot{\stress} + \alpha \dot{p} \,\mathbf{I} \rrangle + \tfrac{1}{2M} \| \dot{p} \|^2,\\
 \mathcal{P}^\star_\mathrm{ext}(\dot{\generalizedstress}) &= \llangle \dot{\u}_\Gamma, \dot{\stress}\n \rrangle_{\Gamma_{\stress}} + \llangle \masssource, \dot{p} \rrangle + \llangle q_\mathrm{\Gamma,n}, \dot{p} \rrangle_{\Gamma_{\flux}}. 
\end{align*}
Here, $\mathbb{A}=\mathbb{C}^{-1}$ denotes the compliance tensor; for homogeneous, isotropic materials, it satisfies for $\stress \in \mathbb{R}^{d \times d}$, with deviatoric and hydrostatic components $\stress^\mathrm{d} := \stress - \sigma^\mathrm{h}\, \mathbf{I}$ and $\sigma^\mathrm{h} := \tfrac{1}{d} \mathrm{tr}\,\stress$ , respectively,
\begin{align}
\label{elasticity-compliance-tensor-quadratic}
 (\mathcal{A} \stress) : \stress = \frac{1}{2\mu} \left| \stress^\mathrm{d} \right|^2 + \frac{1}{K_\mathrm{dr}} |\sigma^\mathrm{h}|^2
\end{align}

Finally, the evolution of the generalized stress $\generalizedstress$ is prescribed by the generalized gradient flow
\begin{align}
\label{generalized-gradient-flow:dual-formulation:linear-poro-elasticity}
 \dot{\generalizedstress} = \underset{\testgeneralizedstress\in\dot{\mathcal{H}}^\star}{\mathrm{arg\,min}} \,\Big\{ \mathcal{D}^\star(\testgeneralizedstress) + \llangle \GRAD \mathcal{E}^\star(\generalizedstress), \testgeneralizedstress \rrangle - \mathcal{P}^\star_\mathrm{ext}(\testgeneralizedstress) \Big\},
\end{align}
subject to compatible, initial data $\generalizedstress=\generalizedstress_0$ at time $t=0$. Evidently one major advantage of the dual formulation~\eqref{generalized-gradient-flow:dual-formulation:linear-poro-elasticity} compared to the primal formulation~\eqref{biot-generalized-gradient-flow-structure-U} is, it allows for incompressible solids and fluids.

The corresponding optimality conditions can be shown to be identical to the four field formulation of linear poro-elasticity~\cite{Baerland2017}, employing the total stress and the fluid pressure as primary variables, and mechanical displacement and rotation as Lagrange multipliers.

\subsection{Well-posedness of linear poro-elasticity}\label{section:linear-biot:well-posedness}

In the following, we establish existence and uniqueness of a weak solution to the primal formulation of linear poro-elasticity and discuss its regularity. For this, we apply the abstract well-posedness result, Thm.~\ref{thm:abstract-well-posedness-of-dne}.

\begin{lemma}[Well-posedness and regularity for linear poro-elasticity]\label{lemma:well-posedness-poro}
Let $\mathcal{V},\mathcal{V}_0,\mathcal{Z}_0$ and $\mathcal{Z}_{\int}$ as defined in~\eqref{linear-poro:zero-spaces:start}--\eqref{linear-poro:zero-spaces:end} and~\eqref{linear-poro:zero-spaces:second} with 
\begin{align*}
 \u_\Gamma &\in C(0,T;H^{1/2}(\Gamma_{\u})^d) \cap H^1(0,T;H^{1/2}(\Gamma_{\u})^d), \\
 q_\mathrm{\Gamma,n} &\in C(0,T;H^{-1/2}(\Gamma_{\flux})^d).
\end{align*}
For the external loadings, and natural boundary and initial conditions assume
\begin{align*} 
 \stress_\mathrm{\Gamma,n} &\in C(0,T;H^{-1/2}(\Gamma_{\stress})^d) \cap H^1(0,T;H^{-1/2}(\Gamma_{\stress})^d), \\
 p_\Gamma &\in C(0,T;H^{1/2}(\Gamma_{\flux})^d)\cap H^1(0,T;H^{1/2}(\Gamma_{\flux}),\\
 \fext &\in C(0,T;(\bm{H}^1(\Omega))^\star) \cap H^1(0,T;(\bm{H}^1(\Omega))^\star),\\
 \gext &\in C(0,T;H(\mathrm{div},\Omega)^\star) \cap H^1(0,T;H(\mathrm{div};\Omega)^\star),\\
 \masssource &\in L^2(0,T;\mathcal{V}_0^\star \cap \mathcal{Z}_0^\star), \\
 \fluidmass_0 &\in \mathcal{V}_0^\star \cap \mathcal{Z}_0^\star \\
 \u_0 &\in \bm{H}^1(\Omega),\text{ such that }{{\u_0}_|}_{\Gamma_{\u}} = \u_\Gamma(0),
\end{align*}
with the initial conditions satisfying the compatibility condition
\begin{align*}
 \llangle \mathbb{C} \eps{\u_0}, \eps{v} \rrangle  - \llangle M \left( \theta_0 - \alpha \DIV \u_0\right), \DIV \v \rrangle = \llangle \fext(0), \v \rrangle, \quad \forall \v\in \mathcal{V}_0.
\end{align*}
Then there exists a unique solution $\bm{U}=(\u,\flux_{\int})$ of~\eqref{biot-generalized-gradient-flow-structure-U} and equivalently~\eqref{linear-biot-generalized-gradient-flow-structure:start}--\eqref{linear-biot-generalized-gradient-flow-structure:end}, satisfying
 \begin{align}
 \label{lemma:poro-regularity-1}
  \u &\in L^\infty(0,T; H^1(\Omega)), \\
 \label{lemma:poro-regularity-2}
  \flux_{\int} &\in H^1(0,T;L^2(\Omega))\cap L^\infty(0,T;H(\mathrm{div};\Omega)),\\
 \label{lemma:poro-regularity-3}
  \DIV \flux_{\int} &\in H^1(0,T; H^{-1}(\Omega)) \cap L^\infty(0,T; L^2(\Omega)).
 \end{align}
 For the volumetric flux $\flux$, fluid content $\fluidmass$, fluid pressure $p$ and stress $\stress$, associated with the states by~\eqref{accumulated-darcy-flux},~\eqref{accumulated-mass-conservation},~\eqref{linear-biot-interpretation-m} and~\eqref{linear-biot-interpretation-stress}, respectively, it holds
 \begin{align}
 \label{lemma:poro-regularity-extra:start}
  \flux &\in L^2(0,T;L^2(\Omega)),\ \DIV \flux \in L^2(0,T;H^{-1}(\Omega)), \\
  \fluidmass &\in H^1(0,T;H^{-1}(\Omega)) \cap L^\infty(0,T;L^2(\Omega)), \\
 \label{lemma:poro-regularity-extra:pressure-regularity}
  p&\in L^\infty(0,T;L^2(\Omega)), \\
 \label{lemma:poro-regularity-extra:end}
  \stress &\in L^\infty(0,T;L^2(\Omega)).
 \end{align}
\end{lemma}

\begin{proof}
The well-posedness result is a direct consequence of Thm.~\ref{thm:abstract-well-posedness-of-dne}, applied to poro-elasticity written as doubly non-linear evolution equation, cf.\ Sec.~\ref{section:abstract-dne}. For this, we reformulate~\eqref{biot-generalized-gradient-flow-structure-U} having two objectives in mind: (i) time-dependent contributions due to essential boundary conditions and external sources are required to impact only the external work rates; (ii) the final structure is required to match the abstract setting of Thm.~\ref{thm:abstract-well-posedness-of-dne}.

Due to objective~(i), the problem has to be formulated for homogeneous contributions of the generalized displacement. It will be denoted by 
\begin{align*}
 \left(\u_\mathrm{hom},\flux_\mathrm{\int,hom}\right) := \left(\u,\flux_\mathrm{\int}\right) - \left(\tilde{\u}_\Gamma,\tilde{\flux}_\mathrm{\int,\Gamma}\right) \in \mathcal{V}_0 \times \mathcal{Z}_0,
\end{align*}
where
\begin{align*}
 \tilde{\u}_\Gamma &\in C(0,T;\bm{H}^1(\Omega)) \cap H^1(0,T;\bm{H}^1(\Omega)), \\
 \tilde{\flux}_{\int,\Gamma} &\in C^1(0,T;H(\mathrm{div};\Omega)).
\end{align*}
are extensions of the essential boundary conditions onto $\Omega$, such that $\tilde{\u}(0)=\u_0$, ${\left.{\tilde{\u}_\Gamma}\right|}_{\Gamma_{\u}}(t) = \u_\Gamma(t)$ and ${\left.\tilde{\flux}_\Gamma \cdot \n\right|}_{\Gamma_{\flux}}(t) = \int_0^t q_\mathrm{\Gamma,n}\,dt$ for a.e.\ $t\in(0,T)$.

Using the notation of Thm.~\ref{thm:abstract-well-posedness-of-dne}, we define for $(\v,\z)\in\mathcal{V}_0\times\mathcal{Z}_0$
\begin{align*}
\Psi(\z) &:= \mathcal{D}_\mathrm{fluid}(\z) 
  & \mathcal{E}_1(\v) &:= \mathcal{E}_\mathrm{eff}(\v), \\
\mathcal{E}_2(m) &:= \tfrac{M}{2} \|m\|^2,
  & \Lambda(\v,\z) &:= \fluidmass_0 - \DIV \z - \alpha \DIV \v,\\
  &&\mathcal{E}(\v,\z) &:= \mathcal{E}_1(\v) + \mathcal{E}_2(\Lambda(\v,\z)),
\end{align*}
and in order to fulfill objective~(i), we set 
\begin{align*}
 \mathcal{P}_1(t,\v)&:=
 \mathcal{P}_\mathrm{ext,mech}(t,\v)
 -
 \llangle \mathbb{C} \eps{\tilde{\u}_\Gamma(t)}, \eps{\v} \rrangle \\
 &\quad+
 M \llangle \intmasssource(t) - \DIV \tilde{\flux}_{\int,\Gamma}(t) - \alpha \DIV \tilde{\u}_\Gamma(t), \alpha\DIV \v \rrangle, \\[1pt]
 \mathcal{P}_2(t,\z)&:= 
 \mathcal{P}_\mathrm{ext,fluid}(t,\z)
 -
 \llangle \permeability^{-1} \dot{\tilde{\flux}}_{\int,\Gamma}(t), \z \rrangle \\
 &\quad+
 M \llangle \intmasssource(t) - \DIV \tilde{\flux}_{\int,\Gamma}(t) - \alpha \DIV \tilde{\u}_\Gamma(t), \DIV \z \rrangle.
\end{align*}
One can simply verify that~\eqref{biot-generalized-gradient-flow-structure-U} is equivalent to 
\begin{align}
\label{linear-poro:proof:aux1}
 (\dot{\u}_\mathrm{hom},\dot{\flux}_\mathrm{\int,hom}) = \underset{(\v,\z)\in\mathcal{V}_0 \times \mathcal{Z}_0}{\mathrm{arg\,min}}\, 
 \bigg\{ 
 \Psi(\z) + \llangle \GRAD \mathcal{E}\left(\u_\mathrm{hom},\flux_\mathrm{\int,hom}\right), (\v,\z) \rrangle \\
 \nonumber
 - \mathcal{P}_1(t,\v) - \mathcal{P}_2(t,\z) 
 \bigg\}
\end{align}
together with zero initial conditions. Furthermore, it is simple to verify~(P1)--(P6); we just note that $\mathcal{V}_1=\mathcal{V}_0$, $\mathcal{V}_2=\mathcal{Z}_0$ and $\mathcal{B}_2=\bm{L}^2(\Omega)$ using the notation of Thm.~\ref{thm:abstract-well-posedness-of-dne}. Consequently, we obtain existence and uniqueness of a solution to~\eqref{linear-poro:proof:aux1}, and consequently of~\eqref{biot-generalized-gradient-flow-structure-U}, satisfying~\eqref{lemma:poro-regularity-1}--\eqref{lemma:poro-regularity-2}. Since
 \begin{align*}
  \int_0^T \frac{\llangle  \DIV \dot{\flux}_\mathrm{\int,hom}, \phi\rrangle}{\|\GRAD\phi\|}\,dt
  \leq
  \| \dot{\flux}_\mathrm{\int,hom} \|_{L^2(0,T;L^2(\Omega))}^2,
 \end{align*}
 and the regularity of $\tilde{\flux}_{\int,\Gamma}$, it follows~\eqref{lemma:poro-regularity-3}. Finally,~\eqref{lemma:poro-regularity-extra:start}--\eqref{lemma:poro-regularity-extra:end} follow directly using~\eqref{accumulated-darcy-flux},~\eqref{accumulated-mass-conservation} and~\eqref{linear-biot-interpretation-m}--\eqref{linear-biot-interpretation-stress}.
\end{proof}

\section{Linear poro-visco-elasticity as generalized gradient flow}\label{section:poro-visco-elasticity-wellposedness}  

Biot's consolidation model considers solely primary consolidation, which results in a  characteristic, quasi-static, mechanical response of the poro-elastic system. Modelling-wise, this originates from neglecting viscous dissipation due to mechanical deformation. However, in physical, deformable porous media viscous dissipation always occurs, and it leads to  partially non-instantaneous deformation, also called \textit{secondary consolidation}. The theory of poro-visco-elasticity incorporates such visco-elastic effects. 

Classical models for visco-elasticity consider a separation of the total strain into an elastic and a visco-elastic strain~\cite{Coussy2004}; the elastic contribution is instantaneously recovered during an unloading process, whereas the visco-elastic contribution is not. In the extreme case, the elastic behavior can be neglected and the total strain can be assumed to be identical to the visco-elastic strain~\cite{Showalter2000,Bociu2016}; the corresponding model is also referred to as linear, quasi-static poro-elasticity with secondary consolidation. Here, we treat the general case, including the elastic and visco-elastic strains.

Linear poro-visco-elasticity can be naturally formulated as generalized gradient flow by suitably enhancing the primal model for linear poro-elasticity.
As state we choose the generalized displacement $(\u,\flux_{\int},\viscostrain)$, incorporating now also the visco-elastic strain $\viscostrain$, living for a.e.\ $t\in (0,T)$ in 
\begin{align*}
 \mathcal{T}:=\{\testviscostrain \in L^2(\Omega)^{d\times d} \,|\, \testviscostrain\ \text{symmetric}\}.
\end{align*}

We distinguish between the standard elastic strain energy and a stored, visco-elastic energy. Additionally, we allow for different impacts of the elastic and visco-elastic strains on the energy corresponding to the fluid. The associated energy functional is consistently defined as by~\cite{Coussy2004}
\begin{align*}
 \mathcal{E}_\mathrm{v}(\u,\flux_{\int},\viscostrain) =& 
 \tfrac{1}{2} \langle \mathbb{C}(\eps{\u} - \viscostrain),\eps{\u} -\viscostrain\rangle + \tfrac{1}{2} \langle \mathbb{C}_\mathrm{v} \viscostrain, \viscostrain \rangle \\[-3pt]
 &+
 \tfrac{M}{2} \left\| \fluidmass_0 + \intmasssource - \DIV \flux_{\int} - \alpha_\mathrm{v} \, \mathrm{tr}\,\viscostrain - \alpha (\DIV \u - \mathrm{tr}\,\viscostrain) \right\|^2,
\end{align*}
where $\mathbb{C}_\mathrm{v}$ is a symmetric, uniformly positive definite, fourth-order tensor. 
 
Taking also into account the dissipation of the stored, visco-elastic energy, the dissipation potential is defined by
\begin{align*}
 \mathcal{D}_\mathrm{v}(\dot{\u},\dot{\flux}_{\int},\dotviscostrain) =& \tfrac{1}{2} \langle \permeability^{-1} \flux_{\int}, \flux_{\int} \rangle + \tfrac{1}{2} \langle \mathbb{C}_\mathrm{v}' \dotviscostrain, \dotviscostrain \rangle.
\end{align*}
For many materials a visco-elastic effect is only encountered, e.g., for the volumetric part~\cite{Coussy2004,Bociu2016}. Consequently, the symmetric, positive semi-definite, fourth-order tensor $\mathbb{C}_\mathrm{v}'$ may be singular. 

Again, restricting to isotropic materials, the fourth-order tensors $\mathbb{C}_\mathrm{v}$, $\mathbb{C}_\mathrm{v}'$ can be associated with Lam\'e parameters $\mu_\mathrm{v}> 0,\lambda_\mathrm{v}\geq0,\mu_\mathrm{v}'\geq 0,\lambda_\mathrm{v}'>0$, and corresponding bulk moduli $K_\mathrm{dr,v}= \tfrac{2\mu_\mathrm{v}}{d} + \lambda_\mathrm{v}$ and $K_\mathrm{dr,v}' = \tfrac{2\mu_\mathrm{v}'}{d} + \lambda_\mathrm{v}'$, via
\begin{align*}
 \mathbb{C}_\mathrm{v} \viscostrain &= 2\mu_\mathrm{v}\viscostrain + \lambda_\mathrm{v} \mathrm{tr}\,\viscostrain, \\
 \mathbb{C}_\mathrm{v}' \dotviscostrain &= 2\mu_\mathrm{v}'\dotviscostrain + \lambda_\mathrm{v}' \mathrm{tr}\,\dotviscostrain.
\end{align*}

Since $\viscostrain$ is interpreted as internal variable, the external work rates can be chosen as in the context of linear poro-elasticity
\begin{align*}
 \mathcal{P}_\mathrm{ext}(\dot{\u},\dot{\flux}_{\int},\dotviscostrain) = 
 \mathcal{P}_\mathrm{ext,mech}(\dot{\u})
 +
 \mathcal{P}_\mathrm{ext,fluid}(\dot{\flux}_{\int}).
\end{align*}

Finally, within the framework introduced in Sec.~\ref{section:general-model-biot}, the resulting evolution equation reads for current state $(\u,\flux_{\int},\viscostrain)$
\begin{align}
 \label{viscoporo-elasticity:gradient-flow-structure}
 (\dot{\u},\dot{\flux}_{\int}, \dotviscostrain) 
 \,=\, 
 \underset{(\v,\z,\testviscostrain)\in \dot{\mathcal{V}}\times\dot{\mathcal{Z}}_{\int}\times\mathcal{T}}{\mathrm{arg\,min}}\,
 \Big\{
 \mathcal{D}_\mathrm{v}(\v,\z,\testviscostrain) + \llangle \GRAD \mathcal{E}_\mathrm{v}(\u,\flux_{\int},\viscostrain), (\v,\z,\testviscostrain) \rrangle &\\[-1em]
 \nonumber
 - \mathcal{P}_\mathrm{ext}(\v,\z,\testviscostrain) \Big\}&.
\end{align}
The corresponding optimality conditions yield the model for linear poro-visco-elasticity as discussed by~\cite{Coussy2004}
\begin{align}
\label{visco-biot:mechanics}
\llangle \mathbb{C}(\eps{\u} - \viscostrain), \eps{\v} \rrangle - \alpha \llangle p, \DIV \v \rrangle &= \mathcal{P}_\mathrm{ext,mech}(\v) &&\forall\v\in\mathcal{V}_0,\\
\label{visco-biot:darcy}
\llangle \permeability^{-1}\flux, \z \rrangle - \llangle p, \DIV \z \rrangle &= \mathcal{P}_\mathrm{ext,fluid}(\z), &&\forall\z\in\mathcal{Z}_0,\\
\label{visco-biot:visco-elasticity}
\llangle \mathbb{C}_\mathrm{v}' \dotviscostrain + \mathbb{C}_\mathrm{v} \viscostrain - \alpha_\mathrm{v} p \, \mathbf{I}, \testviscostrain \rrangle 
&= \llangle \stress, \testviscostrain \rrangle &&\forall\testviscostrain\in\mathcal{T},\\
\label{visco-biot:mass-conservation}
\fluidmass- \fluidmass_0 + \DIV \flux_{\int} &= \intmasssource, &&\text{in }L^2(\Omega),
\end{align}
where we explicitly introduced the fluid pressure and total stress by
\begin{align}
 \label{poro-visco-interpretation-pressure}
 p
 &=\partial_m\mathcal{E}_\mathrm{v}
 = M\left(\fluidmass_0 + \intmasssource - \DIV \flux_{\int} - \alpha_\mathrm{v} \, \mathrm{tr}\,\viscostrain - \alpha (\DIV \u - \mathrm{tr}\,\viscostrain)\right), \\
 \label{poro-visco-interpretation-stress}
 \stress &= \partial_{\GRAD\u}\mathcal{E}_\mathrm{v} = \mathbb{C} (\eps{\u} - \viscostrain)  - \alpha p \, \mathbf{I}.
\end{align}
Well-posedness can be analyzed analogously to Lemma~\ref{lemma:well-posedness-poro}. 

\begin{lemma}[Well-posedness of linear poro-visco-elasticity]\label{lemma:poro-visco-elasticity-wellposedness}
 Let $\zeroviscostrain\in L^2(\Omega)^{d \times d}$, satisfying the compatibility condition for the deviatoric components
 \begin{align*}
  2\mu_\mathrm{v} \llangle \bm{\varepsilon}_\mathrm{v,0}^\mathrm{d}, \testviscostrain \rrangle = 2\mu \llangle \eps{\u_0}^\mathrm{d} - \bm{\varepsilon}_\mathrm{v,0}^\mathrm{d}, \testviscostrain \rrangle, \quad \forall \testviscostrain\in\mathcal{T},
 \end{align*}
 in case $\mu_\mathrm{v}'=0$, where the deviatoric components are defined according to~\eqref{definition:deviatoric-hydrostatic-visco-strain} and~\eqref{definition:deviatoric-hydrostatic-strain}. Then under the same regularity assumptions as in Lemma~\ref{lemma:well-posedness-poro} there exists a unique solution $(\u,\flux_{\int},\viscostrain)$ of~\eqref{visco-biot:mechanics}--\eqref{visco-biot:mass-conservation}, satisfying~\eqref{lemma:poro-regularity-1}--\eqref{lemma:poro-regularity-3} as for linear poro-elasticity. Additionally, for $\mu_\mathrm{v}'>0$ it holds 
 \begin{align}
  \viscostrain \in H^1(0,T;L^2(\Omega)),
 \end{align} 
 whereas for $\mu_\mathrm{v}'=0$, $\lambda_\mathrm{v}'>0$ it holds
 \begin{align}
  \viscostrain &\in  L^\infty(0,T;L^2(\Omega)), \\
  \mathrm{tr}\,\viscostrain &\in H^1(0,T;L^2(\Omega)).
 \end{align}
 For the flux $\flux$, mass $m$, pressure $p$ and stress $\stress$, associated with the states by~\eqref{accumulated-darcy-flux},~\eqref{accumulated-mass-conservation},~\eqref{poro-visco-interpretation-pressure} and~\eqref{poro-visco-interpretation-stress}, respectively, same regularity holds as for linear poro-elasticity, cf.~\eqref{lemma:poro-regularity-extra:start}--\eqref{lemma:poro-regularity-extra:end}.
\end{lemma}
 
\begin{proof}
 Non-homogeneous boundary conditions $\u_\Gamma$, $q_{\Gamma,n}$ and a non-zero, external source term $\masssource$ can be discussed as in the proof of Lemma~\ref{lemma:well-posedness-poro}. In the following, the focus is exclusively on the visco-elastic contribution, and only the case of zero boundary data and source terms is discussed.  

 The case $\mu_\mathrm{v}'>0$ follows analogously to the proof of Lemma~\ref{lemma:well-posedness-poro}; employ Thm.~\ref{thm:abstract-well-posedness-of-dne} with the partition of state variables $\{\u\}$ and $\{\flux_{\int},\viscostrain\}$.

 The second case, $\mu_\mathrm{v}'=0$, $\lambda_\mathrm{v}'>0$, requires a problem reformulation before applying Thm.~\ref{thm:abstract-well-posedness-of-dne}. As $\mathbb{C}_\mathrm{v}'$ is singular, $\mathcal{D}_\mathrm{v}(\cdot)$ is not coercive on $\bm{L}^2(\Omega)\times \mathcal{T}$. This can be fixed by decomposing strains. We introduce an orthogonal decomposition of visco-elastic strains into their hydrostatic and deviatoric parts. Let
 \begin{align}
 \label{definition:deviatoric-hydrostatic-visco-strain}
  \viscostrainhydrostatic&:= \tfrac{1}{d} \mathrm{tr}\, \viscostrain\in \mathcal{T}^\mathrm{h}, \qquad
  \viscostraindeviatoric:=  \viscostrain - \viscostrainhydrostatic \, \mathbf{I}\in \mathcal{T}^\mathrm{d},\\
  \nonumber
  \mathcal{T}^\mathrm{h} &:= \left\{ t\,\mathbf{I} \,\left|\, t\in L^2(\Omega) \right.\right\},\\
  \nonumber
  \mathcal{T}^\mathrm{d} &:= \left\{ \left. \testviscostrain \in L^2(\Omega)^{d\times d} \, \right| \, \mathrm{tr}\, \testviscostrain = 0 \right\}.
 \end{align}
 such that $\llangle \mathbb{C}_\mathrm{v} \viscostrain, \viscostrain \rrangle = 2\mu_\mathrm{v} \| \viscostraindeviatoric \|^2 + d^2\left( \tfrac{2\mu_\mathrm{v}}{d} + \lambda_\mathrm{v} \right) \|\viscostrainhydrostatic \|^2$. Similarly, we introduce
 \begin{align} 
 \label{definition:deviatoric-hydrostatic-strain}
  \epshydrostatic{\u} &:= \tfrac{1}{d} \mathrm{tr}\, \eps{\u}, \qquad
  \epsdeviatoric{\u}  := \eps{\u} - \epshydrostatic{\u}\, \mathbf{I}.
 \end{align}
 We re-interpret the energy and dissipation potential as functions of $\u,\flux_{\int},\viscostraindeviatoric,\viscostrainhydrostatic$ and their temporal changes, respectively,
 \begin{align*}
  \mathcal{E}_\mathrm{v}(\u,\flux_{\int},\viscostraindeviatoric,\viscostrainhydrostatic) 
  &= 
  \tfrac{1}{2} \left( 2\mu \| \epsdeviatoric{\u} - \viscostraindeviatoric\|^2 + 2\mu_\mathrm{v} \| \viscostraindeviatoric \|^2 \right) \\
  &\quad
  + \tfrac{d^2}{2} \left( K_\mathrm{dr} \| \epshydrostatic{\u} - \viscostrainhydrostatic\|^2 +  K_\mathrm{dr,v} \| \viscostrainhydrostatic \|^2 \right) \\
  &\quad
  + \tfrac{M}{2} \left\| \fluidmass_0 - \DIV \flux_{\int} - d\alpha_\mathrm{v} \viscostrainhydrostatic - d\alpha \left( \epshydrostatic{\u} - \viscostrainhydrostatic \right) \right\|^2, \\
  \mathcal{D}_\mathrm{v}(\dot{\u},\dot{\flux}_{\int},\dotviscostraindeviatoric,\dotviscostrainhydrostatic) 
  &= \tfrac{1}{2} \llangle \permeability^{-1} \dot{\flux}_{\int}, \dot{\flux}_{\int} \rrangle + \tfrac{d^2 \lambda_\mathrm{v}'}{2} \| \dotviscostrainhydrostatic\|^2.
 \end{align*}
 where $K_\mathrm{dr}=\tfrac{2\mu}{d} + \lambda$ and $K_\mathrm{dr,v}=\tfrac{2\mu_\mathrm{v}}{d} + \lambda_\mathrm{v}$. The external work rate $\mathcal{P}_\mathrm{ext}$ is independent of $\dotviscostrain$, and hence, remains unaltered after re-interpretation.
  
 The dissipation potential $\mathcal{D}_\mathrm{v}$ defines a norm for $(\dot{\flux}_{\int},\dotviscostrainhydrostatic)$ and hence is coercive on $\dot{\mathcal{Z}}_{\int} \times \mathcal{T}^\mathrm{h}$. In order to apply Thm.~\ref{thm:abstract-well-posedness-of-dne}, we decompose $\mathcal{E}_\mathrm{v}$ into a strictly convex part, depending only on the complementary part, $(\u,\viscostraindeviatoric)$, and a convex remainder. Let
 \begin{align*}
  \Psi_\mathrm{v}(\dot{\flux}_{\int},\viscostrainhydrostatic) &:= \tfrac{1}{2} \llangle \permeability^{-1} \dot{\flux}_{\int}, \dot{\flux}_{\int} \rrangle + \tfrac{d^2 \lambda_\mathrm{v}'}{2} \| \dotviscostrainhydrostatic\|^2, \\
  \mathcal{E}_\mathrm{v,1}(\u,\viscostraindeviatoric)
  &:=
  \tfrac{1}{2} \left( 2\mu \| \epsdeviatoric{\u} - \viscostraindeviatoric\|^2 + 2\mu_\mathrm{v} \| \viscostraindeviatoric \|^2 \right), \\
  \mathcal{E}_\mathrm{v,2}(x_1, x_2, x_3)
  &:=
  \tfrac{d^2}{2} \left( K_\mathrm{dr} \|x_1 \|^2 +  K_\mathrm{dr,v} \| x_2 \|^2 \right) + \tfrac{M}{2} \left\| x_3 \right\|^2, \\
  \Lambda_\mathrm{v}(\u,\flux_{\int},\viscostraindeviatoric,\viscostrainhydrostatic)
  &:=
  \left[ \epshydrostatic{\u} - \viscostrainhydrostatic, \, \viscostrainhydrostatic, \, \fluidmass_0 - \DIV \flux_{\int} - d\alpha_\mathrm{v} \viscostrainhydrostatic - d\alpha \left( \epshydrostatic{\u} - \viscostrainhydrostatic \right) \right],
 \end{align*}
 satisfying
 \begin{align*}
  \mathcal{E}_\mathrm{v}\left(\u,\flux_{\int},\viscostraindeviatoric,\viscostrainhydrostatic\right) 
  =
  \mathcal{E}_\mathrm{v,1}\left(\u,\viscostraindeviatoric\right)
  +
  \mathcal{E}_\mathrm{v,2}\left(\Lambda_\mathrm{v}\left(\u,\flux_{\int},\viscostraindeviatoric,\viscostrainhydrostatic\right)\right).
 \end{align*}
 Finally,~\eqref{viscoporo-elasticity:gradient-flow-structure} takes the form
  \begin{align*}
  &(\dot{\u},\dot{\flux}_{\int}, \dotviscostraindeviatoric, \dotviscostrainhydrostatic) \,=\, 
 \underset{(\v,\z,\testviscostrain^\mathrm{d},t^\mathrm{h})\in \mathcal{V}_0 \times \mathcal{Z}_0 \times \mathcal{T}^\mathrm{d} \times \mathcal{T}^\mathrm{h}}{\mathrm{arg\,min}}\,
 \bigg\{
 \Psi_\mathrm{v}(\v,t^\mathrm{h}) \\
 &\qquad 
 + \Big\langle \GRAD \mathcal{E}_\mathrm{v}\left(\u,\flux_{\int},\viscostraindeviatoric,\viscostrainhydrostatic\right) , (\v,\z,\testviscostrain^\mathrm{d},t^\mathrm{h}) \Big\rangle -
 \mathcal{P}_\mathrm{ext,mech}(\u)
 -
 \mathcal{P}_\mathrm{ext,fluid}(\z)
 \bigg\}.
 \end{align*}
 complying with the abstract structure in Thm.~\ref{thm:abstract-well-posedness-of-dne}. Additionally, the properties~(P1)--(P6) are satisfied. Hence, by Thm.~\ref{thm:abstract-well-posedness-of-dne}, we obtain the existence of a unique solution to~\eqref{viscoporo-elasticity:gradient-flow-structure}. Regularity follows along the lines of Lemma~\ref{lemma:well-posedness-poro}.
\end{proof}

\begin{remark}[Purely visco-elastic deformation] As mentioned above, the mechanical displacement may be assumed to be purely visco-elastic, i.e., $\viscostrain=\eps{\u}$. This corresponds to $\mu,\lambda\rightarrow \infty$, while $\mathbb{C}_\mathrm{v}$ and $\mathbb{C}_\mathrm{v}'$ remain finite, i.e., $\mathbb{C}$ acts as penalty parameter. In the limit, following from Lemma~\ref{lemma:poro-visco-elasticity-wellposedness}, $\eps{\u}$ inherits the regularity of $\viscostrain$. This yields an alternative approach to~\cite{Showalter2000,Bociu2016} for the analysis of quasi-static, linear  poro-elasticity with secondary consolidation.	
\end{remark}

\section{Non-linear poro-elasticity under infinitesimal strains as generalized gradient flow}\label{section:non-linear-biot-linear-coupling}

In many applications linear constitutive laws are not sufficient in order to describe the physical behavior of a fluid-saturated, deformable porous medium -- even when restricted to the infinitesimal strain regime~\cite{vanDerKnaap1959,Biot1973}; similar to soil mechanics~\cite{Hardin1972} and solid mechanics~\cite{Temam2000}. The gradient flow modeling framework introduced in Sec.~\ref{section:general-model-biot} allows for involving a variety of non-linear relationships. In the following, we consider a non-linear stress-strain relationship together with a non-linear fluid compressibility, assuming infinitesimal strains. Based on the gradient flow structure of the resulting models, we analyze their well-posedness along the lines of Sec.~\ref{section:linear-biot-gradient-flow}. This setting has been also studied numerically by~\cite{Borregales2018}.

We generalize the primal formulation of linear poro-elasticity, cf.\ Sec.~\ref{section:linear-biot:primal-formulation}: We choose the same state variables, $(\u,\flux_{\int})$, living in the same function spaces as before. In the spirit of hyperelasticity~\cite{Temam2000}, we consider energies
\begin{align}
 \label{non-linear-poro-elasticity:energy:start}
 \mathcal{E}_\mathrm{nl}(\u,\flux_{\int}) &:= \mathcal{E}_\mathrm{nl,eff}(\u) + \mathcal{E}_\mathrm{nl,fluid}(\u,\flux_{\int}), \\
 \mathcal{E}_\mathrm{nl,eff}(\u) &:= \int_\Omega W(\eps{\u}) \, dx, \\
 \mathcal{E}_\mathrm{nl,fluid}(\u,\flux_{\int})
 & := \int_\Omega \int_0^{\fluidmass_0 + \intmasssource - \DIV \flux_{\int} - \alpha \DIV \u} b^{-1}(s)\, ds \, dx,
 \label{non-linear-poro-elasticity:energy:end}
\end{align}
for some convex strain energy density $W$ and an invertible function $b$. This choice results in the generalized, (implicit) definition of the fluid pressure and mechanical stress, cf.~\eqref{linear-biot-interpretation-m}--\eqref{linear-biot-interpretation-stress}, via
\begin{align}
\label{non-linear-biot-interpretation-m}
 m &= b(p) + \alpha \DIV \u,\\
\label{non-linear-biot-interpretation-stress}
 \stress &= \frac{\partial W(\eps{\u})}{\partial \bm{\varepsilon}(\u)} - \alpha p \,\mathbf{I},
\end{align}
which follows directly from~\eqref{thermodynamic-interpretation}. Examples for strain energy densities $W$ and the corresponding effective stress $\stress_\mathrm{eff}$, cf.~\eqref{thermodynamic-interpretation-2}, are given in Tab.~\ref{table:examples-non-linear-constitutive-laws}; in principle also non-convex potentials~\cite{Temam2000} could be employed, but are not considered here. For $b$ we allow arbitrary invertible, increasing, Lipschitz continuous functions with $b(0)=0$, generalizing the linear compressibility law $b(p)=\tfrac{1}{M}p$, employed in linear poro-elasticity. We refer to~\cite{Borregales2018} for possible choices.

\begin{table}[h!]
\centering
\def\arraystretch{2}
{\footnotesize
 \begin{tabular}{ll}
  \hline
  Linear elasticity (cf.\ Sec.~\ref{section:linear-biot-gradient-flow}) 
    & {\arraycolsep=1pt$\begin{array}{rcl}
       W(\eps{\u})          & = & \tfrac{1}{2} \left( 2\mu |\eps{\u}|^2 + \lambda (\DIV \u)^2 \right) \\[-0.75em]
       \stress_\mathrm{eff} & = & 2\mu \eps{\u} + \lambda (\DIV \u) \,\mathbf{I}
      \end{array}$}\\
  \hline
  Non-linear compressibility (cf.~\cite{Biot1973,Borregales2018}) 
    & {\arraycolsep=1pt$\begin{array}{rcl}
       W(\eps{\u}) & = & \tfrac{1}{2} \left( 2\mu |\eps{\u}|^2 + \int_0^{\DIV \u} l(s)\, ds \right),\ l\text{ increasing},\ l(0)=0 \\[-0.75em]
       \stress_\mathrm{eff} & = & 2\mu \eps{\u} + l(\DIV \u) \,\mathbf{I}
      \end{array}$}\\
  \hline
  Non-linear shear modulus  (cf.~\cite{Hardin1972,Barucq2005})
    & {\arraycolsep=1pt$\begin{array}{rcl}
       W(\eps{\u}) & = & \int_0^{|\eps{\u}|} sf(s)\, ds + \tfrac{\lambda}{2} (\DIV \u)^2,\ f\text{ unif. pos. and non-decr.}\\[-0.75em]
       \stress_\mathrm{eff} & = & f(|\eps{\u}|) \eps{\u} + \lambda (\DIV \u) \,\mathbf{I}
       \end{array}$} \\
  \hline
      Simple visco-elasto-plasticity (cf.~\cite{Temam2000})
    & {\arraycolsep=1pt$\begin{array}{rcl}
       W(\eps{\u}) & = & \int_0^{|\epsdeviatoric{\u}|} sf(s)\, ds + K_\mathrm{dr} (\DIV \u)^2, \\[0.25em]
       f(s) & = & \left\{ 
	{\arraycolsep=4pt\begin{array}{ll} 
	 \\[-3em] 
	 2\mu \epsdeviatoric{\u}, & 2\mu|\epsdeviatoric{\u}| < K \\[-0.75em]
	 2\mu + \tfrac{K}{|\epsdeviatoric{\u}|} \epsdeviatoric{\u}, & \mathrm{else}
	\end{array}} \right.\\[0.25em]
       \stress_\mathrm{eff} & = & \stress_\mathrm{eff}^\mathrm{d} + K_\mathrm{dr} (\DIV \u) \,\mathbf{I} \\[0.25em]
       \stress_\mathrm{eff}^\mathrm{d} & = & \left\{ 
	{\arraycolsep=4pt\begin{array}{ll} 
	 \\[-3em] 
	 2\mu \epsdeviatoric{\u}, & |\stress_\mathrm{eff}^\mathrm{d}|=2\mu|\epsdeviatoric{\u}| < K \\[-0.75em]
	 \left( 2\mu + \tfrac{K}{|\epsdeviatoric{\u}|} \right) \epsdeviatoric{\u}, & \mathrm{else, i.e., }\, |\stress_\mathrm{eff}^\mathrm{d}|\geq K
	\end{array}} \right.
      \end{array}$} \\ \\[-1.75em]
  \hline \\[-1em]
 \end{tabular}}
 \caption{\label{table:examples-non-linear-constitutive-laws} Examples for the strain energy density $W$ used for the definition of $\mathcal{E}_\mathrm{nl,eff}$ and the corresponding effective stress $\stress_\mathrm{eff}$.}
\end{table}
Other than that, we employ the external work rate, the dissipation potential as for linear poro-elasticity, cf.\ Sec.~\ref{section:linear-biot:primal-formulation}. Inserting all components into the gradient flow framework from Sec.~\ref{section:general-model-biot}, yields the final model
\begin{align}
 \label{non-linear-biot:gradient-flow-structure}
 (\dot{\u},\dot{\flux}_{\int})
 &= \underset{(\bm{v},\z)\in\dot{\mathcal{V}}\times\dot{\mathcal{Z}}_{\int}}{\text{arg\,min}}\, 
 \Big\{ 
 \mathcal{D}_\mathrm{fluid}(\z)
 + \llangle \GRAD \mathcal{E}_\mathrm{nl}(\u,\flux_{\int}), (\v,\z) \rrangle 
 - \mathcal{P}_\mathrm{ext}(\v,\z) \Big\}.
\end{align} 
Considering, e.g., constant shear modulus and non-linear compressibility, cf.\ Tab.~\ref{table:examples-non-linear-constitutive-laws}, and the fluid pressure as defined by~\eqref{non-linear-biot-interpretation-m}, the corresponding optimality conditions are consistent with the model considered by~\cite{Borregales2018}
\begin{align*}
 2\mu \llangle \eps{\u}, \eps{\v} \rrangle + \llangle l(\DIV \u), \DIV \v \rrangle -\alpha \llangle p, \DIV \v \rrangle &= \mathcal{P}_\mathrm{ext,mech}(\v) &&\forall\v\in\mathcal{V}_0,\\
 \llangle \permeability^{-1}\dot{\flux}_{\int}, \z \rrangle - \llangle p, \DIV \z \rrangle &= \mathcal{P}_\mathrm{ext,fluid}(\z), &&\forall\z\in\mathcal{Z}_0,\\
 b(p) + \alpha \DIV \u + \DIV \flux_{\int} &= \fluidmass_0 + \intmasssource, && \text{in } L^2(\Omega).
\end{align*}

Well-posedness of non-linear poro-elasticity described by the generalized gradient flow~\eqref{non-linear-biot:gradient-flow-structure} follows directly with same argumentation as in the case of linear poro-elasticity.

\begin{lemma}[Well-posedness for non-linear poro-elasticity]\label{lemma:non-linear-manuel:well-posedness-poro}
 Let $\mathcal{E}_\mathrm{nl}$, $W$ and $b$ as in~\eqref{non-linear-poro-elasticity:energy:start}--\eqref{non-linear-poro-elasticity:energy:end}. Furthermore, let $W$ be strongly convex in $\eps{\u}$ with $W(\eps{\bm{0}})=0$ and $\partial_{\bm{\varepsilon}}W(\eps{\bm{0}})=\bm{0}$, and let $b^{-1}$ be uniformly increasing with $b^{-1}(0)=0$. Consider homogeneous boundary conditions and conservation of mass, i.e., $\u_\Gamma = \bm{0}$, $q_\mathrm{\Gamma,n}=0$, and $\masssource=0$. Other than that assume regularity as in Lemma~\ref{lemma:well-posedness-poro}. And assume finite energy $\mathcal{E}_\mathrm{nl}(\u_0,\bm{0})<\infty$ and the compatibility condition
 \begin{align*}
  \llangle \partial_{\u} \mathcal{E}_\mathrm{nl}(\u_0,\bm{0}), \v \rrangle = \mathcal{P}_\mathrm{ext,mech}(\v),\quad \forall \v \in \mathcal{V}_0.
 \end{align*}
 Then there exists a unique solution $(\u,\flux_{\int})$ of~\eqref{non-linear-biot:gradient-flow-structure}. Its regularity is the same as in the case of linear poro-elasticity, cf.~\eqref{lemma:poro-regularity-1}--\eqref{lemma:poro-regularity-3}.
\end{lemma}

\begin{proof} 
 The proof goes along the lines of Lemma~\ref{lemma:well-posedness-poro}. We choose the partition $\{\u\}$ and $\{\flux_{\int}\}$ and define
 \begin{align*}
  \Psi_\mathrm{nl}(\dot{\flux}_{\int}) &:= \mathcal{D}_\mathrm{fluid}(\dot{\flux}_{\int}) 
    & \mathcal{E}_\mathrm{nl,1}(\u) &:= \mathcal{E}_\mathrm{nl,eff}(\u), \\
  \mathcal{E}_\mathrm{nl,2}(m) &:= \int_\Omega \int_0^m b^{-1}(s) \, ds\, dx,
    & \Lambda_\mathrm{nl}(\u,\flux_{\int}) &:= \fluidmass_0 - \DIV \flux_{\int} - \alpha \DIV \u.
 \end{align*}
 Employing this notation,~\eqref{non-linear-biot:gradient-flow-structure} complies with the abstract structured discussed in Thm.~\ref{thm:abstract-well-posedness-of-dne}. Furthermore, $W$ strongly convex and $b$ uniformly increasing guarantee growth conditions for the energy functionals, cf.~(P4). The remaining properties~(P1)--(P6) are simple to verify. Finally, the result is a consequence of Thm.~\ref{thm:abstract-well-posedness-of-dne}
\end{proof}

\begin{remark}
 All models, listed in Tab.~\ref{table:examples-non-linear-constitutive-laws}, satisfy the assumptions from Lemma.~\ref{lemma:non-linear-manuel:well-posedness-poro}, assuming $\mu>0$.
\end{remark}


\section{Extensions of Darcy flow in poro-elastic media as generalized gradient flow}\label{section:non-newtonian-fluids}

In the presence of non-Newtonian, non-laminar or transitional flow between boundaries, the linear Darcy law is not appropriate anymore to relate the volumetric flux with the fluid pressure gradient. For that reason, extensions of Darcy's law have been established in the literature. In the following, we discuss extensions for non-Newtonian flow, Darcy-Forchheimer flow, and Darcy-Brinkman flow. We incorporate such in the gradient flow modelling framework by an adequate, alternative choice of a dissipation potential $\mathcal{D}_\mathrm{\star,fluid}$ corresponding to viscous flow. By keeping previous choices for energy functionals $\mathcal{E}$, external work rates $\mathcal{P}_\mathrm{ext}$ etc., and preserving the modelling ansatz
\begin{align}
 \label{non-newtonian-fluid:gradient-flow-structure}
 (\dot{\u},\dot{\flux}_{\int})
 &= \underset{(\bm{v},\z)\in\dot{\mathcal{V}}\times\dot{\mathcal{Z}}_{\int}}{\text{arg\,min}}\, 
 \Big\{ 
 \mathcal{D}_\mathrm{\star,fluid}(\z)
 + \llangle \GRAD \mathcal{E}(\u,\flux_{\int}), (\v,\z) \rrangle 
 - \mathcal{P}_\mathrm{ext}(\v,\z) \Big\},
\end{align}
the previously discussed poro-elasticity models get simply enhanced by the corresponding extension of Darcy's law.

\paragraph{Non-Newtonian flow:} Explicitly distinguishing between the permeability $\permeability$ and the fluid shear viscosity $\nu$ (not as in the previous sections), Darcy's law with potentially variable viscosity reads
\begin{align}
\label{non-newtonian:darcy}
 \nu(|\flux|)\flux = -\permeability (\GRAD p - \gext).
\end{align}
Common, constitutive shear viscosity models employed in the literature, cf., e.g.,~\cite{Owens2002}, are given in Tab.~\ref{table:non-linear-viscosity-models}. For non-constant viscosity we assume an isotropic, uniformly bounded permeability $\permeability=\kappa \,\mathbf{I}$ -- a commmon assumption in modelling non-Newtonian fluid flow in porous media, cf., e.g.,~\cite{Ambartsumyan2018}.
\begin{table}[h!]
\centering
\def\arraystretch{2.5}
{
\footnotesize
 \begin{tabular}{ll}
 \hline
  Newtonian fluid & $\nu(s) = \nu_\infty$ \\[0.2em]
 \hline 
  Carreau model   & $\nu(s) = \nu_\infty + \frac{\nu_0 - \nu_\infty}{(1+K_\mathrm{f} |s|^2)^{\frac{2-r}{2}}}$ \\[0.8em]
 \hline
  Cross model     & $\nu(s) = \nu_\infty + \frac{\nu_0 - \nu_\infty}{1 + K_\mathrm{f} |s|^{2-r}}$ \\[0.2em]
 \hline
  Power law       & $\nu(s) = \frac{1}{K_\mathrm{f} |s|^{2-r}}$ \\[0.2em]
 \hline
 \end{tabular}}
 \vspace{0.2cm}
 \caption{\label{table:non-linear-viscosity-models} Constitutive models for the fluid shear viscosity $\nu$~\cite{Owens2002}; let $0<\nu_\infty<\nu_0$, $r\in(1,2)$ and $K_\mathrm{f}>0$.}
\end{table}
The corresponding dissipation potential to be used in~\eqref{non-newtonian-fluid:gradient-flow-structure} is given by
\begin{align}
 \label{non-linear-darcy:dissipation-potential:non-newtonian}
 \mathcal{D}_\mathrm{\nu,fluid}(\flux) = \int_\Omega \kappa^{-1}\int_0^{|\flux|} s\nu(s)\, ds \, dx.
\end{align}

\paragraph{Darcy-Forchheimer flow:}
For flow in porous media with an elevated Reynolds number, Darcy's law is enhanced by the so-called \textit{Forchheimer term}, accounting for inertial effects. The resulting non-linear, constitutive relation reads
\begin{align}
\label{non-linear-darcy:forchheimer}
 \nu\flux + \permeability F |\flux|\flux = -\permeability (\GRAD p - \gext),
\end{align}
where $F\geq 0$ denotes the Forchheimer number~\cite{Forchheimer1901}. The corresponding dissipation potential to be used in~\eqref{non-newtonian-fluid:gradient-flow-structure} is given by
\begin{align}
 \label{non-linear-darcy:dissipation-potential:forchheimer}
 \mathcal{D}_\mathrm{F,fluid}(\flux) = \frac{\nu}{2}\int_\Omega \permeability^{-1}\flux \cdot \flux \, dx + \frac{F}{2} \int_\Omega |\flux|^3 \, dx.
\end{align}

\paragraph{Darcy-Brinkman flow:} 
For transitional flow between boundaries, Darcy's law may be enhanced by the so-called \textit{Brinkman term}. The resulting linear extension of Darcy's law reads
\begin{align}
\label{non-linear-darcy:brinkman}
 \nu\flux - \permeability \DIV (\nu_\mathrm{eff} \GRAD \flux)= -\permeability (\GRAD p - \gext),
\end{align}
where $\nu_\mathrm{eff}\geq 0$ denotes the effective viscosity related to the viscous drag effects~\cite{Brinkman1949}. The corresponding dissipation potential to be used in~\eqref{non-newtonian-fluid:gradient-flow-structure} is given by
\begin{align}
 \label{non-linear-darcy:dissipation-potential:brinkman}
 \mathcal{D}_\mathrm{B,fluid}(\flux) 
 = \frac{\nu}{2} \int_\Omega \permeability^{-1}\flux \cdot \flux \, dx 
 + \frac{\nu_\mathrm{eff}}{2} \int_\Omega |\GRAD \flux|^2 \, dx.
\end{align}

Independent of the specific choices of the energy functionals etc., well-posedness can be again discussed employing the abstract well-posedness result, cf.\ Thm.~\ref{thm:abstract-well-posedness-of-dne}.
 
\begin{lemma}[Well-posedness for extensions of linear Darcy flow in poro-elastic media]\label{lemma:existence-non-newtonian-fluids}
 For $p\geq 1$, let $H^p_\mathrm{div}(\Omega):= \left\{ \left.\z \in L^p(\Omega)^d \,\right|\, \DIV \z \in L^2(\Omega) \right\}$.
 Let the energy functional $\mathcal{E}$ be as in~\eqref{non-linear-poro-elasticity:energy:start} with $W$ and $b$ as in Lemma~\ref{lemma:non-linear-manuel:well-posedness-poro}.
 And assume finite energy $\mathcal{E}(\u_0,\bm{0})<\infty$ and the compatibility condition
 \begin{align*}
  \llangle \partial_{\u} \mathcal{E}(\u_0,\bm{0}), \v \rrangle = \mathcal{P}_\mathrm{ext,mech}(\v),\quad \forall \v \in \mathcal{V}_0.
 \end{align*}
 Consider homogeneous boundary conditions and conservation of mass, i.e., $\u_\Gamma = \bm{0}$, $q_\mathrm{\Gamma,n}=0$, and $\masssource=0$. Other than that assume regularity as in Lemma~\ref{lemma:well-posedness-poro}.
 
 \begin{itemize}
  \item[(1)]  Non-Newtonian fluid flow:
 Let $\nu=\nu(s)$ denote the fluid shear viscosity model. Let $s\mapsto s\nu(s)$ be non-decreasing, and assume there exists a $p\in(1,\infty)$ satisfying $\nu(s)\gtrsim s^{p-2}$, $s> 0$. Then there exists a solution $(\u,\flux_{\int})$  to~\eqref{non-newtonian-fluid:gradient-flow-structure} with the dissipation potential~\eqref{non-linear-darcy:dissipation-potential:non-newtonian}, satisfying 
 \begin{align*}
  \u &\in L^\infty(0,T;\bm{H}^1(\Omega)),\\
  \flux_{\int} &\in H^1(0,T;L^p(\Omega)^d)\cap L^\infty\left(0,T;H^p_\mathrm{div}(\Omega)\right).
 \end{align*}
 In case the energy $\mathcal{E}$ is quadratic, the solution is unique.\\
 
 \item[(2)] Darcy-Forchheimer flow:
 There exists a solution $(\u,\flux_{\int})$ to~\eqref{non-newtonian-fluid:gradient-flow-structure} with the dissipation potential~\eqref{non-linear-darcy:forchheimer}, satisfying
 \begin{align*}
  \u &\in L^\infty(0,T;\bm{H}^1(\Omega)),\\
  \flux_{\int} &\in H^1(0,T;L^3(\Omega)^d)\cap L^\infty\left(0,T;H^3_\mathrm{div}(\Omega)\right).
 \end{align*}
 In case the energy $\mathcal{E}$ is quadratic, the solution is unique.\\
 
 \item[(3)] Darcy-Brinkman flow:
 There exists a unique solution $(\u,\flux_{\int})$ to~\eqref{non-newtonian-fluid:gradient-flow-structure} with the dissipation potential~\eqref{non-linear-darcy:brinkman}, satisfying
 \begin{align*}
  \u &\in L^\infty(0,T;\bm{H}^1(\Omega)),\\
  \flux_{\int} &\in H^1(0,T;\bm{H}^1(\Omega)).
 \end{align*}

 \end{itemize}

\end{lemma}

\begin{proof}
 The result is a direct consequence of Thm.~\ref{thm:abstract-well-posedness-of-dne}, and the proof is analogous to the proofs of Lemma~\ref{lemma:well-posedness-poro} and Lemma~\ref{lemma:non-linear-manuel:well-posedness-poro}. It suffices to verify~(P3) for the different dissipation potentials $\mathcal{D}_\mathrm{\star,fluid}$.
 
 \paragraph{Non-Newtonian fluid flow:} The dissipation potential $\mathcal{D}_\mathrm{\nu,fluid}$ is coercive wrt.\ $L^p(\Omega)^d$ since it holds
 \begin{align*}
  \mathcal{D}_\mathrm{\nu,fluid}(\flux)
  \gtrsim \int_\Omega \int_0^{|\flux|} s\nu(s)\,ds\,dx 
  \gtrsim \int_\Omega \int_0^{|\flux|} s^{p-1} \,ds\,dx
  \gtrsim \| \flux \|_{L^p(\Omega)}^p.
 \end{align*}
 Furthermore, $\mathcal{D}_\mathrm{\nu,fluid}$ is convex as composition of two convex maps; indeed, $\flux\mapsto |\flux|$ is convex, and $x\mapsto \int_0^x s\nu(s) \, ds$ is convex, since $s\mapsto s\nu(s)$ is increasing. All in all, (P3) is fulfilled.
 
 \paragraph{Darcy-Forchheimer flow:} 
 The dissipation potential $\mathcal{D}_\mathrm{F,fluid}$ is by construction coercive wrt.\ $L^3(\Omega)^d$. As sum of convex functions, it is convex. Hence, (P3) is fulfilled. 
 
 \paragraph{Darcy-Brinkman flow:}
 The dissipation potential $\mathcal{D}_\mathrm{B,fluid}$ defines a norm on $\bm{H}^1(\Omega)$. Hence, (P3) is fulfilled. Furthermore, $\mathcal{D}_\mathrm{B,fluid}$ is quadratic, and uniqueness of solutions to~\eqref{non-newtonian-fluid:gradient-flow-structure} is guaranteed.
 
\end{proof}

\begin{remark}[Well-posedness for different viscosity models from Tab.~\ref{table:non-linear-viscosity-models}]
 All models mentioned in Tab.~\ref{table:non-linear-viscosity-models} satisfy the assumptions of Lemma~\ref{lemma:existence-non-newtonian-fluids}. For fluid shear viscosities modelled by the Carreau model,  the Cross model, as well as for Newtonian fluids, one can choose $p=2$, since it holds $\nu(s)\geq \nu_\infty $, $s>0$. For the power law, it holds $\nu(s)\gtrsim s^{r-2}$, $s>0$. Hence, only reduced regularity is obtained with $p=r\in(1,2)$.
\end{remark}

\section{Thermo-poro-elasticity as generalized gradient flow}\label{section:thermo-poro-elasticity-gradient-flow}

Non-isothermal fluid flow in deformable porous media has in general a strongly non-linear, coupled character, compared to linear poro-elasticity. Even under the hypothesis of infinitesimal strains, three non-linearities may occur, cf., e.g.~\cite{Coussy2004}: (i) thermal convection, coupled to the fluid problem; (ii) non-linear viscous dissipation, associated with Darcy's law, acting as a heat source; (iii) and a temperature weighted time derivative of the total entropy in the energy equation. In certain situations, those non-linearities can be neglected~\cite{Coussy2004}: (i) for a small P\'eclet number, which quantifies the heat convectively transported by the fluid in comparison with the heat supplied by diffusion through the porous medium; (ii) for small Brinkman number, which quantifies the order of magnitude of the heat source due to viscous dissipation in comparison with heat supplied by conduction; and (iii) small variations of temperature. Under assumptions~(ii) and~(iii), the model for linear thermo-poro-elasticity with non-linear convection has been  derived using homogenization~\cite{Brun2018,Vanduijn2017}. For a discussion of the general, fully non-linear model we refer to~\cite{Coussy2004}.

Assuming all three non-linear effects~(i)--(iii) can be neglected, allows  for linearizing the general thermo-poro-elasticity model. Using mechanical displacement $\u$, fluid pressure $p$ and temperature $T$ as primary variables, the linear, reduced thermo-poro-elasticity model including linearized fluid state equations reads
\begin{align}
\label{thermo-biot:mechanics}
 -\DIV \left[ \mathbb{C} \eps{\u} - \alpha p \,\mathbf{I} - 3\alpha_\mathrm{T} K_\mathrm{dr} T \,\mathbf{I} \right] &= \fext, \\
%
\label{thermo-biot:mass}
 \tfrac{1}{M} \dot{p} + \alpha \DIV \dot{\u} - 3 \alpha_\phi \dot{T} - \DIV \left(\permeability (\GRAD p - \gext)\right) &= \masssource,\\ 
\label{thermo-biot:energy}
 C_\mathrm{d} \dot{T} + 3\alpha_\mathrm{T} T_0 K_\mathrm{dr} \DIV \dot{\u} - 3\alpha_\phi T_0 \dot{p} -  \DIV \left(\conductivity \GRAD T\right) &= T_0 \entropysource,
\end{align}
subject to suitable boundary and initial conditions, cf., e.g.,~\cite{Coussy2004}. Here, $K_\mathrm{dr}$ denotes the bulk modulus, $\alpha_\mathrm{T}$ is the Biot coefficient associated with thermal effects, $\alpha_\phi$ governs the pressure-temperature coupling of the fluid, $C_\mathrm{d}$ is the total volumetric heat capacity, $T_0$ is a constant reference temperature, and $\conductivity$ denotes the thermal conductivity.
Other than that, same notation is used as in the previous sections.

The linearized thermo-poro-elasticity model~\eqref{thermo-biot:mechanics}--\eqref{thermo-biot:energy} has a similar structure as Biot's consolidation model. In the following, we provide a generalized gradient flow formulation of~\eqref{thermo-biot:mechanics}--\eqref{thermo-biot:energy}, which will be later exploited in the context of robust splitting schemes, cf.\ Sec.~\ref{section:splitting-thermo-poro-elasticity}.

In the context of the abstract gradient flow modelling framework introduced in Sec.~\ref{section:general-model-biot}, we choose $(\u,\fluidmass,S)$ as state variables, i.e., the mechanical displacement, the fluid content and the total entropy. Motivated by~\eqref{thermo-biot:mass}--\eqref{thermo-biot:energy}, the latter two will be related to $(\u,p,T)$ by
\begin{align}
\label{thermo-biot:effective-mass}
 \fluidmass &= \tfrac{1}{M} p  + \alpha \DIV \u - 3 \alpha_\phi T, \\ 
\label{thermo-biot:effective-entropy}
 S &= \tfrac{C_\mathrm{d}}{T_0} T + 3\alpha_\mathrm{T} K_\mathrm{dr} \DIV \u - 3\alpha_\phi p.
\end{align}

Changes of states are associated to $(\dot{\u},\flux,\entropyflux)$, i.e., the rate of mechanical displacement, the volumetric flux and the entropy flux, by conservation of volume and balance of entropy
\begin{align*}
 \dot{\fluidmass} + \DIV\flux&= \masssource, \\
 \dot{S} + \DIV\entropyflux &= \entropysource.
\end{align*}
Under above, linearizing assumptions, the entropy flux can be identified with the heat flux scaled by $T_0^{-1}$ such that due to Darcy's law and Fourier's law, we obtain
\begin{align}
\label{thermo-biot:darcy}
 \flux &= -\permeability(\GRAD p - \gext), \\
\label{thermo-biot:fourier}
 \entropyflux &= -\frac{\conductivity}{T_0} \GRAD T.
\end{align}
Focussing on the inherent gradient flow structure of linearized thermo-poro-elasticity, we omit specifying the regularity of all variables; we refer to the formal function spaces including essential boundary conditions as defined in Sec.~\ref{section:general-model-biot}. Generalizing linear poro-elasticity, a natural choice for the dissipation potential is 
\begin{align*}
 \mathcal{D}_\mathrm{th}(\flux,\entropyflux) &= \frac{1}{2} \llangle \permeability^{-1} \flux, \flux \rrangle + \frac{1}{2} \llangle T_0 \conductivity^{-1} \entropyflux, \entropyflux \rrangle.
\end{align*}

The Helmholtz free energy associated to linearized thermo-poro-elasticity for given state is defined by
\begin{align*}
 \mathcal{E}_\mathrm{th}(\u,\fluidmass,S) &:= \mathcal{E}_\mathrm{eff}(\u) + \mathcal{E}_\mathrm{th,fluid}(\u,\fluidmass,S)
\end{align*}
with $\mathcal{E}_\mathrm{eff}$ defined as for linear poro-elasticity and the fluid contribution
\begin{align*}
 \mathcal{E}_\mathrm{th,fluid}(\u,\fluidmass,S) &:= 
 \frac{1}{2} \llangle \begin{bmatrix} \tfrac{1}{M} & -3\alpha_\phi \\ -3\alpha_\phi & \tfrac{C_\mathrm{d}}{T_0} \end{bmatrix}^{-1} \begin{bmatrix} \fluidmass - \alpha \DIV \u \\ S - 3\alpha_\mathrm{T} K_\mathrm{dr} \DIV \u \end{bmatrix}, \begin{bmatrix} \fluidmass - \alpha \DIV \u \\ S - 3\alpha_\mathrm{T} K_\mathrm{dr} \DIV \u \end{bmatrix} \rrangle.
\end{align*}
Using~\eqref{thermo-biot:effective-mass}--\eqref{thermo-biot:effective-entropy}, $\mathcal{E}_\mathrm{th,fluid}$ can be also written as function of the primary variables
\begin{align*}
 \mathcal{E}_\mathrm{th,fluid}(\u,p,T) &= 
 \frac{1}{2} \llangle \begin{bmatrix} \tfrac{1}{M} & -3\alpha_\phi \\ -3\alpha_\phi & \tfrac{C_\mathrm{d}}{T_0} \end{bmatrix} \begin{bmatrix} p \\ T \end{bmatrix}, \begin{bmatrix} p \\ T  \end{bmatrix} \rrangle.
\end{align*}

Finally, we formulate the linearized thermo-poro-elasticity model as generalized gradient flow: Given the current state $(\u,\fluidmass,S)$, its change is described by
\begin{align}
\label{thermo-poro-elasticity:gradient-flow:start}
 \dot{\fluidmass} &= \masssource -\DIV\flux, \\
 \dot{S} &= \entropysource -\DIV\entropyflux, \\
 \nonumber
 (\dot{\u},\flux,\entropyflux) &= \underset{(\v,\z,\w)\in\dot{\mathcal{V}}\times\mathcal{Z}\times\mathcal{W}}{\text{arg\,min}}\, \Big\{
 \mathcal{D}_\mathrm{th}(\z,\w)  
 + \llangle \GRAD \mathcal{E}_\mathrm{th}(\u,\fluidmass,S), [\v,-\DIV \z,-\DIV\w] \rrangle \\[-1em]
 &\qquad\qquad\qquad\qquad\qquad- \mathcal{P}_\mathrm{ext,th}(\v,\z,\w) \Big\}.
\label{thermo-poro-elasticity:gradient-flow:end}
\end{align}
Using simple calculations, one can verify that the corresponding optimality conditions are equivalent to the original problem formulation~\eqref{thermo-biot:mechanics}--\eqref{thermo-biot:energy}. Well-posedness can be established analogously to linear poro-elasticity, exploiting that linearized thermo-poro-elasticity is essentially a vectorized version of linear poro-elasticity.

\begin{remark}[Including non-monotone perturbations]\label{remark:porothermo-perturbed}
For larger P\'eclet and Brink\-man numbers, the contributions 
\begin{alignat}{2}
\label{perturbations:start}
 W_{(i)}(\flux,\entropyflux) &= -c \llangle \conductivity^{-1}\entropyflux,\flux \rrangle,\ &&\left(\text{for some }c\in\mathbb{R}\right),\\
 W_{(ii)}(\flux)   &= \left(\tfrac{1}{T_0}-3\alpha_\phi\right)\llangle \permeability^{-1} \flux, \flux \rrangle,
\label{perturbations:end}
\end{alignat}
corresponding to thermal convection and the heat production due to viscous dissipation, respectively, are non-negligible and have to be incorporated in the energy equation~\eqref{thermo-biot:energy}, which then becomes
\begin{align*}
  \tfrac{C_\mathrm{d}}{T_0} \dot{T} + 3\alpha_\mathrm{T} K_\mathrm{dr} \DIV \dot{\u} - 3\alpha_\phi \dot{p} + W_{(i)}(\flux,\entropyflux) + \DIV \entropyflux &= \entropysource + W_{(ii)}(\flux).
\end{align*}
Eq.~\eqref{thermo-biot:mechanics}--\eqref{thermo-biot:mass} remain unchanged. Based on the above discussion, the resulting equations can be interpreted as perturbed gradient flow (or doubly non-linear evolution equation with non-monotone perturbations). In the context of operator splitting schemes, we will discuss possibilities to still exploit the part containing a gradient flow structure for deriving robust splitting schemes for the general problem, cf.\ Sec.~\ref{section:non-linear-thermo-poro}.
\end{remark}

\section*{Part II -- Efficient discrete approximation schemes for thermo-poro-visco-elasticity}
\addcontentsline{toc}{section}{Part II -- Efficient discrete approximation schemes for thermo-poro-visco-elasticity} 

The gradient flow structure of thermo-poro-visco-elasticity revealed in Part~I allows for a unified framework for deriving stable temporal and spatial discretizations, as well as the development and analysis of robust operator splitting schemes. The abstract workflow taken in this paper is visualized in Fig.~\ref{figure:procedure-derivation-splitting} can be summarized as follows: Given a time-continuous gradient flow formulation, a time-discrete approximation is introduced by applying the minimizing movement scheme, resulting in a (convex and hence well-posed) minimization problem for each time step. A corresponding dual problem is derived by applying the Legendre-Fenchel duality theory~\cite{Ekeland1999}. Provided that the resulting minimization problems are block-separable, alternating minimization is utilized, decoupling physical subproblems -- it comes with strong robustness under fairly weak assumptions and allows for an abstract convergence analysis~\cite{Grippof1999,Grippo2000}. Furthermore, the underlying minimization structure of the coupled problems enables simple acceleration of iterative solvers via cheap line search strategies. 

Although based on semi-discrete approximations, the results also immediately translate to fully-discrete approximations obtained by the Galerkin method, i.e., the well-posedness as convex minimization problems, and the efficient numerical solution by block-coordinate descent methods.

\begin{figure}[h!]    
\centering  
 \begin{overpic}[width=0.96\textwidth]{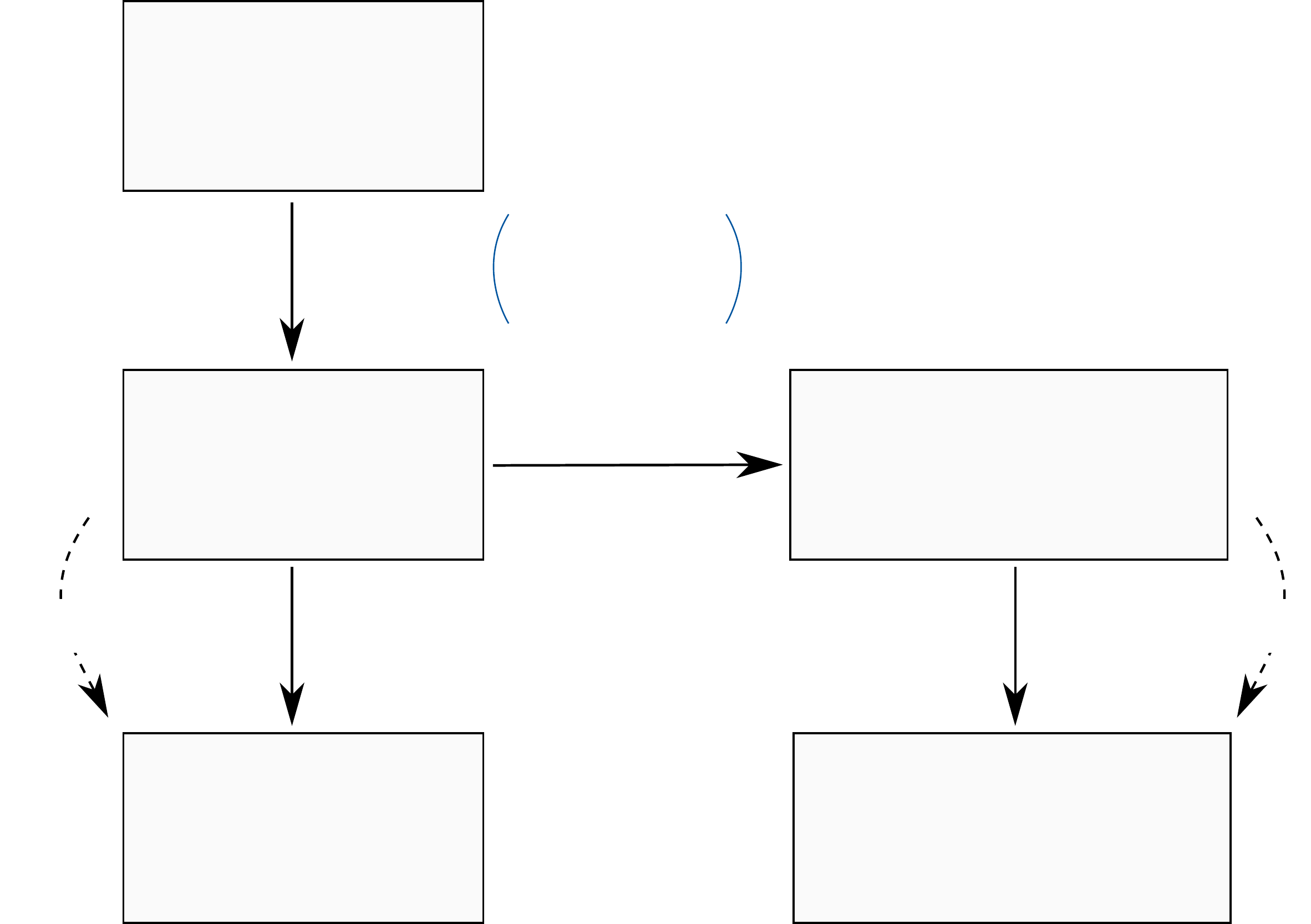}
  \put(39,32){\textcolor{grau}{Duality theory}}
  \put(1.1,22.15){\textcolor{grau}{Line search}}
  \put(89,22.15){\textcolor{grau}{Line search}}
  \put(23.5,51.5){\scalebox{1.02}{\textcolor{grau}{Minimizing}}}
  \put(23.5,49.4){\scalebox{1.02}{\textcolor{grau}{movement}}}
  \put(23.5,47.3){\scalebox{1.02}{\textcolor{grau}{scheme}}}
  \put(38.5,49.4){\scalebox{1.02}{\textcolor{grau}{+}}}
  \put(41,51.5){\scalebox{1.02}{\textcolor{grau}{Conforming}}} 
  \put(41,49.3){\scalebox{1.02}{\textcolor{grau}{Galerkin}}}
  \put(41,47.2){\scalebox{1.02}{\textcolor{grau}{method}}}
  \put(23.5,22.25){\textcolor{grau}{Alternating}}
  \put(23.5,20.25){\textcolor{grau}{minimization}} 
  \put(63,22.25){\textcolor{grau}{Alternating}}
  \put(61.25,20.25){\textcolor{grau}{minimization}} 
  \put(10.5,67.5){Generalized gradient}  
  \put(10.5,64.75){flow formulation} 
  \put(12,60){$\dot{x}+\GRAD \mathcal{E}(x)=f_\mathrm{ext}$}
  \put(10.5,39.25){\underline{Primal} semi-discrete}
  \put(10.5,36){approximation} 
  \put(11,31.5){$x^n= \underset{x}{\mathrm{arg\,min}} \, \mathcal{E}_\mathrm{tot}^{\Delta t}(x)$}
  \put(61.5,39.25){\underline{Dual} semi-discrete}
  \put(61.5,36){approximation} 
  \put(61.5,31.5){$(x^{\star})^{n}= \underset{x^\star}{\mathrm{arg\,min}} \, \mathcal{E}_\mathrm{tot}^{\star,\Delta t}(x^\star)$}
  \put(10.5,11.5){\underline{Undrained}-type split}
  \put(10.5,7.25){\scalebox{0.8}{$x_1^{i}= \underset{x_1}{\mathrm{arg\,min}} \, \mathcal{E}_\mathrm{tot}^{\Delta t}\Big(x_1,x_2^{i-1}\Big)$}}
  \put(10.5,2.5){\scalebox{0.8}{$x_2^{i}= \underset{x_2}{\mathrm{arg\,min}} \, \mathcal{E}_\mathrm{tot}^{\Delta t}\Big(x_1^i,x_2\Big)$}}
  \put(61.5,11.5){\underline{Fixed-stress}-type split}  
  \put(61.5,7.6){\scalebox{0.8}{$(x_1^{\star})^{i}= \underset{x_1^\star}{\mathrm{arg\,min}} \, \mathcal{E}_\mathrm{tot}^{\star,\Delta t}\Big(x_1^\star,(x_2^{\star})^{i-1}\Big)$}}
  \put(61.5,2.85){\scalebox{0.8}{$(x_2^{\star})^{i}= \underset{x_2^\star}{\mathrm{arg\,min}} \, \mathcal{E}_\mathrm{tot}^{\star,\Delta t}\Big((x_1^{\star})^{i},x_2^\star\Big)$}} 
 \end{overpic}
 \caption{\label{figure:procedure-derivation-splitting} Workflow for the derivation of splitting schemes for linear thermo-poro-visco-elasticity, illustrated for simplicity for classical gradient flows.}
\end{figure}
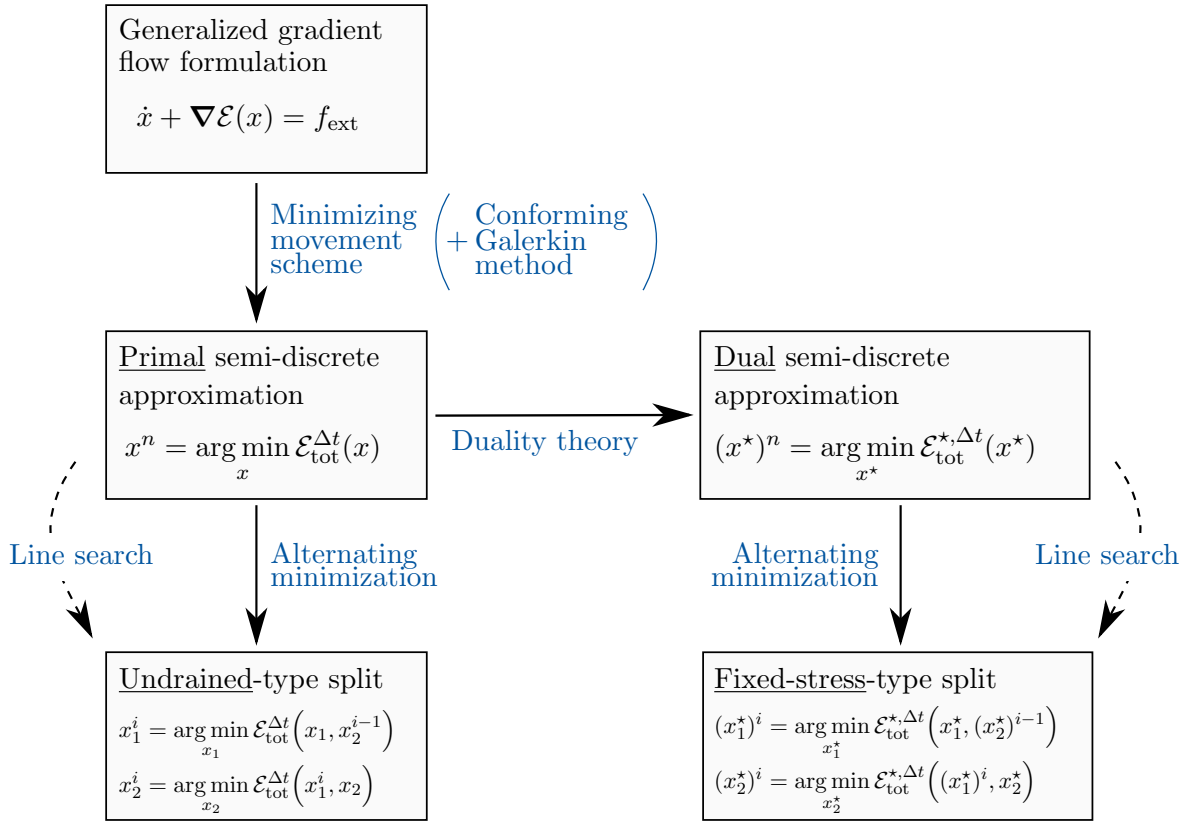 

By applying the workflow in particular to linear poro-elasticity, we derive the well-known undrained and fixed-stress splits. Thereby, we provide a novel interpretation of the widely used splitting schemes as inexact minimization. Motivated by that, the abstract approach is further applied to distinct representatives of three generalizations of linear poro-elasticity: poro-visco-elasticity, non-linear poro-elasticity under infinitesimal strains, and thermo-poro-elasticity. Ultimately, novel splitting schemes are derived for poro-visco-elasticity and nonlinear poro-elasticity, structurally similar to the undrained and fixed-stress splits. In the context of thermo-poro-elasticity, the recently proposed undrained-adiabatic and extended fixed-stress splits~\cite{Kim2018a} are derived and by that justified mathematically.

\section{Energy-driven time discretization via minimizing movements}\label{section:general-time-discretization}

Gradient flows allow for stable time discretization. Throughout this part of the paper, we apply the so-called \textit{minimizing movement} scheme~\cite{Ambrosio1995}, which is energy dissipating and closely related to the implicit Euler method, most often the first choice time discretization for poro-elasticity models. Utilizing a minimization formulation,  the minimizing movement scheme retains the structure and possible convexity properties of the problem. However, we note, it does not preserve a discrete energy identity analogous to~\eqref{appendix:well-posedness-dne:energy-identity}; for structure-preserving time-discretizations we refer, e.g., to~\cite{Jungel2018} and the references within. 
 
For simplicity, we consider an equidistant partition $0=t_0<t_1<...<t_N=T$ of the time interval $[0,T]$ with time step size $\Delta t$. Fields, functionals and function spaces evaluated at discrete time $t_n$ are enhanced by an exponent $n$, e.g., $x^n:=x(t_n,\cdot)\in\mathcal{X}^n:=\mathcal{X}(t_n)$ and $\mathcal{P}_\mathrm{ext}^n(\cdot):=\mathcal{P}_\mathrm{ext}(t_n)$.

The minimizing movement scheme applied to the abstract, generalized gradient flow~\eqref{generalized-gradient-flow-structure} is identical with a semi-implicit Euler method, where state-dependent functions are lagged in time. For time step $n$, it reads: Given $x^{n-1}\in\mathcal{X}^{n-1}$, find $x^n\in\mathcal{X}^n$ and $p^n\in\mathcal{P}^n$ satisfying
\begin{align}
\label{general-time-discrete:start}
 &\frac{x^n-x^{n-1}}{\Delta t} + \mathcal{T}\left(x^{n-1}\right)p^n = 0 \\
 &p^n \,=\, \underset{p\in \mathcal{P}^n}{\mathrm{arg\,min}}\,\left\{ \Delta t\,  \mathcal{D}\left(x^{n-1};p\right) + \mathcal{E}(x^{n}) - \Delta t  \mathcal{P}_\mathrm{ext}^n\left(x^{n-1};p\right) \right\}.
\label{general-time-discrete:end}
\end{align}
As the structure of the original problem is retained, solvability of the time-discrete problem is automatically inherited from the continuous problem. In this work, all dissipation potentials, external work rates and process operators are state-independent. Thus, we omit the explicit dependence from now on.

The coupled problem~\eqref{general-time-discrete:start}--\eqref{general-time-discrete:end} can be obviously decoupled by reducing~\eqref{general-time-discrete:end} to a minimization problem for process vectors
\begin{align}
\label{time-discrete-minimization-flux:start}
 &\frac{x^n-x^{n-1}}{\Delta t} + \mathcal{T}p^n = 0 \\
 &p^n \,=\, \underset{p\in \mathcal{P}^n}{\mathrm{arg\,min}}\,\left\{ \Delta t \, \mathcal{D}\left(p\right) + \mathcal{E}\left(x^{n-1} - \mathcal{T}p\right) - \Delta t \, \mathcal{P}_\mathrm{ext}^n\left(p\right) \right\}.
\label{time-discrete-minimization-flux:end}
\end{align}
Alternatively, provided that the change of state is directly associated to its rate, it is $\mathcal{T} = -Id$; in the context of poro-elasticity, this particular case occurs, e.g., for mechanical deformation. Consequently,
\eqref{time-discrete-minimization-flux:start}--\eqref{time-discrete-minimization-flux:end} 
becomes a minimization problem for the state itself
\begin{align}
\label{time-discrete-minimization-state} 
 &x^n \,=\, \underset{x \in \mathcal{X}^n}{\mathrm{arg\,min}}\,\left\{ \Delta t \, \mathcal{D}\left(\frac{x - x^{n-1}}{\Delta t}\right) + \mathcal{E}(x) - \Delta t \, \mathcal{P}_\mathrm{ext}^n\left(\frac{x-x^{n-1}}{\Delta t}\right) \right\}.
\end{align}
The Euler-Lagrange equation is indeed equivalent to the \textit{implicit Euler} scheme
\begin{align*}
 \GRAD \mathcal{D}\left(\frac{x^n - x^{n-1}}{\Delta t}\right) + \GRAD \mathcal{E}(x^n) = \GRAD \mathcal{P}_\mathrm{ext}^n\left(\frac{x^n - x^{n-1}}{\Delta t}\right),\quad \text{in }\mathcal{X}_0^\star,
\end{align*}
where $\mathcal{X}_0$ is the linear tangent space to $\mathcal{X}^n$.

For the thermo-poro-visco-elasticity models discussed in this paper, we employ a combination of~\eqref{time-discrete-minimization-flux:start}--\eqref{time-discrete-minimization-flux:end} and~\eqref{time-discrete-minimization-state}, depending on the nature of the particular variables and available process vectors. A fully-discrete approximation may then be obtained by the conforming Galerkin method; see Sec.~\ref{section:spatial-discretization} for an exemplary discussion in the context of linear poro-elasticity.

\section{Minimization formulations of discrete linear poro-elasticity and robust splitting schemes via alternating minimization}\label{section:splitting-linear-biot}

In the literature, various formulations of linear poro-elasticity are employed, differing in the choice of primal variables. In this spirit, we present various minimization formulations of time-discrete, linear poro-elasticity, after applying the minimizing movement scheme (Sec.~\ref{section:general-time-discretization}). In particular, we discuss the widely used two-, three-, and five-field formulations, as well as the primal and the dual formulations naturally arising from Part~I. The specific minimization formulation is relevant when applying a line search strategy for the acceleration of iterative solvers as splitting schemes, cf.~\ Sec.~\ref{section:minimization:line-search}. Fully-discrete approximations with same properties are obtained by the conforming Galerkin method. We also note that minimization formulations can be derived in the context of the least-squares finite element method, cf., e.g.,~\cite{Korsawe2005}; however, such usually do not stem naturally from a physical, gradient flow formulation but build directly on classical PDE-based models.

Following the workflow visualized in Fig.~\ref{figure:procedure-derivation-splitting}, we derive the widely used undrained split and fixed-stress split as alternating minimization applied to the primal and dual formulations, respectively. Splitting schemes for linear poro-elasticity have been well studied in recent literature, and as such, much of the material in this section represents a new perspective, and indeed also new proofs, of known results. However, even in this case the discussion in this section lays the foundation for the more advanced applications in the subsequent sections, wherein the gradient flow framework leads to new schemes not previously reported.

\subsection{Minimization formulations of time-discrete linear poro-elasticity}\label{section:poro:minimization-formulations}

In the following, we introduce various minimization formulations of time-discrete, linear poro-elasticity differing in the choice of the primary variables. We present the primal and dual formulations, as well as the more widely used two-, three-, and five-field formulations.

\subsubsection{Primal formulation of time-discrete linear poro-elasticity}\label{section:minimization:primal-two-field}

Time discretization of the continuous, primal formulation of linear poro-elasticity~\eqref{biot-generalized-gradient-flow-structure-U} via the minimizing movement scheme, yields the primal formulation of time-discrete, linear poro-elasticity. It is formulated as a series of minimization problems:
At time step $n\geq 1$, let
\begin{align}
\label{primal-discrete-poro-elasticity-spaces:start}
 \mathcal{V}^n &:= \left\{ \left.\v\in H^1(\Omega) \,\right|\, \v=\u_\Gamma^n \text{ on }\Gamma_{\u}\right\}, \\
 \mathcal{Z}^n &:= \left\{ \z\in H(\mathrm{div};\Omega) \,\left|\, \z\cdot\n = q_\mathrm{\Gamma,n}^n \text{ on }\Gamma_{\flux}\right.\right\}.
\label{primal-discrete-poro-elasticity-spaces:end}
\end{align}
Then given $\fluidmass^{n-1}$, define $(\u^n,\flux^n)\in\mathcal{V}^n\times\mathcal{Z}^n$ by
\begin{align}
\label{primal-two-field-minimization:start}
 (\u^n,\flux^n) :=& \underset{(\u,\flux)\in\mathcal{V}^n\times\mathcal{Z}^n}{\mathrm{arg\,min}}\,\mathcal{E}_\mathrm{tot}^{\Delta t}(\fluidmass^{n-1};\u,\flux), \quad \text{where }\\[0.5em]
 \nonumber
 \mathcal{E}_\mathrm{tot}^{\Delta t}(\fluidmass^{n-1};\u,\flux) 
 :=&
 \tfrac{1}{2} \llangle \mathbb{C} \eps{\u}, \eps{\u} \rrangle + \tfrac{\Delta t}{2} \llangle \permeability^{-1} \flux, \flux \rrangle 
 \\
\nonumber
 &+ \tfrac{M}{2} \left\| \fluidmass^{n-1} +\Delta t \, \masssource^n - \Delta t \DIV \flux - \alpha \DIV \u \right\|^2 \\
\nonumber
 &- \mathcal{P}^n_\mathrm{ext,mech}(\u)- \Delta t \, \mathcal{P}^n_\mathrm{ext,fluid}(\flux),
\end{align}
and set $\fluidmass^n:=\fluidmass^{n-1} +\Delta t \, \masssource^n- \Delta t \DIV \flux^n$. Since the energy $\mathcal{E}_\mathrm{tot}^{\Delta t}$ is strictly convex, existence and uniqueness of a solution to~\eqref{primal-two-field-minimization:start} follow by classical results from convex analysis, cf.\ Thm.~\ref{appendix:well-posedness:convex-minimization}.

\subsubsection{Dual formulation of time-discrete linear poro-elasticity}\label{section:minimization:dual-formulation}
We introduce a dual formulation of~\eqref{primal-two-field-minimization:start}. It can be derived using tools from convex analysis; or equivalently, by employing the minimizing movement scheme to the continuous, dual formulation~\eqref{generalized-gradient-flow:dual-formulation:linear-poro-elasticity}:
At time step $n\geq 1$, let
\begin{align*}
 \mathcal{S}^n &:= \left\{ \bm{\tau} \in H(\text{div};\Omega)^d \,\left|\, 
 \begin{array}{l}
  \bm{\tau}\bm{n} = \stress_\mathrm{\Gamma,n}^n \text{ on }\Gamma_{\stress}, \\[2pt]                                                                  
  \DIV \bm{\tau} + \fext^n = \bm{0} \text{ in } L^2(\Omega),\\[2pt]
  \llangle \bm{\tau},\bm{\gamma}\rrangle = 0 \ \forall \bm{\gamma}\in\bm{Q}_\mathrm{AS}
 \end{array} \right. \right\}, \\[2pt]
 \mathcal{Q}^n &:= \left\{\left. q \in H^1(\Omega) \,\right|\, q=p_\Gamma^n \text{ on }\Gamma_{p} \right\},
\end{align*} 
where $\bm{Q}_\mathrm{AS}$ as defined in~\eqref{definition:skew-symmetric-tensors}. Then given $(\stress^{n-1},p^{n-1})\in\mathcal{S}^{n-1}\times\mathcal{Q}^{n-1}$, set $\fluidmass^{n-1}:=\tfrac{1}{M}p^{n-1}+ \alpha \,\mathrm{tr}\left(\mathbb{A}(\stress^{n-1} + \alpha p^{n-1} \mathbf{I})\right)$, and define $(\stress^{n},p^{n})\in\mathcal{S}^n\times\mathcal{Q}^n$ to be the solution of the dual minimization problem
\begin{align}
\label{dual-two-field-minimization:start}
  (\stress^{n},p^{n}) :=& \underset{(\stress,p)\in\mathcal{S}^n\times\mathcal{Q}^n}{\mathrm{arg\,min}}\,\mathcal{E}_\mathrm{tot}^{\star,\Delta t}(\fluidmass^{n-1};\stress,p), \quad \text{where}\\[0.5em]
 \nonumber
 \mathcal{E}_\mathrm{tot}^{\star,\Delta t}(\fluidmass^{n-1};\stress,p) :=&\, 
 \tfrac{1}{2} \llangle \mathbb{A} (\stress+\alpha p \mathbf{I}),\stress + \alpha p \mathbf{I} \rrangle 
 \\ 
 \nonumber
 &+ \tfrac{1}{2M} \left\| p \right\|^2 + \tfrac{\Delta t}{2} \llangle \permeability (\GRAD p - \gext^n), \GRAD p - \gext^n \rrangle \\
 \nonumber
 &
 -\llangle  \u_\Gamma^n, \stress\bm{n} \rrangle_{\Gamma_{\u}} 
 -\llangle \fluidmass^{n-1} + \Delta t\, \masssource^n, p \rrangle 
 - \Delta t \llangle  q_\mathrm{\Gamma,n}^n, p \rrangle_{\Gamma_{\q}}.
\end{align}
Since the energy $\mathcal{E}_\mathrm{tot}^{\star,\Delta t}$ is strictly convex, and the feasible set $\mathcal{S}^n\times\mathcal{Q}^n$ is non-empty and convex, existence and uniqueness of a solution to~\eqref{dual-two-field-minimization:start} follow by classical results from convex analysis, cf.\ Thm.~\ref{appendix:well-posedness:convex-minimization}.

\subsubsection{Two-field saddle point formulation of time-discrete linear poro-elasticity}\label{section:minimization:classical-two-field}

In the literature, linear poro-elasticity is often studied both numerically and analytically based on a two-field saddle point formulation of the linear Biot equations. It employs the mechanical displacement $\u$ and the fluid pressure $p$ as primary variables, cf., e.g.,~\cite{Showalter2000,Mikelic2013,Kim2011,Kim2011b,Storvik2019}. Employing the implicit Euler method for time-discretization, time step $n\geq 1$ reads: Given $\fluidmass^{n-1} := \frac{1}{M}p^{n-1}+\alpha \DIV\u^{n-1}$, find $(\u^n,p^n)\in \mathcal{V}^n \times \mathcal{Q}^n$ satisfying for all $(\v,q)\in\mathcal{V}_0\times \mathcal{Q}_0$
\begin{align*}
 \llangle \mathbb{C}\eps{\u^n}, \eps{\v} \rrangle - \alpha \llangle p^n, \DIV \v \rrangle &= \llangle \fext^n, \v \rrangle + \llangle \stress_\mathrm{\Gamma,n}^n, \v \rrangle_{\Gamma_{\stress}},\\
 \frac{1}{M} \llangle p^n, q \rrangle + \alpha \llangle \DIV \u^n , q \rrangle + \Delta t \, \llangle \permeability (\GRAD p^n - \gext^n ) , \GRAD q \rrangle &= \llangle \fluidmass^{n-1} + \Delta t \, \masssource^n, q \rrangle + \Delta t \llangle q_\mathrm{\Gamma,n}^n, q \rrangle_{\Gamma_{\flux}},
\end{align*}
where $\mathcal{Q}_0 := \{ q \in H^1(\Omega) \,|\, q=0 \text{ on }\Gamma_{p} \}$.

Saddle point formulations can be in general transformed to a constrained minimization problem~\cite{Boffi2013}. In the context of linear poro-elasticity, the constraint has to explicitly impose one of the physical subproblems. In the following, we choose to impose the balance of linear momentum and define a suitable product space for each time step $n$
\begin{align*}
 \tilde{\mathcal{H}}^n
 :=
 \left\{
 (\u,p)\in\mathcal{V}^n \times\mathcal{Q}^n
 \ \left| \ 
 \begin{array}{rl}
 \llangle \mathbb{C} \eps{\u}, \eps{\v} \rrangle - \alpha \llangle p, \DIV \v \rrangle \ \ \ &\\
 = \llangle \fext^n, \v \rrangle + \llangle \stress_\mathrm{\Gamma,n}^n, \v \rrangle_{\Gamma_{\stress}}& \ \ \forall \v\in\mathcal{V}_0
 \end{array}
 \right.
 \right\}.
\end{align*}
The time-discrete, linear Biot equations at fixed time step $n$, formulated as constrained minimization problem, read: Given $(\u^{n-1},p^{n-1})\in\tilde{\mathcal{H}}^{n-1}$, set $\fluidmass^{n-1} := \frac{1}{M}p^{n-1}+\alpha \DIV\u^{n-1}$, and find $(\u^n, p^n)\in\tilde{\mathcal{H}}^n$, satisfying
\begin{align}
\label{two-field-up-minimization:start}
 (\u^n,p^n) :=& \underset{(\u,p)\in\tilde{\mathcal{H}}^n}{\mathrm{arg\,min}}\,\mathcal{E}_\mathrm{tot}^{\Delta t}(\fluidmass^{n-1};\u,p), \quad \text{where}\\[0.5em]
 \nonumber
 \mathcal{E}_\mathrm{tot}^{\Delta t}(\fluidmass^{n-1};\u,p) 
 :=&
 \tfrac{1}{2} \llangle \mathbb{C} \eps{\u}, \eps{\u} \rrangle
 + \tfrac{1}{2M} \left\| p \right\|^2 \\
\nonumber
 &+ \tfrac{\Delta t}{2} \llangle \permeability (\GRAD p - \gext^n), \GRAD p - \gext^n \rrangle \\
\nonumber
 &- \llangle \fluidmass^{n-1} + \Delta t \, \masssource^n, p \rrangle
 - \Delta t \llangle q_\mathrm{\Gamma,n}^n, p \rrangle_{\Gamma_{\q}}.
\end{align}
Since the energy $\mathcal{E}_\mathrm{tot}^{\Delta t}$ is strictly convex and the feasible set $\tilde{\mathcal{H}}^{n}$ is non-empty and convex, existence and uniqueness of a solution to~\eqref{dual-two-field-minimization:start} follow by classical results from convex analysis, cf.\ Thm.~\ref{appendix:well-posedness:convex-minimization}. We emphasize, that well-posedness also follows in the extreme case of an incompressible fluid and impermeable medium, as long as $\mathcal{V}_0 \times \mathcal{Q}_0$ is inf-sup stable such that $\tilde{\mathcal{H}}^n$ is essentially constrained by a one-dimensional relation. Finally, we note, compared to the primal and dual formulations,~\eqref{two-field-up-minimization:start} is not block-separable.

\begin{remark}[Mass conservation as constraint]
 Alternatively, mass conservation can imposed as constraint resulting in an alternative minimization formulation of the time-discrete, linear Biot equations. However, having splitting schemes accelerated by relaxation in mind, cf.\ Sec.~\ref{section:minimization:line-search}, the particular choice matters. The constraint has to be satisfied after each splitting iteration; consequently, the above formulation~\eqref{two-field-up-minimization:start} suits the fixed-stress split, whereas the use of mass conservation as constraint contrarily suits the undrained splitting scheme, cf.\ Sec.~\ref{section:splitting-schemes:alternating-minimization}.
\end{remark}

\subsubsection{Three-field formulation of time-discrete linear poro-elasticity}\label{section:minimization:three-field-formulation}

A conforming Galerkin finite element discretization of the classical two-field saddle point formulation is not locally mass conservative. Therefore, in the literature, often a mixed formulation of the fluid flow problem is employed, cf., e.g.,~\cite{Both2017,Both2018a,Borregales2018,Castelletto2016,Ferronato2010,Bause2017}; this results in a three-field formulation employing the mechanical displacement $\u$, the fluid pressure $p$ and the volumetric flux $\flux$ as primary variables. Based on the primal two-field minimization formulation of linear poro-elasticity (Sec.~\ref{section:minimization:primal-two-field}), we state an unconstrained minimization formulation corresponding to the three-field formulation of linear poro-elasticity. For this, we essentially modify slightly the energy used in~\eqref{primal-two-field-minimization:start} and define the pressure as a post-processed quantity. Consistent with~\eqref{linear-biot-interpretation-m}, we define for given $\u$ and $\fluidmass$ by
\begin{align}
\label{three-field:pressure-projection}
 p := \Pi_{\tilde{\mathcal{Q}}}\big(M(\fluidmass - \alpha \DIV \u) \big),
\end{align}
where $\Pi_{\tilde{\mathcal{Q}}}$ denotes the orthogonal projection onto $\tilde{\mathcal{Q}}:= L^2(\Omega)$; the particular choice for the pressure space $\tilde{\mathcal{Q}}$ originates from the expected regularity, cf.\ Lemma~\ref{lemma:well-posedness-poro}.

Finally, we define the minimization formulation for time step $n$: Given $(\u^{n-1},\flux^{n-1},p^{n-1})\in\mathcal{V}^{n-1}\times\mathcal{Z}^{n-1}\times\tilde{\mathcal{Q}}$, set $\fluidmass^{n-1} := \frac{1}{M}p^{n-1}+\alpha \DIV\u^{n-1}$, and define $(\u^n,\flux^n,p^n)\in\mathcal{V}^n\times\mathcal{Z}^n\times\tilde{\mathcal{Q}}$ as solution to
\begin{align}
 \label{three-field-minimization:start}
 (\u^n,\flux^n) :=& \underset{(\u,\flux)\in\mathcal{V}^n\times\mathcal{Z}^n}{\mathrm{arg\,min}}\,\mathcal{E}_\mathrm{tot}^{\Delta t}(\fluidmass^{n-1};\u,\flux), \\
 \label{three-field-minimization:end}
  p^n :=& \, \Pi_{\tilde{\mathcal{Q}}} \left(M(\fluidmass^{n-1} +\Delta t\, \masssource^n - \Delta t \, \DIV \flux^n - \alpha \DIV \u^n) \right),
\end{align}
where
\begin{align}
  \nonumber
 \mathcal{E}_\mathrm{tot}^{\Delta t}(\fluidmass^{n-1};\u,\flux) 
 :=&
 \, \tfrac{1}{2} \llangle \mathbb{C} \eps{\u}, \eps{\u} \rrangle + \tfrac{\Delta t}{2} \llangle \permeability^{-1} \flux, \flux \rrangle \\
\nonumber
 & + \tfrac{M}{2} \left\| \Pi_{\tilde{\mathcal{Q}}}( \fluidmass^{n-1} + \Delta t\, \masssource^n - \Delta t \DIV \flux - \alpha \DIV \u) \right\|^2 \\
\nonumber
 &- \mathcal{P}_\mathrm{ext,mech}^n(\u)- \Delta t \, \mathcal{P}_\mathrm{ext,fluid}^n(\flux).
\end{align}
The minimization problem is strictly convex and the projection is well-defined; existence and uniqueness of a solution to~\eqref{three-field-minimization:start}--\eqref{three-field-minimization:end} follow by classical results from convex analysis, cf.\ Thm.~\ref{appendix:well-posedness:convex-minimization}. Furthermore, it is simple to verify that the corresponding optimality conditions yield the classical three-field formulation of time-discrete, linear poro-elasticity, cf.\ Sec.~\ref{section:optimality-conditions:three-field}.


\begin{remark}[Constrained minimization formulation]\label{remark:3field-poro-constrained-minimization}
 Based on the inherent double saddle point structure of the three-field formulation, alternatively a non-block-separable minimization formulation, constrained by mass conservation, could be utilized, similar to Sec.~\ref{section:minimization:classical-two-field}. This would allow in particular for the discussion of the incompressible case $M=\infty$. However, unlike for $M\in(0,\infty)$, it becomes evident that inf-sup stability is required for the uniqueness of weak solutions.
\end{remark}

\subsubsection{Five-field formulation of time-discrete linear poro-elasticity}\label{section:minimization:five-field}

So far, no minimization formulation presented above lays a foundation for a fully structure-preserving, conforming Galerkin finite element discretization, which is  conserving locally both mass and linear momentum. In order to achieve this, a fully mixed five-field formulation can be used, i.e., mixed formulations for both the mechanical and the fluid flow subproblems, cf., e.g.,~\cite{Ahmed2019a,Ambartsumyan2018b,Keilegavlen2017}. Consequently, both independent subproblems incorporate themselves a saddle point structure; however, different to the two- and three-field formulations, the coupling of the two subproblems is symmetric. 

After all, we combine ideas from previous sections and state a minimization formulation corresponding to the five-field formulation. In particular, starting from the dual formulation~\eqref{section:minimization:dual-formulation}, we add the volumetric flux $\flux$ as variable and impose Darcy's law as constraint. For fixed each time step $n$, we define a suitable product space for the fluid flow variables
\begin{align*}
 \mathcal{F}^n := 
 \left\{ 
 (\flux,p)\in\mathcal{Z}^n \times \tilde{\mathcal{Q}} 
 \,\left|\, 
 \begin{array}{ll}
\llangle \permeability^{-1} \flux, \z \rrangle  - \llangle p, \DIV \z \rrangle \ \ \  &\\
 \ \ = \llangle \gext^n, \z \rrangle - \llangle p_\Gamma^n, \z\cdot\n \rrangle_{\Gamma_p} & \ \forall \z\in\mathcal{Z}_0
 \end{array}
 \right.\right\}.
\end{align*}
Finally, the minimization formulation reads: Given $(\stress^{n-1},\flux^{n-1},p^{n-1})\in\mathcal{S}^{n-1}\times\mathcal{F}^{n-1}$, set $\fluidmass^{n-1}:=\tfrac{1}{M}p^{n-1}+ \alpha \,\mathrm{tr}\left(\mathbb{A}(\stress^{n-1} + \alpha p^{n-1} \mathbf{I})\right)$, and define $(\stress^{n},\flux^n,p^{n}) \in \mathcal{S}^n\times\mathcal{F}^n$ to be the solution of the minimization problem
\begin{align}
\label{dual-five-field-minimization:start}
  (\stress^{n},\flux^n,p^{n}) :=& \underset{(\stress,\flux,p) \in \mathcal{S}^n \times \mathcal{F}^n}{\mathrm{arg\,min}}\,\mathcal{E}_\mathrm{tot}^{\star,\Delta t}(\fluidmass^{n-1};\stress,\flux,p), \quad \text{where}\\[0.5em]
  \nonumber
 \mathcal{E}_\mathrm{tot}^{\star,\Delta t}(\fluidmass^{n-1};\stress,\flux,p) :=& 
 \tfrac{1}{2} \llangle \mathbb{A} (\stress+\alpha p \mathbf{I}),\stress + \alpha p \mathbf{I} \rrangle 
 + \tfrac{1}{2M} \left\| p \right\|^2 \\ 
 \nonumber
 &+ \tfrac{\Delta t}{2} \llangle \permeability^{-1} \flux, \flux \rrangle -\llangle \fluidmass^{n-1} + \Delta t \, \masssource^n, p \rrangle -\llangle  \u_\Gamma^n, \stress\bm{n} \rrangle_{\Gamma_{\u}}.
\end{align}
Since the energy $\mathcal{E}_\mathrm{tot}^{\star,\Delta t}$ is strictly convex and the feasible set $\mathcal{S}^n\times\mathcal{F}^n$ is non-empty and convex, existence and uniqueness of a solution to~\eqref{dual-five-field-minimization:start} follow by classical results from convex analysis, cf.\ Thm.~\ref{appendix:well-posedness:convex-minimization}. We refer to Sec.~\ref{section:optimality-conditions:five-field} for the derivation of the corresponding optimality conditions. We emphasize, that well-posedness also follows in the extreme case of an incompressible fluid, as long as $\mathcal{Z}_0 \times \tilde{\mathcal{Q}}$ is inf-sup stable such that $\mathcal{F}^n$ is essentially constrained by a one-dimensional relation.

\subsection{Optimality conditions}\label{section:splitting-schemes:optimality-conditions}

For each of the minimization formulations of the linear Biot equations presented in Sec.~\ref{section:minimization:primal-two-field}--\ref{section:minimization:five-field}, the corresponding optimality conditions can be derived as the first variation. Constraints are handled via the method of Lagrange multipliers. For better illustration of the undrained and fixed-stress split in the following section, we derive the three-field and five-field formulation of the linear Biot equations below.

\subsubsection{Three-field formulation of the linear Biot equations derived from minimization}\label{section:optimality-conditions:three-field}

We consider the minimization formulation of the linear Biot equations from Sec.~\ref{section:minimization:three-field-formulation}. We recall that the mechanical displacement and volumetric flux $(\u^n,\flux^n)$ are determined by minimization and the pressure $p^n$ is defined using a post-processing. Hence, by definition of $p^n\in\tilde{\mathcal{Q}}$ and the orthogonal projection $\Pi_{\tilde{\mathcal{Q}}}$, it holds
\begin{alignat*}{3}
\llangle M \Pi_{\tilde{\mathcal{Q}}}( \fluidmass^{n-1} + \Delta t\, \masssource^n - \Delta t \DIV \flux^n - \alpha \DIV \u^n)\right.&,\left. \Pi_{\tilde{\mathcal{Q}}}( \alpha \DIV \v) \rrangle
&&=
\alpha \llangle p^n, \DIV \v \rrangle, \\
\llangle M \Pi_{\tilde{\mathcal{Q}}}( \fluidmass^{n-1} + \Delta t\, \masssource^n - \Delta t \DIV \flux^n - \alpha \DIV \u^n)\right.&,\left. \Pi_{\tilde{\mathcal{Q}}}( \Delta t \DIV \z) \rrangle
&&=
\Delta t \llangle p^n, \DIV \z \rrangle
\end{alignat*}
for all $(\v,\z)\in\mathcal{V}_0\times\mathcal{Z}_0$. From that, the optimality conditions for $(\u^n,\flux^n)$ and the definition of $p^n$ read for all $(\v,\z,q)\in\mathcal{V}_0 \times \mathcal{Z}_0 \times \tilde{\mathcal{Q}}$
\begin{alignat}{2}
\label{optimality-conditions:three-field-biot:start}
 \llangle \mathbb{C} \eps{\u^n}, \eps{\v} \rrangle - \alpha \llangle p^n , \DIV \v \rrangle &= \mathcal{P}_\mathrm{ext,mech}^n(\v), &&  \\    
 \label{optimality-conditions:three-field-biot:mid}
 \llangle \permeability^{-1} \flux^n, \z \rrangle - \llangle p^n, \DIV \z \rrangle &= \mathcal{P}_\mathrm{ext,fluid}^n(\z), && \\   
  \tfrac{1}{M} \llangle p^n, q \rrangle  + \alpha \llangle \DIV \u^n, q \rrangle + \Delta t \llangle \DIV \flux^n, q \rrangle &=\llangle \fluidmass^{n-1} + \Delta t \, \masssource^n, q \rrangle, &&
\label{optimality-conditions:three-field-biot:end}
\end{alignat}
which yields the classical three-field formulation of the linear Biot equations.

\subsubsection{Five-field formulation of the linear Biot equations derived as optimality conditions}\label{section:optimality-conditions:five-field}

We consider the minimization formulation of the linear Biot equations from Sec.~\ref{section:minimization:five-field}. First, we define function spaces for the stress variable
\begin{align*}
 \tilde{\mathcal{S}}^n &= \left\{ \stress \in H(\text{div};\Omega)^d \,\left|\, \stress\bm{n} = \stress_\mathrm{\Gamma,n}^n \text{ on }\Gamma_{\stress} \right.\right\}, \\
 \tilde{\mathcal{S}}_0 &= \left\{ \stress \in H(\text{div};\Omega)^d \,\left|\, \stress\bm{n} = \bm{0} \text{ on }\Gamma_{\stress} \right. \right\}.
\end{align*}
Additionally, we introduce the mechanical displacement $\u\in \bm{L}^2(\Omega)$, the rotation $\bm{\zeta}\in\bm{Q}_\mathrm{AS}$ and an artificial fluid flux $\tilde{\flux}\in \mathcal{Z}^n$ as Lagrange multipliers. A suitable Lagrangian, incorporating the balance of linear momentum, weak symmetry of the stress tensor and Darcy's law, is given by
\begin{align*}
 \mathcal{L}(\fluidmass^{n-1};\stress,\flux,p,\u,\bm{\zeta},\tilde{\flux})
 &:=
 \mathcal{E}_\mathrm{tot}^{\star,\Delta t}(\fluidmass^{n-1}; \stress, \flux, p) +
 \llangle \DIV \stress + \fext^n, \u \rrangle + \llangle \stress, \bm{\zeta} \rrangle \\
 &-
 \Delta t \left ( \llangle \permeability^{-1} \flux, \tilde{\flux} \rrangle - \llangle p, \DIV \tilde{\flux} \rrangle - \llangle \gext^n, \tilde{\flux} \rrangle + \llangle p_{\Gamma}^n, \tilde{\flux} \cdot \bm{n} \rrangle_{\Gamma_p}\right).
\end{align*}
For the isotropic case, the corresponding saddle point $(\stress,\u,\bm{\zeta}, p, \tilde{\flux}, \flux)\in\tilde{\mathcal{S}}^n\times\times\bm{L}^2(\Omega) \times \bm{Q}_\mathrm{AS} \times  \tilde{\mathcal{Q}}\times \mathcal{Z}^n \times \mathcal{Z}^n$ (omitting the index $^n$) is characterized by
\begin{align}
\label{optimality-conditions:five-field-biot:start}
 \llangle \mathbb{A} \stress, \bm{\tau} \rrangle + \llangle \u , \DIV \bm{\tau} \rrangle + \llangle \bm{\zeta}, \bm{\tau} \rrangle + \tfrac{\alpha}{d K_\mathrm{dr}} \llangle p, \text{tr}\, \bm{\tau} \rrangle &= \llangle \u_\Gamma^n, \bm{\tau}\n \rrangle_{\Gamma_{\u}}, \\
 \llangle \DIV \stress, \v \rrangle &=-\llangle \fext^n, \v\rrangle, \\
\label{optimality-conditions:five-field-biot:end-mechanics} 
 \llangle \stress, \bm{\gamma} \rrangle &=0, \\
 \label{optimality-conditions:five-field-biot:start-fluid}
 \left(\tfrac{1}{M} + \tfrac{\alpha^2}{K_\mathrm{dr}}\right) \llangle p, q \rrangle + \tfrac{\alpha}{d K_\mathrm{dr}} \llangle \text{tr}\,\stress, q \rrangle  + \Delta t \llangle \DIV \tilde{\flux}, q \rrangle &= \llangle \fluidmass^{n-1} + \Delta t\, \masssource^n, q \rrangle,\\
 \llangle \permeability^{-1} \flux, \tilde{\z} \rrangle - \llangle \permeability^{-1} \tilde{\flux}, \tilde{\z} \rrangle &= 0,\\
 \llangle \permeability^{-1} \flux, \z \rrangle - \llangle p, \DIV \z \rrangle &= \llangle \gext^n, \z \rrangle - \llangle p_\Gamma^n, \z \cdot \n \rrangle_{\Gamma_p}.
\label{optimality-conditions:five-field-biot:end}
\end{align}
for all variations $(\bm{\tau}, \v, \bm{\gamma}, q,  \tilde{\z},  \z) \in\tilde{\mathcal{S}}_0 \times \bm{L}^2(\Omega) \times \bm{Q}_\mathrm{AS} \times \tilde{\mathcal{Q}} \times \mathcal{Z}_0 \times \mathcal{Z}_0$. The artificial volumetric flux $\tilde{\flux}$ can be identified with the actual volumetric flux $\flux$, which finally yields the five-field formulation of the linear Biot equations.

\subsection{Consequences for fully-discrete approximations}\label{section:spatial-discretization}

Exemplarily for all sections in Part II, we comment: The various minimization formulations of semi-discrete linear poro-(visco-thermo-)elasticity can be discretized in space by the finite element method using the conforming Galerkin method. Well-posedness of the different fully-discrete formulations follows then by the same arguments as for their semi-discrete versions. For variants described by constrained minimization, the need for inf-sup stable finite element pairs immediately translates from the continuous to the discrete setting in order to ensure non-empty feasible sets subject to one-dimensional constraints, and thereby strict convexity of the energies.

We note, the discussion on efficient splitting schemes for linear poro-(visco-thermo-)elasticity can be equally based on the semi- as well as the corresponding fully-discrete approximations.

\subsection{Classical splitting schemes derived as alternating minimization}\label{section:splitting-schemes:alternating-minimization}

Summarizing the previous subsections: Time-discrete, linear poro-elasticity can be formulated as strictly convex, quadratic minimization problem, possibly subject to affine constraints depending on the choice of primary variables. By applying the conforming Galerkin method, the same translates to corresponding fully-discrete approximations. Various strategies can be applied to solve the resulting fully-discrete minimization problem numerically. We mention three: 

\begin{enumerate}[label=(\roman*)]

 \item The corresponding optimality conditions are derived as in  Sec.~\ref{section:splitting-schemes:optimality-conditions} and are solved in a monolithic fashion.

 \item Popular in the poro-elasticity community, the coupled optimality conditions are solved using an iterative splitting scheme, decoupling the mechanics and fluid flow subproblems, cf., e.g.,~\cite{Kim2011,Kim2011b,Mikelic2013,Castelletto2016,Both2017}.

 \item Based on the minimization formulation, some (inexact) minimization algorithm from the vast convex optimization literature, cf., e.g.,~\cite{Bertsekas1999,Temam2000} is applied. In the poro-elasticity literature, such an approach has not yet been pursued.
 
\end{enumerate}

Conforming simultaneously with options (ii) and (iii), we propose applying classical (exact) alternating minimization for solving linear poro-elasticity. Alternating minimization is equivalent to a two-block coordinate descent method as well as a successive two-subspace correction method for orthogonal space decompositions, alternating between minimizing the energy wrt.\ two different blocks of variables while constantly updating the complementary block. Partitioning the set of primal variables into a block of mechanical variables and a block of flow variables, yields a splitting scheme conforming with~(ii).

As resulting schemes, we in fact obtain the previously introduced and now widely used \textit{undrained split} and \textit{fixed-stress split}, cf., e.g.,~\cite{Kim2011,Kim2011b}. Originally, they have been physically motivated as predictor-corrector methods: For the undrained split, in the predictor step, the mechanical subproblem is solved under undrained conditions; in the corrector step, the unaltered fluid flow problem is solved with updated mechanical variables. Instead, for the fixed-stress split, in the predictor step, the fluid flow subproblem is solved under fixed volumetric stress; in the corrector step, the unaltered mechanics subproblem is solved. In order to explain their robustness and convergence properties, so far problem-specific analyses have been required, cf., e.g., ~\cite{Mikelic2013,Both2017,Borregales2018,Storvik2019}.

Originally, physically motivated, they are now endowed with a simple, mathematical intuition, providing an immediate understanding on their robust convergence properties. 
Alternating minimization exhibits guaranteed robustness under fairly weak assumptions; cf.~\cite{Luo1992,Luo1993,Grippof1999,Grippo2000} from the perspective of block coordinate descent methods, or~\cite{Tai1998,Tai2001,Xu2001} from the perspective of successive subspace correction methods. Furthermore, under stronger convexity and continuity assumptions on the energy, theoretical convergence rates can be analyzed using abstract theory~\cite{Beck2013}. By improving the abstract result from the aforementioned work to constrained minimization problems in infinite dimensions, cf.\ Appendix~\ref{appendix:section:alternating-minimization}, we establish theoretical convergence rates for the undrained split and fixed-stress split consistent with problem-specific analyses in the literature, cf., e.g., ~\cite{Mikelic2013,Both2017,Borregales2018,Storvik2019}.


\subsubsection{Derivation and analysis of the undrained split as alternating minimization}\label{section:undrained-split}

In this section, we identify the widely used undrained split~\cite{Kim2011} as alternating minimization applied to the primal two-field formulation of time-discrete, linear poro-elasticity, cf.\ Sec.~\ref{section:minimization:primal-two-field}. As the primal formulation is less frequently used in the literature, we illustrate the resulting scheme in the following with reference to the closely related three-field formulation, cf.\ Sec.~\ref{section:minimization:three-field-formulation}. We note the undrained split can be in fact equivalently derived based on the three-field formulation, but the analysis requires unnecessarily more involved notation.
\begin{algorithm}
 \caption{Single iteration of the undrained split}
 \label{algorithm:undrained-split}
 \SetAlgoLined
 \DontPrintSemicolon
 
  \vspace{0.5em}
 
  Input: $(\u^{n,i-1},\flux^{n,i-1})\in\mathcal{V}^n\times\mathcal{Z}^n$ \; \vspace{0.8em}
   
  Determine $\u^{n,i} := \underset{\u\in\mathcal{V}^n}{\mathrm{arg\,min}}\, \mathcal{E}_\mathrm{tot}^{\Delta t}(\fluidmass^{n-1};\u,\flux^{n,i-1})$\; \vspace{0.25em}

  Determine $\flux^{n,i} := \underset{\flux\in\mathcal{Z}^n}{\mathrm{arg\,min}}\, \mathcal{E}_\mathrm{tot}^{\Delta t}(\fluidmass^{n-1};\u^{n,i},\flux)$\; \vspace{0.2em}
\end{algorithm}

Alternating minimization is applied respecting the natural block structure of the problem, cf.\ Alg.\ref{algorithm:undrained-split} for a single iteration.
The first step is equivalent to solving a stabilized mechanics problem, cf.~\eqref{optimality-conditions:three-field-biot:start}: For given $(\u^{n,i-1},\flux^{n,i-1})\in \mathcal{V}^n\times\mathcal{Z}^n$, find $\u^{n,i}\in \mathcal{V}^n$ satisfying for all $\v\in\mathcal{V}_0$
\begin{align*}
 \llangle \mathbb{C}\eps{\u^{n,i}}, \eps{\v} \rrangle + \llangle M\alpha^2\  \mathrm{tr}\, \eps{\u^{n,i} - \u^{n,i-1}}, \mathrm{tr}\,\eps{\v} \rrangle \\
 -\alpha \llangle p^{n,i-1}, \DIV \v \rrangle = \mathcal{P}_\mathrm{ext,mech}^n(\v),
\end{align*}
where the pressure $p^{n,i-1}$ is formally defined, consistent with~\eqref{poro-visco-interpretation-pressure},
\begin{align*}
 p^{n,i-1} := M\left(\fluidmass^{n-1} + \Delta t\, \masssource^n - \Delta t \, \DIV \flux^{n,i-1} - \alpha \DIV \u^{n,i-1}\right).
\end{align*}
The second step is equivalent to solving the fluid flow problem~\eqref{optimality-conditions:three-field-biot:mid}--\eqref{optimality-conditions:three-field-biot:end} with updated mechanical variables. 

We highlight that the final iterative scheme is equivalent to the undrained split for linear poro-elasticity. We establish convergence employing an abstract convergence result for alternating minimization leading to consistent previous problem-specific discussions~\cite{Mikelic2013,Borregales2018}.
 
\begin{lemma}[Linear convergence of the undrained split]\label{lemma:undrained-split:convergence-rate}
 The undrained split converges linearly, independent of the initial guess. Let $\U^n:=(\u^{n},\flux^{n})$ denote the solution of the coupled problem~\eqref{primal-two-field-minimization:start} and let $\U^{n,i}:=(\u^{n,i},\flux^{n,i})$ denote the iterates defined by the undrained split, cf.\ Alg.~\ref{algorithm:undrained-split}. For $i\in\mathbb{N}$, define the errors $\erroru^{n,i}:=\u^{n,i}-\u^n$, $\errorflux^{n,i}:=\flux^{n,i} - \flux^n$. Let $|||\cdot|||$ denote the norm induced by the quadratic part of $\mathcal{E}_\mathrm{tot}^{\Delta t}$
 \begin{align*}
  \left|\left|\left|(\u,\flux)\right|\right|\right|^2 := 
  \tfrac{1}{2} \llangle \mathbb{C} \eps{\u}, \eps{\u} \rrangle +
  \tfrac{\Delta t}{2} \llangle \permeability^{-1} \flux, \flux \rrangle 
 + \tfrac{M}{2} \left\|  \Delta t \DIV \flux + \alpha \DIV \u \right\|^2.
 \end{align*}
 And let $K_\mathrm{dr}^\star\geq K_\mathrm{dr}$ be largest constant such that
 \begin{align}\label{proof:poro:undrained:kdr}
  K_\mathrm{dr}^\star \| \mathrm{tr}\,\eps{\v} \|^2 \leq  \llangle \mathbb{C} \eps{\v},  \eps{\v}\rrangle, \quad \forall\, \v\in\mathcal{V}_0.
\end{align}
 It holds the \textit{a priori} result
 \begin{align}
 \label{lemma:undrained-split:convergence-rate:result-2}
  &\left|\left|\left|(\erroru^{n,i},\errorflux^{n,i})\right|\right|\right| \leq \left(\frac{\tfrac{\alpha^2}{K_\mathrm{dr}^\star}}{\tfrac{1}{M} + \tfrac{\alpha^2}{K_\mathrm{dr}^\star}}\right)^i 
  \left( \mathcal{E}^{n,0} - \mathcal{E}^n \right)^{1/2},
 \end{align}
 and the \textit{a posteriori} result
 \begin{align}
 \label{lemma:undrained-split:convergence-rate:result-3}
  \left|\left|\left|(\erroru^{n,i},\errorflux^{n,i})\right|\right|\right| \leq 
  \left(\frac{\tfrac{\alpha^2}{K_\mathrm{dr}^\star}}{\tfrac{1}{M} + \tfrac{\alpha^2}{K_\mathrm{dr}^\star}}\right) \left[ 1 - \left(\frac{\tfrac{\alpha^2}{K_\mathrm{dr}^\star}}{\tfrac{1}{M} + \tfrac{\alpha^2}{K_\mathrm{dr}^\star}}\right)^2\right]^{- 1/2}
  \, \left( \mathcal{E}^{n,i-1} - \mathcal{E}^{n,i} \right)^{1/2},
 \end{align}
 where $\mathcal{E}^{n} := \mathcal{E}_\mathrm{tot}^{\Delta t}(\fluidmass^{n-1}; \U^{n})$, and $\mathcal{E}^{n,j} := \mathcal{E}_\mathrm{tot}^{\Delta t}(\fluidmass^{n-1}; \U^{n,j}),\ j\in\mathbb{N}$.

\end{lemma} 

\begin{proof}
 We apply Lemma~\ref{appendix:lemma:alternating-minimization}. For this, we introduce two semi-norms
 \begin{align*}
  \left\| \U \right\|_{1,\Delta t}^2 &:= \llangle \mathbb{C} \eps{\u}, \eps{\u} \rrangle, \\
  \left\| \U \right\|_{2,\Delta t}^2 &:= \Delta t \llangle \permeability^{-1} \flux, \flux \rrangle + \Delta t^2
  \left( \tfrac{1}{M} + \tfrac{\alpha^2}{K_\mathrm{dr}^\star} \right)^{-1} \left\| \DIV \flux \right\|^2.
 \end{align*}
 We show that $\mathcal{E}_\mathrm{tot}^{\Delta t}$ is (i) strongly convex wrt.\ $\|\cdot\|_{1,\Delta t}$ and $\|\cdot\|_{2,\Delta t}$, (ii) $\GRAD_{\u}\mathcal{E}_\mathrm{tot}^{\Delta t}$ is Lipschitz continuous wrt.\ $\|\cdot\|_{1,\Delta t}$, and (iii) $\GRAD_{\flux}\mathcal{E}_\mathrm{tot}^{\Delta t}$ is Lipschitz continuous $\|\cdot\|_{2,\Delta t}$. Throughout the proof, for lighter notation, we omit the explicit dependence of $\mathcal{E}_\mathrm{tot}^{\Delta t}$ on $\fluidmass^{n-1}$.
 
\paragraph{(i) Strong convexity of $\mathcal{E}_\mathrm{tot}^{\Delta t}$.}
Let $\U_i=(\v_i,\z_i)\in\mathcal{V}^n\times\mathcal{Z}^n$, $i=1,2$. As $\mathcal{E}_\mathrm{tot}^{\Delta t}$ is quadratic, it holds
\begin{align*}
 &\llangle \GRAD \mathcal{E}_\mathrm{tot}^{\Delta t}(\U_1) - \GRAD \mathcal{E}_\mathrm{tot}^{\Delta t}(\U_2),\U_1 - \U_2 \rrangle 
 =
 2 \left|\left|\left| \U_1 - \U_2\right|\right|\right|^2.
\end{align*}
Clearly, $2\left|\left|\left| \U \right|\right|\right|^2 \geq \left\| \U \right\|_{1,\Delta t}^2$.
By applying Young's inequality with optimally balanced weights and using~\eqref{proof:poro:undrained:kdr}, one can show for all $(\u,\flux)\in\mathcal{V}_0\times\mathcal{Z}_0$
\begin{align}
 \label{proof:poro:undrained:balancing-young}
 \llangle \mathbb{C} \eps{\u}, \eps{\u} \rrangle
 + M \left\|  \Delta t \DIV \flux + \alpha \DIV \u \right\|^2
 \geq 
 -
 \frac{\tfrac{\alpha^2}{K_\mathrm{dr}^\star}}{\tfrac{1}{M} + \tfrac{\alpha^2}{K_\mathrm{dr}^\star}} \, 
 \left\|\Delta t \, \DIV \flux \right\|^2.
\end{align}
Hence, $2\,||| \U_1 - \U_2 |||^2 \geq \left\| \U_1 - \U_2 \right\|_{2,\Delta t}^2$. All in all, $\mathcal{E}_\mathrm{tot}^{\Delta t}$ is strongly convex wrt.\ $\|\cdot\|_{i,\Delta t}$ with constant $\sigma_i=1$, $i=1,2$.
 
\paragraph{(ii) Lipschitz continuity of $\GRAD_{\u}\mathcal{E}_\mathrm{tot}^{\Delta t}$.}
Let $\U=(\v,\z)\in\mathcal{V}^n\times\mathcal{Z}^n$. It holds
\begin{align*}
 &\underset{\bm{h}\in\mathcal{V}_0}{\mathrm{sup}}\, \frac{\llangle \GRAD_{\u} \mathcal{E}_\mathrm{tot}^{\Delta t}\left(\U + (\bm{h},\bm{0})\right) - \GRAD_{\u} \mathcal{E}_\mathrm{tot}^{\Delta t}(\U), (\bm{h},\bm{0})\rrangle}{\| (\bm{h},\bm{0}) \|_{1,\Delta t}^2} \\
 &\quad=
 \underset{\bm{h}\in\mathcal{V}_0}{\mathrm{sup}}\, \frac{\llangle \mathbb{C} \eps{\bm{h}}, \eps{\bm{h}}\rrangle + M\alpha^2 \| \DIV \bm{h} \|^2}{\llangle \mathbb{C} \eps{\bm{h}}, \eps{\bm{h}}\rrangle} \\
 &\quad\leq
 1 + \frac{M\alpha^2}{K_\mathrm{dr}^\star},
\end{align*}
where we used~\eqref{proof:poro:undrained:kdr}. All in all, $\GRAD_{\u}\mathcal{E}_\mathrm{tot}^{\Delta t}$ is Lipschitz continuous wrt.\ $\|\cdot\|_{1,\Delta t}$ with Lipschitz constant $L_1 = 1 + \frac{M\alpha^2}{K_\mathrm{dr}^\star}$.
 
\paragraph{(iii) Lipschitz continuity of $\GRAD_{\flux}\mathcal{E}_\mathrm{tot}^{\Delta t}$.}
Let $\U=(\v,\z)\in\mathcal{V}^n\times\mathcal{Z}^n$. It holds
\begin{align*}
 &\underset{\bm{h}\in\mathcal{Z}_0}{\mathrm{sup}}\, \frac{\llangle \GRAD_{\flux} \mathcal{E}_\mathrm{tot}^{\Delta t}\left(\U + (\bm{0} , \bm{h})\right) - \GRAD_{\flux} \mathcal{E}_\mathrm{tot}^{\Delta t}(\U), (\bm{0},\bm{h})\rrangle}{\| (\bm{0},\bm{h}) \|_{2,\Delta t}^2} \\
 &\quad=
 \underset{\bm{h}\in\mathcal{Z}_0}{\mathrm{sup}}\, \frac{ \Delta t \llangle \permeability^{-1} \bm{h}, \bm{h} \rrangle + \Delta t^2 M \| \DIV \bm{h} \|^2}{ \Delta t \llangle \permeability^{-1} \bm{h}, \bm{h} \rrangle + \Delta t^2 \left(\tfrac{1}{M} + \tfrac{\alpha^2}{K_\mathrm{dr}^\star} \right)^{-1} \| \DIV \bm{h} \|^2} \\
 &\quad\leq
 1 + \frac{M\alpha^2}{K_\mathrm{dr}^\star}.
\end{align*}
All in all, $\GRAD_{\flux}\mathcal{E}_\mathrm{tot}^{\Delta t}$ is Lipschitz continuous wrt.\ $\|\cdot\|_{2,\Delta t}$ with Lipschitz constant $L_2=1 + \frac{M\alpha^2}{K_\mathrm{dr}^\star}$.

\paragraph{Consequences.}
By Lemma~\ref{appendix:lemma:alternating-minimization}, it follows
\begin{align*}
  \mathcal{E}_\mathrm{tot}^{\Delta t}(\U^{n,i}) - \mathcal{E}_\mathrm{tot}^{\Delta t}(\U^n)
  &\leq
  \left(1 - \tfrac{1}{L_1}\right) \left(1 - \tfrac{1}{L_2}\right)  \left(\mathcal{E}_\mathrm{tot}^{\Delta t}(\U^{n,i-1}) - \mathcal{E}_\mathrm{tot}^{\Delta t}(\U^n) \right), \\
  \mathcal{E}_\mathrm{tot}^{\Delta t}(\U^{n,i}) - \mathcal{E}_\mathrm{tot}^{\Delta t}(\U^n)
  &\leq
  L_1 L_2
  \left(\mathcal{E}_\mathrm{tot}^{\Delta t}(\U^{n,i-1}) - \mathcal{E}_\mathrm{tot}^{\Delta t}(\U^{n,i}) \right).
\end{align*}
Moreover, as $\mathcal{E}_\mathrm{tot}^{\Delta t}$ is quadratic and $\GRAD \mathcal{E}_\mathrm{tot}^{\Delta t}(\U^n)=\bm{0}$, it holds
\begin{align*}
 \mathcal{E}_\mathrm{tot}^{\Delta t}(\U^{n,i}) - \mathcal{E}_\mathrm{tot}^{\Delta t}(\U^n) = \left|\left|\left| (\erroru^{n,i},\errorflux^{n,i})\right|\right|\right|^2.
\end{align*}
The results~\eqref{lemma:undrained-split:convergence-rate:result-2} and~\eqref{lemma:undrained-split:convergence-rate:result-3} follow directly. 
\end{proof}

\begin{remark} We emphasize that $K_\mathrm{dr}^\star$ in~\eqref{proof:poro:undrained:kdr} enters only the theoretical result, and exact knowledge is not required for practical use of the splitting scheme.
\end{remark}

\subsubsection{Derivation and analysis of the fixed-stress split as alternating minimization}

In this section, we identify the widely used fixed-stress split~\cite{Kim2011} as alternating minimization applied to the dual formulation of time-discrete, linear poro-elasticity, cf.\ Sec.~\ref{section:minimization:dual-formulation}. In the following, the resulting scheme is illustrated  with reference to the closely related five-field formulation, cf.\ Sec.~\ref{section:minimization:five-field}; in fact the five-field formulation (Sec.~\ref{section:minimization:five-field}) can be equally used as basis leading to the same scheme.

\begin{algorithm}
 \caption{Single iteration of the fixed-stress split}
 \label{algorithm:fixed-stress-split}
 \SetAlgoLined
 \DontPrintSemicolon
 
  \vspace{0.5em}
 
  Input: $(\stress^{n,i-1},p^{n,i-1})\in\mathcal{S}^n\times\mathcal{Q}^n$ \; \vspace{0.8em}
   
  Determine $p^{n,i} := \underset{p\in\mathcal{Q}^n}{\mathrm{arg\,min}}\, \mathcal{E}_\mathrm{tot}^{\star,\Delta t}(\fluidmass^{n-1};\stress^{n,i-1},p)$\; \vspace{0.25em}

  Determine $\stress^{n,i} := \underset{\stress\in\mathcal{S}^n}{\mathrm{arg\,min}}\, \mathcal{E}_\mathrm{tot}^{\star,\Delta t}(\fluidmass^{n-1};\stress,p^{n,i})$\; \vspace{0.2em}
\end{algorithm}

Alternating minimization is applied respecting the natural block structure of the dual problem, cf.\ Alg.~\ref{algorithm:fixed-stress-split} for a single iteration. It is important to note that the problem is block separable; constraints decouple into purely mechanical and fluid flow constraints. Hence, alternating minimization can be applied without violating constraints. The first step is equivalent to solving a stabilized flow problem, cf.~\eqref{optimality-conditions:five-field-biot:start-fluid}
--\eqref{optimality-conditions:five-field-biot:end}: For given $(\stress^{n,i-1},p^{n,i-1})\in \mathcal{S}^n\times\mathcal{Q}^n$, find $p^{n,i}\in \mathcal{Q}^n$ satisfying for all $q \in \mathcal{Q}_0$
\begin{align}
\label{fixed-stress-pressure-equation}
 &\tfrac{1}{M} \llangle p^{n,i}, q \rrangle + \llangle \alpha^2 \left(\mathbf{I} : \mathbb{A} : \mathbf{I} \right) \left(p^{n,i} - p^{n,i-1}\right), q \rrangle \\
 &\quad + \alpha \llangle \mathrm{tr}\, \bm{\varepsilon}_{\u}^{n,i-1}, q \rrangle 
 + \Delta t \llangle \DIV \flux^{n,i}, q \rrangle = \llangle \fluidmass^{n-1} + \Delta t \, \masssource^n, q \rrangle,
\end{align}
where we have formally abbreviated the mechanical strain and the volumetric flux, consistent with~\eqref{linear-biot-interpretation-stress} and Darcy's law
\begin{align*}
 \bm{\varepsilon}_{\u}^{n,i} &:= \mathbb{A}(\stress^{n,i} + \alpha p^{n,i} \,\mathbf{I}), \\
 \flux^{n,i} &:= -\permeability \left( \GRAD p^{n,i} - \gext^n \right).
\end{align*}
For an isotropic, homogeneous material, the stabilization term equals 
\begin{align*}
 \tfrac{\alpha^2}{K_\mathrm{dr}}\llangle p^{n,i} - p^{n,i-1} , q \rrangle.
\end{align*}
The second step is equivalent to solving the mechanics problem with updated fluid flow variables, cf.~\eqref{optimality-conditions:five-field-biot:start}--\eqref{optimality-conditions:five-field-biot:end-mechanics}.

We highlight that the resulting scheme is equivalent to the fixed-stress split for linear poro-elasticity. As for the undrained split, we establish convergence based on an abstract convergence result for alternating minimization. After all, the following convergence result is consistent with results based on previous problem-specific \textit{a priori} error analyses~\cite{Mikelic2013,Both2017,Storvik2019} and an \textit{a posteriori} error analysis based on Ostrowski-type estimates for contraction mappings~\cite{Kumar2018}.

\begin{lemma}[Linear convergence of the fixed-stress split]\label{lemma:fixed-stress:convergence-rate}
 The fixed-stress split converges linearly, independent of the initial guess. Let $\generalizedstress^n:=(\stress^{n},p^{n})$ denote the solution of the coupled problem~\eqref{dual-two-field-minimization:start} and let $\generalizedstress^{n,i}:=(\stress^{n,i},p^{n,i})$ denote the iterates defined by the fixed-stress split, cf.\ Alg.~\ref{algorithm:fixed-stress-split}. For $i\in\mathbb{N}$, define the errors $\errorstress^{n,i}:=\stress^{n,i}-\stress^n$, $\errorpressure^{n,i}:=p^{n,i} - p^n$. 
 Let $|||\cdot|||_\star$ denote the norm induced by the quadratic part of $\mathcal{E}_\mathrm{tot}^{\star,\Delta t}$
 \begin{align*}
  \left|\left|\left|(\stress,p)\right|\right|\right|_\star^2 :=
  \tfrac{1}{2} \llangle \mathbb{A} (\stress+\alpha p \mathbf{I}),\stress + \alpha p \mathbf{I} \rrangle 
 + \tfrac{1}{2M} \| p \|^2 + \tfrac{\Delta t}{2} \llangle \permeability \GRAD p, \GRAD p \rrangle.
 \end{align*}
 It holds the \textit{a priori} result
\begin{align}
 \label{lemma:fixed-stress:convergence-rate:result-2}
  \left|\left|\left| (\errorstress^{n,i},\errorpressure^{n,i}) \right|\right|\right|_\star 
  &\leq 
  \left(\frac{\tfrac{\alpha^2}{K_\mathrm{dr}}}{\tfrac{1}{M} + \tfrac{\Delta t\, \kappa_\mathrm{m}}{C_\Omega^2} + \tfrac{\alpha^2}{K_\mathrm{dr}}}\right)^i \left( \mathcal{E}^{n,0} - \mathcal{E}^n \right)^{1/2},
 \end{align}
 and the \textit{a posteriori} result
 \begin{align}
 \label{lemma:fixed-stress:convergence-rate:result-3}
  \left|\left|\left| (\errorstress^{n,i}, \errorpressure^{n,i})\right|\right|\right|_\star
  &\leq 
  \left(\frac{\tfrac{\alpha^2}{K_\mathrm{dr}}}{\tfrac{1}{M} + \tfrac{\Delta t\, \kappa_\mathrm{m}}{C_\Omega^2} + \tfrac{\alpha^2}{K_\mathrm{dr}}}\right)
  \left[1 - \left(\frac{\tfrac{\alpha^2}{K_\mathrm{dr}}}{\tfrac{1}{M} + \tfrac{\Delta t\, \kappa_\mathrm{m}}{C_\Omega^2} + \tfrac{\alpha^2}{K_\mathrm{dr}}}\right)^2 \right]^{-1/2}
  \left( \mathcal{E}^{n,i-1} - \mathcal{E}^{n,i} \right)^{1/2}
 \end{align}
 where $\mathcal{E}^n := \mathcal{E}_\mathrm{tot}^{\star,\Delta t}(\fluidmass^{n-1}; \generalizedstress^{n})$ and $\mathcal{E}^{n,j} := \mathcal{E}_\mathrm{tot}^{\star,\Delta t}(\fluidmass^{n-1}; \generalizedstress^{n,j}), \ j\in\mathbb{N}$, and $C_\Omega$ and $\kappa_m$ denote the Poincar\'e constant and the smallest eigenvalue of $\permeability$, respectively.
\end{lemma}
 
\begin{proof}
 We apply Lemma~\ref{appendix:lemma:alternating-minimization}. For this, we introduce two semi-norms
 \begin{align*}
  \| \generalizedstress \|_{1,\star,\Delta t}^2 &:= \tfrac{1}{M} \|p\|^2 + \Delta t \llangle \permeability \GRAD p, \GRAD p \rrangle, \\
  \| \generalizedstress \|_{2,\star,\Delta t}^2 &:= 
  \llangle \mathbb{A} \stress, \stress\rrangle - \frac{\tfrac{\alpha^2}{K_\mathrm{dr}}}{\tfrac{1}{M} + \tfrac{\Delta t \kappa_\mathrm{m}}{C_\Omega^2}  + \tfrac{\alpha^2}{K_\mathrm{dr}}} \, \frac{1}{K_\mathrm{dr}}  \| \sigma^\mathrm{h} \|^2,
 \end{align*}
 where $\sigma^h$ denotes the hydrostatic component of $\stress$. Positive semi-definiteness of $\|\cdot\|_\mathrm{2,\star,\Delta t}$ follows from~\eqref{elasticity-compliance-tensor-quadratic}. We show that $\mathcal{E}_\mathrm{tot}^{\star,\Delta t}$ is (i) strongly convex wrt.\ $\|\cdot\|_{1,\star,\Delta t}$ and $\|\cdot\|_{2,\star,\Delta t}$, (ii) $\GRAD_{p}\mathcal{E}_\mathrm{tot}^{\star,\Delta t}$ is Lipschitz continuous wrt.\ $\|\cdot\|_{1,\star,\Delta t}$, and (iii) $\GRAD_{\stress}\mathcal{E}_\mathrm{tot}^{\star,\Delta t}$ is Lipschitz continuous $\|\cdot\|_{2,\star,\Delta t}$. Throughout the proof, for lighter notation, we omit the explicit dependence of $\mathcal{E}_\mathrm{tot}^{\star,\Delta t}$ on $\fluidmass^{n-1}$.

\paragraph{(i) Strong convexity of $\mathcal{E}_\mathrm{tot}^{\star,\Delta t}$.}
Let $\generalizedstress_i=(\stress_i,p_i)\in\mathcal{S}^n\times\mathcal{Q}^n$, $i=1,2$. As $\mathcal{E}_\mathrm{tot}^{\star,\Delta t}$ is quadratic, it holds
\begin{align*}
  &\llangle \GRAD \mathcal{E}_\mathrm{tot}^{\star,\Delta t}(\generalizedstress_1) - \GRAD \mathcal{E}_\mathrm{tot}^{\star,\Delta t}(\generalizedstress_2),\generalizedstress_1 - \generalizedstress_2 \rrangle 
 =
 2 \left|\left|\left| \generalizedstress_1 - \generalizedstress_2\right|\right|\right|_\star^2.
\end{align*}
It follows directly, that $2\left|\left|\left| \generalizedstress_1 - \generalizedstress_2 \right|\right|\right|_\star^2 \geq \| \generalizedstress_1 - \generalizedstress_2 \|_{1,\star,\Delta t}^2$. Furthermore, utilizing the Poincar\'e inequality and decomposing the stress into its deviatoric and hydrostatic components, we obtain for all $(\stress,p)\in\tilde{\mathcal{S}}_0\times \mathcal{Q}_0$
\begin{align}\label{proof:fixed-stress:strong-convexity-2}
2\left|\left|\left| \generalizedstress \right|\right|\right|_\star^2
&\geq 
 \llangle \mathbb{A} (\stress + \alpha p \,\mathbf{I}), \stress + \alpha p \,\mathbf{I} \rrangle + \left(\tfrac{1}{M} + \tfrac{\Delta t \kappa_\mathrm{m}}{C_\Omega^2} \right) \| p \|^2 \\
 \nonumber
&= 
 \llangle \mathbb{A} \stress, \stress \rrangle 
 + 2 \tfrac{\alpha}{K_\mathrm{dr}} \langle \sigma^\mathrm{h}, p \rangle + \left( \tfrac{1}{M} + \tfrac{\Delta t \kappa_\mathrm{m}}{C_\Omega^2} + \tfrac{\alpha^2}{K_\mathrm{dr}} \right) \| p \|^2 \\
 \nonumber
&\geq
\| \generalizedstress \|_{2,\star,\Delta t}^2.
\end{align}
By applying Young's inequality adequately and rearranging terms, we obtain $2\,\left|\left|\left| \generalizedstress_1 - \generalizedstress_2 \right|\right|\right|_\star^2 \geq \| \generalizedstress_1 - \generalizedstress_2 \|_{2,\star,\Delta t}^2$. All in all, $\mathcal{E}_\mathrm{tot}^{\star,\Delta t}$ is strongly convex wrt.\ $\|\cdot\|_{i,\star,\Delta t}$ with constant $\sigma_i=1$, $i=1,2$.

\paragraph{(ii) Lipschitz continuity of $\GRAD_{p}\mathcal{E}_\mathrm{tot}^{\star,\Delta t}$.}
Let $\generalizedstress=(\stress,p)\in\mathcal{S}^n\times\mathcal{Q}^n$. It holds
\begin{align*}
 &\underset{h\in\mathcal{Q}_0}{\mathrm{sup}}\, \frac{\llangle \GRAD_{p} \mathcal{E}_\mathrm{tot}^{\star,\Delta t}\left(\generalizedstress + (\bm{0},h)\right) - \GRAD_{p} \mathcal{E}_\mathrm{tot}^{\star,\Delta t}(\generalizedstress), (\bm{0},h) \rrangle}{\| (\bm{0},h) \|_{1,\star,\Delta}^2} \\
 &\quad=
 \underset{h\in\mathcal{Q}_0}{\mathrm{sup}}\, \frac{\tfrac{\alpha^2}{K_\mathrm{dr}} \| h \|^2 + \| (\bm{0},h) \|_{1,\star,\Delta}^2}{\| (\bm{0},h) \|_{1,\star,\Delta}^2}.
\end{align*}
Decomposing and bounding $\|h\|$ optimally by $\|h\|$ and $\|\GRAD h\|$, using the Poincar\'e inequality, we obtain
\begin{align*}
 \|h\|^2 &
 \leq \left(\frac{1}{M} + \frac{\Delta t \kappa_\mathrm{m}}{C_\Omega^2}\right)^{-1} \|(\bm{0},h)\|_{1,\star,\Delta t}^2.
\end{align*}
All in all, $\GRAD_{p}\mathcal{E}_\mathrm{tot}^{\star,\Delta t}$ is Lipschitz continuous wrt.\ $\|\cdot\|_{1,\star,\Delta t}$ with Lipschitz constant $L_1=1 + \frac{\alpha^2}{K_\mathrm{dr}} \, \left(\tfrac{1}{M} + \tfrac{\Delta t\kappa_\mathrm{m}}{C_\Omega^2 }\right)^{-1}$.

\paragraph{(iii) Lipschitz continuity of $\GRAD_{\stress}\mathcal{E}_\mathrm{tot}^{\star,\Delta t}$.}
Let $\generalizedstress=(\stress,p)\in\mathcal{S}^n\times\mathcal{Q}^n$. It holds
\begin{align*}
 &\underset{\bm{h}\in\tilde{\mathcal{S}}_0}{\mathrm{sup}}\, \frac{\llangle \GRAD_{\stress} \mathcal{E}_\mathrm{tot}^{\star,\Delta t}\left(\generalizedstress + (\bm{h},0)\right) - \GRAD_{\stress} \mathcal{E}_\mathrm{tot}^{\star,\Delta t}(\generalizedstress), (\bm{h},0)\rrangle}{\| (\bm{h},0) \|_{2,\star,\Delta}^2} \\
 &\quad=
 \underset{\bm{h}\in\tilde{\mathcal{S}}_0}{\mathrm{sup}}\, \frac{\tfrac{1}{2\mu} \| \bm{h}^\mathrm{d} \|^2 + \tfrac{1}{K_\mathrm{dr}} \| h^\mathrm{h} \|^2}{\tfrac{1}{2\mu} \|\bm{h}^\mathrm{d}\|^2 + \frac{1}{K_\mathrm{dr}} \left(1 - \frac{\tfrac{\alpha^2}{K_\mathrm{dr}}}{\tfrac{1}{M} + \tfrac{\Delta t \kappa_\mathrm{m}}{C_\Omega^2}  + \tfrac{\alpha^2}{K_\mathrm{dr}}} \right) \| h^\mathrm{h} \|^2}.
\end{align*}
We conclude, that $\GRAD_{\stress}\mathcal{E}_\mathrm{tot}^{\star,\Delta t}$ is Lipschitz continuous wrt.\ $\|\cdot\|_{2,\star,\Delta t}$ with Lipschitz constant $L_2=1 + \frac{\alpha^2}{K_\mathrm{dr}} \, \left(\tfrac{1}{M} + \tfrac{\Delta t\kappa_\mathrm{m}}{C_\Omega^2 }\right)^{-1}$.

\paragraph{Consequences.}
By Lemma~\ref{appendix:lemma:alternating-minimization}, it follows
\begin{align*}
  \mathcal{E}_\mathrm{tot}^{\star,\Delta t}(\generalizedstress^{n,i}) - \mathcal{E}_\mathrm{tot}^{\star,\Delta t}(\generalizedstress^n)
  &\leq
  \left(1 - \tfrac{1}{L_1}\right) \left(1 - \tfrac{1}{L_2}\right)
  \left(\mathcal{E}_\mathrm{tot}^{\star,\Delta t}(\generalizedstress^{n,i-1}) - \mathcal{E}_\mathrm{tot}^{\star,\Delta t}(\generalizedstress^n) \right), \\
  \mathcal{E}_\mathrm{tot}^{\star,\Delta t}(\generalizedstress^{n,i}) - \mathcal{E}_\mathrm{tot}^{\star,\Delta t}(\generalizedstress^n)
  &\leq
  L_1L_2 \left(\mathcal{E}_\mathrm{tot}^{\star,\Delta t}(\generalizedstress^{n,i-1}) - \mathcal{E}_\mathrm{tot}^{\star,\Delta t}(\generalizedstress^{n,i}) \right).
\end{align*}
Moreover, since $\mathcal{E}_\mathrm{tot}^{\star,\Delta t}$ is quadratic and $\llangle \GRAD \mathcal{E}_\mathrm{tot}^{\star,\Delta t}(\generalizedstress^n), \generalizedstress - \generalizedstress^n \rrangle \geq 0$ for all $\generalizedstress\in \tilde{\mathcal{S}}^n \times \mathcal{Q}^n$, it holds
\begin{align*}
 \mathcal{E}_\mathrm{tot}^{\star,\Delta t}(\generalizedstress^{n,i}) - \mathcal{E}_\mathrm{tot}^{\star,\Delta t}(\generalizedstress^n) \geq \left|\left|\left| (\errorstress^{n,i},\errorpressure^{n,i})\right|\right|\right|_\star^2.
\end{align*}
The results~\eqref{lemma:fixed-stress:convergence-rate:result-2} and~\eqref{lemma:fixed-stress:convergence-rate:result-3} follow directly. 
\end{proof}

\subsubsection{General remarks and implications}

We close the section on splitting schemes for linear poro-elasticity with general remarks. Most remain true for the subsequent sections.

\begin{enumerate}[label=(\roman*)]

 \item \textit{Order of minimization steps:}  The order of the steps within the alternating minimization algorithm is not relevant for robust convergence; however, we have chosen the specific orders as above to demonstrate the closer connection to the undrained and the fixed-stress splits.
 
 \item \textit{Splitting schemes for particular formulation of the semi-discrete Biot equations:} The different, presented minimization formulations in Sec.~\ref{section:minimization:primal-two-field}--\ref{section:minimization:five-field} are all equivalent. Hence, for each specific formulation a splitting scheme can be constructed by equivalent reformulation of the splitting schemes presented above. 
 
 \item \textit{Splitting schemes for fully-discrete linear Biot equations:} Fully-discrete Biot equations can be constructed by applying the conforming Galerkin method to the different minimization formulations from Sec.~\ref{section:minimization:primal-two-field}--\ref{section:minimization:five-field}. In contrast to~(ii), they are not equivalent. Hence, for a particular fully-discrete formulation, splitting schemes are derived from their corresponding semi-discrete versions, cf.~(ii).
  
 
 \item \textit{Stable spatial discretization under splitting:} 
 In practice it has been observed for the two-field saddle point formulation that inf-sup unstable pairs of finite elements are actually robust under the fixed-stress split~\cite{Yoon2018}. Given that the fixed-stress split is equivalent to a two-block coordinate descent method, which converges already provided that each of the subproblems is uniquely solvable~\cite{Grippof1999}, this observation can now be theoretically explained. After all, it is nevertheless noteworthy that inf-sup stability can be beneficial for the performance of the fixed-stress split when applied to problems with a saddle point structure; e.g., for the two-field saddle point formulation, inf-sup stability adds artificial compressibility~\cite{Storvik2019}.
 
 \item \textit{Different meshes for different subproblems:}  The discussion in~(v) also explains intuitively why splitting schemes allow the use of different meshes for different subproblems, without losing robustness~\cite{Dana2018}. In particular, the minimization structure allows for a natural development of specific two-mesh formulations retaining the symmetric character of the problem.
  
 \item \textit{Heterogeneous, anisotropic media:} The minimization structure of time-discrete poro-elasticity remains inherent for heterogeneous, anisotropic media. Consequently, alternating minimization can be again employed for constructing robust splitting schemes. In particular, convergence properties of the undrained split and fixed-stress split are retained, consistent with problem-specific analyses~\cite{Both2017,Dana2018b}. 
 
 \item \textit{Inexact alternating minimization:} Clearly, instead of employing exact alternating minimization, each step may be also solved inexactly. As long as the energy is sufficiently decreased, convergence is still guaranteed. This allows for a more efficient implementation of splitting schemes employing, e.g., iterative solvers with coarse stopping criteria for each subproblem. 
\end{enumerate}

We will return to points (ii), (iii), and (vii) in the numerical examples in Sec.~\ref{section:numerical-results}.

\section{Robust splitting schemes for discrete linear poro-visco-elasticity}\label{section:splitting-poro-visco-elasticity}

In the previous section, popular splitting schemes for linear poro-elasticity have been identified as alternating minimization applied to suitable minimization formulations of semi-discrete, linear poro-elasticity. In this section, we apply the same workflow, cf.\ Fig~\ref{figure:procedure-derivation-splitting}, and analogously derive novel extensions of the undrained and fixed-stress splits, now applicable to semi-discrete, linear poro-visco-elasticity. In this regard, we additionally establish for the first time guaranteed, linear convergence rates utilizing abstract optimization theory. After all, the key observation for the following efforts is the fact that semi-discrete, linear poro-visco-elasticity is a vectorized version of semi-discrete, linear poro-elasticity. Consequently, the subsequent discussion appears as a natural extension of Sec.~\ref{section:splitting-linear-biot}. To highlight the analogy, we attempt to employ visually related notation.

\subsection{Minimization formulations of time-discrete linear poro-visco-elasticity}
We introduce two minimization formulations of time-discrete, linear poro-visco-elasticity. We obtain the primal formulation by applying the minimizing movement scheme to the primal formulation of time-continuous, linear poro-visco-elasticity (Sec.~\ref{section:poro-visco-elasticity-wellposedness}). A dual formulation is then proposed based on the close, structural connection between poro-visco-elasticity and poro-elasticity. Both formulations will serve as bases for the development of robust splitting schemes.
 
\subsubsection{Primal formulation of time-discrete linear poro-visco-elasticity}
\label{section:poro-visco:primal}

The primal formulation of time-discrete, linear poro-visco-elasticity is directly obtained by applying the minimizing movement scheme to linear poro-visco-elasticity~\eqref{viscoporo-elasticity:gradient-flow-structure}. By gathering terms, the resulting formulation can be interpreted as vectorized version of the primal formulation of time-discrete, linear poro-elasticity with a tensorial stiffness matrix (of sixth order) and Biot coefficient
\begin{align*}
 \bm{\mathcal{C}}_\mathrm{v}&:= \begingroup \setlength\arraycolsep{4pt}\begin{bmatrix} 1 & -1 \\ -1  & 1 \end{bmatrix} \endgroup \otimes \mathbb{C} + \begin{bmatrix} 0 & 0\\ 0 & 1 \end{bmatrix} \otimes \left( \mathbb{C}_\mathrm{v} + \tfrac{1}{\Delta t}\mathbb{C}_\mathrm{v}'\right) ,\qquad 
 \bm{\alphamatrix}_\mathrm{v} := \begin{bmatrix} \alpha \\ \alpha_\mathrm{v} - \alpha  \end{bmatrix},
\end{align*}
such that for arbitrary $\bm{\varepsilon}_1$, $\bm{\varepsilon}_2\in \mathbb{R}^{d\times d}$, $\begin{bmatrix} \bm{\varepsilon}_1 \\ \bm{\varepsilon}_2 \end{bmatrix}$ is a third-order tensor and it holds
\begin{align*} 
 \bm{\mathcal{C}}_\mathrm{v}:\begin{bmatrix} \bm{\varepsilon}_1 \\ \bm{\varepsilon}_2 \end{bmatrix} 
 &= \begin{bmatrix} \mathbb{C} \left( \bm{\varepsilon}_1 - \bm{\varepsilon}_2 \right) \\ - \mathbb{C} \left( \bm{\varepsilon}_1 - \bm{\varepsilon}_2 \right) + \left(\mathbb{C}_\mathrm{v} + \tfrac{1}{\Delta t}\mathbb{C}_\mathrm{v}'\right) \bm{\varepsilon}_2 \end{bmatrix},\\
 \left(\bm{\alphamatrix}_\mathrm{v} \otimes \mathbf{I}\right) : \begin{bmatrix} \bm{\varepsilon}_1 \\ \bm{\varepsilon}_2 \end{bmatrix} &= \bm{\alphamatrix}_\mathrm{v}^\top \begin{bmatrix} \mathrm{tr}\, \bm{\varepsilon}_1 \\ \mathrm{tr}\, \bm{\varepsilon}_2 \end{bmatrix}.
\end{align*}
Let the spaces $\mathcal{V}^n$ and $\mathcal{Z}^n$ be as defined in~\eqref{primal-discrete-poro-elasticity-spaces:start}--\eqref{primal-discrete-poro-elasticity-spaces:end}, and define additionally $\mathcal{T}^n:=\mathcal{T}$. We obtain the time-discrete, primal formulation: For time step $n\geq 1$, given $\fluidmass^{n-1},\viscostrain^{n-1}$, define $(\u^n,\viscostrain^n,\flux^n)\in\mathcal{V}^n\times\mathcal{T}^n\times\mathcal{Z}^n$ to be the solution of the minimization problem
\begin{align}
\label{primal-porovisco-minimization:start}
 (\u^n,\viscostrain^n,\flux^n) :=& \underset{(\u,\viscostrain,\flux)\in\mathcal{V}^n\times\mathcal{T}^n\times\mathcal{Z}^n}{\mathrm{arg\,min}}\,\mathcal{E}_\mathrm{v,tot}^{\Delta t}(\fluidmass^{n-1},\viscostrain^{n-1};\u,\viscostrain,\flux),
\end{align}
where
\begin{align*}
  &\mathcal{E}_\mathrm{v,tot}^{\Delta t}(\fluidmass^{n-1},\viscostrain^{n-1};\u,\viscostrain,\flux) \\
 \nonumber
 &\quad:=
 \tfrac{1}{2} \llangle  \bm{\mathcal{C}}_\mathrm{v} : \begin{bmatrix} \eps{\u} \\ \viscostrain \end{bmatrix},  \begin{bmatrix} \eps{\u} \\ \viscostrain \end{bmatrix} \rrangle + \tfrac{\Delta t}{2} \llangle \permeability^{-1} \flux, \flux \rrangle \\
\nonumber
 &\qquad+ \tfrac{M}{2} \left\| \fluidmass^{n-1}  + \Delta t\, \masssource^n - \Delta t \, \DIV \flux - \left( \bm{\alphamatrix}_\mathrm{v} \otimes \mathbf{I} \right) : \begin{bmatrix} \eps{\u} \\ \viscostrain\end{bmatrix} \right\|^2 \\
\nonumber
 &\qquad- \mathcal{P}^n_\mathrm{ext,mech}(\u) 
 - \llangle \tfrac{1}{\Delta t} \mathbb{C}_\mathrm{v}' \viscostrain^{n-1}, \viscostrain \rrangle
 - \Delta t \, \mathcal{P}^n_\mathrm{ext,fluid}(\flux),
\end{align*}
and set $\fluidmass^n:=\fluidmass^{n-1} +\Delta t\, \masssource^n - \Delta t\, \DIV \flux^n$. Since the minimization problem is strictly convex and coercive, existence and uniqueness of a solution to~\eqref{primal-porovisco-minimization:start} follow by classical results from convex analysis, cf.\ Thm.~\ref{appendix:well-posedness:convex-minimization}.

\begin{remark}[Explicit reduction to linear poro-elasticity]
 The first variation of $\mathcal{E}_\mathrm{v,tot}^{\Delta t}(\fluidmass^{n-1},\viscostrain^{n-1};\u,\viscostrain,\flux) $ wrt.\ $\viscostrain$ can be locally inverted for $\viscostrain$. Consequently, the coupled problem~\eqref{primal-porovisco-minimization:start} can be easily reduced to a problem for $(\u,\flux)$. This allows in particular for reusing code written for linear poro-elasticity.
\end{remark}

\subsubsection{Dual formulation of time-discrete linear poro-visco-elasticity}
\label{section:poro-visco:dual}

By adopting the analogy between the primal and dual formulations of time-discrete, linear poro-elasticity (Sec.~\ref{section:splitting-linear-biot}) to its vectorized form, we introduce a natural dual minimization formulation of time-discrete, linear poro-visco-elasticity.
Based on the corresponding primal formulation written as vectorized Biot equations (Sec.~\ref{section:poro-visco:primal}), we introduce natural dual variables $(\stress,\stress_\mathrm{v},p)$: a pair of stresses $(\stress,\stress_\mathrm{v})$, consisting of the total stress $\stress$ and a stress-type field $\stress_\mathrm{v}$ enforcing the visco-elastic strain, and a fluid pressure $p$, formally related to the primal variables by
\begin{align}
\label{poro-visco:dual:dual-variables:start}
\begin{bmatrix} \stress \\ \stress_\mathrm{v} \end{bmatrix}
&=  \bm{\mathcal{C}}_\mathrm{v} : \left( 
\begin{bmatrix} \eps{\u} \\ \viscostrain \end{bmatrix}
  -\left(\bm{\alphamatrix}_\mathrm{v} \otimes \mathbf{I}\right)\, p \right), \\
  p &= M \left(\fluidmass^{n-1}  + \Delta t\, \masssource^n - \Delta t \, \DIV \flux - \left(\bm{\alphamatrix}_\mathrm{v} \otimes \mathbf{I} \right) : \begin{bmatrix} \eps{\u} \\ \viscostrain\end{bmatrix}  \right).
\label{poro-visco:dual:dual-variables:end}
\end{align}

Analogous to linear poro-elasticity, constrained function spaces for the stress variables are used, with the constraints dictated by the primal formulation. Formally, it holds $\stress_\mathrm{v} = \tfrac{1}{\Delta t}\mathbb{C}_\mathrm{v}' \viscostrain^{n-1}$. Hence, given $\viscostrain^{n-1}$, we set 
\begin{align}
\label{porovisco-visco-elastic-stress-space}
 \mathcal{S}_\mathrm{v}^n := \left\{ \tfrac{1}{\Delta t} \mathbb{C}_\mathrm{v}' \viscostrain^{n-1} \right\}.
\end{align}
Then $\mathcal{S}^n \times \mathcal{S}_\mathrm{v}^n \times \mathcal{Q}^n$ is a suitable function space for the dual variables $(\stress,\stress_\mathrm{v},p)$.
  
Let $\bm{\mathcal{A}}_\mathrm{v}:=\bm{\mathcal{C}}_\mathrm{v}^{-1}$ denote the generalized compliance tensor. It satisfies for all $\stress,\stress_\mathrm{v}$ with deviatoric components $\stress^\mathrm{d},\stress_\mathrm{v}^\mathrm{d}$ and hydrostatic components $\sigma^\mathrm{h},\sigma_\mathrm{v}^\mathrm{h}$
\begin{align}\label{poro-visco:generalized-compliance-tensor}
 \llangle \bm{\mathcal{A}}_\mathrm{v} : \begin{bmatrix} \stress \\ \stress_\mathrm{v} \end{bmatrix},  \begin{bmatrix} \stress \\ \stress_\mathrm{v} \end{bmatrix} \rrangle
 =
 \llangle (2\bm{\mathcal{M}})^{-1}  : \begin{bmatrix} \stress^\mathrm{d} \\ \stress_\mathrm{v}^\mathrm{d} \end{bmatrix},  \begin{bmatrix} \stress^\mathrm{d} \\ \stress_\mathrm{v}^\mathrm{d} \end{bmatrix} \rrangle
 +
 \llangle \mathbf{K}^{-1} \begin{bmatrix} \sigma^\mathrm{h} \\ \sigma_\mathrm{v}^\mathrm{h} \end{bmatrix},  \begin{bmatrix} \sigma^\mathrm{h} \\ \sigma_\mathrm{v}^\mathrm{h} \end{bmatrix} \rrangle
\end{align}
analogous to~\eqref{elasticity-compliance-tensor-quadratic},
where $\mathbbm{1}$ is the fourth-order identity tensor, and
\begin{align*} 
 \bm{\mathcal{M}} &:= 
 \mu \begingroup \setlength \arraycolsep{4pt}\begin{bmatrix} 1 &-1 \\ -1 & 1 \end{bmatrix} \endgroup \otimes \mathbbm{1} 
 +
 \left( \mu_\mathrm{v} + \tfrac{1}{\Delta t} \mu_\mathrm{v}'\right)\begin{bmatrix} 0 & 0 \\ 0  & 1 \end{bmatrix} \otimes \mathbbm{1},
 \\
 \mathbf{K} &:= 
  K_\mathrm{dr} \begingroup \setlength \arraycolsep{4pt}\begin{bmatrix} 1 &-1 \\ -1 & 1 \end{bmatrix} \endgroup + \left( K_\mathrm{dr,v} + \tfrac{1}{\Delta t} K_\mathrm{dr,v}'\right)\begin{bmatrix} 0 & 0 \\ 0 & 1 \end{bmatrix}.
\end{align*}

Then the dual formulation of time-discrete, linear poro-visco-elasticity for time step $n\geq1$ reads: Given $(\stress^{n-1},\stress_\mathrm{v}^{n-1},p^{n-1})\in \mathcal{S}^{n-1} \times \mathcal{S}_\mathrm{v}^{n-1}\times\mathcal{Q}^{n-1}$, set 
\begin{align}
\nonumber
 \fluidmass^{n-1}&:=\tfrac{1}{M}p^{n-1} 
 +
 \left(\bm{\alphamatrix}_\mathrm{v} \otimes \mathbf{I} \right) :  \bm{\mathcal{A}}_\mathrm{v}:
 \left( \begin{bmatrix} \stress^{n-1} \\ \stress_\mathrm{v}^{n-1} \end{bmatrix} +\left(\bm{\alphamatrix}_\mathrm{v} \otimes \mathbf{I}\right)\, p^{n-1} \right),\\ 
 \label{porovisco-visco-elastic-strain-definition}
 \viscostrain^{n-1}&:=\begin{bmatrix} \bm{0}, & \mathbf{I} \end{bmatrix}\, \left(\bm{\mathcal{A}}_\mathrm{v} :\left( \begin{bmatrix} \stress^{n-1} \\ \stress_\mathrm{v}^{n-1} \end{bmatrix} +\left(\bm{\alphamatrix}_\mathrm{v} \otimes \mathbf{I}\right)\, p^{n-1} \right)\right),
\end{align}
and define $(\stress^{n},\stress_\mathrm{v}^n,p^{n})\in \mathcal{S}^n\times\mathcal{S}_\mathrm{v}^n\times\mathcal{Q}^n$ to be the solution of the block-separable, constrained minimization problem
\begin{align}
\label{structure-preserving-porovisco-minimization:start}
  (\stress^{n},\stress_\mathrm{v}^n,p^{n}) :=& \underset{(\stress,\stress_\mathrm{v},p)\in \mathcal{S}^n \times \mathcal{S}_\mathrm{v}^n \times \mathcal{Q}^n}{\mathrm{arg\,min}}\,\mathcal{E}_\mathrm{v,tot}^{\star,\Delta t}(\fluidmass^{n-1};\stress,\stress_\mathrm{v},p),
\end{align}
where
\begin{align*}
 \mathcal{E}_\mathrm{v,tot}^{\star,\Delta t}(\fluidmass^{n-1};\stress,\stress_\mathrm{v},p) 
 :=& 
 \tfrac{1}{2} \llangle \bm{\mathcal{A}}_\mathrm{v} : \left( \begin{bmatrix} \stress \\ \stress_\mathrm{v} \end{bmatrix} +\left(\bm{\alphamatrix}_\mathrm{v} \otimes \mathbf{I}\right) \, p \right), \begin{bmatrix} \stress \\ \stress_\mathrm{v} \end{bmatrix} + \left(\bm{\alphamatrix}_\mathrm{v} \otimes  \mathbf{I}\right) \, p\rrangle \\
 &+ \tfrac{1}{2M} \| p \|^2 + \tfrac{\Delta t}{2} \llangle \permeability \left( \GRAD p -\gext^n\right) , \GRAD p - \gext^n \rrangle \\
 &- \llangle  \u_\Gamma^n, \stress\bm{n} \rrangle_{\Gamma_{\u}}
 -\llangle \fluidmass^{n-1} + \Delta t \, \masssource^n , p \rrangle 
 - \llangle q_\mathrm{\Gamma,n}, p \rrangle_{\Gamma_{\flux}}.
\end{align*}
The minimization problem is strictly convex and the feasible set is non-empty and convex; existence and uniqueness of a solution to~\eqref{structure-preserving-porovisco-minimization:start} follow by classical results from convex analysis, cf.\ Thm.~\ref{appendix:well-posedness:convex-minimization}.

\begin{remark}[Relations to previous formulations]\label{poro-visco-elasticity-formulations}
Similar to the primal formulation, we highlight the vectorized character of~\eqref{structure-preserving-porovisco-minimization:start} compared to the dual formulation of time-discrete, linear poro-elasticity~\eqref{dual-two-field-minimization:start}. Furthermore, it is evident, that including $\stress_\mathrm{v}$ as variable is redundant as it is determined beforehand. Hence, in practice, an equivalent formulation can be obtained by simple modification of the dual formulation of time-discrete poro-elasticity. Finally, along the lines of the five-field formulation of time-discrete, linear poro-elasticity (Sec.~\ref{section:minimization:five-field}), a fully structure-preserving formulation can be also obtained for poro-visco-elasticity; for this essentially a flux variable has to be included as primary variable, and Darcy's law has to be enforced.
\end{remark}

\subsection{Physical splitting schemes for time-discrete linear poro-visco-elasticity}  

Since time-discrete, linear poro-visco-elasticity is simply a vectorized generalization of time-discrete, linear poro-elasticity, the robust undrained split and fixed-stress split for poro-elasticity can be generalized to poro-visco-elasticity in a natural fashion. Again the detailed construction and analysis of the splitting schemes utilizes the natural interpretation as alternating minimization.

\subsubsection{Undrained split for poro-visco-elasticity}

We derive a robust splitting scheme by applying alternating minimization to the primal formulation of time-discrete poro-visco-elasticity~\eqref{primal-porovisco-minimization:start}. As before, we choose to minimize successively in the directions of the mechanical and fluid flow variables, cf.\ Alg.~\ref{algorithm:undrained-split-porovisco}. The resulting scheme can be identified as undrained split for poro-visco-elasticity. 

\begin{algorithm}
 \caption{Single iteration of the undrained split for poro-visco-elasticity}
 \label{algorithm:undrained-split-porovisco}
 \SetAlgoLined
 \DontPrintSemicolon
 
  \vspace{0.5em}
 
  Input: $(\u^{n,i-1},\viscostrain^{n,i-1},\flux^{n,i-1})\in\mathcal{V}^n\times\mathcal{T}^n\times\mathcal{Z}^n$ \; \vspace{0.8em}
   
  Determine $(\u^{n,i},\viscostrain^{n,i}) := \underset{(\u,\viscostrain)\in\mathcal{V}^n\times\mathcal{T}^n}{\mathrm{arg\,min}}\, \mathcal{E}_\mathrm{v,tot}^{\Delta t}(\fluidmass^{n-1},\viscostrain^{n-1};\u,\viscostrain,\flux^{n,i-1})$\; \vspace{0.25em}

  Determine $\flux^{n,i} := \underset{\flux\in\mathcal{Z}^n}{\mathrm{arg\,min}}\, \mathcal{E}_\mathrm{v,tot}^{\Delta t}(\fluidmass^{n-1},\viscostrain^{n-1};\u^{n,i},\viscostrain^{n,i},\flux)$\; \vspace{0.2em}
\end{algorithm}

As for linear poro-elasticity, the first step is equivalent to solving a stabilized mechanics problem: For given $(\u^{n,i-1},\viscostrain^{n,i-1},\flux^{n,i-1})\in \mathcal{V}^n\times\mathcal{T}^n\times \mathcal{Z}^n$, find $(\u^{n,i},\viscostrain^{n,i})\in \mathcal{V}^n \times \mathcal{T}^n$ satisfying for all $(\v,\testviscostrain)\in\mathcal{V}_0 \times \mathcal{T}$
\begin{align*}
 \llangle \bm{\mathcal{C}}_\mathrm{v} : \begin{bmatrix} \eps{\u^{n,i}} \\ \viscostrain^{n,i} \end{bmatrix}, \begin{bmatrix} \eps{\v} \\ \testviscostrain \end{bmatrix} \rrangle + \llangle M \bm{\alphamatrix}_\mathrm{v}\bm{\alphamatrix}_\mathrm{v}^\top \begin{bmatrix} \mathrm{tr}\, \eps{\u^{n,i} - \u^{n,i-1}} \\ \mathrm{tr}\left(\viscostrain^{n,i} - \viscostrain^{n,i-1}\right) \end{bmatrix} , \begin{bmatrix} \mathrm{tr}\,\eps{\v} \\ \mathrm{tr}\,\testviscostrain \end{bmatrix} \rrangle \\
 -\llangle \left(\bm{\alphamatrix}_\mathrm{v} \otimes \mathbf{I}\right) \, p^{n,i-1}, \begin{bmatrix} \eps{\v} \\ \testviscostrain \end{bmatrix} \rrangle = \mathcal{P}_\mathrm{ext,mech}^n(\v) + \llangle \tfrac{1}{\Delta t}\mathbb{C}_\mathrm{v}' \viscostrain^{n-1}, \testviscostrain \rrangle
\end{align*}
where the pressure $p^{n,i-1}$ is formally defined, consistent with~\eqref{poro-visco-interpretation-pressure},
\begin{align*}
 p^{n,i-1} &:= M \left(\fluidmass^{n-1}  + \Delta t\, \masssource^n - \Delta t \, \DIV \flux^{n,i-1} - \left(\bm{\alphamatrix}_\mathrm{v} \otimes \mathbf{I} \right) : \begin{bmatrix} \eps{\u^{n,i-1}} \\ \viscostrain^{n,i-1} \end{bmatrix}  \right).
\end{align*}
We highlight a characteristic property: Tensorial stabilization is applied naturally. For instance, the stabilization term equals
\begin{align*}
 \llangle \begingroup \setlength \arraycolsep{3pt} M \begin{bmatrix} \alpha^2 & \alpha (\alpha_\mathrm{v} - \alpha) \\ \alpha (\alpha_\mathrm{v} - \alpha) & (\alpha_\mathrm{v} - \alpha)^2 \end{bmatrix}\endgroup \begin{bmatrix} \mathrm{tr}\, \eps{\u^{n,i} - \u^{n,i-1}} \\ \mathrm{tr}\left(\viscostrain^{n,i} - \viscostrain^{n,i-1}\right) \end{bmatrix} , \begin{bmatrix} \mathrm{tr}\,\eps{\v} \\ \mathrm{tr}\,\testviscostrain \end{bmatrix} \rrangle.
\end{align*}
The second step is equivalent to solving the corresponding fluid flow problem with updated mechanical variables. 

Global convergence follows immediately by abstract analysis on the two-block coordinate descent method. Furthermore, theoretical convergence rates can be derived as for linear poro-elasticity.

\begin{lemma}[Linear convergence of the undrained split for poro-visco-elasticity]
The undrained split converges linearly, independent of the initial guess. Let $(\u^{n},\viscostrain^n,\flux^{n})$ denote the solution of the coupled problem~\eqref{primal-porovisco-minimization:start} and let $(\u^{n,i},\viscostrain^{n,i},\flux^{n,i})$ denote the iterates defined by the undrained split, cf.\ Alg.~\ref{algorithm:undrained-split-porovisco}. For all $i\in\mathbb{N}$, define the errors $\erroru^{n,i}:=\u^{n,i}-\u^n$, $\errorviscostrain^{n,i}:=\viscostrain^{n,i} - \viscostrain^n$, $\errorflux^{n,i}:=\flux^{n,i} - \flux^n$. Let $|||\cdot|||$ denote the norm induced by the quadratic part of $\mathcal{E}_\mathrm{v,tot}^{\Delta t}$
 \begin{align*}
  \left|\left|\left|(\u,\viscostrain,\flux)\right|\right|\right|^2 
  &:= 
  \tfrac{1}{2} \llangle \bm{\mathcal{C}}_\mathrm{v} : \begin{bmatrix} \eps{\u} \\ \viscostrain \end{bmatrix}, \begin{bmatrix} \eps{\u} \\ \viscostrain \end{bmatrix} \rrangle + 
  \tfrac{\Delta t}{2} \llangle \permeability^{-1} \flux, \flux \rrangle  \\
 &+ \tfrac{M}{2} \left\|  \Delta t \DIV \flux + \left(\bm{\alphamatrix}_\mathrm{v} \otimes \mathbf{I} \right) : \begin{bmatrix} \eps{\u} \\  \viscostrain \end{bmatrix} \right\|^2.
 \end{align*}
 Let $K_\mathrm{dr}^\star$ as in~\eqref{proof:poro:undrained:kdr}, and $A_\mathrm{K,\star}^2:= \frac{\alpha_\mathrm{v}^2}{K_\mathrm{dr,v} + \Delta t^{-1} K_\mathrm{dr,v}'} 
  + \frac{\alpha^2}{K_\mathrm{dr}^\star}$. It holds the \textit{a priori} result
 \begin{align}
 \label{lemma:undrained-porovisco:convergence-rate:result-1}
  &\left|\left|\left|(\erroru^{n,i},\errorviscostrain^{n,i},\errorflux^{n,i})\right|\right|\right|
  \leq 
  \left( \frac{A_\mathrm{K,\star}^2}{\tfrac{1}{M} + A_\mathrm{K,\star}^2} \right)^i  
  \left( \mathcal{E}^{n,0} - \mathcal{E}^{n} \right)^{1/2},
 \end{align}
 and the \textit{a posteriori} result
  \begin{alignat}{2}
 \label{lemma:undrained-porovisco:convergence-rate:result-2}
  \left|\left|\left|(\erroru^{n,i},\errorviscostrain^{n,i},\errorflux^{n,i})\right|\right|\right| &\leq 
  \left( \frac{A_\mathrm{K,\star}^2}{\tfrac{1}{M} + A_\mathrm{K,\star}^2} \right)
  \left[1 - \left( \frac{A_\mathrm{K,\star}^2}{\tfrac{1}{M} + A_\mathrm{K,\star}^2} \right)^2 \right]^{-1/2}
  \left( \mathcal{E}^{n, i-1} - \mathcal{E}^{n,i} \right)^{1/2},
 \end{alignat}
 where
 \begin{align*}
 \mathcal{E}^{n}&:= \mathcal{E}_\mathrm{v,tot}^{\Delta t}(\fluidmass^{n-1},\viscostrain^{n-1}; \u^{n}, \viscostrain^{n,j}, \flux^{n}),\\
  \mathcal{E}^{n,j}&:= \mathcal{E}_\mathrm{v,tot}^{\Delta t}(\fluidmass^{n-1},\viscostrain^{n-1}; \u^{n,j}, \viscostrain^{n,j}, \flux^{n,j}),\ j\in\mathbb{N}.
 \end{align*}

\end{lemma}

\begin{proof}
 We follow the same strategy as in the proof of Lemma~\ref{lemma:undrained-split:convergence-rate}. Due to the similarities, we present only the main steps; we stress notation is attempted to look alike. We define two semi-norms
 \begin{align*}
  \| (\u,\viscostrain,\flux) \|_{v,1,\Delta t}^2 &:= \llangle \bm{\mathcal{C}}_\mathrm{v} : \begin{bmatrix} \eps{\u} \\ \viscostrain \end{bmatrix}, \begin{bmatrix} \eps{\u} \\ \viscostrain \end{bmatrix} \rrangle, \\
  \| (\u,\viscostrain,\flux) \|_{v,2,\Delta t}^2 &:= \Delta t \llangle \permeability^{-1} \flux, \flux \rrangle + \Delta t^2  \left( \tfrac{1}{M} + A_\mathrm{K,\star}^2 \right)^{-1} \left\| \DIV \flux \right\|^2.
 \end{align*}
 
\paragraph{(i) Strong convexity of $\mathcal{E}_\mathrm{v,tot}^{\Delta t}$.}
The semi-norms $\| \cdot \|_{v,i,\Delta t}$ are chosen such that $\mathcal{E}_\mathrm{v,tot}^{\Delta t}$ is strongly convex with constant $\sigma_i=1$, $i=1,2$. This is trivial for $\| \cdot \|_{v,1,\Delta t}$. For $\| \cdot \|_{v,2,\Delta t}$, one has to apply Young's inequality and balance weights optimally, similar to~\eqref{proof:poro:undrained:balancing-young}; for this,~\eqref{proof:poro:undrained:kdr} has to be generalized: It holds
 \begin{align}\label{proof:undrained-porovisco:aux-1}
  &\left\| \left(\bm{\alphamatrix}_\mathrm{v} \otimes \mathbf{I} \right) : \begin{bmatrix} \eps{\u} \\ \viscostrain \end{bmatrix} \right\|^2 
  = 
  \left\| \alpha_\mathrm{v} \, \mathrm{tr}\,\viscostrain + \alpha \, \mathrm{tr}(\eps{u} - \viscostrain) \right\|^2 \\
  \nonumber
  &\quad \leq 
  A_\mathrm{K,\star}^2\left( \llangle \left(\tfrac{1}{\Delta t} \mathbb{C}_\mathrm{v}' + \mathbb{C}_\mathrm{v} \right) \viscostrain, \viscostrain \rrangle + \llangle \mathbb{C} (\eps{\u} - \viscostrain), \eps{\u} - \viscostrain \rrangle \right) \\
  \nonumber
  &\quad =
  A_\mathrm{K,\star}^2 \llangle \bm{\mathcal{C}}_\mathrm{v} : \begin{bmatrix} \eps{\u} \\ \viscostrain \end{bmatrix}, \begin{bmatrix} \eps{\u} \\ \viscostrain \end{bmatrix} \rrangle,\qquad \forall (\u,\viscostrain)\in\mathcal{V}_0\times\mathcal{T}.
 \end{align}

\paragraph{(ii) Lipschitz continuity of $\GRAD_{(\u,\viscostrain)}\mathcal{E}_\mathrm{v,tot}^{\Delta t}$ and $\GRAD_{\flux}\mathcal{E}_\mathrm{v,tot}^{\Delta t}$.}
By applying analogous steps as in the proof of Lemma~\ref{lemma:undrained-split:convergence-rate}, one can show that $\GRAD_{(\u,\viscostrain)}\mathcal{E}_\mathrm{v,tot}^{\Delta t}$ and $\GRAD_{\flux}\mathcal{E}_\mathrm{v,tot}^{\Delta t}$ are Lipschitz continuous wrt.\ $\| \cdot \|_{v,1,\Delta t}$ and $\| \cdot \|_{v,2,\Delta t}$, respectively, with Lipschitz constants $L_1=L_2=1 + M A_\mathrm{K,\star}^2$. For the first, utilize~\eqref{proof:undrained-porovisco:aux-1}.

\paragraph{Consequences.}
The final thesis follows from the abstract convergence result Lemma~\ref{appendix:lemma:alternating-minimization}, and the fact that $\mathcal{E}_\mathrm{v,tot}^{\Delta t}$ is quadratic.
\end{proof}

  
\subsubsection{Fixed-stress split for poro-visco-elasticity}\label{section:poro-visco:fixed-stress}

We derive a second, robust splitting scheme by applying alternating minimization to the dual formulation of time-discrete, linear poro-visco-elasticity~\eqref{structure-preserving-porovisco-minimization:start}. As before, we choose to minimize successively in the directions of mechanical and fluid flow variables, cf.\ Alg.~\ref{algorithm:fixed-stress-porovisco}. The resulting scheme can be interpreted as an fixed-stress split for poro-visco-elasticity. 

\begin{algorithm}
 \caption{Single iteration of the fixed-stress split for poro-visco-elasticity}
 \label{algorithm:fixed-stress-porovisco} 
 \SetAlgoLined
 \DontPrintSemicolon
 
  \vspace{0.5em}
 
  Input: $(\stress^{n,i-1}, \stress_\mathrm{v}^{n,i-1}, p^{n,i-1}) \in \mathcal{S}^n\times\mathcal{S}_\mathrm{v}^n \times\mathcal{Q}^n$ \; \vspace{0.8em}
   
  Determine $ p^{n,i} := \underset{p\in\mathcal{Q}^n}{\mathrm{arg\,min}}\, \mathcal{E}_\mathrm{v,tot}^{\star,\Delta t}(\fluidmass^{n-1};\stress^{n,i-1},\stress_\mathrm{v}^{n,i-1},p)$\; \vspace{0.25em}

  Determine $(\stress^{n,i},\stress_\mathrm{v}^{n,i}) = \underset{(\stress,\stress_\mathrm{v})\in \mathcal{S}^n\times\mathcal{S}_\mathrm{v}^n}{\mathrm{arg\,min}}\, \mathcal{E}_\mathrm{v,tot}^{\star,\Delta t}(\fluidmass^{n-1};\stress,\stress_\mathrm{v},p^{n,i})$\; \vspace{0.2em}
\end{algorithm}

The first step is equivalent to solving a stabilized flow problem: For given $(\stress^{n,i-1},\stress_\mathrm{v}^{n,i-1},p^{n,i-1})\in \mathcal{S}^n\times\tilde{\mathcal{S}}^n\times\mathcal{Q}^n$, find $p^{n,i}\in \mathcal{Q}^n$ satisfying for all $q \in \mathcal{Q}_0$
\begin{align*}
&\tfrac{1}{M} \llangle p^{n,i}, q \rrangle 
+ \llangle 
\left(\bm{\alphamatrix}_\mathrm{v} \otimes \mathbf{I} \right) : \bm{\mathcal{A}}_\mathrm{v} : \left(\bm{\alphamatrix}_\mathrm{v} \otimes \mathbf{I} \right)
\left(p^{n,i} - p^{n,i-1}\right), q \rrangle \\
&\quad + \llangle \left(\bm{\alphamatrix}_\mathrm{v} \otimes \mathbf{I} \right) : \begin{bmatrix}  \viscostrain^{n,i-1} \\ \bm{\varepsilon}_{\u}^{n,i-1}\end{bmatrix}, q \rrangle
+ \Delta t \llangle \permeability \left(\GRAD p^{n,i} - \gext^n\right), \GRAD q \rrangle 
= \llangle \fluidmass^{n-1} + \Delta t\, \masssource^n, q \rrangle,
\end{align*}
where we formally abbreviate the total and visco-elastic strains at the previous iteration
\begin{align*}
\begin{bmatrix} \bm{\varepsilon}_{\u}^{n,i-1} \\ \viscostrain^{n,i-1} \end{bmatrix} :=
\bm{\mathcal{A}}_\mathrm{v}:
 \left( \begin{bmatrix} \stress^{n,i-1} \\ \stress_\mathrm{v}^{n,i-1} \end{bmatrix} +\left(\bm{\alphamatrix}_\mathrm{v} \otimes \mathbf{I}\right)\, p^{n,i-1} \right).
\end{align*}
For homogeneous, isotropic materials, the stabilization term equals
\begin{align*}
\llangle \bm{\alphamatrix}_\mathrm{v}^\top \mathbf{K}^{-1} \bm{\alphamatrix}_\mathrm{v} \left(p^{n,i} - p^{n,i-1} \right), q \rrangle
=
\left(\tfrac{\alpha^2}{K_\mathrm{dr}} + \tfrac{\alpha_\mathrm{v}^2}{K_\mathrm{dr,v} + \Delta t^{-1} K_\mathrm{dr,v}'}\right) \llangle p^{n,i} - p^{n,i-1}, q \rrangle.
\end{align*}
The second step is equivalent to solving the mechanics problem with updated fluid flow variables.

Linear convergence can be established based on an abstract convergence result for alternating minimization. Due to the structural similarities of semi-discrete, linear poro-visco-elasticity and poro-elasticity, the following lemma reads as corollary of Lemma~\ref{lemma:fixed-stress:convergence-rate}.

\begin{lemma}[Linear convergence of the fixed-stress split for poro-visco-elasticity]\label{lemma:porovisco:fixed-stress:convergence-rate}
 The fixed-stress split for poro-visco-elasticity converges linearly, independent of the initial guess. Let $(\stress^{n},\stress_\mathrm{v}^n, p^{n})$ denote the solution of the coupled problem~\eqref{structure-preserving-porovisco-minimization:start}
 and let $(\stress^{n,i},\stress_\mathrm{v}^{n,i},p^{n,i})$ denote the iterates defined by the fixed-stress split, cf.\ Alg.~\ref{algorithm:fixed-stress-porovisco}. For $i\in\mathbb{N}$, define the errors $\errorstress^{n,i}:=\stress^{n,i}-\stress^n$, $\bm{e}_{\stress_\mathrm{v}}^{n,i}:=\stress_\mathrm{v}^{n,i} - \stress_\mathrm{v}^n$, $\errorpressure^{n,i}:=p^{n,i} - p^n$. Let $|||\cdot|||$ denote the norm induced by the quadratic part of $\mathcal{E}_\mathrm{v,tot}^{\star,\Delta t}$
 \begin{align*}
  \left|\left|\left|(\stress,\stress_\mathrm{v},p)\right|\right|\right|_\star^2 &:=
  \tfrac{1}{2} \llangle \bm{\mathcal{A}}_\mathrm{v} :\left(\begin{bmatrix} \stress \\ \stress_\mathrm{v} \end{bmatrix} +\left( \bm{\alphamatrix}_\mathrm{v} \otimes \mathbf{I} \right) \,p \right),\begin{bmatrix} \stress \\ \stress_\mathrm{v} \end{bmatrix} +\left( \bm{\alphamatrix}_\mathrm{v} \otimes \mathbf{I} \right) \,p  \rrangle \\
  &\qquad+ \tfrac{1}{2M} \| p \|^2 + \tfrac{\Delta t}{2} \llangle \permeability \GRAD p, \GRAD p\rrangle.
 \end{align*}
 Let $A_\mathrm{K}^2:= \frac{\alpha_\mathrm{v}^2}{K_\mathrm{dr,v} + \Delta t^{-1} K_\mathrm{dr,v}'} 
  + \frac{\alpha^2}{K_\mathrm{dr}}$. It holds the \textit{a priori} result
\begin{align*}
  &\left|\left|\left| (\errorstress^{n,i},\bm{e}_{\stress_\mathrm{v}}^{n,i},\errorpressure^{n,i}) \right|\right|\right|_\star 
  \leq 
  \left(\frac{A_\mathrm{K}^2}{\tfrac{1}{M} + \tfrac{\Delta t \kappa_m}{C_\Omega^2} + A_\mathrm{K}^2}\right)^i 
  \left( \mathcal{E}^{n,0} - \mathcal{E}^n \right)^{1/2},
 \end{align*}
 and the \textit{a posteriori} result
 \begin{align*}
  \left|\left|\left|(\errorstress^{n,i},\bm{e}_{\stress_\mathrm{v}}^{n,i},\errorpressure^{n,i})\right|\right|\right|_\star
  &\leq 
  \left(\frac{A_\mathrm{K}^2}{\tfrac{1}{M} + \tfrac{\Delta t \kappa_m}{C_\Omega^2} + A_\mathrm{K}^2}\right)
  \left[ 1 - \left(\frac{A_\mathrm{K}^2}{\tfrac{1}{M} + \tfrac{\Delta t \kappa_m}{C_\Omega^2} + A_\mathrm{K}^2}\right)^2 \right]^{-1/2}
  \left( \mathcal{E}^{n,i-1} - \mathcal{E}^{n,i} \right)^{1/2},
 \end{align*}
 where
 \begin{align*}
  \mathcal{E}^n &:=\mathcal{E}_\mathrm{v,tot}^{\star,\Delta t}(\fluidmass^{n-1}; \stress^{n}, \stress_\mathrm{v}^{n}, p^{n}),\\
  \mathcal{E}^{n,j}&:=
  \mathcal{E}_\mathrm{v,tot}^{\star,\Delta t}(\fluidmass^{n-1}; \stress^{n,j}, \stress_\mathrm{v}^{n,j}, p^{n,j}),\ j\in\mathbb{N},
 \end{align*}
 and $C_\Omega$ denotes a Poincar\'e-like constant and $\kappa_m$ is the smallest eigenvalue of $\permeability$.
\end{lemma}

\begin{proof}
 We follow the same strategy as in the proof of Lemma~\ref{lemma:fixed-stress:convergence-rate}. Due to the similarities, we present only the main steps; we stress notation is attempted to like alike. We define two semi-norms
 \begin{align*}
  \| (\stress,\stress_\mathrm{v}, p) \|_{1,\mathrm{v},\star,\Delta t}^2 &:= \tfrac{1}{M} \|p\|^2 + \Delta t \llangle \permeability \GRAD p, \GRAD p\rrangle, \\
  \| (\stress,\stress_\mathrm{v}, p)  \|_{2,\mathrm{v},\star,\Delta t}^2 &:= 
  \llangle \bm{\mathcal{A}}_\mathrm{v} : \begin{bmatrix} \stress \\ \stress_\mathrm{v} \end{bmatrix},\begin{bmatrix} \stress \\ \stress_\mathrm{v} \end{bmatrix} \rrangle
  -  \frac{A_\mathrm{K}^2}{\tfrac{1}{M} + \tfrac{\Delta t \kappa_\mathrm{m}}{C_\Omega^2} + A_\mathrm{K}^2}
  \llangle \mathbf{K}^{-1} \begin{bmatrix} \sigma^\mathrm{d} \\ \sigma_\mathrm{v}^\mathrm{d} \end{bmatrix},\begin{bmatrix} \sigma^\mathrm{d} \\ \sigma_\mathrm{v}^\mathrm{d} \end{bmatrix} \rrangle.
 \end{align*}
 Positive semi-definiteness of $\|\cdot\|_\mathrm{2,v,\star,\Delta t}$ holds due to~\eqref{poro-visco:generalized-compliance-tensor}.

\paragraph{(i) Strong convexity of $\mathcal{E}_\mathrm{v,tot}^{\star,\Delta t}$.}

The semi-norms $\| \cdot \|_{i,\mathrm{v},\star,\Delta t}$ are chosen such that $\mathcal{E}_\mathrm{v,tot}^{\star,\Delta t}$ is strongly convex with constant $\sigma_i=1$, $i=1,2$. This is trivial for $\| \cdot \|_{1,\mathrm{v},\star,\Delta t}$. For $\| \cdot \|_{2,\mathrm{v},\star,\Delta t}$, we employ an argument analogous to~\eqref{proof:fixed-stress:strong-convexity-2}. Employing the Poincar\'e inequality, expanding the quadratic terms, and applying the Cauchy inequality and Young's inequality, yields for all $(\stress,stress_\mathrm{v},p)$
\begin{align*}
 &2 \left|\left|\left| (\stress,\stress_\mathrm{v},p) \right|\right|\right|_\mathrm{\star}^2\\
 &\quad\geq 
 \llangle \bm{\mathcal{A}}_\mathrm{v} : \begin{bmatrix} \stress \\ \stress_\mathrm{v} \end{bmatrix},\begin{bmatrix} \stress \\ \stress_\mathrm{v} \end{bmatrix} \rrangle
 +
 2 \llangle \bm{\alphamatrix}_\mathrm{v}^\top (d\mathbf{K})^{-1} \begin{bmatrix} \sigma^\mathrm{h} \\ \sigma_\mathrm{v}^\mathrm{h} \end{bmatrix}, p\rrangle
 +
 \left( \tfrac{1}{M} + \tfrac{\Delta t\kappa_m}{C_\Omega^2} + A_\mathrm{K}^2 \right) \| p\|^2 \\
 &\quad \geq
 \llangle \bm{\mathcal{A}}_\mathrm{v} : \begin{bmatrix} \stress \\ \stress_\mathrm{v} \end{bmatrix},\begin{bmatrix} \stress \\ \stress_\mathrm{v} \end{bmatrix} \rrangle
 -
 \frac{\bm{\alphamatrix}_\mathrm{v}^\top \mathbf{K}^{-1} \bm{\alphamatrix}_\mathrm{v}}{\tfrac{1}{M} + \tfrac{\Delta t\kappa_m}{C_\Omega^2} + A_\mathrm{K}^2 } \llangle \mathbf{K}^{-1} \begin{bmatrix} \sigma^\mathrm{h} \\ \sigma_\mathrm{v}^\mathrm{h} \end{bmatrix}, \begin{bmatrix} \sigma^\mathrm{h} \\ \sigma_\mathrm{v}^\mathrm{h} \end{bmatrix}\rrangle\\
 &\quad =\left\| (\stress,\stress_\mathrm{v},p ) \right\|_\mathrm{2,v,\star,\Delta t}^2.
\end{align*}
 
\paragraph{(ii) Lipschitz continuity of $\GRAD_{p}\mathcal{E}_\mathrm{v,tot}^{\star,\Delta t}$ and $\GRAD_{(\stress,\stress_\mathrm{v})}\mathcal{E}_\mathrm{tot}^{\star,\Delta t}$.}
Analogously to the proof of Lemma~\ref{lemma:fixed-stress:convergence-rate} it can be showed that $\GRAD_{p}\mathcal{E}_\mathrm{v,tot}^{\star,\Delta t}$ and $\GRAD_{(\stress,\stress_\mathrm{v})}\mathcal{E}_\mathrm{tot}^{\star,\Delta t}$ are Lipschitz continuous wrt.\ $\| \cdot \|_{1,\mathrm{v},\star,\Delta t}$ and $\| \cdot \|_{2,\mathrm{v},\star,\Delta t}$, respectively, with Lipschitz constants $L_1=L_2 = 1 +  A_\mathrm{K}^2 \left(\tfrac{1}{M} + \tfrac{\kappa_m \Delta t}{C_\Omega^2}\right)^{-1}$.

\paragraph{Consequences.}
The thesis follows by Lemma~\ref{appendix:lemma:alternating-minimization} and the fact that $\mathcal{E}_\mathrm{v,tot}^{\star,\Delta t}$ is quadratic.

\end{proof}

\section{Robust splitting schemes for discrete non-linear poro-elasticity under infinitesimal strains}\label{section:splitting-non-linear-poro}

So far, part II dealt with quadratic minimization problems related to linear thermo-poro-visco-elasticity. In the following, we briefly demonstrate that the workflow illustrated in Fig.~\ref{figure:procedure-derivation-splitting} can be likewise utilized for discussing non-linear poro-elasticity originating from convex minimization. As an example for non-quadratic, convex minimization problems, we consider non-linear poro-elasticity under infinitesimal strain (Sec.~\ref{section:non-linear-biot-linear-coupling}), and provide the first mathematically justified derivation of a fixed-stress split. Contrary to the previous sections, the coupled problem is decoupled into non-linear subproblems. Therefore, considering inexact solutions of those, different linearization techniques effectively lead to different splitting schemes. In the course of this work, we mention Newton's method and the so-called L-scheme, employing constant approximations of derivatives.

We just remark, a corresponding undrained split can be easily derived within the general framework, based on the primal formulation in Sec.~\ref{section:non-linear-biot-linear-coupling}. Employing inexact solution of the resulting non-linear subproblems by single L-scheme iterations, yields essentially a specific splitting scheme recently derived and analyzed by~\cite{Borregales2018}.

\subsection{Minimization formulation for the three-field formulation}
We recall the primal formulation~\eqref{non-linear-biot:gradient-flow-structure} of non-linear poro-elasticity under infinitesimal strains, allowing for non-linear mechanics and fluid compressibility
\begin{align*}
 (\dot{\u},\dot{\flux}_{\int})
 &= \underset{(\bm{v},\z)\in\dot{\mathcal{V}}\times\dot{\mathcal{Z}}_{\int}}{\text{arg\,min}}\, 
 \Big\{ 
 \mathcal{D}_\mathrm{fluid}(\z)
 + \llangle \GRAD \mathcal{E}_\mathrm{nl}(\u,\flux_{\int}), (\v,\z) \rrangle 
 - \mathcal{P}_\mathrm{ext}(\v,\z) \Big\}.
\end{align*} 
A semi-discrete approximation is directly obtained by applying the minimizing movement schemes (Sec.~\ref{section:general-time-discretization}). By explicit introduction of the fluid pressure consistent with~\eqref{thermodynamic-interpretation}, we consider the more common three-field saddle point formulation, incorporating the structural displacement, volumetric flux and fluid pressure as primary variables. All in all, we obtain a generalization of the three-field formulation of linear poro-elasticity (Sec.~\ref{section:minimization:three-field-formulation}).

Reusing notation, we define the minimization formulation for time step $n$: Given $(\u^{n-1},\flux^{n-1},p^{n-1})\in\mathcal{V}^{n-1}\times\mathcal{Z}^{n-1}\times\tilde{\mathcal{Q}}$, set $\fluidmass^{n-1} := b(p^{n-1})+\alpha \DIV\u^{n-1}$, and define $(\u^n,\flux^n,p^n)\in\mathcal{V}^n\times\mathcal{Z}^n\times\tilde{\mathcal{Q}}$ as solution to
\begin{align}
 \label{three-field-minimization:start}
 (\u^n,\flux^n) :=& \underset{(\u,\flux)\in\mathcal{V}^n\times\mathcal{Z}^n}{\mathrm{arg\,min}}\,\mathcal{E}_\mathrm{nl,tot}^{\Delta t}(\fluidmass^{n-1};\u,\flux), \\
 \label{three-field-minimization:end}
  p^n :=& \, \Pi_{\tilde{\mathcal{Q}}}\left(b^{-1}\left(\Pi_{\tilde{\mathcal{Q}}}(\fluidmass^{n-1} +\Delta t\, \masssource^n - \Delta t \, \DIV \flux^n - \alpha \DIV \u^n) \right)\right),
\end{align}
where
\begin{align}
  \nonumber
 \mathcal{E}_\mathrm{tot}^{\Delta t}(\fluidmass^{n-1};\u,\flux) 
 :=&
 \, \int_\Omega W(\eps{\u})\,dx + \tfrac{\Delta t}{2} \llangle \permeability^{-1} \flux, \flux \rrangle \\
\nonumber
 & + \int_\Omega \int_0^{\Pi_{\tilde{\mathcal{Q}}}( \fluidmass^{n-1} + \Delta t\, \masssource^n - \Delta t \DIV \flux - \alpha \DIV \u) } b^{-1}(s)\,ds \, dx \\
\nonumber
 &- \mathcal{P}_\mathrm{ext,mech}^n(\u)- \Delta t \, \mathcal{P}_\mathrm{ext,fluid}^n(\flux).
\end{align}
Provided that $W$ is strictly convex and $b$ is Lipschitz continuous, the minimization problem is strictly convex. Since also the projection is well-defined; existence and uniqueness of a solution to~\eqref{three-field-minimization:start}--\eqref{three-field-minimization:end} follow by classical results from convex analysis, cf.\ Thm.~\ref{appendix:well-posedness:convex-minimization}. Following Sec.~\ref{section:optimality-conditions:three-field}, the corresponding optimality conditions are given by
\begin{align}
\label{non-linear-biot:semi-discrete:monolithic:start}
 \llangle \GRAD W(\eps{\u^n}), \eps{\v} \rrangle -\alpha \llangle p^n, \DIV \v \rrangle &= \mathcal{P}^n_\mathrm{ext,mech}(\v) &&\forall\v\in\mathcal{V}_0,\\
 \label{non-linear-biot:semi-discrete:monolithic:mid}
 \llangle \permeability^{-1}\flux^n, \z \rrangle - \llangle p^n, \DIV \z \rrangle &= \mathcal{P}^n_\mathrm{ext,fluid}(\z), &&\forall\z\in\mathcal{Z}_0,\\
 \llangle b(p^n) + \alpha \DIV \u + \Delta t \, \DIV \flux^n, q \rrangle &= \llangle \fluidmass^{n-1} + \Delta t \masssource^n, q \rrangle, && \forall q \in \tilde{\mathcal{Q}}.
 \label{non-linear-biot:semi-discrete:monolithic:end}
\end{align}

\subsection{Foundation for an exact fixed-stress split for the dual formulation}\label{section:non-linear-fixed-stress:derivation-dual}

For the derivation of a fixed-stress split for the three-field formulation~\eqref{non-linear-biot:semi-discrete:monolithic:start}--\eqref{non-linear-biot:semi-discrete:monolithic:end}, we utilize a natural dual formulation, generalizing the dual formulation for linear poro-elasticity (Sec.~\ref{section:minimization:dual-formulation}). For this, we first note that $\GRAD W$ is invertible for strictly convex $W$, and there exists a dual scalar potential $U:\,\mathbb{R}^{d \times d} \rightarrow \mathbb{R}$, which by the inverse function theorem satisfies for all $\stress\in\mathbb{R}^{d \times d}$ 
\begin{align*}
\GRAD U(\stress) &= (\GRAD W)^{-1}(\stress),\\ 
\GRAD^2 U (\stress) &= \GRAD^2 W \left( \left(\GRAD W\right)^{-1} (\stress)\right)^{-1}.
\end{align*}
Similarly, let $B:\,\mathbb{R} \rightarrow \mathbb{R}$ a primitive of $b$, satisfying $B'=b$. 

Then the dual minimization formulation reads: Given $(\stress^{n-1},p^{n-1})\in\mathcal{S}^{n-1}\times\mathcal{Q}^{n-1}$, set $\fluidmass^{n-1}:=b(p^{n-1})+ \alpha \,\mathrm{tr}\, \GRAD U\left(\stress^{n-1} + \alpha p^{n-1} \mathbf{I}\right)$, and define $(\stress^{n},p^{n})\in\mathcal{S}^n\times\mathcal{Q}^n$ to be the solution of the dual minimization problem
\begin{align}
\label{non-linear:dual-two-field-minimization:start}
  (\stress^{n},p^{n}) :=& \underset{(\stress,p)\in\mathcal{S}^n\times\mathcal{Q}^n}{\mathrm{arg\,min}}\,\mathcal{E}_\mathrm{nl,tot}^{\star,\Delta t}(\fluidmass^{n-1};\stress,p), \quad \text{where}\\[0.5em]
 \nonumber
 \mathcal{E}_\mathrm{tot}^{\star,\Delta t}(\fluidmass^{n-1};\stress,p) :=&\, 
 \int_\Omega U(\stress+\alpha p \mathbf{I})\, dx\\
 \nonumber
 &+
 \int_\Omega B(p) \, dx + \tfrac{\Delta t}{2} \llangle \permeability (\GRAD p - \gext^n), \GRAD p - \gext^n \rrangle \\
 \nonumber
 &
 -\llangle  \u_\Gamma^n, \stress\bm{n} \rrangle_{\Gamma_{\u}} 
 -\llangle \fluidmass^{n-1} + \Delta t\, \masssource^n, p \rrangle 
 - \Delta t \llangle  q_\mathrm{\Gamma,n}^n, p \rrangle_{\Gamma_{\q}}.
\end{align}

The \textit{exact fixed-stress split} is then defined as (exact) alternating minimization applied to~\eqref{non-linear:dual-two-field-minimization:start}, cf.\ Alg.~\ref{algorithm:fixed-stress-split-non-linear}. Convergence follows directly, and theoretical convergence rates can be studied as previously. When employing inexact minimization in one of the steps, we refer to an \textit{inexact fixed-stress split}.
\begin{algorithm}
 \caption{Single iteration of the exact fixed-stress split for non-linear poro-elasticity under infinitesimal strain}
 \label{algorithm:fixed-stress-split-non-linear}
 \SetAlgoLined
 \DontPrintSemicolon
 
  \vspace{0.5em}
 
  Input: $(\stress^{n,i-1},p^{n,i-1})\in\mathcal{S}^n\times\mathcal{Q}^n$ \; \vspace{0.8em}
   
  Determine $p^{n,i} := \underset{p\in\mathcal{Q}^n}{\mathrm{arg\,min}}\, \mathcal{E}_\mathrm{nl,tot}^{\star,\Delta t}(\fluidmass^{n-1};\stress^{n,i-1},p)$\; \vspace{0.25em}

  Determine $\stress^{n,i} := \underset{\stress\in\mathcal{S}^n}{\mathrm{arg\,min}}\, \mathcal{E}_\mathrm{nl,tot}^{\star,\Delta t}(\fluidmass^{n-1};\stress,p^{n,i})$\; \vspace{0.2em}
\end{algorithm}

\subsection{Fixed-stress splits for the three-field formulation of non-linear poro-elasticity under infinitesimal strains}\label{section:fixed-stress-non-linear-three-field}

Pursuing the previous philosophy, the fixed stress split for the three-three field formulation~\eqref{non-linear-biot:semi-discrete:monolithic:start}--\eqref{non-linear-biot:semi-discrete:monolithic:end} is equivalent with solving first a pressure-stabilized version of the flow problem~\eqref{non-linear-biot:semi-discrete:monolithic:mid}--\eqref{non-linear-biot:semi-discrete:monolithic:end}, and second the mechanics problem~\eqref{non-linear-biot:semi-discrete:monolithic:start} with updated fluid flow variables. The stabilization term can be concluded from the discussion in Sec.~\ref{section:non-linear-fixed-stress:derivation-dual}. In particular, the first step of the exact fixed-stress split for the dual problem, cf.\ Alg.~\ref{algorithm:fixed-stress-split-non-linear}, reads: Find $p^{n,i}\in\mathcal{Q}^n$, satisfying
\begin{align}
\label{non-linear-fixed-stress:pressure-equation}
 &\llangle b(p^{n,i}), q \rrangle  
 + \alpha \llangle \mathbf{I} : \GRAD U(\stress^{n,i-1} + \alpha p^{n,i} \mathbf{I}), q \rrangle \\
 &\quad + \Delta t\, \llangle \permeability \GRAD (p^{n,i} - \gext), \GRAD q \rrangle 
 = \llangle \fluidmass^{n-1} + \Delta t\, \masssource^n, q \rrangle\ \ \forall q\in\mathcal{Q}_0.
 \nonumber
\end{align}
Utilizing the natural linearization of the non-linear coupling term
\begin{align}
\nonumber
 &\alpha \llangle \mathbf{I} : \GRAD U(\stress^{n,i-1} + \alpha p^{n,i} \mathbf{I}), q \rrangle \\
\nonumber
 &\quad \approx
 \alpha \llangle \underbrace{\mathbf{I} : \GRAD U(\stress^{n,i-1} + \alpha p^{n,i-1} \mathbf{I})}_{\hat{=} \mathrm{tr}\, \eps{\u^{n,i-1}}}, q \rrangle \\
 &\qquad + \alpha^2 \llangle \underbrace{\left(\mathbf{I} : \GRAD^2 U(\stress^{n,i-1} + \alpha p^{n,i-1} \mathbf{I}) : \mathbf{I}\right)}_{\hat{=} \mathbf{I} : \GRAD^2 W \left( \eps{\u^{n,i-1}} \right)^{-1} :\mathbf{I} =: K_\mathrm{dr}(\eps{\u^{n,i-1}})^{-1}} (p^{n,i} - p^{n,i-1}), q \rrangle
 \label{definition:kdr-u}
\end{align}
combined with different linearization techniques, we propose, two versions of the exact fixed-stress split for the three-field formulation~\eqref{non-linear-biot:semi-discrete:monolithic:start}--\eqref{non-linear-biot:semi-discrete:monolithic:end}. For direct comparison, we define natural residuals; for $(\v,q)\in\mathcal{V}_0 \times \tilde{\mathcal{Q}}$, let
\begin{align*}
 R^n_{\u}(\u,\flux,p;\v) &:= \llangle \GRAD W(\eps{\u}), \eps{\v} \rrangle -\alpha \llangle p, \DIV \v \rrangle - \mathcal{P}^n_\mathrm{ext,mech}(\v),\\
 R^n_{p}(\u,\flux,p;q) &:= \llangle b(p) + \alpha \DIV \u + \Delta t \, \DIV \flux - \fluidmass^{n-1} - \Delta t \masssource^n, q \rrangle.
\end{align*}

\paragraph{Newton-based fixed-stress split.} In the first step, set $(\flux^{n,i,0},p^{n,i,0})=(\flux^{n,i-1},p^{n,i-1})$, and iterate over $j\geq 1$ until convergence: Given $\u^{n,i-1}\in \mathcal{V}^n$ and $(\flux^{n,i,j-1},p^{n,i,j-1})\in \mathcal{Z}^n \times \tilde{\mathcal{Q}}$, find $(\flux^{n,i,j},p^{n,i,j}) \in \mathcal{Z}^n \times \tilde{\mathcal{Q}}$, satisfying for all $(\z,q)\in\mathcal{Z}_0 \times \tilde{\mathcal{Q}}$
\begin{align*}
&\llangle \permeability^{-1}\flux^{n,i,j}, \z \rrangle - \llangle p^{n,i,j}, \DIV \z \rrangle = \mathcal{P}^n_\mathrm{ext,fluid}(\z),\\
&\llangle  \left[b'(p^{n,i,j-1}) + \frac{\alpha^2}{K_\mathrm{dr}(\eps{\u^{n,i-1}})} \right](p^{n,i,j} - p^{n,i,j-1}), q \rrangle \\
&\quad  + \Delta t\, \llangle \DIV (\flux^{n,i,j} - \flux^{n,i,j-1}), q \rrangle
 = R_p^n(\u^{n,i-1}, \flux^{n,i,j-1}, p^{n,i,j-1}; q).
\end{align*}
In the second step, set $\u^{n,i,0} = \u^{n,i-1}$, and iterate over $k\geq 1$ until convergence: Given $\u^{n,i,k-1}\in \mathcal{V}^n$ and $(\flux^{n,i},p^{n,i})\in \mathcal{Z}^n \times \tilde{\mathcal{Q}}$, find $\u^{n,i,k}\in \mathcal{V}^n$, satisfying for all $\v \in \mathcal{V}_0$
\begin{align*}
 \llangle \GRAD^2 W(\eps{\u^{n,i,k-1}}) \eps{\u^{n,i,k} - \u^{n,i,k-1}}, \eps{\v} \rrangle = R_u^n(\u^{n,i,k-1}, \flux^{n,i}, p^{n,i}; \v).
\end{align*}

\paragraph{L-scheme-based fixed-stress split.} Having in mind the inexact solution of non-linear subproblems, and motivated by the fact, that any fixed-stress split at most is linearly convergent, we disregard Newton's method and choose a very simple linearization instead -- the so-called L-scheme, which employs a constant Jacobian. Let $L_b$, $L_\mathrm{FS}\geq 0$ and $\mathbb{L}\in\mathbb{R}^{d \times d \times d \times d}$ symmetric positive definite (in the same sense as $\mathbb{C}$). In the first step, set $(\flux^{n,i,0},p^{n,i,0})=(\flux^{n,i-1},p^{n,i-1})$, and iterate over $j\geq 1$ until convergence: Given $\u^{n,i-1}\in \mathcal{V}^n$ and $(\flux^{n,i,j-1},p^{n,i,j-1})\in \mathcal{Z}^n \times \tilde{\mathcal{Q}}$, find $(\flux^{n,i,j},p^{n,i,j}) \in \mathcal{Z}^n \times \tilde{\mathcal{Q}}$, satisfying for all $(\z,q)\in\mathcal{Z}_0 \times \tilde{\mathcal{Q}}$
\begin{align*}
&\llangle \permeability^{-1}\flux^{n,i,j}, \z \rrangle - \llangle p^{n,i,j}, \DIV \z \rrangle = \mathcal{P}^n_\mathrm{ext,fluid}(\z),\\
&\llangle  (L_b + L_\mathrm{FS}) (p^{n,i,j} - p^{n,i,j-1}), q \rrangle + \Delta t\, \llangle \DIV (\flux^{n,i,j} - \flux^{n,i,j-1}), q \rrangle\\
&\qquad = R_p^n(\u^{n,i-1}, \flux^{n,i,j-1}, p^{n,i,j-1}; q).
\end{align*}
In the second step, set $\u^{n,i,0} = \u^{n,i-1}$, and iterate over $k\geq 1$ until convergence: Given $\u^{n,i,k-1}\in \mathcal{V}^n$ and $(\flux^{n,i},p^{n,i})\in \mathcal{Z}^n \times \tilde{\mathcal{Q}}$, find $\u^{n,i,k}\in \mathcal{V}^n$, satisfying for all $\v \in \mathcal{V}_0$
\begin{align*}
 \llangle \mathbb{L} \eps{\u^{n,i,k} - \u^{n,i,k-1}}, \eps{\v} \rrangle = R_u^n(\u^{n,i,k-1}, \flux^{n,i}, p^{n,i}; \v).
\end{align*}
Following previous studies on the L-scheme, cf., e.g.~\cite{Borregales2018,List2016}, choosing $L_b,L_\mathrm{FS},\mathbb{L}$ sufficiently large may be expected to yield robust convergence. For instance, for Lipschitz continuous non-linearities, the Lipschitz constants are suitable candidates; or solution-dependent choices as $L_b = \underset{x,t}{\mathrm{max}}\, |b'(p(x,t))|$ and $L_\mathrm{FS}=\tfrac{\alpha^2}{\underset{x,t}{\mathrm{min}}\, K_\mathrm{dr}(\eps{\u(x,t)})}$.

\begin{remark}[Inexact fixed-stress splits] By choosing coarse tolerances or applying only a fixed amount of non-linear iterations in each of the two steps yields an inexact version of the fixed-stress split. In particular, for strongly coupled problems, one can expect that inexact fixed-stress splits are potentially more efficient than the exact fixed-stress split.
\end{remark}

\section{Robust splitting schemes for discrete thermo-poro-elasticity}\label{section:splitting-thermo-poro-elasticity}

As a result of the gradient flow structure of linear thermo-poro-elasticity, cf.\ Sec.~\ref{section:thermo-poro-elasticity-gradient-flow}, robust iterative splitting schemes can be derived for the implicit Euler time-discrete approximation employing the workflow visualized in Fig.~\ref{figure:procedure-derivation-splitting}. We in particular observe that semi-discrete, linear thermo-poro-elasticity can be formulated as vectorized, semi-discrete, linear poro-elasticity, similar to linear poro-visco-elasticity (Sec.~\ref{section:splitting-poro-visco-elasticity}). Thus, technicalities can be immediately adopted from linear poro-elasticity including the construction of a dual problem, the derivation of two-stage splitting schemes and their analyses. After all, we identify the recently proposed undrained-adiabatic and extended fixed-stress splits proposed for non-linear thermo-poro-elasticity~\cite{Kim2018a} as alternating minimization. This new perspective endows the originally physically motivated schemes with mathematical justification. Motivated by the three-way coupling of thermo-poro-elasticity, we also derive a novel, robust three-stage splitting scheme by applying a cyclic three-block coordinate descent method. Finally, we close the section, commenting on possible applications of the splitting schemes to non-linear thermo-poro-elasticity including for instance thermal convection.



\subsection{Minimization formulations for time-discrete linear thermo-poro-elasticity}

Following the abstract workflow visualized in Fig.~\ref{figure:procedure-derivation-splitting}, we introduce a primal and a dual formulation for time-discrete linear thermo-poro-elasticity. In the second part of this section, both formulations will serve as bases for the derivation of practical operator splitting schemes. 

\subsubsection{Primal formulation of time-discrete linear thermo-poro-elasticity}

The primal formulation of time-discrete, linear thermo-poro-elasticity is obtained by applying the minimizing movement scheme to the time-continuous model~\eqref{thermo-poro-elasticity:gradient-flow:start}--\eqref{thermo-poro-elasticity:gradient-flow:end}. Similar to the case of semi-discrete poro-visco-elasticity, the resulting formulation can be interpreted as vectorized version of the primal formulation of time-discrete, linear poro-elasticity, but now with a vectorized flow problem --  a key characteristic which will be utilized throughout the entire section. For this, we introduce a tensorial diffusion, compressibility and Biot coefficient, respectively, by
\begin{align*}
 \mathbf{K}_\mathrm{T} &:= \begin{bmatrix} \permeability & \bm{0} \\ \bm{0} & \frac{\conductivity}{T_0} \end{bmatrix}, \qquad \mathbf{M}_\mathrm{T}^{-1} := \begin{bmatrix} \tfrac{1}{M} & -3\alpha_\phi \\ -3\alpha_\phi & \tfrac{C_\mathrm{d}}{T_0} \end{bmatrix}, \qquad
 \bm{\alphamatrix}_\mathrm{T} := \begin{bmatrix} \alpha \\ 3 \alpha_\mathrm{T} K_\mathrm{dr} \end{bmatrix}.
\end{align*}
Let the spaces $\mathcal{V}^n$ and $\mathcal{Z}^n$ be as defined in~\eqref{primal-discrete-poro-elasticity-spaces:start}--\eqref{primal-discrete-poro-elasticity-spaces:end}, and define additionally
\begin{align*}
 \mathcal{W}^n &:= \left\{ \w \in H(\mathrm{div}; \Omega) \,\left|\, \w\cdot\n = j_\Gamma^n\text{ on }\Gamma_{\entropyflux} \right. \right\}.
\end{align*}
Finally, we state the time-discrete, primal formulation for time step $n\geq 1$: Given $\fluidmass^{n-1}$ and $S^{n-1}$, define $(\u^n,\flux^n,\entropyflux^n)\in\mathcal{V}^n\times\mathcal{Z}^n\times\mathcal{W}^n$ to be the solution of the minimization problem
\begin{align}
 \label{primal-minimization:thermoporo:start}
 &(\u^n,\flux^n,\entropyflux^n) 
 := \underset{(\u,\flux,\entropyflux)\in\mathcal{V}^n\times\mathcal{Z}^n\times\mathcal{W}^n}{\mathrm{arg\,min}}\,\mathcal{E}_\mathrm{th,tot}^{\Delta t}(\fluidmass^{n-1},S^{n-1};\u,\flux,\entropyflux),
\end{align}
where
\begin{align}
 \nonumber
 &\mathcal{E}_\mathrm{th,tot}^{\Delta t}(\fluidmass^{n-1},S^{n-1};\u,\flux,\entropyflux) \\
\nonumber
&\quad :=
 \frac{1}{2} \llangle \mathbb{C} \eps{\u}, \eps{\u} \rrangle
 + \frac{\Delta t}{2} \llangle \mathbf{K}_\mathrm{T}^{-1} \begin{bmatrix} \flux \\ \entropyflux \end{bmatrix}, \begin{bmatrix} \flux \\ \entropyflux \end{bmatrix} \rrangle \\
\nonumber
 &\quad+ \frac{1}{2} \llangle \mathbf{M}_\mathrm{T} \left(\begingroup\setlength\arraycolsep{0pt}\begin{bmatrix} \fluidmass^{n-1} &+ \Delta t\, \masssource^n \\ S^{n-1} &+ \Delta t\, \entropysource^n \end{bmatrix}\endgroup - \Delta t\, \begin{bmatrix} \DIV\flux \\ \DIV\entropyflux \end{bmatrix}  - \left(\bm{\alphamatrix}_\mathrm{T} \otimes \mathbf{I} \right) : \eps{\u} \right), \right.\\
\nonumber
 &\quad \qquad \qquad \quad\,\hspace{0.025cm} \left.
\begingroup\setlength\arraycolsep{0pt}\begin{bmatrix} \fluidmass^{n-1} &+ \Delta t\, \masssource^n \\ S^{n-1} &+ \Delta t\, \entropysource^n \end{bmatrix}\endgroup - \Delta t\,  \begin{bmatrix} \DIV\flux \\ \DIV\entropyflux \end{bmatrix}  - \left(\bm{\alphamatrix}_\mathrm{T} \otimes \mathbf{I} \right) : \eps{\u}\rrangle \\
\nonumber
 &\quad- \mathcal{P}_\mathrm{ext,mech}^n(\u)- \Delta t \, \mathcal{P}_\mathrm{ext,fluid}^n(\flux) - \Delta t \, \mathcal{P}_\mathrm{ext,temp}^n(\entropyflux),
\end{align}
and set 
\begin{align*}
 \begin{bmatrix}\fluidmass^n \\ S^n \end{bmatrix}  &:= \begingroup \setlength\arraycolsep{0pt}\begin{bmatrix} \fluidmass^{n-1} &+ \Delta t\, \masssource^n \\ S^{n-1} &+ \Delta t\,\entropysource^n \end{bmatrix} \endgroup - \Delta t \, \begin{bmatrix} \DIV \flux^n \\ \DIV \entropyflux^n \end{bmatrix}.
\end{align*}
For $\mathbf{K}_\mathrm{T}$ and $\mathbf{M}_\mathrm{T}$ positive definite, the resulting minimization problem is strictly convex and coercive; existence and uniqueness of a solution to~\eqref{primal-minimization:thermoporo:start} follow by classical results from convex analysis, cf.\ Thm.~\ref{appendix:well-posedness:convex-minimization}.

\subsubsection{Dual formulation of time-discrete linear thermo-poro-elasticity}

Given the primal formulation of time-discrete, linear thermo-poro-elasticity in vectorized form, we utilize the insights gained from linear poro-elasticity and poro-visco-elasticity and impose a corresponding dual formulation. First, we introduce natural dual variables: the total stress $\stress$, the fluid pressure $p$ and the temperature of the bulk $T$, formally related to the primal variables by
\begin{align*}
 \stress &= \mathbb{C} \left( \eps{\u} - \left(\mathbf{I} \otimes \bm{\alphamatrix}_\mathrm{T} \right) : \begin{bmatrix} p \\ T \end{bmatrix}\right),\\
 \begin{bmatrix} p \\ T \end{bmatrix} &= \mathbf{M}_\mathrm{T} \left(\begingroup \setlength\arraycolsep{0pt} \begin{bmatrix} \fluidmass^{n-1} & +\Delta t\, \masssource^n \\ S^{n-1} &+\Delta t\, \entropysource^n \end{bmatrix}\endgroup - \Delta t \begin{bmatrix} \DIV \flux \\ \DIV \entropyflux \end{bmatrix} - \left( \bm{\alphamatrix}_\mathrm{T} \otimes \mathbf{I} \right) : \eps{\u} \right).
\end{align*}
For fixed time step $n$, we introduce suitable trial and test function spaces
\begin{align*}
 \mathcal{R}^n &:= \left\{ r\in H^1(\Omega) \,\left|\, r=T_\Gamma^n \text{ on }\Gamma_T \right. \right\}, \\
 \mathcal{R}_0 &:= \left\{ r\in H^1(\Omega) \,\left|\, r=0 \text{ on }\Gamma_T \right. \right\},
\end{align*}
corresponding to the temperature variable. Then $\mathcal{S}^n \times \mathcal{Q}^n \times \mathcal{R}^n$ yields a suitable function space for the dual variables $(\stress,p,T)$.

Finally, the dual formulation of time-discrete, linear thermo-poro-elasticity for time step $n\geq 1$ reads: Given $(\stress^{n-1},p^{n-1},T^{n-1})\in\mathcal{S}^{n-1}\times\mathcal{Q}^{n-1}\times\mathcal{R}^{n-1}$, set 
\begin{align*}
 \bm{\varepsilon}_{\u}^{n-1} &:= \mathbb{A}\left(\stress^{n-1} + \left(\mathbf{I} \otimes \bm{\alphamatrix}_\mathrm{T}\right) : \begin{bmatrix} p^{n-1} \\ T^{n-1} \end{bmatrix}\right), \\
 \begin{bmatrix}\fluidmass^{n-1} \\ S^{n-1} \end{bmatrix} &:= \mathbf{M}_\mathrm{T}^{-1} \begin{bmatrix} p^{n-1} \\ T^{n-1} \end{bmatrix} + \left(\bm{\alphamatrix}_\mathrm{T} \otimes \mathbf{I} \right) : \eps{\u^{n-1}},
\end{align*}
and define $(\stress^{n},p^{n},T^n)\in\mathcal{S}^n\times\mathcal{Q}^n\times\mathcal{R}^n$ to be the solution of the block-separable, constrained minimization problem
\begin{align}
\label{dual-minimization:thermoporo:start}
  &(\stress^{n},p^{n},T^n) := \underset{(\stress,p,T)\in\mathcal{S}^n\times\mathcal{Q}^n\times\mathcal{R}^n}{\mathrm{arg\,min}}\,\mathcal{E}_\mathrm{th,tot}^{\star,\Delta t}(\fluidmass^{n-1}, S^{n-1};\stress,p,T), \quad \text{where}\\[0.5em]
  \nonumber
 &\mathcal{E}_\mathrm{th,tot}^{\star,\Delta t}(\fluidmass^{n-1},S^{n-1};\stress,p,T) \\
 \nonumber
 &\quad :=
 \frac{1}{2} \llangle \mathbb{A} \left(\stress + \left(\mathbf{I} \otimes \bm{\alphamatrix}_\mathrm{T}\right) : \begin{bmatrix} p \\ T \end{bmatrix}\right),\stress + \left(\mathbf{I} \otimes \bm{\alphamatrix}_\mathrm{T}\right) : \begin{bmatrix} p \\ T \end{bmatrix} \rrangle \\
 \nonumber
 &\quad+ \frac{1}{2} \llangle \mathbf{M}_\mathrm{T}^{-1} \begin{bmatrix} p \\ T \end{bmatrix}, \begin{bmatrix} p \\ T  \end{bmatrix} \rrangle 
 + \frac{\Delta t}{2} \llangle \mathbf{K}_\mathrm{T} \begin{bmatrix} \GRAD p - \gext^n \\ \GRAD T \end{bmatrix} , \begin{bmatrix} \GRAD p - \gext^n \\ \GRAD T \end{bmatrix}  \rrangle \\
 \nonumber
 &\quad
 - \llangle  \u_\Gamma^n, \stress\bm{n} \rrangle_{\Gamma_{\u}}
 -\llangle \begingroup \setlength\arraycolsep{0pt}\begin{bmatrix}\fluidmass^{n-1} &+ \Delta t\, \masssource^n \\ S^{n-1} &+ \Delta t \, \entropysource^n \end{bmatrix}\endgroup, \begin{bmatrix} p \\ T \end{bmatrix} \rrangle  - \Delta t \llangle  q_\mathrm{\Gamma,n}^n, p \rrangle_{\Gamma_{\flux}} - \Delta t \llangle j_\mathrm{F,\Gamma}^n, T \rrangle_{\Gamma_{\entropyflux}}.
\end{align}
The minimization problem is strictly convex and the feasible set is non-empty and convex; existence and uniqueness of a solution to~\eqref{dual-minimization:thermoporo:start} follow by classical results from convex analysis, cf.\ Thm.~\ref{appendix:well-posedness:convex-minimization}.

\subsection{Splitting schemes for linear thermo-poro-elasticity derived as alternating minimization}

Due to the convexity properties, any cyclic block coordinate descent method applied to either the primal or the dual formulation, which respects the block structure of the problem, is globally convergent~\cite{Grippof1999,Grippo2000}; in particular two- and three-block coordinate descent methods, decoupling the fully-coupled problem into its physical subproblems. Based on that fact, we derive the undrained-adiabatic and extended fixed-stress splits~\cite{Kim2018a} as two-block coordinate descent methods,  following the abstract workflow, cf.\ Fig.~\ref{figure:procedure-derivation-splitting}, and additionally propose a robust three-block coordinate descent method for linear thermo-poro-elasticity. Theoretical convergence can be showed by adjusting the proofs for the corresponding results in the context of linear poro-elasticity.

\subsubsection{Undrained-adiabatic split based on primal thermo-poro-elasticity}

Applying alternating minimization to the primal formulation of semi-discrete thermo-poro-elasticity yields a generalized undrained split, decoupling the mechanics problem from the rest. For this, the primal variables corresponding to the fluid flow and thermal subproblems are considered a single block, cf.\ Alg.~\ref{algorithm:undrained-adiabatic-split} for a single iteration of the resulting scheme.
\begin{algorithm}
 \caption{Single iteration of undrained-adiabatic split} 
 \label{algorithm:undrained-adiabatic-split}
 \SetAlgoLined
 \DontPrintSemicolon
 
  \vspace{0.5em}
 
  Input: $(\u^{n,i-1},\flux^{n,i-1},\entropyflux^{n,i-1})\in\mathcal{V}^n\times\mathcal{Z}^n\times\mathcal{W}^n$ \; \vspace{0.8em}
   
  Determine $\u^{n,i} := \underset{\u\in\mathcal{V}^n}{\mathrm{arg\,min}}\, \mathcal{E}_\mathrm{th,tot}^{\Delta t}(\fluidmass^{n-1},S^{n-1};\u,\flux^{n,i-1},\entropyflux^{n,i-1})$\; \vspace{0.25em}

  Determine $(\flux^{n,i},\entropyflux^{n,i}) := \underset{(\flux,\entropyflux)\in\mathcal{Z}^n\times\mathcal{W}^n}{\mathrm{arg\,min}}\, \mathcal{E}_\mathrm{th,tot}^{\Delta t}(\fluidmass^{n-1},S^{n-1};\u^{n,i},\flux,\entropyflux)$\; \vspace{0.2em}
\end{algorithm}

By construction, the resulting splitting scheme is equivalent to a predictor-corrector method, solving the mechanics problem under undrained and adiabatic conditions in the predictor step. This is equivalent to the stabilized mechanics problem: Find $\u^{n,i}\in\mathcal{V}^n$ satisfying for all $\v \in \mathcal{V}_0$
\begin{align}
\nonumber
  \llangle \mathbb{C}\eps{\u^{n,i}}, \eps{\v} \rrangle 
  + \llangle \bm{\alphamatrix}_\mathrm{T}^\top \mathbf{M}_\mathrm{T} \bm{\alphamatrix}_\mathrm{T} \  \mathrm{tr}\, \eps{\u^{n,i} - \u^{n,i-1}}, \mathrm{tr}\, \eps{\v} \rrangle &\\
  \nonumber
  - \llangle \bm{\alphamatrix}_\mathrm{T}^\top \begin{bmatrix} p^{n,i-1} \\ T^{n,i-1} \end{bmatrix}, \DIV \v \rrangle
 = \mathcal{P}_\mathrm{ext,mech}^n(\v),
\end{align}
where we formally abbreviated the fluid pressure and temperature by
\begin{align*}
 \begin{bmatrix}
  p^{n,i-1} \\
  T^{n,i-1} 
 \end{bmatrix}
 :=
 \mathbf{M}_\mathrm{T} \left(\begingroup \setlength\arraycolsep{0pt} \begin{bmatrix} \fluidmass^{n-1}&+\Delta t\, \masssource^n \\ S^{n-1} &+\Delta t\, \entropysource^n \end{bmatrix} \endgroup - \Delta t \begin{bmatrix} \DIV \flux^{n,i-1} \\ \DIV \entropyflux^{n,i-1} \end{bmatrix} - \left( \bm{\alphamatrix}_\mathrm{T} \otimes \mathbf{I} \right) : \eps{\u^{n,i-1}} \right).
\end{align*}
For homogeneous, isotropic materials, the stabilization equals
\begin{align*}
  \left(M\alpha^2 + 9 \tfrac{\left(\alpha_\mathrm{T} K_\mathrm{dr} + M\alpha\alpha_\phi \right)^2}{\tfrac{C_\mathrm{d}}{T_0} - 9M\alpha_\phi^2} \right)\, \llangle \DIV (\u^{n,i} - \u^{n,i-1}),\DIV \v \rrangle.
\end{align*} 
The second step of Alg.~\ref{algorithm:undrained-adiabatic-split} (the corrector step) is equivalent to solving the unmodified, coupled fluid flow and thermal subproblems with updated displacement. 
After all, the resulting stabilization term is identical with that employed within the undrained-adiabatic split for thermo-poro-elasticity including thermal convection~\cite{Kim2018a}. 

By adopting the ideas of the proof for the undrained split for poro-elasticity, cf.\ Lemma~\ref{lemma:undrained-split:convergence-rate}, to vectorized poro-elasticity, analogous convergence results can be deduced for the undrained-adiabatic split for linear thermo-poro-elasticity. 

\begin{corollary}[Linear convergence of the undrained-adiabatic split]
 The undrained-adiabatic split for linear thermo-poro-elasticity converges linearly, independent of the initial guess . Let $\errorentropyflux^{n,i}:=\entropyflux^{n,i} - \entropyflux^n$, $n,i\in\mathbb{N}$, and let $|||\cdot|||$ denote the norm induced by the quadratic part of $\mathcal{E}_\mathrm{th,tot}^{\Delta t}$. Let $K_\mathrm{dr}^\star$ as in~\eqref{proof:poro:undrained:kdr}. It holds the \textit{a priori} result
\begin{align*}
  \left|\left|\left| (\erroru^{n,i},\errorflux^{n,i},\errorentropyflux^{n,i}) \right|\right|\right| 
  &\leq  
  \left(\frac{\tfrac{|\bm{\alpha}_\mathrm{T}|^2}{K_\mathrm{dr}^\star}}{\tfrac{|\bm{\alpha}_\mathrm{T}|^2}{\bm{\alpha}_\mathrm{T}^\top \mathbf{M}_\mathrm{T} \bm{\alpha}_\mathrm{T}} + \tfrac{|\bm{\alpha}_\mathrm{T}|^2}{K_\mathrm{dr}^\star}}\right)^i 
  \left( \mathcal{E}^{n,0} - \mathcal{E}^n \right)^{1/2},
 \end{align*}
 where $\mathcal{E}^{n,0}$ and $\mathcal{E}^n$ are the energies of the initial iterate and the solution, resp.
\end{corollary}

\subsubsection{Extended fixed-stress split based on dual thermo-poro-elasticity}

A generalized fixed-stress split for thermo-poro-elasticity is derived by applying alternating minimization to the dual formulation of time-discrete thermo-poro-elasticity~\eqref{dual-minimization:thermoporo:start}. For this, the energy is successively minimized for fixed total stress, and simultaneously fixed fluid pressure and temperature variables; cf.\ Alg.~\ref{algorithm:extended-fixed-stress-split} for a single iteration of the resulting scheme. 
\begin{algorithm}
 \caption{Single iteration of the extended fixed-stress split}
 \label{algorithm:extended-fixed-stress-split}
 \SetAlgoLined
 \DontPrintSemicolon
 
  \vspace{0.5em}
 
  Input: $(\stress^{n,i-1},p^{n,i-1},T^{n,i-1})\in\mathcal{S}^n\times\mathcal{Q}^n\times\mathcal{R}^n$ \; \vspace{0.8em}
   
  Determine $(p^{n,i},T^{n,i}) := \underset{(p,T)\in\mathcal{Q}^n\times\mathcal{R}^n}{\mathrm{arg\,min}}\, \mathcal{E}_\mathrm{th,tot}^{\star, \Delta t}(\fluidmass^{n-1},S^{n-1};\stress^{n,i-1}, p, T)$\; \vspace{0.25em}

  Determine $\stress^{n,i} := \underset{\stress\in\mathcal{S}^n}{\mathrm{arg\,min}}\, \mathcal{E}_\mathrm{th,tot}^{\star, \Delta t}(\fluidmass^{n-1},S^{n-1};\stress, p^{n,i}, T^{n,i})$\; \vspace{0.2em}

\end{algorithm}

By construction the generalized fixed-stress split is equivalent to a predictor-corrector method, simultaneously solving the coupled fluid flow and thermal subproblems under fixed stress conditions in the predictor step. This is equivalent to the stabilized problem: Find $(p^{n,i},T^{n,i})\in\mathcal{Q}^n\times\mathcal{R}^n$ satisfying for all $(q,r)\in \mathcal{Q}_0 \times \mathcal{R}_0$
\begin{align*}
 &\llangle \mathbf{M}_\mathrm{T}^{-1}\begin{bmatrix} p^{n,i} \\ T^{n,i} \end{bmatrix} , \begin{bmatrix} q \\ r \end{bmatrix} \rrangle 
 +
 \llangle \bm{\alphamatrix}_\mathrm{T}\bm{\alphamatrix}_\mathrm{T}^\top\ \left(\mathbf{I} : \mathbb{A} : \mathbf{I}\right) \, \begin{bmatrix} p^{n,i} - p^{n,i-1} \\ T^{n,i} - T^{n,i-1} \end{bmatrix}, \begin{bmatrix} q \\ r \end{bmatrix} \rrangle\\
 &\quad 
 + \llangle \left( \bm{\alphamatrix}_\mathrm{T} \otimes \mathbf{I} \right) : \bm{\varepsilon}_{\u}^{n,i-1}, \begin{bmatrix} q \\r \end{bmatrix} \rrangle
 +
 \Delta t \llangle \mathbf{K}_\mathrm{T} \begin{bmatrix} \GRAD p^{n,i} - \gext^n \\ \GRAD T^{n,i} \end{bmatrix}, \begin{bmatrix} \GRAD q \\ \GRAD r \end{bmatrix} \rrangle\\
 &\quad 
 =
 \llangle \begingroup \setlength\arraycolsep{0pt} \begin{bmatrix} \fluidmass^{n-1} & + \Delta t\, \masssource^n \\ S^{n-1} &+\Delta t \, \entropysource^n \end{bmatrix} \endgroup, \begin{bmatrix} q \\ r \end{bmatrix}\rrangle 
 + \Delta t \llangle q_\mathrm{\Gamma,n}^n, q \rrangle_{\Gamma_{\flux}}
 + \Delta t \llangle j_\mathrm{\Gamma,n}^n, r \rrangle_{\Gamma_{\entropyflux}},
\end{align*}
where we used the formal abbreviation of the mechanical strain
\begin{align*}
 \bm{\varepsilon}_{\u}^{n,i-1} := \mathbb{A}\left(\stress^{n,i-1} + \left(\mathbf{I} \otimes \bm{\alphamatrix}_\mathrm{T} \right) : \begin{bmatrix} p^{n,i-1} \\ T^{n,i-1} \end{bmatrix} \right).
\end{align*}
A characteristic property: Tensorial stabilization is applied. For instance, for a homogeneous, isotropic material, the stabilization term equals
\begin{align*}
 \llangle \begingroup \setlength \arraycolsep{3pt} \begin{bmatrix} \tfrac{\alpha^2}{K_\mathrm{dr}} & 3\alpha\alpha_\mathrm{T} \\ 3\alpha\alpha_\mathrm{T} & 9 \alpha_\mathrm{T}^2 K_\mathrm{dr} \end{bmatrix} \endgroup \begin{bmatrix} p^{n,i} - p^{n,i-1} \\ T^{n,i} - T^{n,i-1} \end{bmatrix}, \begin{bmatrix} q \\ r \end{bmatrix} \rrangle.
 \end{align*}
The second step of Alg.~\ref{algorithm:extended-fixed-stress-split} (the corrector step) is equivalent to solving the unmodified, mechanical problem with updated pressure and temperature. The resulting stabilization terms are identical with those employed within the extended fixed-stress split for thermo-poro-elasticity with thermal convection~\cite{Kim2018a}.

By adopting the ideas of the proof for the undrained split for poro-elasticity, cf.\ Lemma~\ref{lemma:undrained-split:convergence-rate}, to vectorized poro-elasticity, analogous convergence results can be deduced for the undrained-adiabatic split for linear thermo-poro-elasticity.

As for the undrained-adiabatic split, the structural similarities to poro-elasticity allows for adopting the convergence results for the standard fixed-stress split, cf. Lemma~\ref{lemma:fixed-stress:convergence-rate}, and deduce analogous results for the extended fixed-stress split. For instance, without presenting the analogous proof, we state the generalized \textit{a priori} convergence result.

\begin{corollary}[Linear convergence of the extended fixed-stress split]
 The extended fixed-stress split for linear thermo-poro-elasticity converges linearly, independent of the initial guess. Let $\errortemperature^{n,i}:=T^{n,i} - T^n$, $n,i\in\mathbb{N}$, and let $|||\cdot|||_\star$ denote the norm induced by the quadratic part of $\mathcal{E}_\mathrm{th,tot}^{\star,\Delta t}$. Assume for brevity, $\permeability=\kappa \mathbf{I}$ and $\conductivity = \kappa_\mathrm{F} \mathbf{I}$ constant in space. It holds the \textit{a priori} result
\begin{align*}
  \left|\left|\left| (\errorstress^{n,i},\errorpressure^{n,i},\errortemperature^{n,i}) \right|\right|\right|_\star 
  &\leq  
  \left(\frac{\tfrac{|\bm{\alpha}_\mathrm{T}|^2}{K_\mathrm{dr}}}{\tfrac{\bm{\alpha}_\mathrm{T}^\top \left(\mathbf{M}_\mathrm{T}^{-1} + \Delta t\, C_\Omega^{-2} \begingroup \footnotesize \begin{bmatrix} \kappa & 0 \\ 0 & \tfrac{\kappa_\mathrm{F}}{T_0}\end{bmatrix}\endgroup \right)\bm{\alpha}_\mathrm{T}}{|\bm{\alpha}_\mathrm{T}|^2} + \tfrac{|\bm{\alpha}_\mathrm{T}|^2}{K_\mathrm{dr}}}\right)^i 
  \left( \mathcal{E}^{n,0} - \mathcal{E}^n \right)^{1/2},
 \end{align*}
 where $\mathcal{E}^{n,0}$ and $\mathcal{E}^n$ are the energies of the initial iterate and the solution, resp.
\end{corollary}
By the Cauchy-Schwarz inequality, the convergence rate of the extended fixed-stress split is lower than for the undrained-adiabatic split -- even for $\mathbf{K}_\mathrm{T}=\bm{0}$.


\subsubsection{Three-block coordinate descent for thermo-poro-elasticity}

By definition thermo-hydro-mechanical models couple three processes. Thus, in the context of splitting schemes, it is a natural ambition to decouple all three subproblems from each other -- with a potential benefit increase of the same kind as two-stage decoupling methods. Three-stage decoupling methods for thermo-poro-elasticity with thermal convection have been recently proposed by~\cite{Brun2019}, including a rigorous convergence analysis. In the following, we briefly demonstrate that similar methods can be derived by applying three-block coordinate descent methods, a natural generalization of alternating minimization.

Since both the primal and the dual formulations of linear thermo-poro-elasticity are block-separable and convex, any cyclic three-block coordinate descent is globally convergent which respects the block structure of the coupled problem, cf.~\cite{Grippof1999,Grippo2000}. We exemplarily state one candidate of six possible combinations based on the dual problem -- we solve successively for pressure, temperature and stress, cf.\ Alg.~\ref{algorithm:thermoporo:3block-dual} for a single iteration. Similarly, the primal problem can serve as basis; we choose an algorithm closer to the extended fixed-stress split expecting better performance.

\begin{algorithm}
 \caption{Single iteration of the three-block coordinate descent for dual thermo-poro-elasticity}
 \label{algorithm:thermoporo:3block-dual}
 \SetAlgoLined
 \DontPrintSemicolon
 
  \vspace{0.5em}
 
  Input: $(\stress^{n,i-1},p^{n,i-1},T^{n,i-1})\in\mathcal{S}^n\times\mathcal{Q}^n\times\mathcal{R}^n$ \; \vspace{0.8em}

  Determine $p^{n,i} := \underset{p\in\mathcal{Q}^n}{\mathrm{arg\,min}}\, \mathcal{E}_\mathrm{th,tot}^{\star,\Delta t}(\fluidmass^{n-1},S^{n-1};\stress^{n,i-1},p,T^{n,i-1})$\; \vspace{0.25em}
  
  Determine $T^{n,i} := \underset{T\in\mathcal{R}^n}{\mathrm{arg\,min}}\, \mathcal{E}_\mathrm{th,tot}^{\star,\Delta t}(\fluidmass^{n-1},S^{n-1};\stress^{n,i-1},p^{n,i},T)$\; \vspace{0.25em}
  
  Determine $\stress^{n,i} := \underset{\stress\in\mathcal{S}^n}{\mathrm{arg\,min}}\, \mathcal{E}_\mathrm{th,tot}^{\star,\Delta t}(\fluidmass^{n-1},S^{n-1};\stress,p^{n,i},T^{n,i})$\; \vspace{0.2em}

\end{algorithm}

The first step of Alg.~\ref{algorithm:thermoporo:3block-dual} is equivalent to solving a fluid flow problem with fixed-stress type pressure stabilization: Find $p^{n,i}\in\mathcal{Q}^n$ satisfying for all $q\in\mathcal{Q}_0$
\begin{align*}
 \tfrac{1}{M} \llangle p^{n,i}, q \rrangle 
 - 3\alpha_\phi \llangle T^{n,i-1}, q \rrangle
 + \tfrac{\alpha^2}{K_\mathrm{dr}} \llangle p^{n,i} - p^{n,i-1}, q \rrangle
 + \alpha \llangle \mathrm{tr}\, \bm{\varepsilon}_{\u}^{n,i-1}, q \rrangle & \\
 + \Delta t \llangle \permeability (\GRAD p^{n,i} - \gext^n), \GRAD q \rrangle 
 =
 \llangle \fluidmass^{n-1} + \Delta t\, \masssource^n , q \rrangle + \Delta t \llangle q_\mathrm{\Gamma,n}^n, q \rrangle_{\Gamma_{\flux}}.
\end{align*}
The second step of Alg.~\ref{algorithm:thermoporo:3block-dual} is equivalent to solving a thermal problem with fixed-stress type temperature stabilization: Find $T^{n,i}\in\mathcal{R}^n$ satisfying 
\begin{align*}
 \tfrac{C_\mathrm{d}}{T_0} \llangle T^{n,i}, r \rrangle 
 - 3\alpha_\phi \llangle p^{n,i}, r \rrangle
 + 9\alpha_\mathrm{T}^2 K_\mathrm{dr} \llangle T^{n,i} - T^{n,i-1}, r \rrangle 
 + 3\alpha_\mathrm{T} K_\mathrm{dr} \llangle \mathrm{tr}\, \bm{\varepsilon}_{\u}^{n,i-1/2}, r \rrangle& \\
 + \Delta t \llangle \tfrac{\conductivity}{T_0} \GRAD T^{n,i}, \GRAD r \rrangle
 =
 \llangle S^{n-1} + \Delta t \, \entropysource^n, r \rrangle + \Delta t \llangle j_\mathrm{F,\Gamma}^n, r \rrangle_{\Gamma_{\entropyflux}},
\end{align*}
for all $r\in\mathcal{R}_0$, where we formally abbreviated the updated mechanical strain
\begin{align*}
 \bm{\varepsilon}_{\u}^{n,i-1/2} := \mathbb{A} \left( \stress^{n,i-1} + \left(\mathbf{I} \otimes \bm{\alphamatrix}_\mathrm{T} \right) : \begin{bmatrix} p^{n,i} \\  T^{n,i-1} \end{bmatrix}\right).
\end{align*}
The final step of Alg.~\ref{algorithm:thermoporo:3block-dual} is identical with solving the pure mechanical problem for updated pressure and temperature. All in all, the main difference of the resulting scheme to the extended fixed-stress split is the diagonal instead of tensorial stabilization due to further decoupling.

\subsection{Comments on splitting schemes for non-linear thermo-poro-elasticity}\label{section:non-linear-thermo-poro}

The splitting schemes derived in this section are in first place only guaranteed to be robust for semi-discrete thermo-poro-elasticity models with an underlying convex minimization structure. As discussed in the modelling section, general thermo-poro-elasticity models do only satisfy a perturbed gradient flow structure, cf.\ Remark~\ref{remark:porothermo-perturbed}. Therefore the minimizing movement scheme does not apply immediately, and implicit semi-discrete thermo-poro-elasticity models do generally not stem from convex minimization. Evidently, by explicitly lagging the perturbations in time, the symmetric character of linear thermo-poro-elasticity can be retained, and the above splitting schemes are robust.
 
Nonetheless, the splitting schemes derived for the simplified, linear case above may as well assist in the construction of splitting schemes for the fully non-linear problem. We mention two possible strategies:

\begin{enumerate}[label=(\roman*)]
 
 \item After decomposing the time-continuous, coupled problem into a sum of a linear and parabolic, and a convective problem, an operator splitting~\cite{Holden2010}, e.g., Strang splitting, is utilized. Then the parabolic problem, essentially identical to linear thermo-poro-elasticity, may be solved efficiently using the above splitting schemes, and the convective problem may be solved by a separate, tailored scheme.
 
 \item Consider the semi-discrete problem obtained after applying the implicit Euler method. Provided that the perturbations and the time step size are sufficiently small, the semi-discrete problem exhibits a non-symmetric but elliptic character. Under that assumptions iterative two- and three-stage splitting schemes with sufficient diagonal stabilization have been rigorously showed to be linearly convergent~\cite{Brun2019}. Consequently, robust convergence may be also expected for stabilization terms replaced by those resulting from the above discussions, i.e., effectively by applying the undrained-adiabatic and extended fixed-stress split as recently proposed by~\cite{Kim2018a}. Numerically, this has been demonstrated by the aforementioned work.
\end{enumerate}

\section{Acceleration of splitting schemes by optimal relaxation}\label{section:minimization:line-search}

Due to the minimization character of the fully coupled, semi-discrete thermo-poro-visco-elasticity equations, the convergence of splitting schemes for such can be effectively improved by relaxation. Alg.~\ref{algorithm:line-search} formulates relaxation by exact line search for a general, inexact minimization algorithm for solving semi-discrete generalized gradient flows discretized by the minimizing movement scheme (Sec.~\ref{section:general-time-discretization}). For quadratic minimization problems with affine constraints (i.e., e.g., linear thermo-poro-visco-elasticity), optimal relaxation in the sense of a classical, exact line search strategy is feasible; minimizing the quadratic interpolation of three energy values is sufficient for computing the optimal weight. However, also for nonlinear thermo-poro-visco-elasticity stemming from non-quadratic, but convex minimization under affine constraints, we propose the same simple (now inexact) line search strategy.

\begin{algorithm}[h!]
 \caption{Relaxation of inexact minimization $\mathcal{IM}$ by exact line search for solving time-discrete generalized gradient flows~\eqref{time-discrete-minimization-state}}
 \label{algorithm:line-search}
 \SetAlgoLined
 \DontPrintSemicolon
  
 \vspace{0.2em}
  
 Given $\mathcal{X}^n$ affine, $x^{n-1}\in\mathcal{X}^{n-1}$, define $\mathcal{E}^\Delta(x) := \Delta t \, \mathcal{D}\left(\frac{x - x^{n-1}}{\Delta t}\right) + \mathcal{E}(x) - \mathcal{P}_\mathrm{ext}^n(x)$ \;\vspace{0.25em}
 Let $\mathcal{IM}:\mathcal{X}^n \rightarrow \mathcal{X}^n$ such that $\mathcal{E}^\Delta(\mathcal{IM}(x)) < \mathcal{E}^\Delta(x)$, where wlog.\ $x$ is not the minimizer \;\vspace{0.5em}
 $x^{n,0} \gets x^{n-1}$, $i\gets 1$ \;\vspace{0.5em}
 \While {\textit{'stopping criterion not satisfied'}}{ \vspace{0.5em}
  Compute $x^{n,i-1/2} \gets \mathcal{IM}(x^{n,i-1})\in\mathcal{X}^n$ \; \vspace{0.5em}
  Obtain descent direction $\Delta x^{n,i} \gets x^{n,i-1/2} - x^{n,i-1}$  \vspace{0.5em} \;
  Solve $\alpha^{n,i} \gets \underset{\alpha}{\mathrm{arg\,min}} \, \mathcal{E}^\Delta\left(x^{n,i-1/2} + \alpha \Delta x^{n,i}\right)$ \;
  Update $x^{n,i} \gets x^{n,i-1/2} + \alpha^{n,i} \Delta x^{n,i}\in\mathcal{X}^n$ \vspace{0.5em} \;
  $i \gets i+1$ \vspace{0.5em}
 }
\end{algorithm}

\section{Numerical examples -- Performance of the relaxed fixed-stress split for a 3D footing problem}\label{section:numerical-results}
 
Splitting schemes for solving thermo-hydro-mechanical processes have been numerically studied from various angles in the literature. In the following, we focus on three of the main new contributions obtained from the gradient flow analysis, not previously reported in literature, and study: (i) the impact of relaxation of splitting schemes by exact line search also put in context to the optimization of splitting schemes, (ii) the performance of splitting schemes for poro-visco-elasticity, and (iii) the performance of splitting schemes for nonlinear poro-elasticity. Due to its larger popularity, we restrict the study to fixed-stress-type splits.

All in all, we consider four test cases based on the same geometry but with slightly differing mechanical material behavior -- a unit cube, cf.\ Fig.~\ref{figure:footing-geometry}, fixed at the bottom and subject to a ramped up, normal force at the top, i.e., 
\begin{align*}
 \u_\Gamma=\bm{0}\text{ on }[0,1]^2\times\{0\},\quad \stress_\mathrm{\Gamma,n}(t)= 10^9t\text{\,[N/m$^2$s]\,}\bm{e}_3,\text{ on }[0.25,0.75]^2\times\{1\}.
\end{align*}
No-flow is imposed on the same parts of the boundary. No-stress and zero-pressure boundary conditions are applied on the remaining boundary. Body forces are absent. Similar setups have been considered by~\cite{Gaspar2008,Storvik2019,Adler2019}.

If not mentioned otherwise, the geometry is discretized by a structured $16 \times 16 \times 16$ hexahedral mesh, and $5$ time steps of constant time step size $\Delta t=0.1$ [s] are simulated. For the numerical solution the plain and the relaxed fixed-stress splits are applied. The performance of those is measured in terms of the average number of iterations per time step required for convergence and run times, where as stopping criterion a relative $L^2(\Omega)$ error with tolerance $\epsilon_{r} = 10^{-6}$ is employed. For the implementation of the numerical examples, we use the DUNE project~\cite{dune2016}, with extensive use of the DUNE-functions module~\cite{dunefunctions2015,dunefunctions2018}.
\begin{figure}
\centering
\subfloat[Geometry\label{figure:footing-geometry}]{\begin{overpic}[width=0.35\textwidth]{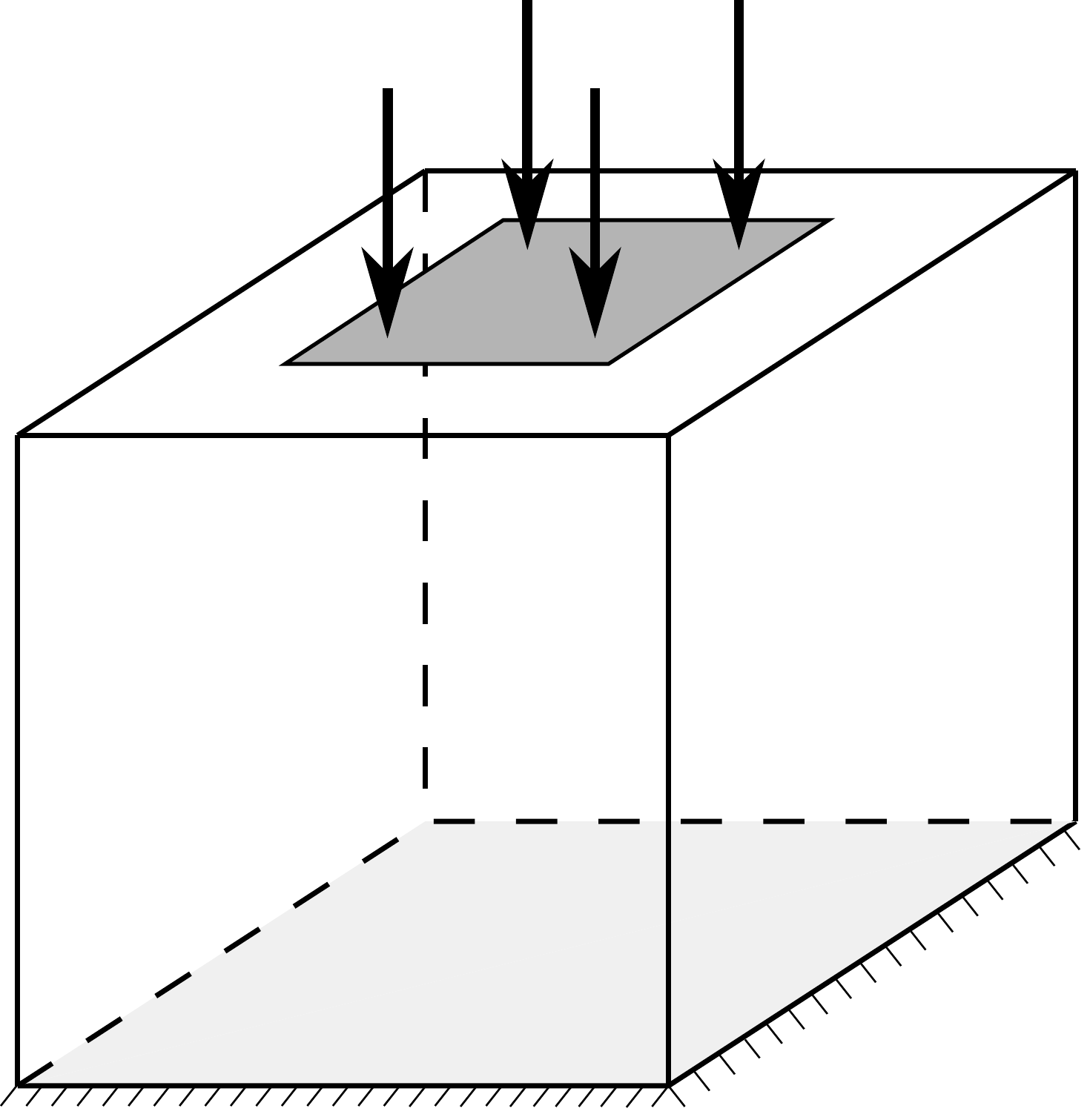}
                     \put(48,96){ $\stress_\mathrm{\Gamma,n}$}
                    \end{overpic}
 }
 \hspace{0.75cm}
\subfloat[Simulation result for test case I\label{figure:footing-result}]
 {\includegraphics[width=0.55\textwidth]{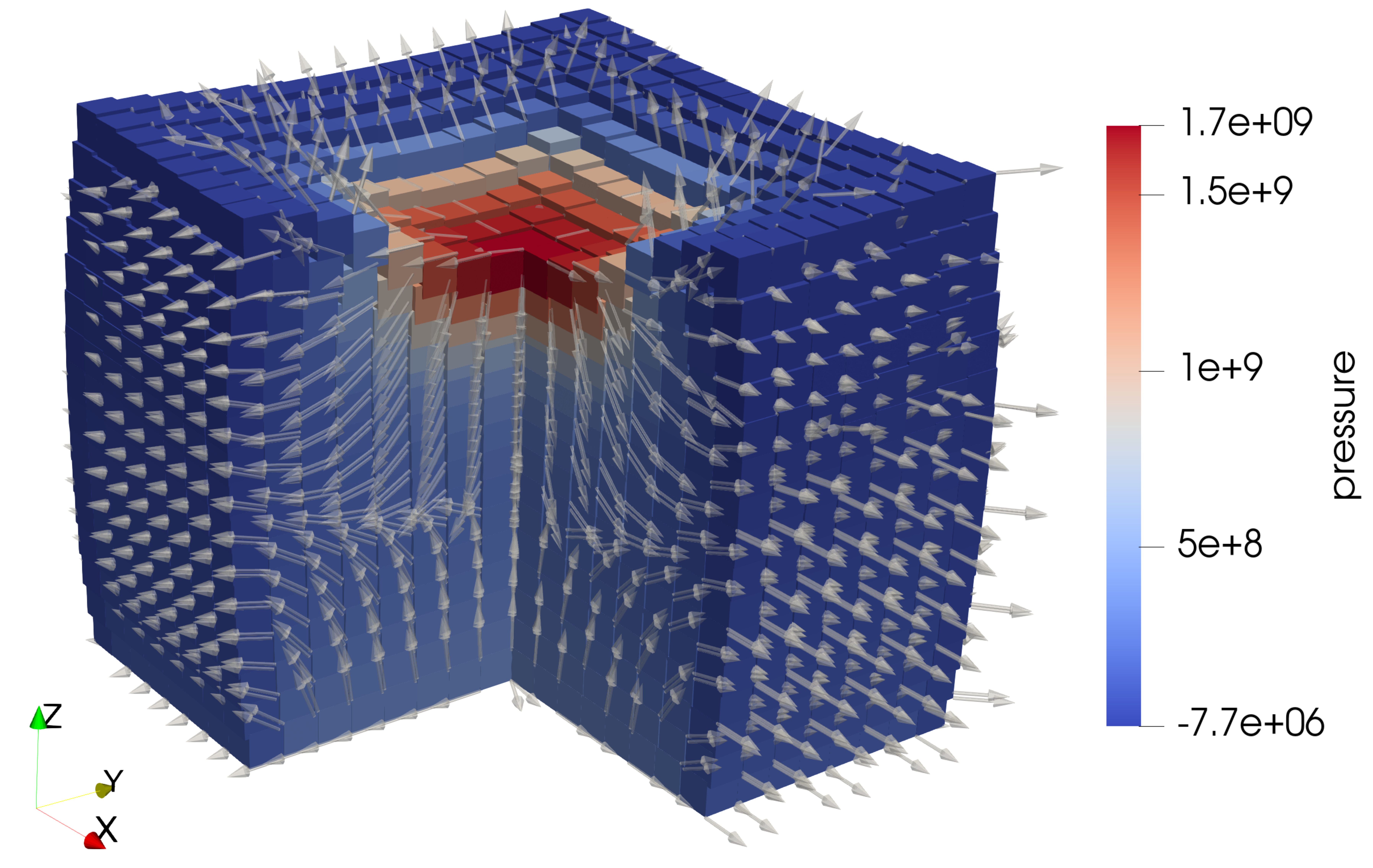}
 }
 \caption{Initial configuration incl.\ boundary conditions for the 3D footing problem; deformed, poro-elastic configuration after 5 time steps for $E=10^{11}$ [Pa], incl. pressure profile and outflow; the deformation is scaled by factor 15.}
 \vspace{-0.3cm}
\end{figure}

\subsection{Poro-elastic test case I -- Line search under varying coupling strength}
The material is assumed to be poro-elastic, homogeneous and isotropic with material parameters as in Tab.~\ref{table:parameters:poro}. In this first part, we study the impact of the relaxation by line search of the fixed-stress split under varying coupling strength. For this, we vary the Young's modulus $E$, which is inversely proportional to the coupling strength. A simulation result for $E=10^{11}$~[Pa] is visualized in Fig.~\ref{figure:footing-result}.
\begin{table}[h]
\begin{center}
\def\arraystretch{1.5}
\begin{tabular}{ l  c  r r l}
\hline
 Name & Symbol & Value& Value & Unit\\[-5pt]
 & & (Test case I--III) & (Test case IV) & \\
\hline
Young's modulus & $E$ & $10^9..10^{12}$ & $10^{10}$ & Pa\\
Poisson's ratio & $\nu$ &  $0.2$ & $0.2,0.495$ & --\\
Biot-Willis constant &$\alpha$ &  $1$ & $1$ & --\\
Compressibility coefficient & M & $10^{11}$ & $10^{11}$ & Pa \\
Permeability & $\kappa$ & $10^{-13}$  & $10^{-11}$ & $\mathrm{m}^2$\\
\hline
\end{tabular}
\caption{\label{table:parameters:poro} Poro-elasticity-specific material parameters for the 3D footing problem, used in test cases I--IV.}
\vspace{-0.8cm}
\end{center}
\end{table}

By applying the Galerkin method to the five-field formulation of the semi-discrete, linear Biot equations, cf.\ Sec.~\ref{section:minimization:five-field}, a fully structure-preserving spatial discretization is employed. As conforming finite element spaces for the mechanical problem, we utilize lowest order Brezzi-Douglas-Marini elements, cf., e.g.,~\cite{Boffi2013}, for the (unsymmetric) stress tensor, piecewise constants for the mechanical displacements and piecewise constant, skew-symmetric tensors for the rotation. For the fluid flow problem, we employ lowest order Brezzi-Douglas-Marini elements for the volumetric flux and piecewise constant elements for the fluid pressure. However, we note, the subsequent results are not strongly depending on the particular formulation or spatial discretization. 

The performance of the plain and the relaxed fixed-stress splits is displayed in Fig.~\ref{figure:performance:testcase-I}.
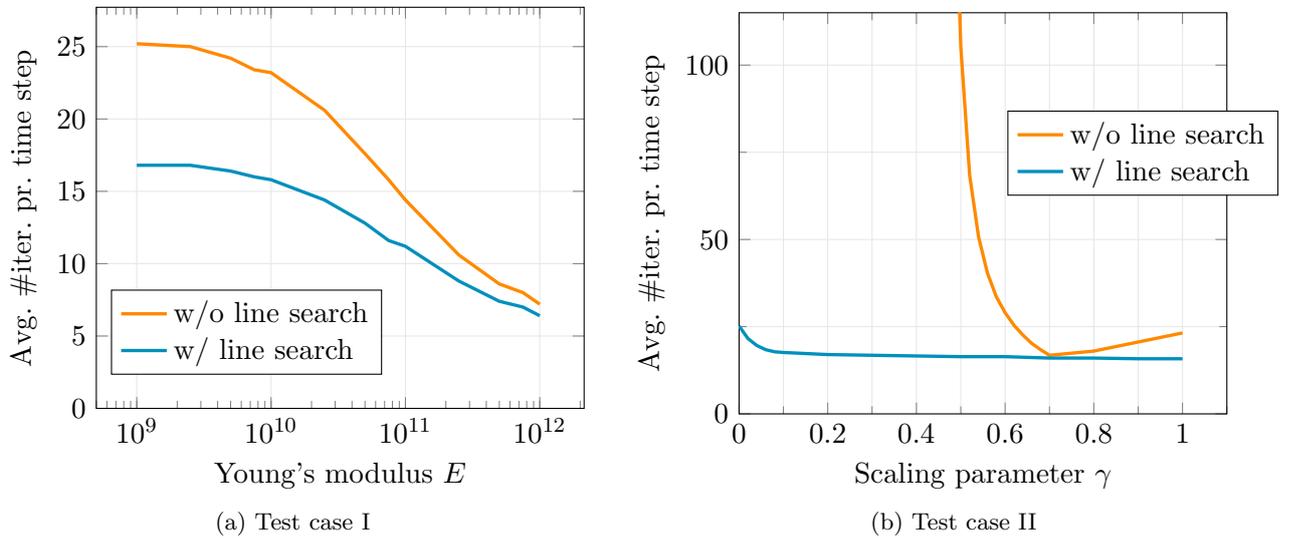
\begin{figure}[h!] 
\subfloat[Test case I\label{figure:performance:testcase-I}]{\centering\begin{tikzpicture}
\begin{semilogxaxis}[
    ylabel={Avg.\ \#iter.\ pr.\ time step},
    width = 0.5\textwidth,
    xlabel={Young's modulus $E$},
    legend entries={{{w/o line search},{w/ line search}}},
    legend cell align=left,
    legend columns=1,
    legend style={at={(0.03,0.19)},anchor=west},
    xmin = 5e8,
    ymin = 0,
    ymajorgrids=true, major grid style={line width=.2pt, draw=gray!20},
    xmajorgrids=true
] 
 
\addplot [solid, line width=1.25pt, color=grun] coordinates { (1,0) };

\addplot [solid, line width=1.25pt, color=blau] coordinates { (1,0) };

\addplot [solid, line width=1.25pt, color=grun]   coordinates {
(10e8, 126/5)
(25e8, 125/5)
(50e8, 121/5)
(75e8, 117/5)
(100e8, 116/5)
(250e8, 103/5)
(500e8, 88/5)
(750e8, 79/5)
(1000e8, 72/5)
(2500e8, 53/5)
(5000e8, 43/5)
(7500e8, 40/5)
(10000e8, 36/5)
}; 

\addplot [solid, line width=1.25pt, color=blau]   coordinates {
(10e8, 84/5)
(25e8, 84/5)
(50e8, 82/5)
(75e8, 80/5)
(100e8, 79/5)
(250e8, 72/5)
(500e8, 64/5)
(750e8, 58/5)
(1000e8, 56/5)
(2500e8, 44/5)
(5000e8, 37/5)
(7500e8, 35/5)
(10000e8, 32/5)
};

%

\end{semilogxaxis}
\end{tikzpicture}}
\hspace{0.5cm}
\subfloat[Test case II\label{figure:performance:testcase-II}]{\centering
\begin{tikzpicture}
\begin{axis}[
    ylabel={Avg.\ \#iter.\ pr.\ time step},
    width = 0.5\textwidth,
    xlabel={Scaling parameter $\gamma$},
    legend entries={{{w/o line search},{w/ line search}}},
    legend cell align=left,
    legend columns=1,
    legend style={at={(0.55,0.65)},anchor=west},
    xmin = 0,
    ymin = 0,
    ymax = 115, 
    ytick={0,50,100},
    grid=both, grid style={line width=.2pt, draw=gray!20}, minor tick num=1
] 
 
\addplot [solid, line width=1.25pt, color=grun] coordinates { (-1,0) };
\addplot [solid, line width=1.25pt, color=blau] coordinates { (-1,0) };

\addplot [solid, line width=1.25pt, color=grun]   coordinates {
(40e-2, 500)
(50e-2, 526/5)
(52e-2, 342/5)
(54e-2, 254/5)
(56e-2, 202/5)
(58e-2, 168/5)
(60e-2, 145/5)
(62e-2, 127/5)
(64e-2, 113/5)
(66e-2, 101/5)
(68e-2, 92/5)
(70e-2, 84/5)
(80e-2, 90/5)
(90e-2, 103/5)
(100e-2, 116/5)
};


\addplot [solid, line width=1.25pt, color=blau]   coordinates {
(0e-2, 126/5)
(2e-2, 108/5)
(4e-2, 98/5)
(6e-2, 92/5)
(8e-2, 89/5)
(10e-2, 88/5)
(20e-2, 85/5)
(30e-2, 84/5)
(40e-2, 83/5)
(50e-2, 82/5)
(60e-2, 82/5)
(70e-2, 80/5)
(80e-2, 80/5)
(90e-2, 79/5)
(100e-2, 79/5)
}; 


\end{axis}
\end{tikzpicture}}
\caption{Number of (poro-elasticity) fixed-stress split iterations for varying coupling strength (Test case I) and varying stabilization parameter (Test case II).}
\vspace{-0.3cm}
\end{figure}
We observe that the relaxation by line search allows for reducing the number of iterations up to a factor of 30\%. A greater impact can be observed for more strongly coupled problems. On the other hand, only small improvement is observed for weakly coupled problems. This is related to the result of the following test case.

\subsection{Poro-elastic test case II -- Line search vs. stabilization tuning}

It has been previously emphasized~\cite{Kim2011b,Mikelic2013} that the fixed-stress split can be tuned by appropriate weighting of the stabilization parameter $\tfrac{\alpha^2}{K_\mathrm{dr}}$ in the fluid flow problem, cf.~\eqref{fixed-stress-pressure-equation}. \textit{A priori} knowledge on optimal tuning however is lacking due to a strong dependence on the specific geometry, material parameters and applied boundary conditions~\cite{Castelletto2015}. It has been numerically demonstrated that optimal weighting may differ substantially from test case to test case~\cite{Both2018a}. Hence, in general, it is difficult to tune the parameter in practice; in~\cite{Storvik2019} the authors discuss a brute-force optimization strategy utilizing a coarse mesh.

In the following, we demonstrate that the application of exact line search yields a flexible, black box-type alternative to tuning the stabilization parameter. For this, we replace the stabilization parameter by $\gamma \tfrac{\alpha^2}{K_\mathrm{dr}}$ with $\gamma\in[0,1]$ and apply again both plain and relaxed fixed-stress splits in order to solve the 3D footing problem. Here, we choose the same parameters as in test case I, but with fixed $E=10^{10}$ [Pa]. The number of iterations required for convergence for varying $\gamma$ are displayed in Fig.~\ref{figure:performance:testcase-II}. We make two observations:
  
\begin{itemize}

\item[$\bullet$] For the plain fixed-stress split we observe practical convergence only for $\gamma\in[0.5,1]$. This is consistent with theoretical considerations, cf., e.g.,~\cite{Mikelic2013,Both2017,Storvik2019}. The line-search enhanced fixed-stress split however shows very robust behavior wrt.\ $\gamma$; despite the strong coupling, convergence is even observed for lacking stabilization ($\gamma=0$). 
 
\item[$\bullet$] For optimally chosen weighting ($\gamma \approx 0.7$) there is no difference in the number of iterations between the plain and the relaxed fixed-stress splits. 

\end{itemize}
Altogether, line search acts here as black-box tuning of the stabilization parameter. However, we note, there is no theoretical guarantee for the optimality of relaxed splitting schemes compared to optimized splitting schemes.

\subsection{Poro-visco-elastic test case III -- Line search under varying coupling strength}

In the following test case, we demonstrate the convergence of the fixed-stress split for poro-visco-elasticity. For this, we re-consider test case I now for a poro-visco-elastic material, and enhance the poro-elastic material parameters (Tab.~\ref{table:parameters:poro}) by visco-elasticity-specific parameters displayed in Tab.~\ref{table:parameters:porovisco}. A simulation result for $E=10^{11}$~[Pa] is visualized in Fig.~\ref{figure:porovisco-footing-result}.
\begin{table}[h!]
\begin{center}
\def\arraystretch{1.5}
\begin{tabular}{ l  c  r l}
\hline
 Name & Symbol & Value& Unit\\
\hline
Young's modulus & $E_\mathrm{v}$ & $10^{10}$ & Pa\\
Poisson's ratio & $\nu_\mathrm{v}$ &  $0.3$ & --\\
Shear modulus   & $\mu_\mathrm{v}'$ & $0$ & Pa\\
Lam\'e constant & $\lambda_\mathrm{v}'$ &  $10^9$ & Pa \\
Biot-Willis constant &$\alpha_\mathrm{v}$ &  $0.8$ & --\\
\hline
\end{tabular}
\caption{\label{table:parameters:porovisco}Poro-visco-elasticity-specific material parameters for the 3D footing problem.}
\vspace{-0.8cm}
\end{center}
\end{table}

For the spatial discretization, we again utilize a fully-structure preserving formulation based on the dual formulation, cf. Remark~\ref{poro-visco-elasticity-formulations}. In particular, the visco-elastic stress $\stress_\mathrm{v}$ is explicitly introduced, cf.~\eqref{porovisco-visco-elastic-stress-space}, with the visco-elastic strain computed from~\eqref{porovisco-visco-elastic-strain-definition} by projection onto piecewise constant, symmetric tensors. Hence, the resulting, spatial discretization has the same complexity as in the case of poro-elasticity. 

The number of iterations for the plain and relaxed fixed-stress splits required for convergence is displayed in Fig.~\ref{figure:performance:testcase-III}.
\begin{figure}[h!] 
\centering 
\subfloat[\label{figure:porovisco-footing-result} Simulation result]{\includegraphics[width=0.475\textwidth]{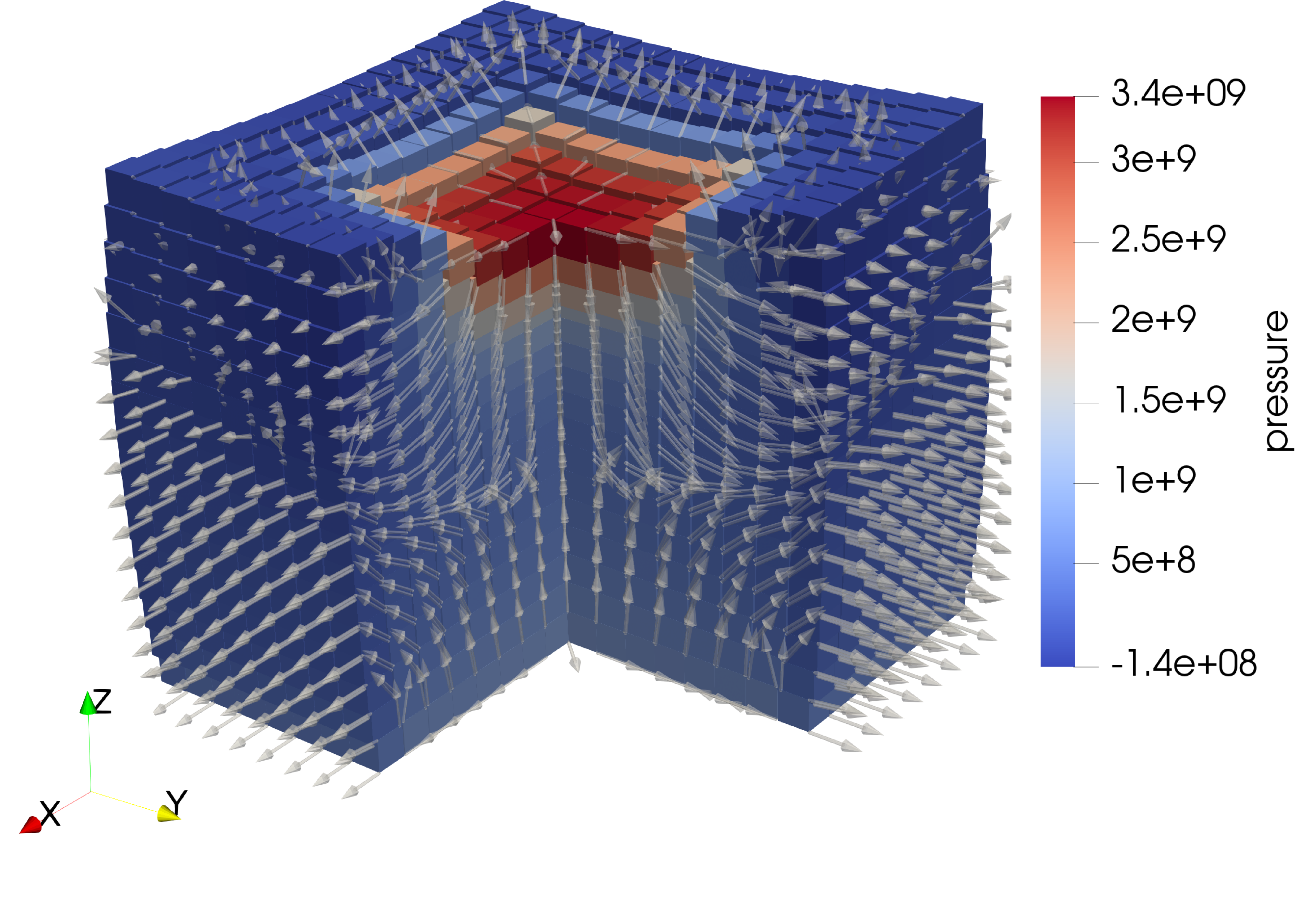}} 
\hspace{0.1cm}
\subfloat[Performance result \label{figure:performance:testcase-III}]{
\begin{tikzpicture}
\begin{semilogxaxis}[
    ylabel={Avg.\ \#iter.\ pr.\ time step},
    width = 0.5\textwidth,
    xlabel={Young's modulus $E$},
    legend entries={{{w/o line search},{w/ line search}}},
    legend cell align=left,
    legend columns=1,
    legend style={at={(0.03,0.19)},anchor=west},
    xmin = 5e8,
    ymin = 0,
    ymajorgrids=true, major grid style={line width=.2pt, draw=gray!20},
    xmajorgrids=true
] 
 

\addplot [solid, line width=1.25pt, color=grun] coordinates { (1,0) };
\addplot [solid, line width=1.25pt, color=blau] coordinates { (1,0) };

\addplot [solid, line width=1.25pt, color=grun]   coordinates {
(10e8, 120/5)
(25e8, 111/5)
(50e8, 99/5)
(75e8, 92/5)
(100e8, 86/5)
(125e8, 82/5)
(250e8, 68/5)
(375e8, 61/5)
(500e8, 57/5)
(635e8, 56/5)
(750e8, 57/5)
(825e8, 58/5)
(1000e8, 61/5)
(2500e8, 67/5)
(5000e8, 71/5)
(7500e8, 72/5)
(10000e8, 73/5)
};

\addplot [solid, line width=1.25pt, color=blau]   coordinates {
(10e8, 82/5)
(25e8, 78/5)
(50e8, 73/5)
(75e8, 70/5)
(100e8, 67/5)
(250e8, 60/5)
(500e8, 56/5)
(750e8, 55/5)
(1000e8, 53/5)
(2500e8, 51/5)
(5000e8, 51/5)
(7500e8, 51/5)
(10000e8, 51/5)
};

%
%
\end{semilogxaxis}
\end{tikzpicture}}
\caption{Test case III:
The deformed, poro-visco-elastic configuration after 5 time steps for $E=10^{11}$ [Pa], incl. pressure profile and outflow. And the number of (poro-visco-elasticity) fixed-stress split iterations for varying coupling strength.}
\vspace{-0.3cm}
\end{figure}
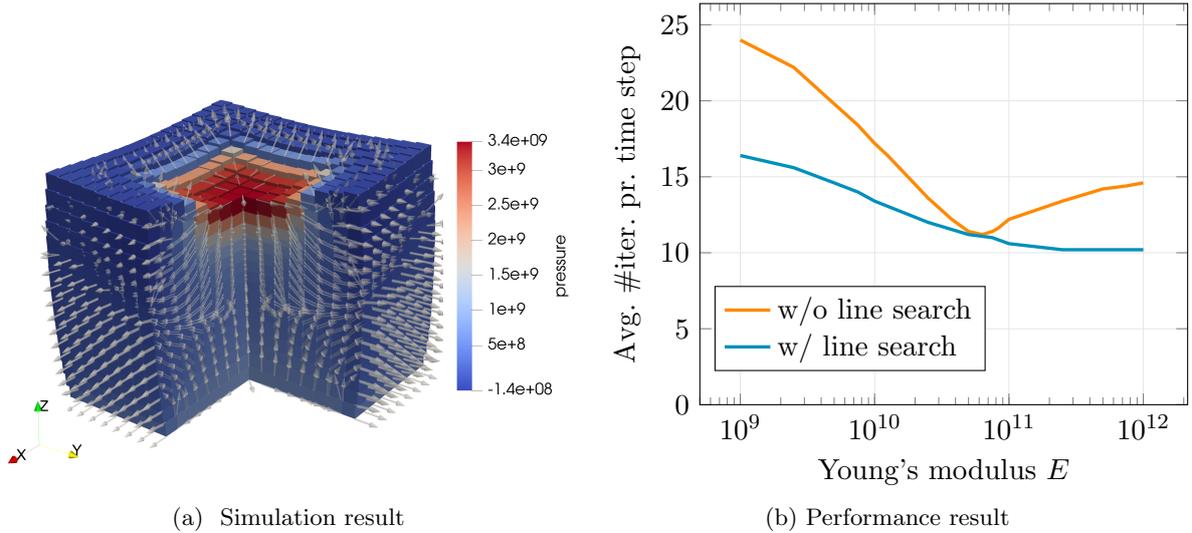
At first glance, the performances of both splitting schemes look qualitatively differently. The relaxed fixed-stress split exhibits a monotone relation between its performance and the coupling strength, consistent with the theoretical convergence result, cf.\ Lemma~\ref{lemma:porovisco:fixed-stress:convergence-rate}. In contrast, the plain fixed-stress split reveals a worsening of the performance for weaker coupling. This can be explained by the findings from test case II.  For varying Young's modulus, the overall, structural behavior of the material alters due to $\nu\neq \nu_\mathrm{v}$. As a consequence, considering the optimized fixed-stress split, the optimal tuning parameter changes with $E$. For smaller and larger $E$, it is further off the natural stabilization parameter employed within the plain fixed-stress split; for intermediate Young's modulus ($E\approx 5\cdot 10^{11}$ [Pa]), both parameters are relatively close. This can be justified by the fact that for that configuration line search relaxation does not yield any improvement of the convergence.

After all, if the optimal tuning parameter had been employed for each Young's modulus, the plain fixed-stress split would exhibit the same monotone behavior as under relaxation. Again, line search relaxation proves successful as black box tuning without \textit{a priori} knowledge of the physical behavior of the medium.

\subsection{Non-linear poro-elastic test case IV -- Acceleration and robustness increase of splitting schemes by line search relaxation}

In the final numerical test case, we demonstrate the convergence of the fixed-stress split for nonlinear poro-elasticity under infinitesimal strains (Sec.~\ref{section:non-linear-biot-linear-coupling}). In particular, we study the impact of line search relaxation for various, inexact fixed-stress splits (Sec.~\ref{section:fixed-stress-non-linear-three-field}) in comparison to the exact fixed-stress split.

For this, we re-consider test case I now for a nonlinearly elastic material. Differently from before, for the spatial discretization, we consider a structured $32 \times 32 \times 32$ hexahedral mesh, inducing a greater challenge to the nonlinear and linear solvers. Furthermore, a three-field formulation, consistent with Sec.~\ref{section:splitting-non-linear-poro}, is considered. We employ linear elements for the structural displacement, lowest order Raviart-Thomas elements for the volumetric flux and piecewise constant elements for the fluid pressure.

In order to pinpoint the impact of the non-linear character of the equations, we introduce only a single non-linearity compared to test case I -- a non-linear (effective) stress-strain relationship corresponding to the non-quadratic p-Laplacian-type energy~\cite{Biot1973,Barucq2005}
\begin{align*}
\mathcal{E}_\mathrm{nl,eff}(\u) = \int_\Omega  W(\eps{\u}) \, dx = \int_\Omega \left( \mu | \eps{\u}|^2 + \frac{\lambda}{3} | \DIV \u |^3 \right)\, dx.
\end{align*}
Apart from that we choose the same model as in test case I, with material parameters from Tab.~\ref{table:parameters:poro}. We consider two setups with two different Poisson ratios: Setup A with $\nu=0.2$, and Setup B with $\nu=0.495$, inducing a comparatively stronger coupling strength and stronger non-linearity, respectively. The simulation result for $\nu=0.2$ is illustrated in Fig.~\ref{figure:testcase-iv}; the qualitative difference in the flow field compared to test case I, cf.\ Fig~\ref{figure:footing-result}, originates from significantly different permeability fields.

The agenda is similar as before. We apply the fixed-stress split in order to solve the coupled problem, and study the impact of line search relaxation. The non-linear character of the problem allows for choosing various exact or inexact non-linear solvers for solving the mechanics subproblems. In addition, we point out, that the exact fixed-stress split (Sec.~\ref{section:fixed-stress-non-linear-three-field}) introduces a displacement-dependent pressure stabilization of the flow equation via the solution-dependent bulk modulus $K_\mathrm{dr}(\eps{u}) = \tfrac{2\mu}{d} + 2\lambda|\DIV \u|$, cf.~\eqref{definition:kdr-u}. Hence, despite the linear character of the flow equation, the exact Jacobian of the stabilized pressure equation alters with each fixed-stress iteration. 

In the following, we apply Newton- and L-scheme-based fixed-stress splits, as introduced in Sec.~\ref{section:fixed-stress-non-linear-three-field}, with the latter chosen due to the low computational cost per iteration; however, the L-scheme requires choosing $L_b,\ L_\mathrm{FS}$ and $\mathbb{L}$. Given user-defined parameters $0\leq |\DIV \u|_\mathrm{min}\leq |\DIV \u|_\mathrm{max} < \infty$, we set
\begin{align}
\label{l-scheme-choices-numerical-testcase}
 L_b=\frac{1}{M},
 \quad
 L_\mathrm{FS}=\frac{\alpha^2}{\tfrac{2\mu}{d} + 2\lambda |\DIV \u|_\mathrm{min}},
 \quad
 \mathbb{L} = 2\mu \mathbb{I} + 2\lambda |\DIV \u|_\mathrm{max} \, \mathbf{I} \otimes \mathbf{I}.
\end{align}
Detailed descriptions of the non-linear solvers used in this section are given in Tab.~\ref{table:numerical-example-non-linear-solvers}. In addition, we apply three relaxation techniques, cf.\ Table~\ref{table:numerical-example-accelerations}. In particular, we also consider applying line search after each non-linear iteration.

\begin{table}[h!]
\centering
\def\arraystretch{2}
{\footnotesize
 \begin{tabular}{|l|l|}
 \hline
  Abbreviation & Description \\
  \hline\hline
  N$_\mathrm{max}$ 
    & Newton's method until convergence, i.e., $\| r^i \| / \| r^0 \| < 10^{-5}$, \\[-1em]
    & where $r^i$ is the residual of the subproblem in the $i$-th Newton\\[-1em]
    & iteration; the Jacobian of the flow equation is reassembled. \\
  \hline 
  N$_1$   
    & As N$_\mathrm{max}$ but employing only a single Newton iteration. \\
  \hline 
  L$_m^\mathrm{ex}$
    & $m$ L-scheme iterations if convergence is not met before (see N$_\mathrm{max}$),\\[-1em]
    & with $L_b$, $L_\mathrm{FS}$ and $\mathbb{L}$ as in~\eqref{l-scheme-choices-numerical-testcase} with $|\DIV \u|_\mathrm{min}= \underset{x,t}{\mathrm{min}}\,|\DIV \u|$\\[-1em]
    & and $|\DIV \u|_\mathrm{max}=\underset{x,t}{\mathrm{max}}\,|\DIV \u|$ \\
  \hline
  & \\[-17pt]
  L$_m^\mathrm{opt}$ & As L$_m^\mathrm{ex}$ but with $|\DIV \u|_\mathrm{min}= |\DIV \u|_\mathrm{max}=\tfrac{1}{10} \underset{x,t}{\mathrm{max}}\,|\DIV \u|$ \\[-17pt]
  & \\
 \hline
 \end{tabular}
}
\caption{\label{table:numerical-example-non-linear-solvers} Non-linear solvers employed in test case IV.}
\end{table}

\begin{table}[h!]
\centering
\def\arraystretch{2}
{\footnotesize
 \begin{tabular}{|l|l|}
 \hline
  Abbr. & Description of the relaxation strategy \\
  \hline\hline
  LS$_\mathrm{-}$ 
    & Plain splitting scheme and non-linear solver without any relaxation.\\
  \hline
  LS$_\mathrm{s}$
    & Line search based on quadratic interpolation applied for the splitting solver.\\
  \hline
  LS$_\mathrm{s/m}$
    & Same as LS$_\mathrm{s}$, but with the same strategy also applied on the inner non-linear\\[-1em]
    & solver for the mechanics subproblem. \\
  \hline
 \end{tabular}
}
\caption{\label{table:numerical-example-accelerations} Relaxation strategies employed in test case IV.}
\end{table}

The solver performances of various relevant combinations of non-linear solvers and relaxation strategies for Setup A and Setup B are displayed in Fig.~\ref{figure:testcase-iv}. Those include the plain number of outer fixed-stress iterations and potential inner extra non-linear iterations if more than one iteration has been applied; in addition, total run times are displayed for Setup B, including run times for assembling matrices and right hand sides, as well as the application of linear solvers. We stress, we use serial, direct solvers. Hence, the Jacobian employed for L-scheme-based splittings is factorized only once, but not for Newton-based splits. Moreover, we mention observations not indicated in the figures:



\pgfplotstableread{
1 13.8
2 14
3 13.2
4 14
}\datatableFinePlainFSIterationsA

\pgfplotstableread{
1 0
2 4
3 0
4 14
}\datatableFinePlainExtraIterationsA

\pgfplotstableread{
1 9.6
2 10
3 12.4
4 10
}\datatableFineLSFSIterationsA

\pgfplotstableread{
1 0
2 4
3 0
4 10
}\datatableFineLSExtraIterationsA

\pgfplotstableread{
1 9.8
2 10
3 10.2
4 10
}\datatableFineLSTwoFSIterationsA

\pgfplotstableread{
1 0
2 10.2
3 0
4 10
}\datatableFineLSTwoExtraIterationsA

\pgfplotstableread{
1  -10
2  -10 
3  -10
4  -10 
}\datatableFineEmptyA


\pgfplotstableread{
1  0
2  0
3  0
4  0
5  0
6  1000
7  1000
8  1000
}\datatableFinePlainNoConvergence

\pgfplotstableread{
1  0
2  0
3  0
4  0
5  0
6  0
7  1000
8  1000
}\datatableFineLSNoConvergence

\pgfplotstableread{
1  0
2  0
3  0
4  0
5  0
6  0
7  0
8  0
}\datatableFineLSTwoNoConvergence

\pgfplotstableread{
1 10.2
2 9.4
3 90.2
4 45.8
5 21.2
6  1000
7  1000
8  1000
}\datatableFinePlainFSIterations

\pgfplotstableread{
1 0
2 8
3 0
4 45.8
5 84.8
6  1000
7  1000
8  1000
}\datatableFinePlainExtraIterations

\pgfplotstableread{
1 9
2 7.4
3 64.2
4 40.8
5 18
6 22.4
7  1000
8  1000
}\datatableFineLSFSIterations

\pgfplotstableread{
1 0
2 8.8
3 0
4 40.8
5 72
6 0
7  1000
8  1000
}\datatableFineLSExtraIterations

\pgfplotstableread{
1 8
2 7.4
3 45.6
4 15.6
5 12.8
6 20.8
7 11.4
8 8.2
}\datatableFineLSTwoFSIterations

\pgfplotstableread{
1 0
2 12
3 0
4 15.6
5 51.2
6 0
7 11.4
8 32.8
}\datatableFineLSTwoExtraIterations

\pgfplotstableread{
8   -10
16  -10 
32  -10
64  -10 
128 -10 
6  -23
7  -23
8  -23
}\datatableFineEmpty


\pgfplotstableread{
1  0
2  0
3  0
4  0
5  0
6  5000
7  5000
8  5000
}\datatableFineTimePlainNoConvergence

\pgfplotstableread{
1  0
2  0
3  0
4  0
5  0
6  0
7  5000
8  5000
}\datatableFineTimeLSNoConvergence

\pgfplotstableread{
1  0
2  0
3  0
4  0
5  0
6  0
7  0
8  0
}\datatableFineTimeLSTwoNoConvergence

\pgfplotstableread{
1 2.406955e+03
2 3.865953e+03
3 1.409206e+03
4 1.127245e+03
5 1.093227e+03
6  1000
7  1000
8  1000
}\datatableFinePlainTime

\pgfplotstableread{
1 2.162559e+03
2 3.579376e+03
3 1.386549e+03
4 1.245197e+03
5 1.031258e+03
6 5.057925e+02
7  1000
8  1000
}\datatableFineLSTime

\pgfplotstableread{
1 1.979742e+03
2 4.345097e+03
3 1.224495e+03
4 6.497442e+02
5 1.060864e+03
6 5.745208e+02
7 4.874249e+02
8 6.959304e+02
}\datatableFineLSTwoTime

\pgfplotstableread{
8   -10
16  -10 
32  -10
64  -10 
128 -10 
6  -23
7  -23
8  -23
}\datatableFineTimeEmpty

\begin{figure} 
\includegraphics[width=\textwidth]{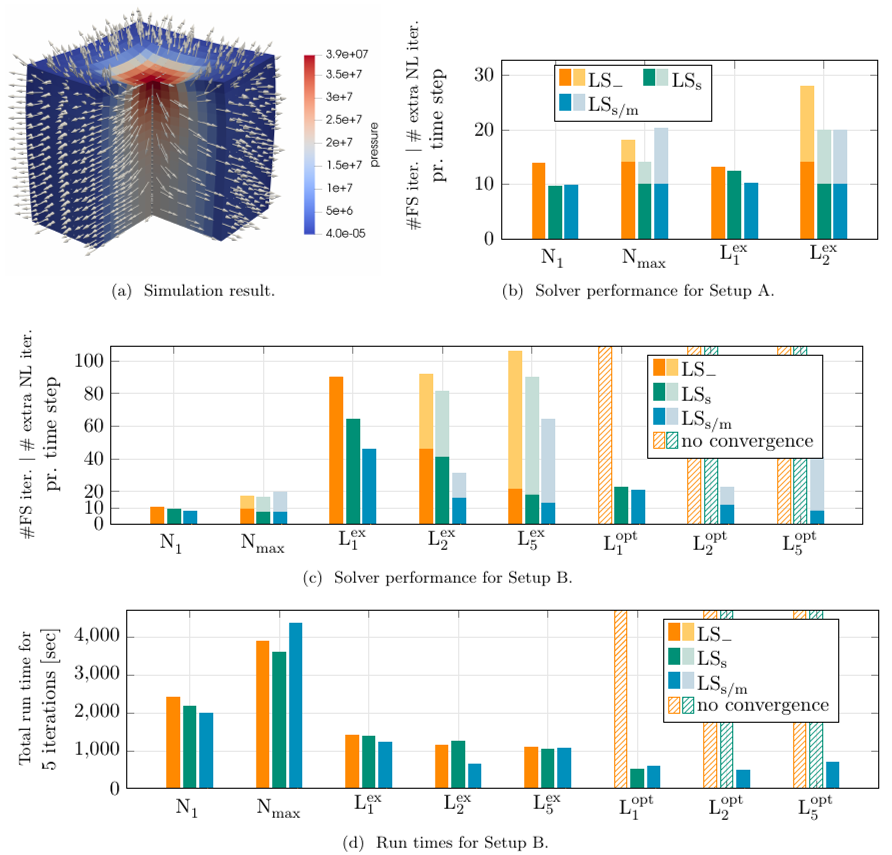}
\caption{\label{figure:testcase-iv} Test case IV: (a) Deformed, poro-elastic configuration after 5 time steps for $\nu=0.2$, incl. pressure profile and outflow; the deformation is scaled by factor 5. (b) and (c): Performance of different non-linear solvers (Tab.~\ref{table:numerical-example-non-linear-solvers}) combined with different relaxation strategies (Tab.~\ref{table:numerical-example-accelerations}), measured in average number of fixed-stress (FS) iterations and extra non-linear (NL) iterations per time step used for solving the mechanics problem, if more than one non-linear iterations per fixed-stress iteration is utilized; they are displayed on top of each other, illustrating the total amount of non-linear iterations required. (d) Total run times (incl.\ assembly and solver) for five time steps corresponding to (c).}
\end{figure}

\begin{itemize}
%
 
 \item[$\bullet$] For Setup A, the number of fixed-stress iterations per time step is approximately the same for all schemes, indicating a dominant coupling strength.
 
 \item[$\bullet$] For Setup B, the number of fixed-stress iterations per time step decreases for the Newton- and $L^\mathrm{ex}$-type methods; it increases for the more optimistic choice  $L^\mathrm{opt}$. Under relaxation on both levels, the iteration counts are practically constant for all methods. 
 
\end{itemize}

\noindent
We conclude, most importantly, inexact alternating minimization can outperform exact alternating minimization. The number of outer fixed-stress iterations might decrease the more accurately the non-linear problems are solved, but on the other hand, the total amount of required inner non-linear iterations increases much more. This makes relaxation by (inexact) line search attractive, which allows for improved solution of the non-linear subproblems and the overall performance of the splitting scheme without requiring to solve a linear system. We observe, relaxation does not only accelerate convergence but it also increases the robustness; similar effects have been previously observed for relaxation by Anderson acceleration for the fixed-stress split~\cite{Both2018b}. 

Finally, if applicable, simple linearizations as the L-scheme might outperform Newton-based linearization techniques. In particular, when combining them with relaxation. The main drawback of the L-scheme is that it includes tuning parameters. Optimal choices may lead to good performance, whereas bad choices might even lead to no convergence. Suitable choices being solution-dependent, makes the final choice rather difficult; however, line search may allow convergence for a wide range of parameters, potentially even faster than more conservative choices of the tuning parameters, for which the plain scheme converges. In the present example, by choosing an L-scheme-based fixed-stress split with optimistic tuning parameters and full line search relaxation, run times 1/8 of those for the non-relaxed, exact Newton-based fixed-stress split have been achieved. The finer the mesh the more drastic the difference as direct solvers are employed in this study.

\section{Concluding remarks and discussion}

The aim of the present work was to examine the inherent gradient flow structures of thermo-hydro-mechanical processes in porous media with focus on consequences for the well-posedness analysis and construction of numerical approximations and solvers. A major finding was that various, existing PDE models from the literature can be formulated as generalized gradient flows utilizing thermodynamic interpretation of energies and dissipation potentials -- for instance, linear poro-elasticity, linear poro-visco-elasticity, non-linear poro-elasticity in the infinitesimal strain regime, non-Newtonian Darcy and non-Darcy flows in poro-elastic media, and thermo-poro-elasticity without thermal convection. Moreover, well-posedness has been established for those models utilizing a unified framework introduced for doubly non-linear evolution equations.

One further significant finding to emerge from this work is that robust, physically based operator splitting schemes for time-discrete approximations are a consequence of a suitable \textit{choice} of primary variables (in fact dictated by the gradient flow structure) and a simple application of plain alternating minimization. Robustness is then an immediate consequence of the naturally underlying minimization structure of the semi-discrete problem arising from suitable time-discretization of gradient flows; in that light, e.g., the undrained and the fixed-stress splits appear to be the natural splitting schemes for linear poro-elasticity. Moreover, abstract convergence theory allows to quantify the energy decrease for each iteration of the splitting schemes only utilizing convexity and Lipschitz continuity properties of the problem -- a fairly simple machinery compared to previous analyses in the literature and also immediately applicable to heterogeneous, anisotropic materials. We derive novel splitting schemes and establish \textit{a priori} and \textit{a posteriori} convergence results in the context of linear poro-elasticity, linear poro-visco-elasticity and linear thermo-poro-elasticity.

The results of this work support the idea that splitting schemes for models with a vector structure ought to utilize tensorial stabilization instead of diagonal stabilization; such has been previously proposed either based on physical intuition or rather more ad hoc, cf., e.g.,~\cite{Kim2018a} in the context of thermo-poro-elasticity or~\cite{Hong2018} in the context of multiple-network poro-elastic theory. The latter has not been covered in this work, but it is essentially a generalization of linear thermo-poro-elasticity; our results can be immediately extrapolated.

Additionally, we highlight the known and simple fact that a minimization formulation enables relaxation of iterative solvers by line search strategies. Such have not been utilized before in the field of poro-elasticity. Our numerical experiments suggest that line search acts as black box optimization of the stabilization and possibly linearization in the context of optimized splitting schemes~\cite{Storvik2019}, which is especially practical for problems with changing geometries or boundary conditions.

Throughout the entire work, we utilize linear poro-elasticity as proof of concept and verify that the provided framework yields consistent results with the literature, but from a new perspective. After all, it seems promising for handling further models as also demonstrated for various extensions of linear poro-elasticity. 


The most important limitation lies in the fact that, evidently, not all thermo-hydro-mechanical processes are suitably modelled by gradient flows, e.g., convective-dominated processes, or materials with limit behavior as incompressible fluids or solids. However, at least in the context of the numerical solution, non-monotone perturbations of gradient flows may be discussed using operator splitting techniques as Strang's splitting or semi-implicit time-discretization, and limit cases may be handled employing duality theory. After all, the provided theory may still assist in various situations -- to what extent is topic of future research. Moreover, in this sense, interesting areas of applications and model extensions include finite strain poro-elasticity, poro-elasticity for fractured media, poroplasticity, and compositional and multi-phase flow in poro-elastic media. In terms of numerical solvers, for strongly non-linear and possibly non-convex problems, a further study could assess the need for more advanced optimization algorithms as primal-dual methods, alternating direction method of multipliers, or proximal alternating minimization for deriving robust linearization or non-linear preconditioners. This would be a fruitful area for further work.

\appendix

\section{Abstract well-posedness results}\label{appendix:classic-theorems} 

The theoretical results in this work are mostly deduced by application of abstract results from literature; we recall two results for doubly non-linear evolution equations and convex optimization.

The following well-posedness result for doubly non-linear evolution equation can be understood as a corollary or refined discussion of previous classical results, e.g.,~\cite{Colli1992}. The main improvement to previous results is a weaker regularity assumption on the external loading. This is compensated with stronger, structural assumptions on the functions spaces, as well as the dissipation potential and energy functional. Here, we consider an energy functional which does not explicitly depend on time. In order to incorporate time-dependent energy functionals, assumptions and proof techniques as, e.g., by~\cite{Mielke2013}, can be additionally applied.

\begin{theorem}[Well-posedness for doubly evolution equations with weakly regular load]\label{preliminaries:thm:well-posedness:gradient-flows}
 Consider the doubly non-linear evolution equation
 \begin{align}
 \label{preliminaries:gradient-flow-theorem:gradient-flow}
 \GRAD\Psi(\dot{x}(t)) + \GRAD \mathcal{E}(t,x(t)) = f(t) \text{ in }\mathcal{V}^\star \text{ a.e.\ in }(0,T);\ \ x(0)=x_0.
 \end{align}
 where 
 \begin{itemize}
 
  \item[$\bullet$] $p_\psi,p_\mathcal{E}\in(1,\infty)$; $p:=\mathrm{min}\,\{ p_\psi, p_\mathcal{E} \}$; $p^\star\in(1,\infty)$ such that $\tfrac{1}{p} + \tfrac{1}{p^\star} = 1$.

  \item[$\bullet$] $\mathcal{B}$ is a separable, reflexive Banach space with norm $\| \cdot \|_\mathcal{B}$.

  \item[$\bullet$] $\mathcal{V}$ is a separable, reflexive Banach space with a semi-norm $| \cdot |_\mathcal{V}$, such that
  \begin{align}
   \| x \|_{\mathcal{V}} := \left( \| x \|_\mathcal{B}^{p} + | x |_\mathcal{V}^{p} \right)^{1/p},\quad x\in\mathcal{V}.
  \end{align}
  defines a norm on $\mathcal{V}$. Furthermore, $\mathcal{V}$ is dense and compactly embedded in $\mathcal{B}$. 

  \item[$\bullet$] $\Psi:\mathcal{B} \rightarrow [0,\infty)$ is convex and continuously differentiable. There exists a constant $C>0$ such that
  \begin{align*}
   \Psi(x) \geq C \| x \|_{\mathcal{B}}^{p_\psi},\quad x\in\mathcal{B}.
  \end{align*}

  \item[$\bullet$] $\mathcal{E}:[0,T] \times \mathcal{V} \rightarrow [0,\infty)$, such that there exist constants $C_1>0$, $C_2\geq 0$, satisfying
  \begin{align*}\mathcal{E}(t,x) \geq C_1 |x|_\mathcal{V}^{p_\mathcal{E}} - C_2 \quad \text{for all }(t,x)\in[0,T]\times \mathcal{V}.
  \end{align*} 
  Furthermore, $\mathcal{E}(t,\cdot):\mathcal{V} \rightarrow (-\infty,\infty)$ is convex, lower-semicontinuous, and continuously differentiable for all $t\in[0,T]$; and $\mathcal{E}(\cdot,x):[0,T]\rightarrow (-\infty,\infty)$ is differentiable a.e.\ for all $x\in\mathcal{V}$ such that there exists a constant $C>0$, satisfying for a.e.\ $t\in(0,T)$
  \begin{align*}
   | \partial_t \mathcal{E}(t,x)| \leq C(1+\mathcal{E}(t,x))\quad \text{for all }x\in \mathcal{V}.
  \end{align*}

  \item[$\bullet$] $f\in C(0,T;\mathcal{V}^\star)\cap W^{1,p^\star}(0,T;\mathcal{V}^\star)$.
  
  \item[$\bullet$] $x_0\in \mathcal{V}$ such that $\mathcal{E}(0,x_0) < \infty$.
  
 \end{itemize}
 Then there exists a solution $x\in W^{1,p}(0,T;\mathcal{B})\cap L^\infty(0,T;\mathcal{V})$ of~\eqref{preliminaries:gradient-flow-theorem:gradient-flow}, satisfying $\mathcal{E}(x)\in L^\infty(0,T)$ and the energy identity
 \begin{align}
 \label{appendix:well-posedness-dne:energy-identity}
  &\int_0^T \Psi(\dot{x}(t))\,dt + \mathcal{E}(x(T)) - \llangle f(T),x(T) \rrangle \\
  \nonumber
  &\quad = \mathcal{E}(x_0) - \llangle f(0), x(0) \rrangle + \int_0^T \partial_t \mathcal{E}(t,x(t))\, dt - \int_0^T \llangle \dot{f}(t), x(t) \rrangle \, dt.
 \end{align}
 If $\GRAD\Psi$ or $\GRAD\mathcal{E}$ is linear and self-adjoint, the solution is unique.
\end{theorem}

\begin{proof}
 The proof is analogous to the proof of Thm.~1 by~\cite{Colli1992}, enhanced by discussions of the time-dependence of the energy functional by~\cite{Francfort2006}: First, the doubly non-linear evolution equation~\eqref{preliminaries:gradient-flow-theorem:gradient-flow} is discretized in time by consecutive convex minimization problems, and second, stability bounds are derived, and finally, compactness arguments are employed in order to pass to the limit, obtaining a solution to the time-continuous problem. Due to the weaker regularity assumptions on the load term, the second step of~\cite{Colli1992} is not applicable here. In the following, we derive stability for the time-discrete approximation under the above assumptions.
 
 As in~\cite{Colli1992,Francfort2006}, we use the minimizing movement scheme to discretize~\eqref{preliminaries:gradient-flow-theorem:gradient-flow} in time. Let $0=t_0<t_1<...<t_N=T$ of $[0,T]$ denote a partition of $[0,T]$ with constant time step size $\Delta t$. Set $x^0 = x(0)$ and define consecutively
 \begin{align*}
  x^n := \underset{x\in\mathcal{B}}{\mathrm{arg\,min}} \, \left\{\Delta t \, \Psi\left( \frac{x - x^{n-1}}{\Delta t} \right) + \tilde{\mathcal{E}}^n(x)\right\}
 \end{align*}
 where $f^n:= \tfrac{1}{\Delta t} \int_{t_{n-1}}^{t_n} f(t)\, dt$, and $\tilde{\mathcal{E}}^n :\mathcal{B} \rightarrow (-\infty,\infty]$ defined by
 \begin{align*}
  \tilde{\mathcal{E}}^n(x) = \left\{ \begin{array}{ll} \mathcal{E}(t_n,x) - \llangle f^n, x \rrangle & x \in \mathcal{V}, \\ \infty, &\mathrm{otherwise}, \end{array} \right.
 \end{align*}
 is a proper, convex, lower-semicontinuous function. By Thm.~\ref{appendix:well-posedness:convex-minimization}, $x^n$ is well-defined. Furthermore, for all $n$, it holds
 \begin{align*}
  \Delta t \, \Psi\left( \frac{x^n - x^{n-1}}{\Delta t} \right) + \tilde{\mathcal{E}}^n(x^n) \leq \tilde{\mathcal{E}}^n\left(x^{n-1}\right)
 \end{align*}
 and hence, by induction $x^n \in \mathcal{V}$ for all $n$, since $x^0\in\mathcal{V}$. Summing over all time steps, employing the definition of $\tilde{\mathcal{E}}$ and manipulating the sum over the load terms, yields
 \begin{align}
 \label{appendix:well-posedness-dne-aux:1}
  &\sum_n \Delta t \, \Psi\left( \frac{x^n - x^{n-1}}{\Delta t} \right) + \mathcal{E}(t_N,x^N) \\
  \nonumber
  &\quad\leq \mathcal{E}(0,x(0)) + \sum_n \int_{t_{n-1}}^{t_n} \partial_t \mathcal{E}(t,x^{n-1})\, dt \\
  \nonumber
  &\qquad+ \llangle f^0, x(0) \rrangle
  - \llangle f^N, x^N \rrangle 
  - \sum_n \Delta t \llangle \frac{f^n - f^{n-1}}{\Delta t}, x^{n-1} \rrangle.
 \end{align}
 As in~\cite{Francfort2006}, we employ the bound on $\partial_t\mathcal{E}$ together with a Gr\"onwall inequality and obtain
 \begin{align*}
  \sum_n \int_{t_{n-1}}^{t_n} \partial_t \mathcal{E}(t,x^{n-1})\, dt 
  \leq 
  C\left(1+\sum_n \Delta t \, \mathcal{E}(t_{n-1},x^{n-1}) \right),
 \end{align*}
 where $C>0$ depends on $T$ and the stability bound on $\partial_t\mathcal{E}$. Inserting into~\eqref{appendix:well-posedness-dne-aux:1}, and, furthermore, utilizing the assumptions on $\Psi$, $\mathcal{E}$ and $f$ yields for arbitrary $\delta >0$
 \begin{align*}
  &\sum_n \Delta t \left\| \frac{x^n - x^{n-1}}{\Delta t} \right\|_{\mathcal{B}}^{p_\psi} + \left| x^N \right|_\mathcal{V}^{p_\mathcal{E}} + \mathcal{E}(t_N,x^N) \\
  &\quad \leq C\left(1+\sum_n \Delta t \, \mathcal{E}(t_{n-1},x^{n-1}) \right) + \delta \left( 
  \left\| x^N \right\|_{\mathcal{V}}^{p} + \sum_n \Delta t \left\| x^{n-1} \right\|_{\mathcal{V}}^{p} 
  \right).
 \end{align*}
 where $C>0$ depends on $T$, $\delta$, the initial data, and regularity of the loading. Using Young's inequality and the definition of $\| \cdot \|_\mathrm{\mathcal{V}}$, it holds
 \begin{align*}
  &\sum_n \Delta t \left\| \frac{x^n - x^{n-1}}{\Delta t} \right\|_{\mathcal{B}}^{p_\psi} + \left| x^N \right|_\mathcal{V}^{p_\mathcal{E}} + \mathcal{E}(t_N,x^N) \\
  &\leq C\left(1+\sum_n \Delta t \, \mathcal{E}(t_{n-1},x^{n-1}) \right) \\
  &\qquad+ \delta \left(
  \left\| x^N \right\|_{\mathcal{B}}^{p_\psi} + \sum_n \Delta t \left\| x^{n-1} \right\|_{\mathcal{B}}^{p_\psi} + \left| x^N \right|_{\mathcal{V}}^{p_\mathcal{E}} + \sum_n \Delta t \left| x^{n-1} \right|_{\mathcal{V}}^{p_\mathcal{E}} 
  \right).
 \end{align*}
 By constructing a telescope sum, exploiting the convexity of $x\mapsto x^{p_\psi}$ and applying H\"older inequalities, we obtain
 \begin{align*}
  \left\| x^N \right\|_{\mathcal{B}}^{p_\psi} + \sum_n \Delta t \left\| x^{n-1} \right\|_{\mathcal{B}}^{p_\psi} 
  \leq 
  C \left(1 + \sum_n \Delta t \left\| \frac{x^n - x^{n-1}}{\Delta t} \right\|_{\mathcal{B}}^{p_\psi} \right)
 \end{align*}
 for $C>0$ depending on $p_\psi$, $x_0$ and $T$. Hence, for $\delta$ sufficiently small it holds
 \begin{align*}
  &\sum_n \Delta t \left\| \frac{x^n - x^{n-1}}{\Delta t} \right\|_{\mathcal{B}}^{p_\psi} + \left| x^N \right|_\mathcal{V}^{p_\mathcal{E}} + \mathcal{E}(t_N,x^N) \\
  &\quad\leq C\left(1+\sum_n \Delta t \, \mathcal{E}(t_{n-1},x^{n-1}) + \sum_n \Delta t \left| x^{n-1} \right|_{\mathcal{V}}^{p_\mathcal{E}} \right).
 \end{align*}
 Finally, by employing a Gr\"onwall inequality, we obtain uniform stability for the left hand side. Based on that, the proof can be continued along the lines of~\cite{Colli1992,Francfort2006}, utilizing compactness arguments in order to pass to the limit $\Delta t \rightarrow 0$ and obtain a solution to the time-continuous doubly non-linear evolution equation, that in particular satisfies the energy identity~\eqref{appendix:well-posedness-dne:energy-identity}.
\end{proof}

\begin{theorem}[Well-posedness for convex minimization~\cite{Ekeland1999}]\label{appendix:well-posedness:convex-minimization}
 Consider the problem
 \begin{align}
 \label{preliminaries:convex-minimization-problem}
  &\text{minimize }f(\x)\\
 \nonumber
  &\text{subject to }\x\in \mathcal{C}, 
 \end{align}
 where $f:\mathcal{X}\rightarrow \mathbb{R}$ is a proper, convex, lower semi-continuous function, and $\mathcal{C}\subset\mathcal{X}$ is non-empty, closed, convex subset of $\mathcal{X}$, a reflexive Banach space. If $\mathcal{C}$ is bounded or $f$ is coercive over $\mathcal{C}$, i.e., $f(x)\rightarrow \infty$ for $x\in\mathcal{C}$ with $\|x\|\rightarrow \infty$, then~\eqref{preliminaries:convex-minimization-problem} has a solution. It is unique if $f$ is strictly convex. 
\end{theorem}

\section{Alternating minimization for block-separable constrained convex minimization in infinitely dimensional Hilbert spaces}\label{appendix:section:alternating-minimization}

In~\cite{Beck2013}, the authors establish an abstract convergence result for alternating minimization, applied to a constrained, strongly convex minimization problem in finite dimensions. Furthermore, convexity and Lipschitz continuity are solely considered wrt.\ Euclidean norms. We generalize the abstract result, allowing for a constrained minimization problem in infinitely dimensional Hilbert spaces. Convexity and Lipschitz continuity are considered wrt.\ the semi-norms.

\paragraph{Hilbert space structure.}
Let $\mathcal{X} = \mathcal{X}_1 \times \mathcal{X}_2$ be a product of Hilbert spaces, equipped with an inner product $\llangle\cdot,\cdot\rrangle$. Assume it is induced by separate inner products $\llangle\cdot,\cdot\rrangle_{1}$ and $\llangle\cdot,\cdot\rrangle_{2}$ on $\mathcal{X}_1$ and $\mathcal{X}_2$, respectively, such that
\begin{align*}
 \llangle (x_1,x_2), (y_1,y_2) \rrangle = \llangle x_1, y_1 \rrangle_{1} + \llangle x_2,y_2 \rrangle_{2}, && (x_1,x_2),(y_1,y_2)\in\mathcal{X}_1\times\mathcal{X}_2.
\end{align*}
The inner product $\llangle\cdot,\cdot\rrangle$ acts naturally also as duality pairing on $\mathcal{X}^\star \times \mathcal{X}$. Additionally, let $|\cdot|_\star$ on $\mathcal{X}$ denote some semi-norm on $\mathcal{X}$.

\paragraph{Function properties.}
Let $f:\mathcal{X}\rightarrow \mathbb{R}$ be differentiable. We introduce two properties:
\begin{enumerate}[label=(\roman*)]
 \item We call $f$ \textit{strongly convex wrt.\ $|\cdot|_\star$} if there exists a constant $\sigma>0$ such that
 \begin{align}
 \label{appendix:strong-convexity}
  f(\bm{y}) \geq f(\bm{x}) + \llangle \GRAD f(\bm{x}),\bm{y} - \bm{x} \rrangle + \frac{\sigma}{2} | \bm{y} - \bm{x} |_\star^2,\quad \forall\bm{x},\bm{y}\in\mathcal{X}.
 \end{align}
 which is equivalent to (see, e.g.,~\cite{Nesterov2004})
 \begin{align*}
  \llangle \GRAD f(\bm{y}) - \GRAD f(\bm{x}),\bm{y} - \bm{x} \rrangle \geq \sigma | \bm{y} - \bm{x} |_\star^2,\quad \forall\bm{x},\bm{y}\in\mathcal{X}.
 \end{align*}

 \item We call the $k$-th block $\GRAD_k f$ of the gradient of $f$ \textit{Lipschitz continuous wrt.\ $|\cdot|_\star$} if there exists a constant $L_k<\infty$ such that for all $\bm{x}\in\mathcal{X}$ and $\bm{h}_k\in\mathcal{X}$, it holds
 \begin{align}
  \label{appendix:lipschitz-continuity}
  \llangle \GRAD_k f(\bm{x}+\bm{h}_k) - \GRAD_k f(\bm{x}), \bm{h}_k \rrangle \leq L_k | \bm{h}_k |_{\star,1}^2,
 \end{align}
 where $\bm{h}_1=(\tilde{h}_1,0)$ and $\bm{h}_2=(0,\tilde{h}_2)$ for some $\tilde{h}_k\in\mathcal{X}_k$, $k=1,2$. The condition~\eqref{appendix:lipschitz-continuity} is equivalent to (see, e.g.,~\cite{Nesterov2004})
 \begin{align}
  \label{appendix:lipschitz-continuity-2}
  f(\bm{x}+\bm{h}_k) \leq f(\bm{x}) + \llangle \GRAD_k f(\bm{x}), \bm{h}_k \rrangle + \frac{L_k}{2} | \bm{h}_k |_{\star,1}^2.
 \end{align}

\end{enumerate}

\paragraph{Alternating minimization.}
Let $\mathcal{X}$ as above and $f:\mathcal{X}\rightarrow \mathbb{R}$. Furthermore, let $\tilde{\mathcal{X}}_1\subset\mathcal{X}_1$ and $\tilde{\mathcal{X}}_2\subset\mathcal{X}_2$ be non-empty, convex subsets. We consider the constrained minimization problem
\begin{align}\label{appendix:minimization-problem}
 \underset{(x_1,x_2)\in\tilde{\mathcal{X}}_1\times\tilde{\mathcal{X}}_2}{\mathrm{inf}}\, f(x_1,x_2).
\end{align}
Under certain assumptions on $f$ and $|\cdot|_\star$, we can show global, linear convergence for alternating minimization, cf.\ Alg.~\ref{algorithm:alternating-minimization}.

\begin{algorithm}
 \caption{Single iteration of alternating minimization}
 \label{algorithm:alternating-minimization}
 \SetAlgoLined
 \DontPrintSemicolon
 
  \vspace{0.5em}
 
  Input: $\bm{x}^{i-1}=(x_1^{i-1},x_2^{i-1})\in\tilde{\mathcal{X}}_1\times\tilde{\mathcal{X}}_2$ \; \vspace{0.8em}
   
  Determine $x_1^{i} := \underset{x_1\in\tilde{\mathcal{X}}_1}{\mathrm{arg\,min}}\, f(x_1,x_2^{i-1})$\; \vspace{0.25em}

  Determine $x_2^{i} := \underset{x_2\in\tilde{\mathcal{X}}_2}{\mathrm{arg\,min}}\, f(x_1^{i},x_2)$\; \vspace{0.2em}
\end{algorithm}

\begin{lemma}[Linear convergence of alternating minimization]\label{appendix:lemma:alternating-minimization}
 Let $\tilde{\mathcal{X}}:=\tilde{\mathcal{X}}_1\times\tilde{\mathcal{X}}_2\subset\mathcal{X}$ as above, and let $f:\mathcal{X}\rightarrow \mathbb{R}$ be a differentiable function. Furthermore, let $|\cdot|_{\star,1}$ denote semi-norm on $\mathcal{X}$ satisfying:
 \begin{itemize}
  \item[$\bullet$] $|(x_1,x_2)|_{\star,1} \geq | (x_1,0) |_{\star,1}$ for all $(x_1,x_2)\in\mathcal{X}$,
  \item[$\bullet$] $f$ is strongly convex wrt.\ $|\cdot |_{\star,1}$ with constant $\sigma_1$,
  \item[$\bullet$] $\GRAD_1 f$ is Lipschitz continuous wrt.\ $|\cdot|_{\star,1}$ with Lipschitz constant $L_1$. 
 \end{itemize}
 Then the alternating minimization Alg.~\ref{algorithm:alternating-minimization} is globally, linearly convergent. In particular, let $(\bm{x}^i)_i\subset\tilde{\mathcal{X}}_1\times\tilde{\mathcal{X}}_2$ be the sequence generated by the alternating minimization Alg.~\ref{algorithm:alternating-minimization} for given initial value $\bm{x}^0\in\mathcal{X}$. And let $\bm{x}^\star\in\mathcal{X}$ denotes the unique solution of~\eqref{appendix:minimization-problem}. It holds for all $i\in\mathbb{N}$
 \begin{align*}
  f(\bm{x}^i) - f(\bm{x}^\star) &\leq \left(1 - \frac{\sigma_1}{L_1} \right) \,\left(f(\bm{x}^{i-1}) - f(\bm{x}^\star)\right) \leq \left( 1 - \frac{\sigma_1}{L_1} \right)^i \left( f(\bm{x}^0 ) - f(\bm{x}^\star)\right),\\
  f(\bm{x}^i) - f(\bm{x}^\star) &\leq \frac{L_1}{\sigma_1} \left( f(\bm{x}^{i-1}) - f(\bm{x}^i) \right).
 \end{align*}
 Assume there additionally exists a second semi-norm $|\cdot|_{\star,2}$ on $\mathcal{X}$ satisfying:
 \begin{itemize}
  \item[$\bullet$] $|(x_1,x_2)|_{\star,2} \geq | (0,x_2) |_{\star,2}$ for all $(x_1,x_2)\in\mathcal{X}$,
  \item[$\bullet$] $f$ is strongly convex wrt.\ $|\cdot |_{\star,2}$ with constant $\sigma_2$,
  \item[$\bullet$] $\GRAD_2 f$ is Lipschitz continuous wrt.\ $|\cdot|_{\star,2}$ with Lipschitz constant $L_2$. 
 \end{itemize}
 Then it holds for all $i\in\mathbb{N}$
 \begin{align*}
  f(\bm{x}^i) - f(\bm{x}^\star) &\leq \prod_{j=1}^2 \left(1 - \frac{\sigma_j}{L_j} \right)\,\left(f(\bm{x}^{i-1}) - f(\bm{x}^\star)\right) \leq \prod_{j=1}^2 \left( 1 - \frac{\sigma_j}{L_j} \right)^i \left( f(\bm{x}^0 ) - f(\bm{x}^\star)\right),\\
  f(\bm{x}^i) - f(\bm{x}^\star) &\leq \frac{\prod_{j=1}^2 \left(1 - \frac{\sigma_j}{L_j} \right)}{ 1 - \prod_{j=1}^2 \left(1 - \frac{\sigma_j}{L_j} \right) } \left( f(\bm{x}^{i-1}) - f(\bm{x}^{i}) \right).
 \end{align*}
\end{lemma}

\begin{proof}
 The proof follows the same line of argumentation as the proof of Theorem~5.2~\cite{Beck2013}, but carefully tailored to the more general setting used above.
 
 \paragraph{Consequence from strong convexity.}
  
 Consider~\eqref{appendix:strong-convexity} for $\bm{x}=\bm{x}^i$. Minimizing both sides wrt.\ $\bm{y}\in\tilde{\mathcal{X}}$ yields
 \begin{align}
  \label{appendix:proof-1:auxiliary-1}
    f(\bm{x}^i) - f(\bm{x}^\star) 
    &\leq - \underset{\bm{y}\in\tilde{\mathcal{X}}}{\mathrm{inf}} \left( \llangle \GRAD f(\bm{x}^i),\bm{y} - \bm{x}^i \rrangle + \frac{\sigma_1}{2} | \bm{y} - \bm{x}^i |_{\star,1}^2 \right).
 \end{align}
 By definition of alternating minimization it holds
 \begin{align*}
  \langle \GRAD_2 f(\bm{x}^i), y_2 - x_2^i \rangle_2 \geq 0\quad \forall y_2\in\tilde{\mathcal{X}}_2.
 \end{align*} 
 Hence, together with~(A1),~\eqref{appendix:proof-1:auxiliary-1} becomes
  \begin{align}
  \label{appendix:proof-1:auxiliary-1b}
    f(\bm{x}^i) - f(\bm{x}^\star) 
    &\leq - \underset{y_1\in\tilde{\mathcal{X}}_1}{\mathrm{inf}} \left( \llangle \GRAD_1 f(\bm{x}^i),y_1 - x_1^i \rrangle_1 + \frac{\sigma_1}{2} \left| (y_1 - x_1^i,0) \right|_{\star,1}^2 \right).
 \end{align}
 Since $\tilde{\mathcal{X}}_1$ is convex and $0<\sigma_1\leq L_1$ by definition, it holds for all $x_1\in\tilde{\mathcal{X}}_1$
\begin{align*}
  \mathcal{B}_{L_1/\sigma_1}(x_1) 
  := 
  \left\{ h \in \mathcal{X}_1 \, \left| \, x_1 + \frac{L_1}{\sigma_1} h \in \tilde{\mathcal{X}}_1 \right.\right\} 
  &\subset
  \left\{ h \in \mathcal{X}_1 \, \left| \, x_1 + h \in \tilde{\mathcal{X}}_1 \right.\right\} 
  =:
  \mathcal{B}_{1}(x_1)
 \end{align*}
 Hence, we obtain 
 \begin{align*}
  &\underset{y_1\in\tilde{\mathcal{X}}_1}{\mathrm{inf}} \left( \llangle \GRAD_1 f(\bm{x}^i),y_1 - x_1^i \rrangle_1 + \frac{\sigma_1}{2} \left| (y_1 - x_1^i,0) \right|_{\star,1}^2 \right) \\
  &\quad =
  \underset{h \in \mathcal{B}_{L_1/\sigma_1}(x_1^i)}{\mathrm{inf}} \left( \llangle \GRAD_1 f(\bm{x}^i),\frac{L_1}{\sigma_1} h \rrangle_1 + \frac{\sigma_1}{2} \left| \left( \frac{L_1}{\sigma_1} h,0\right) \right|_{\star,1}^2 \right) \\
  &\quad \geq 
  \frac{L_1}{\sigma_1} \, \underset{h \in \mathcal{B}_{1}(x_1^i)}{\mathrm{inf}} \left( \llangle \GRAD_1 f(\bm{x}^i),h \rrangle_1 + \frac{\sigma_1}{2} \left| (h,0) \right|_{\star,1}^2 \right).
 \end{align*}
 Altogether, it holds 
 \begin{align}
  \label{appendix:proof-1:auxiliary-1c}
      f(\bm{x}^i) - f(\bm{x}^\star) 
      \leq 
         & - \frac{L_1}{\sigma_1}\, \underset{x_1^i+h\in\tilde{\mathcal{X}}_1}{\mathrm{inf}} \left( \llangle \GRAD_1 f(\bm{x}^i), h \rrangle_1 + \frac{L_1}{2} \left| (h,0) \right|_{\star,1}^2 \right).
 \end{align}

 \paragraph{Consequence from Lipschitz continuity.}
 
 Consider~\eqref{appendix:lipschitz-continuity-2} for $\bm{x}=\bm{x}^i$. Minimizing both sides wrt.\ $\bm{h}_1=(h_1,0)$ such that $x_1^i+\bm{h}_1\in\tilde{\mathcal{X}}_1$ yields
 \begin{align}
  \label{appendix:proof-1:auxiliary-4}
  f(\bm{x}^i) - f(x_1^{i+1},x_2^i) 
  &\geq  - \underset{x_1^i+h_1\in\tilde{\mathcal{X}}_1}{\mathrm{inf}} \left( \llangle \GRAD_1 f(\bm{x}^i), h_1 \rrangle_1 + \frac{L_1}{2} \left| (h_1,0) \right|_{\star,1}^2 \right). 
 \end{align}

 \paragraph{Consequences for alternating minimization.}
 
 By putting together~\eqref{appendix:proof-1:auxiliary-1c} and~\eqref{appendix:proof-1:auxiliary-4}, and exploiting the definition of alternating minimization, we obtain the \textit{a posteriori} estimate
 \begin{align*}
  &f(\bm{x}^i) - f(\bm{x}^\star)  \\
  &\quad \leq \frac{L_1}{\sigma_1} \left( f(\bm{x}^i) - f(x_1^{i+1},x_2^i) \right) \\
  &\quad \leq \frac{L_1}{\sigma_1} \left( f(\bm{x}^i) - f(\bm{x}^{i+1}) \right).
 \end{align*}
 Adding and subtracting $f(\bm{x}^\star)$ on the right hand side and reordering terms, yields
 \begin{align*}
  f(\bm{x}^{i+1}) - f(\bm{x}^\star) \leq \left(1 - \frac{\sigma_1}{L_1} \right) \left( f(x_1^{i+1},x_2^i) - f(\bm{x}^\star) \right) \leq \left(1 - \frac{\sigma_1}{L_1} \right) \left( f(\bm{x}^i) - f(\bm{x}^\star) \right).
 \end{align*}
 The \textit{a priori} and \textit{a posteriori} results of the first part follow immediately. 
 
 The second part of the thesis is proved analogously with focus on the second step of the alternating minimization algorithm. By a symmetry argument it follows
 \begin{align*}
  f(x_1^{i+1},x_2^i) - f(\bm{x}^\star) \leq \frac{L_2}{\sigma} \left( f(x_1^{i+1},x_2^i) - f(\bm{x}^{i+1}) \right).
 \end{align*}
 Hence, we obtain the \textit{a priori} result
 \begin{align*}
  f(\bm{x}^{i+1}) - f(\bm{x}^\star) \leq \left(1 - \frac{\sigma_1}{L_1} \right) \left( f(x_1^{i+1},x_2^i) - f(\bm{x}^\star) \right) \leq \left(1 - \frac{\sigma_1}{L_1} \right) \left(1 - \frac{\sigma_2}{L_2} \right) \left( f(\bm{x}^i) - f(\bm{x}^\star) \right).
 \end{align*}
 Reformulation yields the \textit{a posteriori} result of the second part of the assertion.

\end{proof}

\newpage 

\section{Nomenclature}\label{appendix:notation}
\begin{table}[!h]
\begin{tabular}{ll}
 \multicolumn{2}{l}{\textbf{Space and  time}}\\
 $x$ & Spatial coordinate \\
 $t$ & Time \\
 $d$ & Space dimension \\
 $\Omega$ & Domain \\
 $\Gamma$ & Boundary of $\Omega$ \\
 $C_\Omega$ & Poincar\'e constant \\
 $T$ & Final time \\
 $\Delta t$ & Time increment\\
 $t_n$ & n-th time step \\[10pt]
 
 \multicolumn{2}{l}{\textbf{Generalized gradient flows}}\\
 $\mathcal{X}$ & State space \\
 $\mathcal{P}_{\dot{\mathcal{X}}}$ & Process space \\
 $\mathcal{E}$ & Free energy \\
 $\mathcal{D}$ & Dissipation potential \\
 $\mathcal{P}_\mathrm{ext}$ & External work rate \\[10pt]

 \multicolumn{2}{l}{\textbf{Physical fields}}\\
 $\u$ & Structural displacement \\
 $\eps{\u}$ & Linear strain / symmetric gradient of $\u$ \\
 $\stress$ & Total stress \\
 $\stress^\mathrm{d}$ & Deviatoric stress \\
 $\stress^\mathrm{h}$ & Hydrostatic stress \\
 $\stress_\mathrm{eff}$ & Effective stress\\
 $\bm{\zeta}$ & Rotation \\
 $\fluidmass$ & Fluid content \\
 $p$  & Fluid pressure \\
 $\flux$ & Volumetric flux\\
 $S$ & Entropy \\
 $\entropyflux$ & Entropy flux \\
 $\viscostrain$ & Visco-elastic strain \\
 $\flux_{\int},\ \entropyfluxint$ & Accumulated volumetric and entropy fluxes \\
  \end{tabular}
\end{table}
 
 \newpage

\begin{table}[!h]
\begin{tabular}{ll} 
 
 \multicolumn{2}{l}{\textbf{External sources}}\\
 $\masssource$ & Source for mass conservation \\
 $\entropysource$ & Entropy source \\
 $\intmasssource,\ \intentropysource$ & Accumulated mass and entropy sources\\
 $\fext$ & External body force acting on the bulk \\
 $\gext$ & External body force acting on the fluid \\
 $\u_\Gamma$ & Prescribed displacement \\
 $\stress_{\Gamma,\mathrm{n}}$ & Prescribed surface force onto boundary\\
 $p_\Gamma$ & Prescribed pressure \\
 $q_{\Gamma,\mathrm{n}}$ & Prescribed normal volumetric flux \\
 $T_\Gamma$ & prescribed temperature \\
 $j_{\Gamma,\mathrm{n}}$ & Prescribed normal entropy flux \\[10pt]

 \multicolumn{2}{l}{\textbf{Function spaces}}\\
 $\Omega$ & Porous medium \\
 $d$ & Spatial dimension \\
 $\mathcal{X}^\star$ & Dual space of some function space $\mathcal{X}$ \\
 $\mathcal{X}^n$ & Space $\mathcal{X}$ evaluated at time $t_n$ \\
 $\mathcal{X}_0$ & Tangent space of some function space $\mathcal{X}$\\
 $\mathcal{V}$ & Space for structural displacement \\
 $\mathcal{S}$ & Space for total stress including the balance of momentum \\
 $\tilde{\mathcal{S}}$ & Space for total stress without the balance of momentum \\
 $\bm{Q}_\mathrm{AS}$ & Space of skew-symmetric tensors in $\mathbb{R}^{d \times d}$\\
 $\mathcal{Q},\tilde{\mathcal{Q}}$ & Space for fluid pressure\\
 $\Pi_{\tilde{\mathcal{Q}}}$ & Orthogonal projection onto $\tilde{\mathcal{Q}}$ \\
 $\mathcal{Z}$ & Space for volumetric flux \\
 $\mathcal{Z}_{\int}$ & Space for accumulated flux\\
 $\mathcal{W}$ & Space for entropy flux \\
 $\mathcal{W}_{\int}$ & Space for accumulated entropy flux \\
 $\mathcal{T}$ & Space for visco-elastic strains\\[10pt]

 \multicolumn{2}{l}{\textbf{Material parameters}}\\
 $\mathbb{C}$, $\mathbb{C}_\mathrm{v}$, $\mathbb{C}_\mathrm{v}'$ & Stiffness tensors \\
 $\mathbb{A}$ & Compliance tensor \\
 $\mathcal{C}_\mathrm{v}$ & Generalized stiffness tensor\\
 $\mathcal{A}_\mathrm{v}$ & Generalized compliance tensor\\
 $\mu,\ \lambda$ & Lam\'e parameters \\
 $\mu_\mathrm{v},\ \lambda_\mathrm{v}$, $\mu_\mathrm{v}',\ \lambda_\mathrm{v}'$ & Visco-elasticity-specific Lam\'e parameters \\
 $E$ & Young's modulus \\
 $\nu$ & Poisson's ratio \\
 $K_\mathrm{dr}$ & Drained bulk modulus \\
 $W(\eps{\u})$ & Strain energy density \\
 $\tfrac{1}{M}$ & Storage coefficient \\
 $b(p)$ & Nonlinear compressibility \\
 $\alpha,\ \alpha_\mathrm{v},\ \alpha_\mathrm{T}$ & Biot coefficients\\
 $\alpha_\varphi$ & Thermo-hydro coupling coefficient \\
 $\permeability$ & Hydraulic conductivity / permeability \\
 $\conductivity$ & Thermal conductivity \\
 $\nu(|\flux|)$ & Fluid viscosity \\
 $C_\mathrm{d}$ & Total volumetric heat capacity \\
\end{tabular} 
\end{table}

\section*{Acknowledgements}
This work is supported in part by the Research Council of Norway Project
250223. The authors also acknowledge the support from the University of Bergen. 

\bibliographystyle{ieeetr}      

%
%

\end{document}